\documentclass[a4paper, reqno]{amsart}
\usepackage{amssymb, mathrsfs, bbm, mathtools}

\usepackage{hyperref}

\newtheorem{theorem}{Theorem}[section] 
\newtheorem{lemma}[theorem]{Lemma}     

\newtheorem{corollary}[theorem]{Corollary}
\newtheorem{proposition}[theorem]{Proposition}
\usepackage{dsfont}

\theoremstyle{definition}
\newtheorem{definition}[theorem]{Definition}
\newtheorem{example}[theorem]{Example}

\theoremstyle{remark}
\newtheorem{remark}[theorem]{Remark}

\newcommand{\tnu}{\widetilde{\nu}}
\newcommand{\vngl}{{\mathrm{VN}}_l(G)}
\newcommand{\vng}{{\mathrm{VN}}(G)}

\newcommand{\A}{\mathcal{A}}

\newcommand{\N}{\mathcal{N}}

\newcommand{\D}{\mathcal{D}}
\newcommand{\E}{\mathbb{E}}

\newcommand{\M}{\mathcal{M}}

\newcommand{\cM}{\mathfrak{M}}

\newcommand{\R}{\mathcal{R}}

\newcommand{\n}{\mathfrak{n}}

\newcommand{\re}{\mathrm{Re}}
\newcommand{\im}{\mathrm{Im}}
\newcommand{\bR}{{\mathbb{R}}}
\newcommand{\bC}{{\mathbb{C}}}
\newcommand{\bN}{{\mathbb{N}}}

\newcommand{\I}{{\mathds {1}}}
\newcommand{\clo}[1]{\langle{#1}\rangle}

\begin{document}

\title[Reduction theorem and applications]
{An extension of Haagerup's reduction theorem with applications to subdiagonal subalgebras of general von Neumann algebras}

\author{Louis Labuschagne}
\address{DSI-NRF CoE in Math. and Stat. Sci,\\ Focus Area for Pure and Applied Analytics,\\ Internal Box 209, School of Math.\& Stat. Sci.\\NWU, PVT. BAG X6001, 2520 Potchefstroom\\ South Africa}
\email{Louis.Labuschagne@nwu.ac.za}

\author{Quanhua Xu}
\address{Institute for Advanced Study in Mathematics, Harbin Institute of Technology,  Harbin 150001, China; and
Laboratoire de Math{\'e}matiques, Universit{\'e} Marie \& Louis Pasteur, 25030 Besan\c{c}on Cedex, France}
\email{qxu@univ-fcomte.fr}

\date{\today}

\begin{abstract} We revisit Haagerup's enigmatic reduction theorem \cite[Theorems 2.1 \& 3.1]{HJX} showing how that theorem may be extended to general von Neumann algebras $\M$ equipped with an arbitrary faithful normal semifinite weight in a manner which faithfully captures the essence of the original. In contrast to the proposal in \cite[Remark 2.8]{HJX}, we show how in the non-$\sigma$-finite case the enlargement $\R=\M\rtimes\mathbb{Q}_D$ of $\M$ may be approximated by an increasing \emph{sequence} of \emph{expected} semifinite subalgebras. Using this revised version of the reduction theorem we may then all the applications of this theorem to $H^p$-spaces from the $\sigma$-finite case to general von Neumann algebras. Inspired by the theory of topologically ordered groups we then propose the even more general concept of approximately subdiagonal subalgebras which proves to be general enough to contain all group theoretic examples. This then forms the context for much of the study of Fredholm Toeplitz operators in the closing sections.
\end{abstract}

\maketitle

\tableofcontents

\section{Introduction}\label{S1}

In this introduction we focus on elucidating the role played by the reduction theorem and various other approximation techniques in developing the theory of noncommutative $H^p$ spaces, rather 
than trying to give a complete description of the historical development of the theory of noncommutative $H^p$ spaces. A survey of this historical development will be presented at a later stage once the appropriate context has been adequately justified.  
Although only formally published in 2010 \cite{HJX}, Haagerup's reduction theorem first saw the light of day as a handwritten note date marked 15 May 1978, which for many decades was circulated 
among a select few individuals. Haagerup's intention with this theorem was to set up a framework within which a $\sigma$-finite von Neumann algebra $\M$ could be indirectly approximated with a 
monotonically increasing sequence of finite von Neumann algebras equipped with faithful normal tracial states. More specifically this theorem provides a framework within which one can first enlarge 
a given $\sigma$-finite von Neumann algebra $\M$ to a von Neumann superalgebra $\R$ (still $\sigma$-finite) inside of which $\M$ appears as an expected subalgebra. It is then this enlarged 
algebra $\R$ rather than $\M$ itself, that Haagerup showed can be approximated with a sequence $\{\R_n\}$ of finite von Neumann subalgebras equipped with faithful normal tracial states. 

Haagerup's hope was that in the development of the theory of noncommutative $L^p$-spaces this theorem could be used to extend results regarding noncommutative $L^p$-spaces from the easier to handle 
tracial setting, to the more enigmatic setting of $\sigma$-finite von Neumann algebras. It was with the theory of noncommutative $H^p$-spaces that this vision was quite strongly realised. The 
fountainhead of the theory of noncommutative $H^p$-spaces are the maximal subdiagonal subalgebras of Arveson introduced in his remarkable seminal paper ``Analyticity in Operator Algebras'' 
\cite{AIOA}. Ultimately it is these maximal subdiagonal subalgebras which in this theory make up the category of noncommutative $H^\infty$-spaces. The theory initially focused on finite von Neumann 
algebras equipped with faithful normal tracial states developing at a somewhat slow, but carefully measured, pace. The first systematic description of noncommutative $H^p$ spaces of finite maximal 
subdiagonal subalgebras was given by Marsalli and West \cite{MW}. The verification of Szeg\"o's formula for this context in 2005 (see \cite{L-Szego}) then enabled Blecher and Labuschagne to build on 
this earlier pioneering work and develop a very detailed and successful theory of noncommutative $H^p$-spaces in the context of finite von Neumann algebras (see \cite{BLsurvey}). At almost the 
same time that Blecher and Labuschagne started their development, Xu pioneered the application of the reduction theorem to subdiagonal subalgebras of $\sigma$-finite von Neumann algebras in the 2005 
paper \cite{Xu}. This paper contains a ``no-frills'' version of the reduction theorem for $\sigma$-finite algebras for which Xu then shows how in this setting a subdiagonal subalgebra $\A$ of such a von Neumann 
algebra may be enlarged to a subdiagonal subalgebra $\widehat{\A}$ of the enlargement $\R$ of $\M$ in such a way that $\A$ appears as a non-self-adjoint expected subalgebra of $\widehat{A}$, with 
$\widehat{A}$ in turn being monotonically approximable by a sequence $\{\A_n\}$ of subdiagonal subalgebras of the $\R_n$s. This structure then provides a very powerful formalism for extending large 
parts of the achieved theory for finite von Neumann algebras to the setting of subdiagonal subalgebras of $\sigma$-finite von Neumann algebras. Ji demonstrated the power of the technique developed 
by Xu by using it to lift a large number of blocks of theory regarding noncommutative $H^p$-spaces from the finite to the $\sigma$-finite context (see \cite{jig2, jig3}). Subsequently Labuschagne 
also successfully used this technique to make further inroads into the $\sigma$-finite theory \cite{L-HpIII}.

When Haagerup's reduction theorem did finally appear in print \cite{HJX}, the authors did present a mechanism for lifting Haagerup's reduction theorem to general von Neumann algebras 
\cite[Remark 2.8]{HJX}. The essence of this idea runs as follows: One first notes that any von Neumann algebra $\M$ may be written in the form $\M=\oplus_{j\in J}\,\N_j\overline{\otimes}B(K_j)$ 
where each $\N_j$ is a $\sigma$-finite von Neumann algebra. If we are now given a faithful normal semifinite weight on $\M$ of the form $\oplus_{j\in J}(\nu_j\otimes\mathrm{Tr}_j)$ where each 
$\nu_j$ is a faithful normal state on $\N_j$, we may construct a net of approximating algebras by first applying Haagerup's reduction theorem to each $\N_j$ to produce a superalgebra of the form 
$\oplus_{j\in J}\,\R_j\overline{\otimes}B(K_j)$ (where each $\R_j$ is an appropriate enlargement of $\N_j$), and then using the very regular structure of each $B(K_j)$ to build a net of approximating 
subalgebras for $\oplus_{j\in J}\,\R_j\overline{\otimes}B(K_j)$ from the subalgebras approximating each $\R_j$. However there are some problems with this strategy: 
\begin{itemize}
\item It only works for certain weights on $\M$.
\item The approximating subalgebras are neither monotonic nor a sequence.
\item The approximating subalgebras of the enlargement of $\M$ are no longer all expected unital subalgebras. 
\end{itemize}

One of the main objectives of this paper is to prove a version of Haagerup's reduction theorem which is valid for arbitrary von Neumann algebras equipped with some a priori given faithful normal 
semifinite weight, but much closer in spirit to Haagerup's original theorem than the framework posited by the approach mentioned above. In particular we will show how an arbitrary von Neumann algebra 
$\M$ equipped with an arbitrary faithful normal semifinite weight $\nu$, may be enlarged to a superalgebra $\R$ inside of which $\M$ appears as an expected subalgebra, with the enlarged algebra $\R$ 
in this case being approximable by a monotonically increasing sequence of \emph{semifinite} von Neumann subalgebras each equipped with a faithful normal semifinite trace. 

This then raises the question of the efficacy of this extended reduction theorem for lifting results regarding $H^p$-spaces from the setting originally considered by Marsalli and West \cite{MW} to 
general von Neumann algebras. For this task a two-step procedure seems to present itself. In the paper \cite{Bek-sem}, Bekjan established a number technical lemmata (summarised in Proposition 
\ref{Bek-red}), which over the years have proven to be extremely useful in lifting results regarding $H^p$-spaces from the setting of finite von Neumann algebras, to semifinite algebras. The basic 
idea is to in the case where $\M$ is semifinite and the restriction of the faithful normal semifinite trace $\tau$ on $\M$ to the self-adjoint portion $\D$ of $\A$ also semifinite, to select a net of finite projections 
$(e_\alpha)$ in $\D$ which increases to $\I$ and to then show that the compressions $e_\alpha\A e_\alpha$ form a net of (finite) maximal subdiagonal subalgebras which can be used to approximate 
$\A$. As far as lifting results to the setting of general von Neumann algebras is concerned, this procedure forms the first phase of the two-step procedure. In the first step one uses the procedure 
to lift results regarding noncommutative $H^p$ spaces from the setting of finite to semifinite algebras, whilst in the second, one then uses the extended version of the reduction theorem to then 
further lift results from the semifinite to the general setting. As will be demonstrated, this approach proves to be very successful in developing a theory for general von Neumann algebras.

However despite the success noted above, there is one further obstacle that needs to be overcome before a truly complete theory of noncommutative Hardy spaces can be achieved. A challenge which lies 
beyond the reach of the reduction theorem. The ``standard'' approach to lifting results regarding subdiagonal subalgebras requires the reference weight $\nu$ of the von Neumann superalgebra $\M$ to be semifinite on the 
self-adjoint portion $\D$ of the subdiagonal subalgebra $\A$. It is in part this fact that ensures the applicability of the reduction theorem in studying these objects. Whilst this restriction does 
ensure access to the reduction theorem, it also excludes Hardy space of the upper half-plane. In order to achieve a truly complete theory which includes Hardy space of the upper halfplane, we need 
to come up with a more general concept of subdiagonality which canonically contains all former concepts and which makes room for objects like Hardy space of the upper halfplane. This objective is 
achieved by delving deeply into the theory of topologically ordered groups, the structure of which we use to identify a formalism which bears generalisation to general von Neumann algebras. The 
outcome is the notion of approximately subdiagonal subalgebras of von Neumann algebras, which not only canonically contains former concepts of subdiagonality, but which also allows for the 
possibility of the reference weight of the containing von Neumann algebra not being semifinite on the self-adjoint portion of the given approximately subdiagonal algebra. We believe that the Hardy 
spaces obtained for such subalgebras are at present the most general noncommutative Hardy spaces in print.

To be able to handle the proof of the reduction theorem and its applications to $H^p$-spaces with confidence, a deep understanding of the behaviour of faithful normal conditional expectations with 
regard to noncommutative $L^p$-spaces is required. Any reader familiar with this theory will be well aware of the useful and important contributions of Junge and Xu in this regard, who proved a set 
of very useful technical results on this point valid for $\sigma$-finite von Neumann algebras (see for example \cite[Proposition 2.3]{JX-BurkRosen}). However the proofs of Junge and Xu do not 
directly extend to the general case. The authors are moreover not aware of any references where the validity of their results for the general case is explicitly verified. So whilst it may be a kind of 
``folk-theorem'' that these results do extend to the general case, explicit verification of this belief is required. This verification is then the starting point of our analysis.

\section{Preliminaries}\label{S2}

Throughout the symbols $\M$, $\N$ and $\R$ will be used to denote concrete von Neumann algebras.  For each $a\in \M$, the real part of $a$, namely $\frac{1}{2}(a+a^*)$ will be denoted by $\mathrm{Re}(a)$ and imaginary part $\frac{1}{2i}(a-a^*)$ by $\mathrm{Im}(a)$. For a subset $K$ of $\M$ we shall denote the respective collections of real and imaginary parts of the elements of $K$ by 
$\mathrm{Re}(K)$ and $\mathrm{Im}(K)$. Each von Neumann algebra of course admits of a faithful normal semifinite weight with one such weight 
being used as a reference weight. The phrase \emph{faithful normal semifinite} will occasionally be compressed to just fns. Given $\M$, we shall consistently use the symbol $\nu$ to denote the 
reference weight of $\M$. Now let $\N$ be a von Neumann subalgebra of $\M$. We shall follow standard practice by setting $\mathfrak{n}_\nu=\{a\in\M:\nu(a^*a)<\infty\}$, 
$\mathfrak{p}_\nu=\{a\in\M_+:\nu(a)<\infty\}$, and $\mathfrak{m}_\nu=\mathrm{span}(\mathfrak{p}_\nu)$. However given a subalgebra $\A$ of $\M$, we shall write $\mathfrak{n}_\nu(\A)$ for 
$\{a\in\A:\nu(a^*a)<\infty\}$. If $\N$ is a von Neumann subalgebra, we shall additionally write $\mathfrak{p}_\nu(\N)$ and $\mathfrak{m}_\nu(\N)$ for $\{a\in\N_+:\nu(a)<\infty\}$, and 
$\mathrm{span}(\mathfrak{p}_\nu(\N))$ respectively.

To proceed with our analysis we shall require a bit of insight into the GNS-construction for an \emph{fns} weight. Recall that when a von Neumann algebra $\M$ equipped with a faithful normal state $\omega$ is identified with the GNS-representation engendered by $\omega$, the state $\omega$ then becomes a vector state corresponding to a cyclic and separating vector $\Omega$. The vector $\Omega$ is then in fact cyclic and separating for both $\M$ and $\M'$ (See \cite[Propositions 2.5.3 and 2.5.6]{BR}). In this context one may then develop modular theory by defining antilinear operators $S_0$ and $F_0$ on the dense subspaces $\{a\Omega\colon a\in \M\}$ and $\{a'\Omega\colon a'\in \M'\}$ of $H$ by means of the prescriptions $$S_0(a\Omega)=a^*\Omega, \quad F_0(a'\Omega)=a'^*\Omega$$where $a\in \M$ and $a'\in\M'$, and proceeding from there. These operators turn out to be closable with their closures $S$ and $F$ turning out to be adjoints of each other. In the polar decomposition $J\Delta^{1/2}$ of $S$ the positive operator $\Delta$ is referred to as the \emph{modular operator} and the anti-unitary operator $J$ as \emph{modular conjugation}. The prescription 
$\sigma_t^\omega:a\mapsto \Delta_\omega^{it}a\Delta_\omega^{-it}$ can then be shown to induce a strongly-continuous one-parameter automorphism group (the so-called \emph{modular automorphism group}) on $\M$. 

\begin{remark}\label{5:R left Hilbert} In the case where we have a faithful normal semifinite weight $\nu$ rather than a state, the Hilbert space $H_\nu$ in the GNS-construction for the pair $(\M,\nu)$, is constructed by equipping $\mathfrak{n}_\nu$ with an inner product defined by $\langle x,y\rangle =\nu(y^*x)$ $(x,y\in \mathfrak{n}_\nu)$. Equipped with this structure $\n_\nu$ becomes a pre-Hilbert space, with the Hilbert space $H_\nu$ then defined to be its completion.

Despite the absence of a cyclic and separating vector, one may nevertheless still develop a modular theory that closely rivals that of the $\sigma$-finite setting. The primary ingredient one needs to 
develop modular theory in this setting is a dense subspace of $H_\nu$ which admits an involutive structure which can be used to define the operators $S$ and $F$. The subspace $\eta(\n_\nu\cap\n_\nu^*)$ 
turns out to be just such a subspace. We may specifically define the operator $S_0$ on this subspace by means of the prescription $S_0:\eta(a)\mapsto\eta(a^*)$. This operator extends to a closed 
densely-defined anti-linear operator $S$. The modular operator $\Delta$ is then $\Delta=|S|^2$ with the modular conjugation $J$ the anti-linear isometry in the polar decomposition $S=J\Delta^{1/2}$. 
(The above facts can be seen by considering the discussion preceding \cite[Lemma VI.1.4]{Tak2} alongside \cite[Lemma VI.1.5]{Tak2} and \cite[Theorems VII.2.5 \& VII.2.6]{Tak2}.)
\end{remark}

Let $\M$ as above be equipped with an \emph{fns} weight $\nu$. The so-called \emph{centralizer} of the pair $(\M,\nu)$ is defined to be the subalgebra $\M_\nu=\{a\in\M\colon  \sigma_t^\nu(a)=a\mbox{ for all }t\in\mathbb{R}\}$ of fixed points of the modular group. The use of the term centralizer is justified by the fact that the conditions 
\begin{itemize}
\item $a\mathfrak{m}_\nu\subseteq\mathfrak{m}_\nu$ and $\mathfrak{m}_\nu a\subseteq\mathfrak{m}_\nu$,
\item and that $\nu(ax)=\nu(xa)$ for all $x\in\mathfrak{m}_\nu.$
\end{itemize}
are both necessary and sufficient for $a\in \M$ to belong to the centralizer.

An element $a\in \M$ is then said to be $\sigma^\nu_t$-analytic (or just analytic) if there exists a strip $S_\gamma= \{z\in \mathbb{C}\colon  |\Im(z)| <\gamma\}$ in $\mathbb{C}$, and a function $F:S_\gamma\to \M$ such that
\begin{itemize}
\item $F(t)=\sigma_t(a)$ for each $t\in \mathbb{R}$,
\item with $z\mapsto \rho(F(z))$ analytic for every $\rho\in \M_*.$
\end{itemize}
In such a case we write $\sigma_z(a)$ for $F(z).$ If $F$ even extends to an entire-analytic function, we say that $a\in \M$ is \emph{entire-analytic}. We remind the reader that the set of 
entire-analytic elements of $\M$ (denoted by $\M^a_\sigma$) is a $\sigma$-weakly dense *-subalgebra of $\M$ \cite[Theorem VIII.2.3]{Tak2}. In our analysis we will only ever make use of $\M^a_\sigma$. We 
shall in the ensuing analysis therefore consistently use the term \emph{analytic} to refer to entire-analytic functions.

By the term \emph{normal conditional expectation}, we understand a 
unital normal order-preserving contractive map $\mathbb{E}:\M\to\N$ which additionally satisfies the criterion that $\mathbb{E}(afb) = a\mathbb{E}(f)b$ for all $a,b\in\N$ and all $f\in \M$. If for any $f\in\M_+$ we have that $\mathbb{E}(f)\neq 0$ precisely when 
$f\neq 0$, $\mathbb{E}$ is said to be faithful. A category of faithful normal conditional expectations we shall be particularly interested in, is the category of normal conditional expectations $\mathbb{E}:\M\to\N$ satisfying $\nu\circ \mathbb{E}=\nu$ where 
$\nu{\upharpoonright}\N$ is assumed to still be semifinite.

Using the notion of the extended positive part $\widehat{\M}_+$ of a von Neumann algebra $\M$ (see \S IX.4 of \cite{Tak2}) one may introduce a category of expectation like operators called \emph{Operator Valued Weights}. Given a von Neumann subalgebra $\N$ of $\M$, 
an Operator Valued Weight $\mathscr{W}$ from $\M$ to $\N$ is an operator $\mathscr{W}:\M_+\to\widehat{\N}_+$ which preserves suprema, and additionally satisfies
\begin{itemize}
\item $\mathscr{W}(\kappa f)=\kappa\mathscr{W}(f)$,
\item $\mathscr{W}(f+g)=\mathscr{W}(f)+\mathscr{W}(g)$,
\item and $\mathscr{W}(a^*fa)=a^*\mathscr{W}(f)a$
\end{itemize}
for all $f,g \in \M_+$, all $a\in \N$ and all non-negative scalars $\kappa$.

The concept of a crossed product of $\M$ with an LCA (locally compact abelian) group $G$ is particularly important for this paper. (We shall be particularly interested in the discrete group of dyadic rationals.) We pause to summarise the essentials - further details may be found in for 
example \cite{vD}. Suppose that $\M$ acts on the Hilbert space $H$ with $G$ a locally compact abelian group and $\alpha$ an action of the group on $\M$, by which we mean a point to $\sigma$-weakly continuous mapping $\alpha$ from $G$ into the group of 
$*$-automorphisms on $\M$, which respects the group action in the sense that $\alpha_s\circ\alpha_t=\alpha_{st}$. Now let $L^2(G,H)$ be the space of square Bochner-integrable functions from $G$ to $H$. The space $L^2(G,H)$ may of course be written as the tensor product
$H\otimes L^2(G)$. For every $a\in\M$, the prescription 
\begin{equation}\label{def-piembed} \pi_\alpha(a)(x\otimes f)(s)=\sigma_{s^{-1}}(a)(f(s)x) \qquad x\in H, f\in L^2(G)
\end{equation}
is well-defined on the simple tensors, and extends to a bounded map on $L^2(G,H)$. Specifically the map 
$\pi_\alpha:\M\to B(L^2(G,H)):a\mapsto \pi_\alpha(a)$ $(a\in \M)$ yields a faithful normal representation of $\M$ as a subalgebra of $B(L^2(G,H))$ \cite[Part I: Proposition 2.5]{vD}. For every $g\in G$, we may now define shift operators $\lambda_g \in B(L^2(G,H))$ by 
the prescription $\lambda_g(\xi)(s)=\xi(g^{-1}s)$ where $\xi\in L^2(G,H)$. One then defines the crossed product of $\M$ with the group action of $G$, to be the von Neumann algebra on $L^2(G,H)$ generated by $\pi_\alpha(\M)$ and the translation operators $\lambda_g$ 
($g\in G$). We will denote this von Neumann algebra by $\M\rtimes_\alpha G$. It is of interest to note that the unitary group of shift operators is strongly continuous, and realises the group action in the sense that $\lambda_g\pi(a)\lambda_g^*=\pi(\sigma_g(a))$ for any 
$g\in G$ and any $a\in \M$ \cite[Part I: Proposition 2.8 \& Lemma 2.9]{vD}. 

The dual group $\widehat{G}$ of $G$ moreover induces a dual action $\widehat{\alpha}$ on $\M\rtimes_\alpha G$ uniquely characterised by the prescriptions 
\begin{equation}\label{5:eqn dual} 
\widehat{\alpha}_\gamma(\pi(a))=\pi(a) \mbox{ and }\widehat{\alpha}_\gamma(\lambda_g)= \overline{\gamma(g)}\lambda_g \mbox{ for each } a\in\M\mbox{ and }g\in G.
\end{equation}

It is well-known that this dual action characterises the elements of $\pi_\alpha(\M)$ inside $\M\rtimes_\alpha G$ in the sense that as a subspace of $\M \rtimes_\alpha G$, the algebra $\pi_\alpha(\M)$ corresponds to the fixed points of $\widehat{\alpha}$ \cite[Lemma 3.6]{haag-DW2}. 
This fact may now be exploited to define an operator valued weight $\mathscr{W}_G$ from $\M \rtimes_\alpha G$ to $\pi_\alpha(\M)$ by means of the prescription $\mathscr{W}_G(a)=\int_{\widehat{G}}\widehat{\alpha}_\gamma(a)\,d\gamma$ where 
$a\in(\mathcal{M} \rtimes_\alpha G)_+$. Given any fns weight $\nu$ on $\M$, we now define the dual weight $\tnu$ of $\nu$ on  $\M\rtimes_\alpha G$ by means of the formula $\tnu=\widehat{\nu}\circ \widehat{\pi^{-1}} \circ \mathscr{W}_G$, where $\widehat{\nu}$ denotes the extension of $\nu$ to $\widehat{\M}_+$. This dual weight 
turns out to once again be an fns weight (see \cite{haag-DW2} for details).

We pass to giving some brief background to noncommutative $L^p$-spaces. The theory is particularly rich and elegant in the case where the reference weight $\nu$ on $\M$ is tracial (that is $\nu(a^*a)=\nu(aa^*)$ for each $a\in \M$). In this case it is common practice to write $\tau$ for the reference weight rather than $\nu$. In the tracial case the algebra $\M$ enlarges to the so-called algebra of $\tau$-measurable operators $\widetilde{\M}$ which is defined to be the set of all densely defined closed operators $f$ affiliated to $\M$ which satisfy the requirement that for any $\epsilon>0$ there exists a projection $p\in \M$ for which $p(H)\subset \mathrm{dom}(f)$ (thus $fp$ is then actually in $\M$ by the Closed Graph Theorem) and $\tau(\I-p)\leq \epsilon$. The enlargement $\widetilde{\M}$ turns out to be a complete metrisable algebra which is large enough to contain all the $L^p$-spaces. Given any $1\leq p<\infty$ the space $L^p(\M,\tau)$ is then simply defined to be $L^p(\M,\tau)=\{f\in \widetilde{\M}: \tau(|f|^p)<\infty\}$, with the norm on $L^p(\M,\tau)$ given by $\|f\|_p= \tau(|f|^p)^{1/p}$ (see \cite[Chapter I]{terp}). In fact such is the elegance of this setting, that the theory extends to even admit a very rich and detailed theory of rearrangement invariant Banach function spaces, the study of which has attracted a large number of highly talented and brilliant scholars from across the globe. See \cite{dP} for a recent survey of this topic.

The construction of $L^p$-spaces for general (possibly non-tracial) von Neumann algebras is quite a bit more complicated. Over the years several (ultimately equivalent) approaches have been developed, with the most widely used approach being that of Haagerup which is based on the theory of crossed products. Given a von Neumann algebra $\M$ equipped with an fns weight $\nu$, Haagerup's approach was to use the modular automorphism group $\sigma_t^\nu$ ($t\in \mathbb{R}$) engendered by $\nu$ to construct the crossed product $\M\rtimes_\nu\mathbb{R}$. In this particular case the group of automorphisms $(\theta_s)$ on $\M\rtimes_\nu\mathbb{R}$ realising the dual action takes the form
\begin{equation}\label{5:eqn dual-R} 
\theta_s(\pi(a))=\pi(a) \mbox{ and }\theta_s(\lambda_t)= e^{-ist}\lambda_t \mbox{ for each } a\in\M\mbox{ and }t,s\in \mathbb{R}.
\end{equation}\label{eq:dualmodgp}
With $\tnu$ denoting the dual weight on $\M\rtimes_\nu\mathbb{R}$, it is moreover possible to show that (see \cite[Theorem 4.7]{haag-OV1} and \cite[page 342]{haag-OV2}) in this case 
\begin{equation}\sigma_t^{\tnu}(f)=\lambda_tf\lambda_t^*\mbox{ with } \sigma_t^{\tnu}(\pi(a))=\pi(\sigma_t^\nu(a))\mbox{ for all }f\in(\M\rtimes_\nu\mathbb{R})\mbox{ and all }a\in \M.
\end{equation} 
By Stone's theorem there exists a nonsingular densely defined positive operator $h$ affiliated to $\M\rtimes_\nu\mathbb{R}$ for which we have that $\lambda_t=h^{it}$ for all $t\in \mathbb{R}$. It then follows from \cite[Theorem VIII.3.14]{Tak2} and its proof that 
$\M\rtimes_\nu\mathbb{R}$ is in fact semifinite, with the prescription $\tau(\cdot)=\tnu(h^{-1}\cdot)$ yielding an \emph{fns} trace on $\M\rtimes_\nu\mathbb{R}$ for which we have that $\tau\circ\theta_s=e^{-s}\tau$ for all $s\in \mathbb{R}$ \cite[Lemma 5.2]{haag-OV2}. 
Note that by construction $h=\frac{d\tnu}{d\tau}$ \cite{PT}. The above facts not only ensure that $\cM=\M\rtimes_\nu\mathbb{R}$ admits an enlargement to an algebra $\widetilde{\cM}$ of $\tau$-measurable operators, but also that each $\theta_s$ extends continuously to 
this enlarged algebra \cite[Proposition 4.7]{L-Expo}. For each $1\leq p<\infty$ the Haagerup $L^p$-space is then defined to be the space $L^p(\M)=\{a\in \widetilde{\cM}: \theta_s(a)=e^{-s/p}a\mbox{ for all }s\in\mathbb{R}\}$. The space $L^1(\M)$ admits a so-called trace 
functional $tr$, which can be used to realise the norm on $L^p(\M)$ by means of the prescription $\|a\|_p=tr(|a|^p)^{1/p}$ \cite[Definitions II.13 \& II.14]{terp}.
We shall where convenient denote the norm closure of a subset $\mathcal{U}$ of $L^p(\M)$ by $[\mathcal{U}]_p$.

The theory of standard forms of a von Neumann algebra, and in particular the Haagerup-Terp standard form will play an important role in our analysis (see \cite{haag-stdfm}). We therefore pause to review this theory. 

\begin{definition}\label{7:D defstdfrm} Given a von Neumann algebra $\M$ equipped with a faithful normal semifinite weight $\nu$, a quadruple $(\pi_0(\M), H_0, J, \mathscr{P})$ where $\pi_0$ is a faithful representation of $\M$ on the Hilbert space $H_0$, $J:H_0\to H_0$ anti-linear isometric involution, and $\mathscr{P}$ a self-dual cone of $H_0$, is said to be a standard form of $\M$ if the following conditions hold:
\begin{itemize}
\item $J\M J=\M'$ (the commutant of $\M$), 
\item $JzJ=z^*$ for all $z$ in the centre of $\M$,
\item $J\xi=\xi$ for all $\xi\in \mathscr{P}$,
\item $a(JaJ)\mathscr{P}\subset \mathscr{P}$ for all $a\in \M$.
\end{itemize} 
(Recall that when we say that $\mathscr{P}$ is a self-dual cone, we mean that $\xi\in \mathscr{P}$ if and only if $\langle\xi,\zeta\rangle\geq 0$ for all $\zeta\in \mathscr{P}$.)   
\end{definition} 

\begin{definition} We define left $\lambda$ and right $\rho$ actions of $\M$ on $L^2(\M)$, by the prescriptions 
$$\lambda(a)f=af \qquad f\in L^2(\M),$$
$$\rho(a)f=fa \qquad f\in L^2(\M).$$
\end{definition} 

\begin{theorem}[Haagerup-Terp standard form]\label{7:T stdform} \phantom{qwerty}
\begin{enumerate}
\item[(1)] $\lambda$ is faithful normal *-representation, and $\rho$ a faithful normal *-anti-represen\-tation of $\M$ on the Hilbert space $L^2(\M)$.
\item[(2)] For all $a\in \M$ we have that $J\lambda(a)J=\rho(a^*)$ and $J\rho(a)J=\lambda(a^*)$, where $J$ denotes the anti-linear isometric involution $f\mapsto f^*$ on $L^2(\M)$.
\item[(3)] The von Neumann algebras $\lambda(\M)$ and $\rho(\M)$ are commutants of each other, with $\rho(\M)=J\lambda(\M)J$.
\item[(4)] The quadruple $(\lambda(\M), L^2(\M), J, L^2_+(\M))$ is a standard form of $\M$ in the sense of Definition \ref{7:D defstdfrm}.
\end{enumerate}
In the above standard form the embedding $\n_\nu\to H_\nu: a\mapsto\eta(a)$ corresponds to the map $\mathfrak{j}^{(2)}\colon\n_\nu\to L^2(\M): a\mapsto [ah^{1/2}]$ where $h=\frac{d\widetilde{\nu}}{d\tau}$.
\end{theorem}

The final claim above may be seen by noting that by Remark \ref{7:R trwt} we have that $\nu(b_0^*b_1)=tr((h^{1/2}b_0^*)[b_1h^{1/2}])=tr([b_0h^{1/2}]^*[b_1h^{1/2}])$ for all $b_0,b_1\in\n_\nu$.

\section{A review of conditional expectations on Haagerup $L^p$-spaces}\label{S3}

In our analysis we shall repeatedly have occasion to consider the action of faithful normal conditional expectations on Haagerup $L^p$-spaces. For this we shall need to be familiar with the useful and important technical results of Junge and Xu in this regard (see for example \cite[Proposition 2.3]{JX-BurkRosen}). However as pointed out earlier, most of their proofs are only valid for the case of $\sigma$-finite algebras. Whilst it is surely a known folk theorem that their results extend to the general case, we shall nevertheless pause to explicitly verify this contention. To achieve the same results in the general setting, the proof ideas of Junge and Xu require some non-trivial modifications at various points. In particular to prove these results in the general case, we shall need a deeper understanding of how in this case $\mathfrak{n}_\nu$ and $\mathfrak{m}_\nu$ embed into $L^p(\M)$. We therefore collate some technical results from \cite{GL2}, before surveying the necessary facts about expectations. In the following we shall consistently denote the minimal closure of a closable operator $f$ by $[f]$.

\begin{proposition}\label{GL2-2.2+3}
\begin{itemize}
\item Let $q\in [2,\infty)$. If $a\in \mathfrak{n}_\nu$ then $ah^{1/q}$ is closable with $[ah^{1/q}]$, $h^{1/q}a^{*}\in L^{q}(\M)$. \cite[Proposition 2.2]{GL2}
\item Let $a, b \in \mathfrak{n}_\nu$ and $r_{i},s_{i}\in [2,\infty)$ be given with $r_{1}^{-1}+s_{1}^{-1}=r_{2}^{-1}+s_{2}^{-1}$. 
Then (see \cite[Proposition 2.3]{GL2})
\begin{equation*}
([ah^{1/r_{1}}](h^{1/s_{1}}b^{*}))=([ah^{1/r_{2}}](h^{1/s_{2}}b^{*})).
\end{equation*}
\end{itemize}
\end{proposition}

\begin{definition}\label{7:D defn j-embed} For $q\in [2,\infty)$ define the map 
\begin{equation*}
\mathfrak{j}^{(q)}: \mathfrak{n}_\nu\ni a\mapsto \left[ah^{1/q}\right]\in L^q(\M).
\end{equation*}

\medskip

For $p\in [1,\infty)$, define the map 
\begin{equation*}
\mathfrak{i}^{(p)}: \mathfrak{p}_\nu\ni a\mapsto \mathfrak{j}^{(2p)}(a^{1/2})^{*} \mathfrak{j}^{(2p)}(a^{1/2})\in L^p(\M)
\end{equation*}
\end{definition}

\begin{proposition}\label{7:P j-embed}
\begin{itemize}
\item For $q\in [2,\infty)$, each of the maps $\mathfrak{j}^{(q)}:\mathfrak{n}_\nu\to L^q(\M):a\mapsto [ah^{1/q}]$ is linear and injective. \cite[Proposition 2.7]{GL2}
\item For $p\in[1,\infty)$, we may define a map $\mathfrak{i}^{(p)}:\mathfrak{m}_\nu\to L^p(\M)$ by setting $\mathfrak{i}^{(p)}(b^*a)= \mathfrak{j}^{(2p)}(b)^*\mathfrak{j}^{(2p)}(a)$ and extending by linearity to all of $\mathfrak{m}_\nu$. This map is a well-defined linear, injective and positivity-preserving map. \cite[Proposition 2.8 \& Lemma 2.9]{GL2}
\end{itemize}
\end{proposition}

\begin{remark}\label{7:R trwt} It is worth noting that for any $x\in \mathfrak{m}_\nu$ we have that $\nu(x)=tr(\mathfrak{i}^{(1)}(x))$. See \cite[Proposition 2.13(a)]{GL2}. More generally we will for any $x, y \in \mathfrak{m}_\nu$ have that $tr(\mathfrak{i}^{(q)}(x)\mathfrak{i}^{(p)}(y))=tr(x\mathfrak{i}^{(1)}(y))$ where $p,q \geq 1$ are such that $1=\frac{1}{p}+\frac{1}{q}$. See \cite[Proposition 2.10]{GL2}.
\end{remark}

\begin{lemma}[{\cite[Theorem 2.4 \& Lemma 2.5]{GL2}}]\label{GL2-2.4+5}
Let $\mathfrak{n}_\nu^\infty$ be the collection of all analytic elements $a$ of $\mathfrak{n}_\nu\cap\mathfrak{n}_\nu^*$ for which $\sigma_w^\nu(a)\in\mathfrak{n}_\nu\cap\mathfrak{n}_\nu^*$ for all $w\in\bC$. Then $\mathfrak{n}_\nu^\infty$ and $\mathfrak{m}_\nu^\infty=\mathrm{span}\{b^*a\colon a,b \in \mathfrak{n}_\nu^\infty\}$ are $\sigma$-strongly dense in $\M$. Also for each $q\geq 2$, $\{[ah^{1/q}]: a\in \mathfrak{n}_\nu^\infty\}$ is dense in $L^q(\M)$. Moreover given $z\in\mathbb{C}$ with $0\leq \re(z) \leq 1/2$, we have that 
$$[a h^{z}] = h^{z}\sigma _{iz}(a) \quad \mbox{ for all }a\in \mathfrak{n}_{\infty}$$where $h=\frac{d\widetilde{\nu}}{d\tau}$.
\end{lemma}

The following Proposition now easily follows from the above lemma and the manner in which the embeddings $\mathfrak{i}^{(p)}$ and $\mathfrak{j}^{(q)}$ have been defined. 

\begin{proposition}[Proposition 2.11, \cite{GL2}]\label{P density1}
Let $\mathfrak{n}_\nu^\infty$  and $\mathfrak{m}_\nu^\infty$ be as before. Then the following holds:
\begin{enumerate}
\item[(1)] For any $q\in [2,\infty)$, $\{\mathfrak{j}^{(q)}(a): a\in \mathfrak{n}_\nu^\infty\}$ is dense in $L^{q}(\M)$.
\item[(2)] For any $p\in [1,\infty)$, $\{\mathfrak{i}^{(p)}(a): a\in \mathfrak{m}_\nu^\infty\}$ is dense in $L^{p}(\M)$.
\end{enumerate}
\end{proposition}

We shall also have need of the following fact:

\begin{proposition}\label{analytic-n(A)}
Let $\mathcal{B}$ be a $\sigma$-weakly closed unital subalgebra of $\M$ for which we have that $\sigma_t^\nu(\mathcal{B})=\mathcal{B}$. Writing 
$\mathfrak{n}_\nu(\mathcal{B})$ for $\mathfrak{n}_\nu\cap\mathcal{B}$, we will have that $\mathfrak{n}_\nu^\infty(\mathcal{B})$ is $\sigma$-weakly 
dense in $\mathcal{B}$ whenever $\mathfrak{n}_\nu(\mathcal{B})\cap \mathfrak{n}_\nu(\mathcal{B}^*)^*$ is $\sigma$-weakly dense. 
\end{proposition}

\begin{proof} For any $a\in \mathfrak{n}_\nu(\mathcal{B})\cap \mathfrak{n}_\nu(\mathcal{B}^*)^*$ we set 
$$a_n=\sqrt{\frac{n}{\pi}}\int_{-\infty}^\infty\sigma_t^\nu(a)e^{-n t^2}\,dt.$$The fact that the modular group preserves both $\mathcal{B}$ and $\nu$, 
ensures that we will have that $a_n\in \mathfrak{n}_\nu(\mathcal{B})\cap \mathfrak{n}_\nu(\mathcal{B}^*)^*$. It now follows from 
\cite[Proposition 2.5.22]{BR} that $(a_n)$ converges $\sigma$-weakly to $a$ and is analytic. It remains to show that 
$(\sigma^\nu_z(a_n))\subset \mathfrak{n}_\nu(\mathcal{B})\cap \mathfrak{n}_\nu(\mathcal{B}^*)^*$ for each $z\in\mathbb{C}$.  

It is clear from the proof of \cite[Proposition 2.5.22]{BR} that the values $\sigma_z^\nu(a_n)$ are given by the formula 
$\sigma_z^\nu(a_n)=\sqrt{\frac{n}{\pi}}\int_{-\infty}^\infty\sigma_t^\nu(a)e^{-n (t-z)^2}\,dt$. This then enables us to conclude 
that $$\|\sigma_z^\nu(a_n)\|\leq\sqrt{\frac{n}{\pi}}\int_{-\infty}^\infty\|\sigma_t^\nu(a)\||e^{-n(t-z)^2}|\,dt$$ 
$$\leq\|a\|\sqrt{\frac{n}{\pi}}\int_{-\infty}^\infty|e^{-n(t-z)^2}|\,dt= \|a\|e^{n(\im(z))^2}.$$

The quick way to see that $\sigma_z^\nu(a_n)\in \mathfrak{n}_\nu\cap \mathfrak{n}_\nu^*$ for each $z\in \mathbb{C}$, is to appeal to the technology of left Hilbert algebras. 
The connection of $\mathfrak{n}_\nu\cap\mathfrak{n}_\nu^*$ to left Hilbert algebras may be found in for example Theorem VII.2.6 of \cite{Tak2}. The fact that 
$\sigma_z^\nu(a_n)\in \mathfrak{n}_\nu\cap \mathfrak{n}_\nu^*$ for each $z\in \mathbb{C}$, then follows from for example \cite[Corollary, p272]{SZ}. The verification of this fact 
is also embedded in the proof of \cite[Theorem VI.2.1(i)]{Tak2} (see p 25 of that reference). For the sake of the reader we provide the skeleton of a direct proof of this fact. 
Firstly let $R>0$ be given and let $$S_N=\sqrt{\frac{n}{\pi}} \sum_{k=1}^N e^{-n (\widetilde{t}_k-z)^2}\sigma_{\widetilde{t}_k}^\nu(a)\Delta t_k$$be a Riemann-sum of 
$\sqrt{\frac{n}{\pi}}\int_{-R}^R\sigma_t^\nu(a)e^{-n(t-z)^2}\,dt$. Next recall that in its action on $\n_\nu$, $\nu$ satisfies a Cauchy-Schwarz inequality. If we combine this 
fact with the fact that $\nu\circ\sigma^\nu_t=\nu$ for all $t\in\bR$, then for any $s,t\in \bR$, we will have that 
$$|\nu(\sigma_s^\nu(a)\sigma_t^\nu(a))|\leq\nu(|\sigma_s^\nu(a)|^2)^{1/2}.\nu(|\sigma_t^\nu(a)|^2)^{1/2}=\nu(|a|^2)<\infty.$$One may then use this fact to see that
$$\nu(S_N^*S_N)\leq \frac{n}{\pi}\left(\sum_{k=1}^N |e^{-n(\widetilde{t}_k-z)^2}|\Delta t_k\right)^2\nu(|a|^2).$$Taking the limit yields 
$$\frac{n}{\pi}\nu\left(\big|\int_{-R}^R\sigma_t^\nu(a)e^{-n(t-z)^2}\,dt\big|^2\right)\leq \frac{n}{\pi}\left(\int_{-R}^R|e^{-n(t-z)^2}|\,dt\right)^2\nu(|a|^2).$$Now let 
$R\to\infty$, to see that $$\nu(|\sigma_z^\nu(a_n)|^2)\leq \frac{n}{\pi}\left(\int_{-\infty}^\infty|e^{-n(t-z)^2}|\,dt\right)^2\nu(|a|^2)=(e^{n(\im(z))^2})^2\nu(|a|^2)<\infty.$$
\end{proof}

We pause to compare the crossed product of an expected algebra with the crossed product of the ambient superalgebra.

\begin{remark}\label{10:R cond-exp}
Let $\M$ be a von Neumann algebra equipped with a faithful normal semifinite weight $\nu$, and $\N$ a von Neumann subalgebra for which 
$\nu{\upharpoonright}\N$ is a semifinite weight on $\N$. In this setting the requirement $\sigma_t^\nu(\N)=\N$ for all $t\in\bR$ is equivalent to the 
existence of a normal conditional expectation $\mathbb{E}$ from $\M$ onto $\N$ for which $\nu\circ\mathbb{E}=\nu$ 
\cite[Theorem IX.4.2]{Tak2}. When $\N$ is equipped with the restriction of the weight $\nu$ to $\N$, the associated modular group will 
therefore be the restriction of the modular group of $\M$ \cite[Lemma IX.4.21]{Tak2}. We pause to highlight the implications of the theory developed thus far for 
normal conditional expectations of this form. All facts pointed out here are in some sense implicit in the theory presented in sections 4 and 5 of \cite{HJX}. However the importance of conditional expectations justifies a focused discussion of this nature. 

We firstly note that since both von Neumann algebras act on the same Hilbert space with the modular group of $\N$ being nothing more 
than the restriction of the modular group of $\M$, when their modular groups are used to represent them as algebras acting on 
$L^2(\bR,H)$, 
\begin{itemize}
\item $\pi_\nu(\N)$ will appear as a subalgebra of $\pi_\nu(\M)$ (see equation (\ref{def-piembed}));
\item and they will by definition share the same left-shift operators.
\end{itemize}
Hence $\mathfrak{N}=\N\rtimes_\nu\bR$ will therefore be a subspace of $\cM=\M\rtimes_\nu\bR$ for which the modular automorphism group 
of the dual weight is a restriction of the modular automorphism group of $\cM$ since both modular groups are implemented by the same set of shift operators - see equation (\ref{eq:dualmodgp}). Stone's theorem guarantees the existence of a positive non-singular operator $h$ such that $\lambda_t=h^{it}$ for each $t\in \mathbb{R}$. We may now compare the proofs of \cite[Theorem VIII.3.14]{Tak2} and \cite[Lemma 5.2]{haag-OV2} to see that the canonical traces on $\mathfrak{N}$ and $\cM$ may respectively be defined by $\tnu_{\N}(h^{-1}\cdot)$ and $\tnu(h^{-1}\cdot)$. So by construction $\frac{d\tnu}{d\tau_{\cM}}=\frac{d\tnu_{\N}}{d\tau_{\mathfrak{\N}}}$. Thus by construction the canonical trace on $\mathfrak{N}$ will then be a restriction of the trace on $\cM$, which in turn ensures that 
$\widetilde{\mathfrak{N}}$ is a subspace of $\widetilde{\mathfrak{M}}$. It is also clear from the description of the dual action given in equation (\ref{5:eqn dual-R}), that the dual action $(\theta_t^{\N})$ corresponding to $\mathfrak{N}$ is just the restriction of the dual action of $\cM$. By equations (6) and (9) in \cite[Chapter II]{terp}, this then further ensures that the dual weight on $\mathfrak{N}$ is just the 
restriction of the dual weight on $\cM$. Given any $a\in\mathfrak{p}(\N)_\nu$ we will then by Remark \ref{7:R trwt} have that 
$tr_{\N}(\mathfrak{i}^{(1)}(a))=\nu(a)= tr_{\M}(\mathfrak{i}^{(1)}(a))$. By Proposition \ref{P density1} equality will then hold on all of $L^1(\N)$. So $tr_{\N}$ is clearly nothing more than a restriction of $tr_{\M}$. Thus 
one may safely just write $tr$. However more is true. We know from \cite[Remark 5.6]{HJX} that for each $1\leq p <\infty$, the prescription given there (namely to define $\E_{p}$ on the dense subspace $\mathfrak{i}^{(p)}(\mathfrak{m}_\nu)$ of $L^p$ by the formula $\E_{p}(\mathfrak{i}^{(p)}(x))=\mathfrak{i}^{(p)}(\E(x))$ and extend by continuity), yields a contractive map $\mathbb{E}_{p}$ on $L^p(\M)$, for which we have that 
$\mathbb{E}_{p}\circ\mathbb{E}_{p}= \mathbb{E}_{p}$ and $\mathbb{E}_{p}(f^*)=\mathbb{E}_{p}(f)^*$. Since 
$\nu\circ \mathbb{E}=\nu$, it is clear that $\mathbb{E}(\mathfrak{p}_\nu)=\mathfrak{p}(\N)_\nu$. So for any $b\in \mathfrak{p}_\nu$, we will also have that $tr(\mathbb{E}_{1}(\mathfrak{i}^{(1)}(b)))= tr(\mathfrak{i}^{(1)}(\mathbb{E}(b)))=\nu(\mathbb{E}(b))= \nu(a)=tr(\mathfrak{i}^{(1)}(b))$. The density of $\mathrm{span}(\mathfrak{i}^{(1)}(\mathfrak{p}_\nu))$ in $L^1(\M)$, then ensures that 
$tr(\mathbb{E}_{1}(b))=tr(b)$ for each $b\in L^1(\M)$.

It is now clear from the above discussion that in this setting each $L^p(\N)$ will by definition (see \cite[Lemma 4.10]{LM}) 
be a subspace of $L^p(\M)$, and also that for any $0<p<\infty$ each $L^p(\N)$ is a subspace of $L^p(\M)$. Moreover when the form of the quasinorm 
on $L^\Psi$ (see \cite[Proposition 3.11]{L-Orliczgen}) is considered alongside \cite[Remark 2.3]{FK}, it is clear that the quasinorm on $L^p(\N)$ is just 
a restriction of the quasinorm on $L^p(\M)$. Now \cite[Remark 5.6]{HJX} ensures that $\mathbb{E}$ will 
for any $1\leq p<\infty$, canonically induce a contractive map from $L^p(\M)$ to $L^p(\N)$. In fact for any Orlicz space $L^p$ with 
upper fundamental index strictly less than 1, \cite[Proposition 4.9 \& Theorem 4.11]{LM}, show that these spaces can be equivalently renormed in such a manner that $\mathbb{E}$ will also here induce a contractive map from $L^p(\M)$ to $L^p(\N)$.
\end{remark}

It still remains to show that \cite[Proposition 2.3]{JX-BurkRosen} extends to the general case. For this we need to work a bit harder. 

\begin{proposition}\label{exp-props} Let $\N$ be a von Neumann subalgebra of $\M$ for which there exists a faithful normal conditional expectation $\mathbb{E}$ from $\M$ onto $\N$ satisfying $\nu\circ\mathbb{E}=\nu$. Given $1\leq p,q,r\leq \infty$ such that $\frac{1}{p}+\frac{1}{q}=\frac{1}{r}$ we will for $a\in L^p(\N)$ and $b\in L^q(\M)$ have that $\mathbb{E}_{r}(ab)=a\mathbb{E}_{q}(b)$ and $\mathbb{E}_{r}(ba)=\mathbb{E}_{q}(b)a$.
\end{proposition}

\begin{proof} Let $f\in\mathfrak{n}_\nu^\infty$ and $a,b\in \mathfrak{n}(\N)_\nu^\infty$ be given. We shall prove the claim for the case where $1\leq p,q<\infty$. The first step in the proof is to show that $\mathbb{E}$ preserves analyticity. We pause to note that as far as automorphism groups are concerned, analyticity may be defined either in terms of the automorphism group \cite[Definition 2.5.20]{BR}, or the infinitesimal generator of this group \cite[Definition 3.1.17]{BR}, with the two definitions ultimately being equivalent. (See the discussion following \cite[Definition 3.1.17]{BR} in that reference.) We shall here prefer the second of these definitions. Let $\delta$ be the infinitesimal generator of the modular group $\sigma_t^\nu$ and let $a\in \mathrm{dom}(\delta)$. By definition that means that the expressions $\frac{1}{t}(\sigma_t^\nu(a)-a)$ converge $\sigma$-weakly to $\delta(a)$ as $t\to 0$. But since $\mathbb{E}$ is normal and also commutes with $\sigma_t^\nu$, we have that $\frac{1}{t}\mathbb{E}(\sigma_t^\nu(a)-a)=\frac{1}{t}(\sigma_t^\nu(\mathbb{E}(a))-\mathbb{E}(a))$ converges 
$\sigma$-weakly to $\mathbb{E}(\delta(a))$. So by definition $\mathbb{E}(a)\in\mathrm{dom}(\delta)$ with $\delta(\mathbb{E}(a))= \mathbb{E}(\delta(a))$. The claim is now obviously a fairly direct consequence of \cite[Definition 3.1.17]{BR}.

With $h$ denoting $\frac{d\widetilde{\nu}}{d\tau}$, repeated applications of Proposition \ref{GL2-2.2+3} and Lemma \ref{GL2-2.4+5}, now show that
\begin{eqnarray*}  
&&(h^{1/2p}a)[bh^{1/2p}]\mathbb{E}_{q}(\mathfrak{i}^{(q)}(f^*f))\\
&=&(h^{1/2p}a)[bh^{1/2p}]\mathfrak{i}^{(q)}(\mathbb{E}(f^*f))\\
&=&(h^{1/2p}a)[bh^{1/2p}](h^{1/2q}\sqrt{\mathbb{E}(f^*f)})[\sqrt{\mathbb{E}(f^*f)}h^{1/2q}]\\
&=&(h^{1/2p}a)[bh^{1/2r}]\sqrt{\mathbb{E}(f^*f)}[\sqrt{\mathbb{E}(f^*f)}h^{1/2q}]\\
&=&(h^{1/2p}a)[bh^{1/2r}][\mathbb{E}(f^*f)h^{1/2q}]\\
&=&(h^{1/2p}a)[bh^{1/2r}](h^{1/2q}\sigma_{i/2q}(\mathbb{E}(f^*f)))\\
&=&(h^{1/2p}a)[bh^{1/2q}][\sigma_{-i/2r}(\sigma_{i/2q}(\mathbb{E}(f^*f)))h^{1/2r}]\\
&=&(h^{1/2p}a)[bh^{1/2q}](h^{1/2r}\sigma_{i/2q}(\mathbb{E}(f^*f)))\\
&=&[\sigma_{-i/2p}(a)h^{1/2p}](h^{1/2q}\sigma_{i/2q}(b))[\sigma_{-i/2p}(\mathbb{E}(f^*f))h^{1/2r}]\\
&=&[\sigma_{-i/2p}(a)h^{1/2r}]\sigma_{i/2q}(b)[\sigma_{-i/2p}(\mathbb{E}(f^*f))h^{1/2r}]\\
&=&(h^{1/2r}\sigma_{i/2r}(\sigma_{-i/2p}(a)))\sigma_{i/2q}(b)[\sigma_{-i/2p}(\mathbb{E}(f^*f))h^{1/2r}]\\
&=&\mathfrak{i}^{(r)}(\sigma_{i/2q}(a)\sigma_{i/2q}(b)\sigma_{-i/2p}(\mathbb{E}(f^*f)))\\
&=&\mathfrak{i}^{(r)}(\mathbb{E}(\sigma_{i/2q}(a)\sigma_{i/2q}(b)\sigma_{-i/2p}(f^*f)))\\
&=&\mathbb{E}_{r}(\mathfrak{i}^{(r)}(\sigma_{i/2q}(a)\sigma_{i/2q}(b)\sigma_{-i/2p}(f^*f)))
\end{eqnarray*} 
A similar argument to the one used above shows that 
$$(h^{1/2p}a)[bh^{1/2p}]\mathfrak{i}^{(q)}(f^*f)=\mathfrak{i}^{(r)}(\sigma_{i/2q}(a)\sigma_{i/2q}(b)\sigma_{-i/2p}(f^*f))$$ and hence that 
\begin{eqnarray*}  
(h^{1/2p}a)[bh^{1/2p}]\mathbb{E}_{q}(\mathfrak{i}^{(q)}(f^*f))&=&\mathfrak{i}^{(r)}(\mathbb{E}(\sigma_{i/2q}(a)\sigma_{i/2q}(b)\sigma_{-i/2p}(f^*f)))\\
&=&\mathbb{E}_{r}(\mathfrak{i}^{(r)}(\sigma_{i/2q}(a)\sigma_{i/2q}(b)\sigma_{-i/2p}(f^*f)))\\
&=&\mathbb{E}_{r}((h^{1/2p}a)[bh^{1/2p}]\mathfrak{i}^{(q)}(f^*f)). 
\end{eqnarray*}
The density assertion in Lemma \ref{GL2-2.4+5} now leads to the conclusion that $\mathbb{E}_{r}(ab)=a\mathbb{E}_{q}(b)$ for all $a\in L^p(\N)$ and $b\in L^q(\M)^+$. By linearity the claim holds for all $b\in L^q(\M)$.
\end{proof}

In Proposition \ref{7:P j-embed} we introduced the embedding $\mathfrak{i}^{(p)}$ of $\mathfrak{m}_\nu$ into $L^p$, formally corresponding to sending $x\in \mathfrak{m}_\nu$ to $h^{1/2p}xh^{1/2p}$. However for $1<p <\infty$ more general embeddings are possible. In the ensuing analysis these more general embeddings will be extensively used, and hence we pause to clarify the underlying ideas. In the case $1\leq p <2$ we may select $q,r\geq 1$ so that $\frac{1}{r}+\frac{1}{2}=\frac{1}{p}$ and $\frac{1}{p}+\frac{1}{q}=1$, and for each $0\leq c\leq 1$ introduce an embedding of $\mathfrak{m}_\nu$ into $L^p$ formally corresponding to $x\mapsto h^{c/q}h^{1/r}xh^{1/r}h^{(1-c)/q}$, and for the case $2\leq p \leq \infty$ embeddings formally corresponding to $x \mapsto h^{c/p}xh^{(1-c)/p}$. In each case the thus constructed embedding will still be linear and injective, and the image of $\mathfrak{m}_\nu$ still a dense subspace. Considering the case $2\leq p<\infty$ by way of example, an embedding $\mathfrak{i}_c^{(p)}$ corresponding to the formal map $x \mapsto h^{c/p}xh^{(1-c)/p}$, is constructed by for any $a,b\in \in\mathfrak{n}_\nu$ and any $0<c<1$ first defining $\mathfrak{i}_c^{(p)}(b^*a)$ to be $\mathfrak{j}^{(p/c)}(b)^*\mathfrak{j}^{(p/(1-c))}(a)$, and from there to then linearly extend this map to all of $\mathfrak{m}_\nu$.

Given an expectation $\E$ of the form discussed in Proposition \ref{exp-props}, we may use that Proposition to see that the action of $\E_{p}$ on $L^p$ also harmonises with the above more general embeddings. For that we shall need the following result:

\begin{proposition}[{\cite[Lemma 9]{Terp2}}]\label{7:P Terp2}
There exists a net $(f_\lambda)$ of positive analytic elements in $\mathfrak{n}_\nu$ converging strongly to $\I$, and for which \begin{enumerate}
\item[(1)] $\sigma_z^\nu(f_\lambda)\in \mathfrak{n}_\nu\cap \mathfrak{n}_\nu^*$ for each $z\in \mathbb{C}$ and each $\lambda$,
\item[(2)] $\|\sigma_z^\nu(f_\lambda)\|\leq e^{\delta(\im(z))^2}$ for each $z\in \mathbb{C}$ and each $\lambda$ with $\delta>0$ constant,
\item[(3)] $(\sigma_z^\nu(f_\lambda))$ is $\sigma$-weakly convergent to $\I$ for each $z\in \mathbb{C}$.
\end{enumerate}
\end{proposition}

The final claim made above is not included in the version formulated in \cite[Lemma 9]{Terp2}. To convince the reader of its veracity, we therefore give details as appropriate, merely sketching some parts of the proof.

\begin{proof}[Outline of proof]  
 
One starts by selecting any right approximate identity $(g_\lambda)$ of $\mathfrak{n}_\nu$. Recall that in this case $g_\lambda$ must increase $\sigma$-strong* to $\I$ as $\lambda$ increases. 
Fixing some $\delta>0$, we then define the net $(f_\lambda)$ by means of the prescription $$f_\lambda=\sqrt{\frac{\delta}{\pi}}\int_{-\infty}^\infty\sigma_t^\nu(g_\lambda)e^{-\delta t^2}\,dt.$$The next step 
is to show that the function $$F:\bC\to \M:z\mapsto \sqrt{\frac{\delta}{\pi}}\int_{-\infty}^\infty\sigma_t^\nu(g_\lambda)e^{-\delta (t-z)^2}\,dt$$fulfils the criteria of \cite[Definition 2.5.20]{BR}. (Details 
of this part may be found in the proof of \cite[Proposition 2.5.22]{BR}.) Having verified this fact, the values $\sigma_z^\nu(f_\lambda)$, are then given by the formula 
 $$\sigma_z^\nu(f_\lambda)=\sqrt{\frac{\delta}{\pi}}\int\sigma_t^\nu(g_\lambda)e^{-\delta (t-z)^2}\,dt.$$ 
A slight modification of the latter part of the proof of Proposition \ref{analytic-n(A)} 
then shows that 
 $$\nu(|\sigma_z^\nu(f_\lambda)|^2)\leq (e^{\delta(\im(z))^2})^2\nu(|g_\lambda|^2)<\infty.$$

We next show that $(f_\lambda)$ converges strongly to $\I$. Let $H$ be the Hilbert space on which $\M$ acts. For any $\xi\in H$, we then have that
\begin{eqnarray*}
\lim_\lambda\langle f_\lambda\xi,\xi\rangle &=& \lim_\lambda\left\langle \left(\sqrt{\frac{\delta}{\pi}}\int_{-\infty}^\infty\sigma_t^\nu(g_\lambda)e^{-\delta t^2}\,dt\right)\xi,\xi\right\rangle\\
&=& \lim_\lambda\sqrt{\frac{\delta}{\pi}}\int_{-\infty}^\infty\langle\sigma_t^\nu(g_\lambda)\xi,\xi\rangle e^{-\delta t^2}\,dt\\
&=& \sqrt{\frac{\delta}{\pi}}\int_{-\infty}^\infty\langle\sigma_t^\nu(\I)\xi,\xi\rangle e^{-\delta t^2}\,dt\\
&=& \sqrt{\frac{\delta}{\pi}}\int_{-\infty}^\infty e^{-\delta t^2}\,dt.\|\xi\|^2\\
&=& \|\xi\|^2.
\end{eqnarray*}
If we combine the above formula with the fact that $\|f_\lambda\|\leq 1$, that then enables us to conclude that
$$\limsup_\lambda\|f_\lambda\xi-\xi\|^2=\limsup_\lambda\left(\|f_\lambda\xi\|^2 - \langle f_\lambda\xi,\xi\rangle -\langle \xi, f_\lambda\xi\rangle+\|\xi\|^2\right)\leq 0,$$which proves the claim regarding the strong convergence of $(f_\lambda)$.

We pass to proving (c). For any $\xi,\zeta\in H$ we have that 
\begin{eqnarray*}
\lim_\lambda\langle \sigma_z^\nu(f_\lambda)\xi,\zeta\rangle &=& \lim_\lambda\left\langle \left(\sqrt{\frac{\delta}{\pi}}\int_{-\infty}^\infty\sigma_t^\nu(g_\lambda)e^{-\delta (t-z)^2}\,dt\right)\xi,\zeta\right\rangle\\
&=& \lim_\lambda\sqrt{\frac{\delta}{\pi}}\int_{-\infty}^\infty\langle\sigma_t^\nu(g_\lambda)\xi,\zeta\rangle e^{-\delta (t-z)^2}\,dt\\
&=& \sqrt{\frac{\delta}{\pi}}\int_{-\infty}^\infty\langle\sigma_t^\nu(\I)\xi,\zeta\rangle e^{-\delta (t-z)^2}\,dt\\
&=& \sqrt{\frac{\delta}{\pi}}\int_{-\infty}^\infty e^{-\delta (t-z)^2}\,dt.\langle\xi,\zeta\rangle \\
&=& \langle\xi,\zeta\rangle.
\end{eqnarray*}
Thus $(\sigma_z^\nu(f_\lambda))$ converges to $\I$ in the Weak Operator Topology. But the $\sigma$-weak topology and the Weak Operator Topology agree on the unit ball of $\M$. Hence by part (b), $(\sigma_z^\nu(f_\lambda))$ is $\sigma$-weakly convergent to $\I$ as claimed.
\end{proof}

We close with presenting the promised alternative ways of constructing $\E_{p}$.

\begin{proposition}\label{genformula-exp} Let $\N$ be a von Neumann subalgebra of $\M$ for which there exists a faithful normal conditional expectation $\mathbb{E}$ from $\M$ onto $\N$ satisfying $\nu\circ\mathbb{E}=\nu$. Given $1< p<\infty$ and $c\in (0,1)$ such that both $c, (1-c) \in [0, \frac{p}{2}]$, we will for any $x\in \mathfrak{m}_\nu$ have that $\E_{p}(\mathfrak{i}_c^{(p)}(x))= \mathfrak{i}_c^{(p)}(\E(x))$.
\end{proposition}

\begin{proof} Let $(f_\lambda)$ be a net in $\N$ satisfying the criteria of Proposition \ref{7:P Terp2}. For the sake of 
argument suppose that $c>(1-c)$. We then clearly have that $2c-1=c-(1-c)\in(0,\frac{p}{2}]$. Thus $r=\frac{p}{2c-1}\leq 2$. 
For any $\lambda$ and $a,b\in \mathfrak{n}_\nu$ we may then apply Proposition \ref{GL2-2.2+3} and Lemma \ref{GL2-2.4+5} to 
the fact that $\frac{c}{p}=\frac{(1-c)}{p}+\frac{1}{r}$ to see that $f_\lambda\mathfrak{i}_c^{(p)}(x)= \mathfrak{j}^{(r)}(f_\lambda)\mathfrak{i}^{(p/(1-c))}(x)$. By Proposition \ref{exp-props} we then 
have that $$f_\lambda\E_{p}(\mathfrak{i}_c^{(p)}(x)) = \E_{p}(f_\lambda\mathfrak{i}_c^{(p)}(x))= \E_{p}(\mathfrak{j}^{(r)}(f_\lambda)\mathfrak{i}^{(p/(1-c))}(x)) =$$ 
 $$\mathfrak{j}^{(r)}(f_\lambda)\E_{p/(1-c)}(\mathfrak{i}^{(p/(1-c))}(x)) = \mathfrak{j}^{(r)}(f_\lambda)\mathfrak{i}^{(p/(1-c))}(\E(x)).$$
On once again applying Proposition \ref{GL2-2.2+3} and 
Lemma \ref{GL2-2.4+5}, we may conclude from this that 
 $$f_\lambda\E_{p}(\mathfrak{i}_c^{(p)}(x)) = \mathfrak{j}^{(r)}(f_\lambda)\mathfrak{i}^{(p/(1-c))}(\E(x)) 
 =f_\lambda\mathfrak{i}_c^{(p/(1-c))}(\E(x)).$$ 
Since $(f_\lambda)$ is at least $\sigma$-weakly convergent to $\I$, the left and right hand sides of this set of equalities are 
both $L^p$-weakly convergent to respectively $\E_{p}(\mathfrak{i}_c^{(p)}(x))$ and $\mathfrak{i}_c^{(p)}(\E(x))$, clearly showing that $\E_{p}(\mathfrak{i}_c^{(p)}(x))=\mathfrak{i}_c^{(p)}(\E(x))$ as 
was claimed.
\end{proof}

\section{The Haagerup reduction theorem revisited}\label{S4}

In our presentation of the reduction theorem, we will closely follow, but not clone, the presentation of this theorem in \cite{HJX}. The astute reader will pick up some subtle but important differences. The proof of the reduction theorem is long and at times complicated. Because of the aforementioned differences, the proof of the extended version presented here is by quite some margin even longer and more complex. In order to try and make the proof more digestible we avoid the temptation of discussing only the points of difference, but provide full details. In providing details we closely follow the presentation in \cite{HJX}, with the result that large blocks of the proof will therefore be little more than a restatement of the corresponding blocks in \cite{HJX}. However this is done deliberately in order to enable readers familiar with the original proof to clearly see how the proof of the original version needs to be adapted to yield the extended version presented here. To further aid the reader in picking up on this difference, we will where appropriate insert comments pointing to the points of difference.

We write $\mathbb{Q}_D$ for the dyadic rationals. This group plays a crucial role in the construction in that the enlargement of $\M$ we seek is nothing but the crossed product $\M\rtimes_\nu \mathbb{Q}_D$ produced using the restriction of the modular automorphism group to $\mathbb{Q}_D$. For this reason, we pause to gain a deeper understanding of $\mathbb{Q}_D$. In so doing we will take our cue from the helpful discussion on the site \newline \verb+https://en.wikipedia.org/wiki/Dyadic_rational+. When equipped with the discrete topology the dual group of the dyadic rationals is the so-called dyadic solenoid (a compact group). To see this note that the dyadic rationals are the direct limit of infinite cyclic 
subgroups of the rational numbers, $\displaystyle{\varinjlim \left\{2^{-k}\mathbb{Z}\mid k = 0, 1, 2, \dots \right\}}$ with the dual group then turning out to be the inverse limit of dual groups of 
each of these subgroups, namely the unit circle group under the repeated squaring map ${\displaystyle \zeta \mapsto \zeta ^{2}}$. An element of the dyadic solenoid can be represented as an infinite sequence of complex numbers $q_0, q_1, q_2, ...$, with the properties that each $q_k$ lies on the unit circle and that, for all $k > 0$, $q_k^2 = q_{k - 1}$. The group 
operation on these elements multiplies any two sequences component-wise. Each element of the dyadic solenoid corresponds to a character of the dyadic rationals that maps $a/2^b$ to the complex number $q_b^a$. Conversely, every character $\gamma$ of the dyadic rationals corresponds to an element of the dyadic solenoid given by $q_k$ = $\gamma(1/2^k)$.

Let $G$ be an LCA group which admits a group action on the von Neumann algebra in the form of $*$-automorphisms $\alpha_g:\M\to \M$. Write $\widehat{\alpha}_\gamma$ ($\gamma\in \widehat{G}$) for the action induced by the dual group $\widehat{G}$. One may use these structures to introduce the following definition: 

\begin{definition}\label{5:D dualwt}
We formally define the operator-valued weight $\mathscr{W}_G$ from $(\mathcal{M} \rtimes_\alpha G)_+$ onto the extended positive part of $\pi(M)$ by the prescription 
 $$\mathscr{W}_G(a)=\int_{\widehat{G}}\widehat{\alpha}_\gamma(a)\,d\gamma,\quad a\in(\mathcal{M} \rtimes_\alpha G)_+.$$
\end{definition}

In the case where the group $G$ is discrete, the dual group is compact and the integral in the definition therefore $\pi(M)$-valued on $\mathcal{M} \rtimes_\alpha G$. So in this case $\mathscr{W}_G$ is (up to a positive factor) a faithful normal conditional expectation from $\mathcal{M} \rtimes_\alpha G$ onto $\pi(M)$. Since by definition $\tnu=\nu \circ \pi_\alpha^{-1}\circ \mathscr{W}_G$, we clearly have that $\tnu\circ \mathscr{W}_G=\tnu$. In this case the action of $\mathscr{W}_G$ can be very elegantly described on the $\sigma$-weakly dense subspace $\mathrm{span}\{\lambda_t\pi_\alpha(a):t\in G,\, a\in\M\}$ of $\mathcal{M} \rtimes_\alpha G$. We capture this fact in the following well-known result presented here for the sake of the reader.

\begin{corollary}\label{5:C OVexistcor}
If the group $G$ is discrete, the operator-valued weight $\mathscr{W}_G$ defined above is a positive scalar multiple of a faithful normal conditional expectation from $\mathcal{M} \rtimes_\alpha G$ onto $\pi(\M)$. The action of this conditional expectation is uniquely determined by the formula
\begin{equation}\label{Waction} \mathscr{W}_G(\lambda_g\pi(a))=\left\{\begin{array}{ll} \pi(a) &\quad\mbox{ if }g=0\\ 0 &\quad\mbox{ otherwise }\end{array}\right.\qquad g\in G, a\in \M. \end{equation}
\end{corollary}

\begin{proof} We have already noted that the group $G$ is discrete if and only if the dual group $\widehat{G}$ is compact. Being compact, Haar measure on $\widehat{G}$ will be finite. It is clear that in this case $\mathscr{W}_G(a)=\int_{\widehat{G}}\widehat{\alpha}_\gamma(a)\,d\gamma$ will be an element of $\M$ for each $a\in \mathcal{M} \rtimes_\alpha G$. In fact, on rescaling we may assume Haar measure on $\widehat{G}$ to be a probability measure, in which case $\mathscr{W}_G(\I)=\I$. The fact that the action of the conditional expectation on terms of the form $\lambda_g\pi(a)$ (where $g\in G, a\in \M$) uniquely determines the expectation, follows from the $\sigma$-weak density of $\mathrm{span}\{\lambda_t\pi_\alpha(a):t\in G,\, a\in\M\}$ in $\mathcal{M} \rtimes_\alpha G$, and the noted normality of this expectation. Given such an element, we may apply the fact that $\pi(\M)$ corresponds to the fixed points of the dual action (\cite[Lemma 3.6]{haag-DW1}, to see that
 $$\mathscr{W}_G(\lambda_g\pi(a))=\int_{\widehat{G}}\widehat{\alpha}_\gamma(\lambda_g\pi(a))\,d\gamma =\lambda_g\pi(a) \int_{\widehat{G}}\overline{\gamma(g)}\,d\gamma.$$ 
 Assuming $G$ to be additive, the claim now follows from the known fact that $$\int_{\widehat{G}}\overline{\gamma(g)}\,d\gamma=\left\{\begin{array}{ll} 1 &\quad\mbox{ if }g=0\\ 0 &\quad\mbox{ otherwise }\end{array}\right..$$(See Exercise VII.5.6 of \cite{Katznelson}.)
\end{proof}

Let $\M$ be a von Neumann algebra equipped with a faithful normal weight $\nu$. The restriction of the mapping $t\mapsto \sigma_t^\nu$ to $\mathbb{Q}_D$ determines an action of the group $\mathbb{Q}_D$ on $\M$. We will write $\M\rtimes\mathbb{Q}_D$ for the crossed product 
of $\M$ with respect to this action. By the above discussion $\mathscr{W}_{\mathbb{Q}_D}$ is a $\widehat{\alpha}_\gamma$-invariant faithful normal conditional expectation $\mathscr{W}_{\mathbb{Q}_D}$ from $\M\rtimes\mathbb{Q}_D$ onto $\M$. 

For the rest of this section denote $\M\rtimes\mathbb{Q}_D$ by $\R$. We noted above that the subgroups $\left\{2^{-k}\mathbb{Z}\mid i = 0, 1, 2, \dots \right\}$ increase to $\mathbb{Q}_D$. We shall use this structure to construct a matching sequence of von Neumann algebras increasing to $\R$. From here on, this section will be devoted to proving the version of the reduction theorem stated below. We shall prove this theorem by means of a series of not insignificant lemmata. Readers familiar with the proof in \cite{HJX}, will recognise the cousins of these lemmata in \cite{HJX}.

\begin{theorem}\label{red s-finite}
Let $\M$ and $\R$ as above, there exists an increasing sequence $(\R_{n})_{n \geq 1}$ of von Neumann subalgebras of $\R$
satisfying the following properties:
\begin{enumerate}
\item Each $\R_{n}$ is semifinite, and is even finite admitting a faithful normal tracial state if $\nu$ is a state.
\item $\bigcup_{n\geq 1}\R_n$ is $\sigma$-strong* dense in $\R$.
\item For every $n\in\bN$ there exists a faithful normal conditional expectation $\mathscr{W}_{n}$ from $\R$ onto $\R_{n}$ such that
$$\tnu\circ\mathscr{W}_{n}=\tnu \quad \mbox{ and } \quad \sigma_{t}^{\tnu}\circ\mathscr{W}_{n} =\mathscr{W}_{n}\circ \sigma^{\tnu}_{t}, \quad t\in\bR.$$(Notice that the equality $\tnu\circ\mathscr{W}_{n}=\tnu$ ensures that $\mathscr{W}_{n}$ maps $\mathfrak{m}(\R)_{\tnu}^+$ onto $\mathfrak{m}(\R_n)_{\tnu}^+$. So by the normality of $\mathscr{W}_{n}$, $\mathfrak{m}(\R_n)_{\tnu}^+$ must be $\sigma$-weakly dense in $\R_n^+$ and hence $\tnu{\upharpoonright}\R_n$ semifinite.)
\end{enumerate}
\end{theorem}

Before proceeding with the proof, we have one technical issue to sort out. We computed the crossed product $\R = \M\rtimes\mathbb{Q}_D$ 
by using a \emph{restriction} of the modular group $t\mapsto\sigma_t^\nu$ to $\mathbb{Q}_D$. Let's denote this restriction by 
$t\mapsto\alpha_t$. The unitary group $t\mapsto \lambda_t$ ($t\in\mathbb{Q}_D$) induces an automorphism group $t\mapsto \lambda_t(\cdot)\lambda_t^*$ (
$t\in \mathbb{Q}_D$) on $\R$. This then raises the question of how this automorphism group, compares to the modular group $t\mapsto\sigma_t^{\tnu}$. Since $\alpha_t$ is just the restriction of $\sigma_t^\nu$, $\nu$ is clearly $\alpha_t$-invariant. This fact ensures that the above question has a very elegant answer. 

\begin{lemma}\label{10:L modgpred}
For any $t\in \mathbb{Q}_D$ and any $x\in \R$, we have that $\sigma_t^{\tnu}(x)=\lambda_tx\lambda_t^*$.
\end{lemma}
 
\begin{proof} For the sake of clarity we will here distinguish between $\M$, and the copy thereof inside $\R$, writing $\pi(\M)$ for that copy. For any $a\in \M$ and any $t\in \mathbb{Q}_D$, it follows from for example \cite[Lemma 2.9]{vD} and \cite[Theorem 4.7(1)]{haag-OV1} that
$$\lambda_t\pi(a)\lambda_t^*=\pi(\alpha_t(a))=\pi(\sigma_t^\nu(a))=\sigma_t^{\tnu}(\pi(a)).$$Now observe that for $s, t\in \mathbb{Q}_D$ we also have that $\lambda_t\lambda_s\lambda_t^*=\lambda_s=\sigma_t^{\tnu}(\lambda_s)$, where the second equality follows from \cite[Theorem 3.2(2)]{haag-DW1} and the fact that $\nu$ is $\alpha_t$-invariant. Since $\R= \M\rtimes \mathbb{Q}_D$ is generated by $\pi(\M)$ and the shift operators $\lambda_t$ ($t\in \mathbb{Q}_D$), these observations are enough to prove the claim. 
\end{proof} 

With $\mathbb{T}$ denoting the unit circle of the complex plane equipped with normalized Lebesgue measure $dm$, we obtain the following very elegant formula. Readers should compare this formula with the one presented in the hypothesis of \cite[Lemma 2.2]{HJX}. In \cite[Lemma 2.2]{HJX} the focus was on describing the possible values of $\tnu(f(\lambda_t))$. Here by contrast our interest is in the possible operator values of $\mathscr{W}_{\mathbb{Q}_D}(f(\lambda_t)))$. This subtle but important difference is one of the key aspects making the extension of the reduction theorem possible.

\begin{lemma}\label{distribution of lambda}
For each $f \in L_\infty(\mathbb{T})$ and each $t\in \mathbb{Q}_D\setminus\{ 0 \}$, we have that
\begin{equation}\label{distribution of lambda formula}
\mathscr{W}_{\mathbb{Q}_D}(f(\lambda_t))) = \left(\int_{\mathbb{T}}\ f(z)\,dm(z)\right)\I.
\end{equation}
\end{lemma}

\begin{proof} Let $t \in \mathbb{Q}_D\setminus\{ 0 \}$ be given. We may then apply Corollary \ref{5:C OVexistcor} to see that for each $t\in \mathbb{Q}_D$ and each $n\in \mathbb{Z}$, we will have that
$$\mathscr{W}_{\mathbb{Q}_D}(\lambda_t^{n}) =
\mathscr{W}_{\mathbb{Q}_D}(\lambda_{nt}) = \left\{\begin{array}{ll}
\I & \quad\mbox{ if } n=0, \\
0  & \quad\mbox{ otherwise.}\end{array}\right.$$We of course also have that
$$\int_{\mathbb{T}}\ z^n\,dm(z) = \int_{|z|=1}\ z^{n-1}\,dz = \left\{\begin{array}{ll}
  1 & \quad\mbox{ if } n=0, \\
  0  & \quad\mbox{ otherwise.}\end{array}\right.$$  
Thus \eqref{distribution of lambda formula} holds whenever $f$ is a
trigonometric polynomial. The trigonometric polynomials are of course $\sigma$-weakly dense in $L_\infty(\mathbb{T})$, and hence the normality of $\mathscr{W}_{\mathbb{Q}_D}$ therefore ensures that \eqref{distribution of lambda formula} holds for all $f \in L_\infty(\mathbb{T})$.
\end{proof}

Let $\mathcal{Z}(\R)$ denote the center of $\R$. The following lemma is just a version of \cite[Lemma 2.3]{HJX} with a somewhat more detailed proof.

\begin{lemma}\label{bn}
\begin{enumerate}
\item $\lambda_t \in \mathcal{Z}(\R_{\tnu})$ for any $t\in \mathbb{Q}_D$.
\item For every $n \in\mathbb{N}$ there exists a unique $b_{n} \in
\mathcal{Z}(\R_{\tnu})$ such that $0 \leq b_{n} \leq 2 \pi$ and $e^{ib_n}
= \lambda_{2^{-n}}$.
\end{enumerate}
\end{lemma}

\begin{proof} Recall that by Lemma \ref{10:L modgpred}, we have that 
\begin{equation}\label{modular gp of dual weight bis}
\sigma_{t}^{\tnu} (x) = \lambda_t x\lambda_t^*\mbox{ for all }x \in \R, \ t \in G.
\end{equation}
This on the one hand ensures that each $\lambda_s$ ($s\in \mathbb{Q}_D$) is a fixed point of $\sigma_t^{\tnu}$ and hence an element of $\R_{\tnu}$, and on the other that we must for any $x\in \R_{\tnu}$ have that $x=\sigma_t^{\tnu}(x)=\lambda_t x\lambda_t^*$ for all $t\in\mathbb{Q}_D$.  
This clearly ensures the validity of (1). 

To prove (2) we use the branch of the complex logarithmic function $z\mapsto\log(z)$ given by $\log(z)=\ln|z|+i\arg(z)$ where 
$0\leq \arg(z)<2\pi$. For a unimodular complex number $z$ we of course have that $-i\log(z)=\arg(z)$ where $0\leq\arg(z)\leq 2\pi$. Any unitary $u\in \R$ generates a commutative von Neumann subalgebra $\R_u$ of $\R$ which is $*$-isomorphic to some $L^\infty(X,\Sigma,\mu)$ where $(X,\Sigma,\mu)$ is a localizable measure space. By the Borel functional calculus $\log(u)$ then corresponds to a function for which $0\leq -i\log(u)\leq 2\pi$. Thus $-i\log(u)$ is then a bounded positive map (with $\|-i\log(u)\|\leq 2\pi$) belonging to the commutative subalgebra $\R_u$. In the case of $u=\lambda_{2^{-n}}$, the subalgebra $\R_u$ is in fact a subalgebra of $\mathcal{Z}(\R_{\tnu})$. So setting $b_{n} = -i \log(\lambda_{2^{-n}})$ will do the job once we note that (by construction) $e^{i b_n} = \lambda_{2^{-n}}$. 

The uniqueness of $b_{n}$ follows from the fact that $\lambda_{2^{-n}}$ has
no point spectrum by virtue of Lemma \ref{distribution of lambda} and the faithfulness of $\tnu$.
\end{proof}

Now let $a_{n} = 2^{n} b_{n},$ and define a sequence $(\nu_{n})_{n\geq 1}$ of weights on $\R$  by
\begin{equation}\label{def of fn}
\nu_{n}(x)=\tnu(e^{-a_{n}/2} xe^{-a_{n}/2}),\quad
x \in \R^+,\ n \geq 1.
\end{equation}

Parts (3)-(5) of the next lemma corresponds to \cite[Lemma 2.4]{HJX}. All statements regarding the case where $\nu$ is a weight rather than a state (equivalently the case where the $\R_n$'s are semifinite) are of course new. The additional care that needed to be taken to achieve this generality, unfortunately substantially lengthened the proof and increased its complexity.

\begin{lemma}\label{construction of Rn}
\begin{enumerate}
\item Each $\nu_n$ is a faithful normal semifinite weight on $\R$.
\item We have that 
\begin{equation}\label{formula od sfn}
\sigma_{t}^{\nu_{n}}(x)=e^{-ita_{n}}\sigma_{t}^{\tnu}(x)e^{ita_{n}},
\quad x \in\R,\ t \in\mathbb{R}.
\end{equation}
\item $\sigma_{t}^{\nu_{n}}$ is $2^{-n}$-periodic for all $n\geq 1$.
\item Setting $\R_{n} = \R_{\nu_{n}}, \ n \geq 1$, it follows that there exists 
a unique faithful normal conditional expectation $\mathscr{W}_{n}$ from $\R$ onto $\R_{n}$ such that
 $$\tnu\circ\mathscr{W}_{n} = \tnu
 \quad\mbox{and}\quad
 \sigma_{t}^{\tnu}\circ\mathscr{W}_{n} = \mathscr{W}_{n}\circ
 \sigma_{t}^{\tnu}, \quad t \in\mathbb{R},\ n\geq 1.$$
 \item $\R_{n} \subset \R_{n+1}$.
 \item For each $n\in \mathbb{N}$, the restriction $\tau_n$ of $\nu_n$ to $\R_n$, is a faithful normal semifinite trace on $\R_n$. Moreover if $\nu$ is a state, then each  $\tau_n$ is a finite trace on $\R_n$.
 \end{enumerate}
 \end{lemma}

\begin{proof} (1): The normality of $\nu_n$ follows from the fact that $\tnu$ is normal and the fact that 
 $$\sup_\alpha e^{-a_{n}/2}x_\alpha e^{-a_{n}/2} = e^{-a_{n}/2}\sup_\alpha x_\alpha e^{-a_{n}/2}$$ 
for monotone increasing nets. Faithfulness similarly easily 
follows from the fact that for any $x\in \R^+$, 
 $$0=\nu_{n}(x)=\tnu(e^{-a_{n}/2} xe^{-a_{n}/2})\Rightarrow e^{-a_{n}/2} xe^{-a_{n}/2}=0\Rightarrow x=0.$$ 
 Semifiniteness follows from the fact that the $\sigma$-weak density of 
$\mathrm{span}\{x\in\R^+: \tnu(x)<\infty\}$ in $\R$, ensures the $\sigma$-weak density of $\mathrm{span}\{e^{a_{n}/2} xe^{a_{n}/2}\in\R^+: \tnu(x)<\infty\}$, and the fact that 
 $$\mathrm{span}\{x\in\R^+: \nu_n(x)<\infty\}=\mathrm{span}\{e^{a_{n}/2} xe^{a_{n}/2}\in\R^+: \tnu(x)<\infty\}.$$

\medskip

(2): By part (1) of Lemma \ref{bn}, $a_{n} \in \mathcal{Z}(\R_{\tnu})\subset\R_{\tnu}$. Thus this claim is a direct consequence of \cite[Theorem VIII.2.11]{Tak2}.

\medskip

(3): We know from Lemma \ref{10:L modgpred}, that 
$\sigma_t^{\tnu}(a) =\lambda_ta\lambda_t^*$ for every $t\in\mathbb{Q}_D$ and every $a\in \R$. 
It therefore follows from part (2) of Lemma \ref{bn} that
$$\sigma_{2^{-n}}^{\nu_{n}}(x)
= e^{-i\,b_{n}}\sigma^{\tnu}_{2^{-n}}(x)e^{ib_{n}}
= \lambda_{2^{-n}}^{\ast}\lambda_{2^{-n}}x
\lambda_{2^{-n}}^{\ast}\lambda_{2^{-n}}
= x$$for all $x\in\R$, thereby ensuring the validity of (3).

\medskip

(4):  Define $\mathscr{W}_{n}$ by
$$\mathscr{W}_{n}(x) = 2^{n}\int^{2^{-n}}_{0} \sigma_{t}^{\nu_{n}}(x)\, dt,\quad x \in \R.$$
By  the $2^{-n}$-periodicity of $\sigma^{\nu_{n}}$, we have that
$$\mathscr{W}_{n}(x) = \int^{1}_{0}\sigma_{t}^{\nu_{n}}(x)\,dt, \quad x \in \R.$$
It is then a routine matter to check that $\mathscr{W}_{n}$ is a faithful 
normal conditional expectation from $\R$ onto $\R_{n}.$ 

Next let $x\in \R^+$ be given. First suppose that $\tnu(x)<\infty$. Since $a_n\in \R_{\tnu}$, 
we may then apply \cite[Theorem VIII.2.6]{Tak2} and use the $\sigma^{\tnu}$-invariance of 
$\tnu$ to conclude that
$$\tnu\big(\sigma_{t}^{\nu_{n}} (x)\big)
=\tnu \big(e^{-ita_{n}}\sigma_{t}^{\tnu}(x)e^{ita_{n}} \big)  
=\tnu \big(\sigma_{t}^{\tnu}(x) \big)
=\tnu(x),\quad t \in\mathbb{R}.$$If on the other hand 
$\tnu\big(\sigma_{t}^{\nu_{n}} (x)\big)<\infty$ we may apply what we have just proven 
to conclude that 
$$\tnu\big(\sigma_{t}^{\nu_{n}} (x)\big)
=\tnu\big(\sigma_{-t}^{\nu_{n}}(\sigma_{t}^{\nu_{n}} (x))\big)
=\tnu(x), t \in\mathbb{R}.$$Thus for any $x\in \R^+$ and any $t\in \mathbb{R}$, 
$\tnu(x)<\infty$ if and only if $\tnu\big(\sigma_{t}^{\nu_{n}} (x)\big)<\infty$, 
in which case they are equal. We therefore clearly have that $\tnu\circ\sigma_{t}^{\nu_{n}}=\tnu$ 
for all $t\in \mathbb{R}$ and all $n\geq 1$. But then 
$$\tnu(\mathscr{W}_{n}(x))=\int^{1}_{0}\tnu\big(\sigma_{t}^{\nu_{n}}(x)\big)dt
 =\tnu (x),\quad x \in \R^+;$$
that is $\tnu\circ \mathscr{W}_{n}=\tnu$. The uniqueness of $\mathscr{W}_{n}$ is now ensured by 
\cite[Theorem IX.4.2]{Tak2} with the claimed commutation relation following from equation (\ref{formula od sfn}) 
and the definition of $\mathscr{W}_{n}$. 

\medskip

(5): For every natural number $n$, $a_n$ and $a_{n+1}$ will commute, given that they both belong to 
$\mathcal{Z}(\R_{\tnu})$. It is now an easy exercise to see that $\nu_{n+1}(x) = \nu_{n}(h_{n} x)$ for all $x\in\R$, 
where $h_n=e^{-a_{n+1}}e^{a_n}=e^{-(a_{n+1} - a_{n})}$. If we are able to show that $h_{n} \in \mathcal{Z}(\R_{n})$, it will 
then follow from \cite[Theorem VIII.2.11]{Tak2} that $\sigma_{t}^{\nu_{n+1}}(x)= e^{-ith_{n}}\sigma_{t}^{\nu_n}(x)e^{ith_{n}}= 
\sigma_{t}^{\nu_n}(x)$ for all $x\in\R_n$ and all $t \in\mathbb{R}$. This will clearly ensure that $\R_{n}\subset\R_{n+1}$ as claimed.

By part (2) and the fact that $a_k\in \mathcal{Z}(\R_{\tnu})$, we have that $\R_{\tnu} \subset \R_{\nu_{k}}$ for each $k\in \mathbb{N}$. In particular, we will then have that $h_{n} \in \R_{\nu_{n}}=\R_n$. As in Lemma \ref{bn} we now use the branch of the complex logarithmic function $z\mapsto\log(z)$ given $0\leq \arg(z)<2\pi$. Then 
$$a_{n} = -i2^{n}\log\,\lambda_{2^{-n}} = -i2^{n}\log\big(\lambda_{2^{-n-1}}^{2}\big),$$
whence
 $$a_{n+1} - a_{n} = -i2^{n} \big[ 2\log\,
 \lambda_{2^{-n-1}} -\log\,\big(\lambda_{2^{-n-1}}^{2}\big)\big].$$
However for any $z \in\mathbb{T}$,
 $$2\,\log\, z - \log(z^{2}) = \left\{\begin{array}{ll}
 0 & \textrm{if}\  0 \leq\arg\,z <\pi,\\
 2 \pi i &\mathrm{if}\ \pi \leq \arg\,z <2 \pi.
 \end{array}\right.$$
Hence
 $$a_{n+1} - a_{n} = 2^{n+1}\pi  e_{n},$$
where $e_{n}$ is the spectral projection of $\lambda_{2^{-n-1}}$
corresponding to $\im(z) < 0$. So for all $x\in\R$ and all $t\in\mathbb{R}$,
 $$\sigma_{t}^{\nu_{n+1}}(x) = h_{n}^{it}\sigma_{t}^{\nu_{n}}(x)h_{n}^{-it}
 = e^{-i 2^{n+1} \pi te_{n}}
 \sigma_{t}^{\nu_{n}} (x) e^{i 2^{n+1} \pi te_{n}}\,.$$
Consequently, if $x \in \R_n$, the $2^{-n-1}$-periodicity of $\sigma_{t}^{\nu_{n+1}}$ will ensure that
$$x = e^{-i \pi e_{n}} \sigma_{2^{-n-1}}^{\nu_n}(x) e^{i \pi e_{n}} = e^{-i \pi e_{n}} x e^{i \pi e_{n}}.$$Now observe that the Borel functional calculus ensures that $e^{-i \pi e_{n}}e_n=-e_n$ and that $e^{-i \pi e_{n}}(\I-e_n)=e^{-i \pi 0}(\I-e_n) =(\I-e_n)$. Hence 
$$x = e^{-i \pi e_{n}} x e^{i \pi e_{n}} = (\I-2 e_{n}) x(\I-2e_{n}).$$ 
This clearly shows that $\I-2e_{n} \in \mathcal{Z}(\R_n)$, or equivalently that $e_{n} \in \mathcal{Z}(\R_n)$. Thus
$a_{n+1} - a_{n}\in \mathcal{Z}(\R_n)$, and hence also $h_{n} \in \mathcal{Z}(\R_n)$, which then yields the desired inclusion $\R_{n} \subset \R_{n+1}$.
 
\medskip

(6): The faithfulness of $\nu_n$ on $\R_n$ is a clear consequence of part (1). The fact that $\tnu=\tnu\circ\mathscr{W}_{n}$, ensures that the restriction of $\tnu$ to $\R_n$ is normal and semifinite. We may now use a similar argument to that used in the proof of part (1) to conclude from this that $\nu_n$ is normal and semifinite on $\R_n$. It therefore remains to show that $\nu_n$ satisfies the trace property on $\R_n$. Thus let $x\in \R_n$ be given.

First suppose that $\nu_n(x^*x)<\infty$. Thus $x^*x=|x|^2\in \mathfrak{m}_{\nu_n}$. Let $x=v|x|$ be the polar decomposition of $x$. Then 
we clearly have that $v\in \R_n=\R_{\nu_n}$. But in that case we may use \cite[Theorem VIII.2.6]{Tak2} to conclude that 
$x^*xv^*=|x|^2v^*\in \mathfrak{m}_{\nu_n}$ and that in addition $\nu_n(x^*x)=\nu_n(|x|^2v^*v)=\nu(v|x|^2v^*)=\nu_n(xx^*)$. It therefore 
follows that $\nu_n(x^*x)<\infty$ if and only if $\nu_n(xx^*)<\infty$, in which case the two quantities are equal. But then 
$\nu_n(x^*x)=\nu_n(xx^*)$ for all $x\in \R_n$ as required.

Finally suppose that $\nu$ is a state. Since in this particular case $\mathscr{W}_{\mathbb{Q}_D}$ is a faithful normal conditional expectation, we therefore have that $\tnu(\I)=\nu\circ\mathscr{W}_{\mathbb{Q}_D}(\I)=\nu(\I)=1$. So $\tnu$ is also a state. But then 
$\nu_{n}(\I)=\tnu(e^{-a_{n}})\leq\|e^{-a_{n}}\|_\infty<\infty$ as required.
\end{proof}

It remains to show the $\sigma$-weak density of the union of the $\R_{n}$s in $\R$. The groundwork for this verification will be laid by the following cluster of lemmas. The second is essentially just a special case of the first. We first introduce some necessary notation:
Given a von Neumann algebra $\N$ equipped with some \emph{fns} weight $\psi$, we may for any $w\in\mathbb{C}$ with $\im(w)> 0$, define $D(\sigma_w)$ to be the set of all $a\in \N$ for which the map $t\mapsto \sigma^\psi_t(a)$ ($t\in \mathbb{R}$), may be extended to a $\sigma$-weakly continuous function $f_w$ on the strip $\{z\in\mathbb{C}: 0\leq\im(z)\leq \im(w)\}$ which is analytic on the interior of that strip. In the case where $\im(w)<0$, $D(\sigma_w)$ is defined similarly using the strip $\{z\in\mathbb{C}: 0\geq\im(z)\geq \im(w)\}$.

The following pair of lemmata now take the place \cite[Lemma 2.5]{HJX}.

\begin{lemma}[{\cite[Lemma 3.3]{haag-OV1}}/{\cite[VIII.3.18(1)]{Tak2}}]\label{Aig:L 1}
Let $\nu$ be a faithful normal semifinite weight on a von Neumann algebra $\M$. For any $a\in \M$ and $k\geq 0$ the following are equivalent:
\begin{itemize}
\item $\nu(a\cdot a^*)\leq k^2\nu$;
\item $a\in D(\sigma^\nu_{-i/2})$ and $\|\sigma^\nu_{-i/2}(a)\|\leq k$.
\end{itemize}
\end{lemma}

\begin{lemma}\label{L2 bounded elements}
Let $\psi$ be a faithful normal semifinite weight on the von Neumann algebra $\mathcal{N}$. If additionally $a \in \mathcal{N}_{\psi}$, we may in 
applying Lemma \ref{Aig:L 1} to the pair $(\mathcal{N},\psi)$, choose $k$ to be $\|a\|$.  
\end{lemma}
 
\begin{proof}
Note that under the conditions of the hypothesis, the map $t\mapsto\sigma^\psi_t(a)$ $(t\in \mathbb{R})$ is the constant map $t\mapsto a$. The 
unique continuous extension of this map to the strip $\{z\in\mathbb{C}:0\geq \im(z)\geq -1/2\}$ which is analytic on the interior of the strip, is of course again the constant map $z\mapsto a$. Thus by the discussion preceding Lemma \ref{Aig:L 1}, we have that $a=\sigma^\nu_{-i/2}(a)$. The claim therefore directly follows from Lemma \ref{Aig:L 1}. 
\end{proof}

In the following lemma $[x,\;y]$ denotes the commutator of two operators $x$ and $y$, i.e., $[x,\; y]=xy-yx$. If $\psi$ is a
normal weight on a von Neumann algebra $\mathcal{N}$,  $\|x\|_\psi$ will for any $x\in \mathcal{N}$ denote the quantity 
$\psi(x^*x)^{1/2}$. In the case where $\psi$ is also faithful and semifinite and $x\in \mathfrak{n}_\psi$, $\|x\|_\psi$ is of course 
nothing but $\|\eta_\psi(x)\|$ where $\eta_\psi(x)$ is as in Remark \ref{5:R left Hilbert}. For simplicity of notation we will in the 
following theorem identify $\M$ with the canonical copy thereof inside $\R$.

The state version of parts (b) and (c) of the following lemma correspond to \cite[Lemma 2.6]{HJX}. Many parts of the proof had to be substantially reworked to achieve this generality, resulting in a proof which is almost three times as long as that of \cite[Lemma 2.6]{HJX}. In particular some very careful approximations using ``designer'' $\sigma$-weakly dense subspaces were needed to achieve (b) and (c) in the general case, which then explains the long proof and the need for (a). 

\begin{lemma}\label{L2 estimate}
Let $b_n$, $\tnu$ and $\R$ be as in Lemma \ref{bn}. 
\begin{enumerate}
\item[(1)] The subspaces $\mathrm{span}\{\lambda_t a: t\in\mathbb{Q}_D, a\in \mathfrak{n}_\nu(\M)\cap\mathfrak{n}_\nu(\M)^*\}$ and $\mathrm{span}\{\lambda_t a: t\in\mathbb{Q}_D, a\in \mathfrak{m}_\nu(\M), a \mbox{ analytic}\}$ are both $\sigma$-weakly dense subspaces of $\R$. Moreover $\mathrm{span}\{\lambda_t a: t\in\mathbb{Q}_D, a\in \mathfrak{n}_\nu(\M)\} \subset \mathfrak{n}_{\tnu}$ and $\mathrm{span}\{\lambda_t a: t\in\mathbb{Q}_D, a\in \mathfrak{n}^*_\nu(\M)\} \subset \mathfrak{n}^*_{\tnu}$. Each element of $\mathrm{span}\{\lambda_t a: t\in\mathbb{Q}_D, a\in \mathfrak{m}_\nu(\M), a \mbox{ analytic}\}$ is moreover again analytic with respect to $\sigma_t^{\tnu}$.
\item[(2)] For any $w$ in either $\mathrm{span}\{\lambda_ta: t\in\mathbb{Q}_D, a\in \mathfrak{n}_\nu(\M)\cap\mathfrak{n}_\nu(\M)^*\}$ or $\mathrm{span}\{\lambda_ta: t\in\mathbb{Q}_D, a\in \mathfrak{m}_\nu(\M), a \mbox{ analytic}\}$, we have that 
\begin{enumerate}
\item[(i)] $\displaystyle\lim_{n \to\infty} \| [b_{n},\; w]\|_{\tnu} = 0;$
\item[(ii)] $\displaystyle\lim_{n \to\infty}\,\sup_{t\in [-1,\; 1]}\ \big\|\big[e^{i t b_{n}},\; w \big]\big\|_{\tnu}=0$
\end{enumerate}
\item[(3)] For any $x\in \R$ and any $f\in \mathrm{span}\{\lambda_ta: t\in\mathbb{Q}_D, a\in \mathfrak{m}_\nu(\M), a \mbox{ analytic}\}$, we have that $\displaystyle\lim_{n\to\infty}\,\sup_{t\in\bR} \, \big\|(\sigma_{t}^{\nu_n}(x) - x)f\big\|_{\tnu}=0.$
\end{enumerate}

\end{lemma}

\begin{proof} \noindent\textbf{(1):} Recall that the subspace $\{a\in \mathfrak{m}_\nu(\M), a \mbox{ analytic}\}$ is by Lemma \ref{GL2-2.4+5} $\sigma$-strongly and hence also $\sigma$-weakly dense in $\M$. Thus the $\sigma$-weak closure of $\mathrm{span}\{\lambda_ta: t\in\mathbb{Q}_D, a\in \mathfrak{m}_\nu(\M), a \mbox{ analytic }\}$ must include the $\sigma$-weak closure of 
$\mathrm{span}\{\lambda_ta: t\in\mathbb{Q}_D, a\in \M\}$. But this latter space is known to be $\sigma$-weakly dense in the 
crossed product $\R=\M\rtimes\mathbb{Q}_D$. Hence as required, the subspace $$\mathrm{span}\{\lambda_ta: t\in\mathbb{Q}_D, a\in \mathfrak{m}_\nu(\M), a \mbox{ analytic }\}$$is $\sigma$-weakly dense in $\R$. A similar proof using the $\sigma$-strong density 
of $\mathfrak{n}(\M)_\nu^\infty$ in $\M$ (see Lemma \ref{GL2-2.4+5}) shows that $\mathrm{span}\{\lambda_t a: t\in\mathbb{Q}_D, a\in \mathfrak{n}_\nu(\M)\cap\mathfrak{n}_\nu(\M)^*\}$ is $\sigma$-weakly dense in $\R$.  

We next show that $\mathrm{span}\{\lambda_ta: t\in\mathbb{Q}_D, a\in \mathfrak{n}_\nu(\M)\}$ is contained in 
$\mathfrak{n}_{\tnu}$. For any $a\in \mathfrak{n}_\nu(\M)$ and any $\mathbb{Q}_D$, the fact that $\mathscr{W}_{\mathbb{Q}_D}$ 
is a conditional expectation, ensures that 
$\tnu(|\lambda_t a|^2)=\tnu(|a|^2)=\nu(\mathscr{W}_{\mathbb{Q}_D}(|a|^2)=\nu(|a|^2)<\infty$ as required. If indeed $a\in \mathfrak{n}(M)^*_\nu$ it similarly follows that 
$$\tnu(|(\lambda_ta)^*|^2)=\tnu(|\lambda_ta^*\lambda_t^*|^2)= \tnu(|\sigma_t^{\tnu}(a^*)|^2)= \tnu(\sigma_t^{\tnu}(|a^*|^2))=$$ $$\tnu(|a^*|^2)=\nu(\mathscr{W}_{\mathbb{Q}_D}(|a^*|^2)=\nu(|a^*|^2)<\infty.$$This proves the third claim. The fourth claim now follows by duality once we notice that Lemma \ref{10:L modgpred} and \cite[Proposition 6.40]{GLnotes} ensure that $\lambda_ta = [\lambda_{-t}\sigma_t^{\varphi}(a^*)]^*$ for every $t\in \mathbb{Q}_D$ and every $a\in \mathcal{M}$.

Since for any $a\in \M$ and $t \in \mathbb{Q}_D$ we have that $\sigma_t^\nu(a)=\sigma_t^{\tnu}(a)$  with $\lambda_t\in \R_{\tnu}$ (see Lemma \ref{bn}), it clearly follows that each term of the form $\lambda_ta$ will be analytic with respect to $\sigma_t^{\tnu}$ if $a$ is analytic with respect to $\sigma_t^{\nu}$.

\bigskip

\noindent\textbf{(2)(i):} Let $w$ be in either $\mathrm{span}\{\lambda_ta: t\in\mathbb{Q}_D, a\in \mathfrak{n}_\nu(\M)\cap\mathfrak{n}_\nu(\M)^*\}$ or $\mathrm{span}\{\lambda_ta: t\in\mathbb{Q}_D, a\in \mathfrak{m}_\nu(\M), a \mbox{ analytic}\}$, and let $k\in\mathbb{Z}$ be given. For any $t\in\mathbb{Q}_D$ we have by Lemma \ref{bn} that 
\begin{eqnarray*}
\big\|[\lambda_{2^{-n}}^{k},\; w]\big\|_{\tnu}
&=& \big\|(\lambda_{k2^{-n}}w - w \lambda_{k2^{-n}})\big\|_{\tnu}\\
&=& \big\|(w-\lambda_{k2^{-n}}^*w \lambda_{k2^{-n}})\big\|_{\tnu}\\
&=& \big\|(w-\sigma_{-k2^{-n}}^{\tnu}(z))\big\|_{\tnu}\\
&=& \tnu(w^*w)-\tnu(\sigma_{-k2^{-n}}^{\tnu}(w^*)w) - \tnu(w^*\sigma_{-k2^{-n}}^{\tnu}(w))+ \tnu(|\sigma_{-k2^{-n}}^{\tnu}(w)|^2)\\
&=& 2\tnu(w^*w)-\tnu(\sigma_{-k2^{-n}}^{\tnu}(w^*)w) - \tnu(w^*\sigma_{-k2^{-n}}^{\tnu}(w)).
\end{eqnarray*}
(For the last equality we used the fact that $\tnu\circ\sigma_{t}^{\tnu}=\tnu$.) In the case where $w\in \mathrm{span}\{\lambda_ta: t\in\mathbb{Q}_D, a\in \mathfrak{n}_\nu(\M)\cap\mathfrak{n}_\nu(\M)^*\}$, we may now apply \cite[Theorem VIII.1.2]{Tak2} to see that
\begin{equation}\label{L2 estimate 1}
\lim_{n\to\infty}\big\|[P(\lambda_{2^{-n}}),\; w]\big\|_{\tnu} = 0
\end{equation}
for any monomial, and hence also for any trigonometric polynomial $P$. In the case where $w\in\mathrm{span}\{\lambda_ta: t\in\mathbb{Q}_D, a\in \mathfrak{m}_\nu(\M), a \mbox{ analytic}\}$ this conclusion follows from \cite[Lemma VIII.2.5(ii)]{Tak2}.

We show how to obtain the final conclusion for the case where $w\in\mathrm{span}\{\lambda_ta: t\in\mathbb{Q}_D, a\in \mathfrak{n}_\nu(\M)\cap\mathfrak{n}_\nu(\M)^*\}$. The proof for the other case is entirely analogous.
Let $\log$ be as in Lemma \ref{bn}. On $\mathbb{T}$, $\log$ agrees with the bounded almost everywhere continuous function $\arg$. Thus restricted to $\mathbb{T}$ we have that $\log\in L^\infty(\mathbb{T}) \subset L^2(\mathbb{T})$. It follows that there exists a trigonometric polynomial $P$ such that $\|P + i\ {\log}\|_{L_2(\mathbb{T})} < \epsilon$.

Let $a\in \mathfrak{n}_\nu(\M)\cap\mathfrak{n}_\nu(\M)^*$ and any $s\in \mathbb{Q}_D$ be given. (In the case left as an exercise one would assume that $a$ is an analytic element of $\mathfrak{m}_\nu(\M)$.) Taking into account that $\mathscr{W}_{\mathbb{Q}_D}$ is a conditional expectation and that $\lambda_{2^{-n}}$, $b_n$ and $\lambda_s$ all belong to the same abelian von Neumann subalgebra, it follows from equation (\ref{distribution of lambda formula}) and the definitions of $a_n$ and $b_n$ in Lemma \ref{bn}, that 
\begin{eqnarray*}
\|( b_{n}-P(\lambda_{2^{-n}}))\lambda_{s}a\|_{\tnu}&=&\tnu(a^*\lambda_s^* |b_{n}-P(\lambda_{2^{-n}})|^2\lambda_sa)^{1/2}\\
&=&\tnu(a^*|b_{n}-P(\lambda_{2^{-n}})|^2a)^{1/2}\\
&=&\nu(\mathscr{W}_{\mathbb{Q}_D}(a^* |b_{n}-P(\lambda_{2^{-n}})|^2a))^{1/2}\\
&=&\nu(a^*\mathscr{W}_{\mathbb{Q}_D}(|b_{n}-P(\lambda_{2^{-n}})|^2)a)^{1/2}\\
&=&\|-i\log\, - P \|_{L_2(\mathbb{T})}\nu(|a|^2)^{1/2}\\
&<&\epsilon\nu(|a|^2)^{1/2}.
\end{eqnarray*}

Recall that $\lambda_{2^{-n}}$ is a unitary belonging to $\mathcal{Z}(\R_{\tnu})$. Hence $a_n$ and $b_n$ also belong to $\mathcal{Z}(\R_{\tnu})$. With $a$ and $s$ as before we will of course have that $|a|^2\in \mathfrak{m}_\nu\subset 
\mathfrak{m}_{\tnu}$, and hence we may use \cite[Theorem VIII.2.6]{Tak2} to see that
\begin{eqnarray*}
\|\lambda_sa( b_{n}-P(\lambda_{2^{-n}}))\|_{\tnu}&=&\tnu(|\lambda_sa ( b_{n}-P(\lambda_{2^{-n}}))|^2)^{1/2}\\
&=&\tnu((b_{n}-P(\lambda_{2^{-n}}))^*|a|^2 (b_{n}-P(\lambda_{2^{-n}}))^{1/2}\\
&=&\tnu(|b_{n}-P(\lambda_{2^{-n}})|^2|a|^2)^{1/2}\\
&=&\tnu(\mathscr{W}_{\mathbb{Q}_D}(|b_{n}-P(\lambda_{2^{-n}})|^2|a|^2))^{1/2}
\end{eqnarray*}
On once again applying equation (\ref{distribution of lambda formula}), it therefore follows that
$$\|a( b_{n}-P(\lambda_{2^{-n}}))\lambda_s\|_{\tnu}\leq \|-i\log\, - P \|_{L_2(\mathbb{T})}\nu(|a|^2)^{1/2}<\epsilon\nu(|a|^2)^{1/2}.$$

So if $w$ is of the form $z=\sum_{k=1}^m\lambda_{s_k}a_k$, we will have that 
\begin{eqnarray*}
\|[b_{n},\; w ]\|_{\tnu} &\leq& \|[ P(\lambda_{2^{-n}}),\; w]\|_{\tnu} + \|[b_{n} -P(\lambda_{2^{-n}}),\; w]\|_{\tnu}\\
&\leq& \|[ P(\lambda_{2^{-n}}),\; w]\|_{\tnu} + \|(b_{n}-P(\lambda_{2^{-n}}))w\|_{\tnu}+\|w (b_{n}-P(\lambda_{2^{-n}}))\|_{\tnu}\\
&\leq& \|[ P(\lambda_{2^{-n}}),\; w]\|_{\tnu} + \sum_{k=1}^m \|(b_{n}-P(\lambda_{2^{-n}}))\lambda_{s_k}a_k\|_{\tnu}\\
&&\qquad\qquad+\sum_{k=1}^m \|\lambda_{s_k}a_k(b_{n}-P(\lambda_{2^{-n}}))\|_{\tnu}\\
&\leq& \|[P(\lambda_{2^{-n}}),\; w]\|_{\tnu} + \epsilon[2\sum_{k=1}^m \nu(|a_k|^2)^{1/2}].
\end{eqnarray*} 
Therefore, by the above computations, we then have that 
$$\limsup_{n\to\infty}\|[b_{n},\; w]\|_{\tnu}\leq \epsilon[2\sum_{k=1}^m \nu(|a_k|^2)^{1/2}]$$
whence $\lim_{n \to\infty}\| [b_{n},\; w]\|_{\tnu} = 0$ as required.

\bigskip

\noindent\textbf{(2)(ii):} Let $w$ be an element of either $\mathrm{span}\{\lambda_ta: t\in\mathbb{Q}_D, a\in \mathfrak{n}_\nu(\M)\cap\mathfrak{n}_\nu(\M)^*\}$ or $\mathrm{span}\{\lambda_ta: t\in\mathbb{Q}_D, a\in \mathfrak{n}_\nu(\M), a \mbox{ analytic}\}$. 
By Lemma \ref{L2 bounded elements}, the fact that $b_{n} \in \R_{\tnu}$ with 
$\|b_n\|_\infty\leq 2\pi$, similarly ensures that $b_n^*.\tnu.b_n\leq\|b_n\|_\infty^2\tnu \leq (2\pi)^2 \tnu$. 
We claim that $\|[b_n^m,\,w]\|_{\tnu}\leq m(2\pi)^{m-1}\|[b_n,\,w]\|_{\tnu}$ holds for all $m\in \bN$. The validity of the case $m=1$ clearly follows from what we noted above. Now suppose that the claim holds for $m=k$. Since   
 $$[b_{n}^{k+1},\; w]=b_{n}[b_{n}^{k},\; w]+[ b_{n},\; w]b_{n}^{k},$$we may again use the observations made above to see that 
\begin{eqnarray*}
\| [ b_{n}^{k+1},\; w]\|_{\tnu} &\leq& \|b_{n}[b_{n}^{k},\; w]\|_{\tnu}+\|[ b_{n},\; w]b_{n}^{k}\|_{\tnu}\\
 &\leq& \|b_n\|.\|[b_{n}^{k},\; w]\|_{\tnu}+\|[ b_{n},\; w]b_{n}^{k}\|_{\tnu}\\
 &\leq& k(2\pi)^{k}.\|[b_{n},\; w]\|_{\tnu}+\|[ b_{n},\; w]b_{n}^{k}\|_{\tnu}\\
 &\leq& k(2\pi)^{k}.\|[b_{n},\; w]\|_{\tnu}+\|b_n\|_\infty^k\|[ b_{n},\; w]\|_{\tnu}\\
 &=& (k+1)(2\pi)^{k}\|[b_n,\,w]\|_{\tnu}.
\end{eqnarray*}
The claimed estimate therefore follows by induction.
Hence for any $z \in \bC$,
\begin{eqnarray*} 
\| [e^{z b_{n}},\; w]\|_{\tnu}
&\leq& \sum_{k=1}^\infty\frac{|z|^{k}}{k!}\,\|[ b^{k}_{n},\; w]\|_{\tnu}\\
&\leq& \sum_{k=1}^\infty\frac{|z|^{k}}{(k-1)!}\, (2\pi)^{k-1} \|[b_{n},\; w]\|_{\tnu}\\
&=&|z|\,e^{2\pi|z|}\|[ b_{n},\; w]\|_{\tnu}.
\end{eqnarray*}
Therefore
 $$\sup_{t\in [-1,\; 1]}\|[ e^{it b_{n}},\; w]\|_{\tnu}
 \leq e^{2 \pi} \|[ b_{n}, \; w]\|_{\tnu},$$
which on the strength of (i), implies (ii).

\bigskip

\noindent\textbf{(3):} First let $w$ be an element of either $\mathrm{span}\{\lambda_ta: t\in\mathbb{Q}_D, a\in \mathfrak{n}_\nu(\M)\cap\mathfrak{n}_\nu(\M)^*\}$, or $\mathrm{span}\{\lambda_ta: t\in\mathbb{Q}_D, a\in \mathfrak{m}_\nu(\M), a \mbox{ analytic}\}$, and let $\epsilon > 0$ be given. By (b)(ii) there exists $n_{0} \in
\bN$ such that
\begin{equation}\label{L2 estimate2}
 \|[ e^{i s b_{n}},\; w]\|_{\tnu}\leq\epsilon,
 \quad \mbox{for all }s \in [-1,\; 1],\ \mbox{ and all }\;n\geq n_{0}.
\end{equation}
Next observe that by either \cite[Theorem VIII.1.2]{Tak2} or \cite[Theorem VIII.2.5]{Tak2} (as appropriate), we have that  
\begin{eqnarray*}
\|(\sigma_{s}^{\tnu}(w) - w)\|^2_{\tnu}&=&\tnu(w^*w)-\tnu(\sigma_{s}^{\tnu}(w^*)w)- \tnu(w^*\sigma_{s}^{\tnu}(w))+\tnu(\sigma_{s}^{\tnu}(w^*w))\\
&=&2\tnu(w^*w)-\tnu(\sigma_{s}^{\tnu}(w^*)w)- \tnu(w^*\sigma_{s}^{\tnu}(w))\\
&\to&0
\end{eqnarray*}
Therefore $n_0$ can be chosen so that in addition
\begin{equation}\label{L2 estimate3}
 \|(\sigma_{s}^{\tnu}(w) - w)\|_{\tnu}\leq \epsilon,\quad |s| \leq 2^{-n_{0}}.
\end{equation}
Let $t\in\bR$ and $n\in\bN$ be given with $n\geq n_{0}.$ Write $t = t_{1}
+ t_{2},$ where $t_{1} = k\, 2^{-n}$ for some $k\in\mathbb{Z}$ and $0
\leq t_{2} \leq 2^{-n}.$ Then for any $y \in \R$,
 \begin{eqnarray*}
 \sigma_{t_{1}}^{\tnu} (y)
 &=& \lambda_{k2^{-n}}y\lambda_{k2^{-n}}^{\ast}
 =e^{i k b_{n}}y e^{-ik b_{n}}\\
 &=& e^{i k 2^{-n} a_{n}}ye^{-ik2^{-n}a_{n}}
 = e^{i t_{1}a_{n}}ye^{-i t_{1}a_{n}}.
 \end{eqnarray*}
Since $\|\cdot\|_{\tnu}$ is invariant under $\sigma_{t_{1}}^{\tnu}$
and $a_{n} \in \mathcal{Z}(\R_{\tnu})$, we may deduce that
\begin{eqnarray*}
 \|(\sigma_{t}^{\tnu}(w) - e^{i a_{n}t}we^{-ia_{n}t})\|_{\tnu}
 &=& \|(\sigma_{t_{2}}^{\tnu}(w) - e^{i a_{n} t_{2}}we^{-i a_{n}t_{2}})\|_{\tnu}\\
 &\leq& \|(\sigma_{t_{2}}^{\tnu}(w) - w)\|_{\tnu}
 + \|(w - e^{i a_{n}t_{2}}w e^{-ia_{n}t_{2}})\|_{\tnu}\\
 &=& \|(\sigma_{t_{2}}^{\tnu}(w) - w)\|_{\tnu}
 + \|[e^{-ia_{n}t_{2}},\; w]\|_{\tnu}\,.
 \end{eqnarray*}
Now $a_{n}t_{2}=(2^{n} t_{2}) b_{n}$ and $2^{n}t_{2}\le 1$. Hence
from equations (\ref{L2 estimate2}) and (\ref{L2 estimate3}), it follows that 
$$\|e^{-ia_{n}t}\sigma_{t}^{\tnu}(w)e^{ia_{n}t}-w\|_{\tnu} = \|\sigma_{t}^{\tnu}(w)-e^{ia_{n}t}we^{-ia_{n}t})\|_{\tnu}\leq 2\epsilon.$$
(Here we used the fact that $e^{ia_{n}t}$ is in $\R_{\tnu}$ alongside \cite[Theorem VIII.2.6]{Tak2} to verify the claimed equality.) On the basis of Lemma \ref{construction of Rn}, this then ensures that 
\begin{equation}\label{red-estimate 1:eqn}
\lim_{n\to\infty}\,\sup_{t\in\bR} \, \big\|(\sigma_{t}^{\nu_n}(w) - w_\alpha)f\big\|_{\tnu}=0
\end{equation}

Now let $f\in\mathrm{span}\{\lambda_ta: t\in\mathbb{Q}_D, a\in \mathfrak{m}_\nu(\M), a \mbox{ analytic}\}$, and at first assume that $x\in\mathrm{span}\{\lambda_ta: t\in\mathbb{Q}_D, a\in \mathfrak{n}_\nu(\M)\cap\mathfrak{n}_\nu(\M)^*\}$. 
Since $f$ is analytic, we know from Lemma \ref{Aig:L 1} that there exists 
a positive constant $c_f$ so that $f^*.\tnu.f\leq c_f\tnu$. This in turn ensures that
$$\|(\sigma_t^{\nu_n}(x)-x)f\|_{\tnu} \leq c_f^{1/2}\|\sigma_t^{\nu_n}(x)-x\|_{\tnu}\leq c_f^{1/2}2\epsilon.$$
which on the strength of equation (\ref{red-estimate 1:eqn}) proves the claim for this case.
 
Next let $x$ be an arbitrary element of $\R$. The subspace $\mathrm{span}\{\lambda_ta: t\in\mathbb{Q}_D, a\in \mathfrak{n}_\nu(\M)\cap\mathfrak{n}_\nu(\M)^*\}$ of $\R$ is both convex and $\sigma$-weakly dense. Hence by convexity it is also strongly dense with respect to the GNS-representation engendered by $\tnu$. Thus we may select a net $(x_\alpha) \subset \mathrm{span}\{\lambda_ta: t\in\mathbb{Q}_D, a\in \mathfrak{n}_\nu(\M)\cap\mathfrak{n}_\nu(\M)^*\}$ which converges to $x$ (GNS-)strongly. We know from the first part of the proof that 
\begin{equation} \label{red-estimate 2:eqn}
\lim_{n\to\infty}\,\sup_{t\in\bR} \, \big\|(\sigma_{t}^{\nu_n}(x_\alpha) - x_\alpha)f\big\|_{\tnu}=0
\end{equation} 
for each $\alpha$. Since in the GNS-representation of $\R$ the term $\|(x-x_\alpha)f\|_{\tnu}$ is just $\|\pi_{\tnu}(x-x_\alpha)\eta(f)\|$, the strong convergence noted earlier ensures that
\begin{equation}\label{red-estimate 3:eqn}
\lim_{\alpha}\|(x-x_\alpha)f\|_{\tnu}=0.
\end{equation}
We may next use the fact that each $e^{ia_nt}$ and each $\lambda_t$ belong to $\R_{\tnu}$, to see that 
\begin{eqnarray*}
\|(e^{-ia_{n}t}\sigma_{t}^{\tnu}(x-x_\alpha )e^{ia_{n}t})f\|_{\tnu}^2 &=& \|(e^{-ia_{n}t}\lambda_t(x-x_\alpha)\lambda_t^*e^{ia_{n}t})f\|_{\tnu}^2\\
&=& \|(x-x_\alpha)(\lambda_t^*e^{ia_nt}f)\|_{\tnu}^2\\
&=& \|(x-x_\alpha)(\lambda_t^*e^{ia_{n}t}fe^{-ia_{n}t}\lambda_t)\|_{\tnu}^2\\
&=& \|(x-x_\alpha)(e^{ia_{n}t}\sigma_{-t}^{\tnu}(f)e^{-ia_{n}t})\|_{\tnu}^2.
\end{eqnarray*}
It therefore follows from Lemma \ref{construction of Rn} that 
$$\|\sigma_{t}^{\nu_n}(x-x_\alpha)f\|_{\tnu}^2=\|(x-x_\alpha)\sigma_{-t}^{\nu_n}(f)\|_{\tnu}\mbox{ for all }\alpha.$$
Alongside the Uniform Boundedness Principle the strong convergence of $(\pi(x_\alpha))$, ensures that the net $(\pi(x-x_\alpha))$ - and therefore also $(x-x_\alpha)$ - must be norm bounded. So there must exist some $K>0$ so that $\|x-x_\alpha\|\leq K$ for all $\alpha$. We may now use these two facts to see that 
\begin{eqnarray*}
\|\sigma_{t}^{\nu_n}(x-x_\alpha)f\|_{\tnu} &=& \|(x-x_\alpha)\sigma_{-t}^{\nu_n}(f)\|_{\tnu}\\
&=& \|(x-x_\alpha)(\sigma_{-t}^{\nu_n}(f)-f)\|_{\tnu} + \|(x-x_\alpha)\sigma_{-t}^{\nu_n}(f)\|_{\tnu}\\
&\leq& K\|\sigma_{-t}^{\nu_n}(f)-f\|_{\tnu} + \|(x-x_\alpha)f\|_{\tnu}.
\end{eqnarray*}
A combination of equations (\ref{red-estimate 1:eqn}) and (\ref{red-estimate 3:eqn}), now ensures that 
\begin{equation}\label{red-estimate 4:eqn}
\limsup_{n\to\infty}\,\sup_{t\in\bR} \, \big\|\sigma_{t}^{\nu_n}(x-x_\alpha)f\|_{\tnu} \leq  \|(x-x_\alpha)f\|_{\tnu}.
\end{equation}

Given $\epsilon>0$ we may select $\alpha$ so that $\|(x_\alpha-x)f\|_{\tnu}\leq \epsilon$. So by equations (\ref{red-estimate 2:eqn}) and (\ref{red-estimate 4:eqn}), we will then have that
\begin{eqnarray*}
&&\quad\limsup_{n\to\infty}\,\sup_{t\in\bR} \, \|(\sigma_{t}^{\nu_n}(x)-x)f\|_{\tnu}\\
&&\leq \limsup_{n\to\infty}\,\sup_{t\in\bR} \,[ \|\sigma_{t}^{\nu_n}(x-x_\alpha)f\|_{\tnu} + \|(\sigma_{t}^{\nu_n}(x_\alpha)-x_\alpha)f\|_{\tnu} + \|(x_\alpha-x)f\|_{\tnu}]\\
&&\leq \limsup_{n\to\infty}\,\sup_{t\in\bR} \,[ \|\sigma_{t}^{\nu_n}(x-x_\alpha)f\|_{\tnu} + \|(\sigma_{t}^{\nu_n}(x_\alpha)-x_\alpha)f\|_{\tnu} +\epsilon]\\
&&\leq \limsup_{n\to\infty}\,\sup_{t\in\bR}\|\sigma_{t}^{\nu_n}(x-x_\alpha)f\|_{\tnu} + \limsup_{n\to\infty}\,\sup_{t\in\bR}\|(\sigma_{t}^{\nu_n}(x_\alpha)-x_\alpha)f\|_{\tnu} +\epsilon\\
&&= \epsilon.
\end{eqnarray*}
This clearly suffices to prove the claim.
\end{proof}

We are now finally ready to prove the $\sigma$-weak density of $\cup_{n\in\bN}\R_n$ in $\R$, and hence conclude the 
proof of the reduction theorem. The lemma below corresponds to \cite[Lemma 2.7]{HJX}. The proofs of both lemmata make use of a specially selected dense subspace of the pre-dual. Whilst in \cite{HJX} the proof of \cite[Lemma 2.7]{HJX} could be disposed of in six lines, in the present context a quite different dense subspace of functionals needed to be found in terms of which to work out the details of the proof. With the selection of such a subspace having been made, quite a bit more work was required to obtain the same conclusion.  

\begin{lemma}
For any $x \in \R$, the sequence $(\mathscr{W}_n(x))$ is $\sigma$-weakly convergent to $x$. Consequently, $\cup_{n\in\bN}\R_{n}$ is $\sigma$-strong* dense in $\R$.
\end{lemma}

\begin{proof} 
We claim that the subspace of $\R_*$ spanned by functionals of the form $\tnu(y^*\cdot f)$ where $y\in\mathfrak{n}_{\tnu}$ and $f\in\mathrm{span}\{\lambda_ta: t\in\mathbb{Q}_D, a\in \mathfrak{m}_\nu(\M), a \mbox{ analytic}\}$ is norm dense in $\R_*$. To prove this all we need to show is that if for some $g\in\R$ we have that $\tnu(y^*gf)=0$ for all $y$ and $f$ as above, then $g=0$. To see that this is the case, note that when given such a $g$, we may for each $f$ as above select $y$ to be $y=gf$. We then have that $\tnu(|gf|^2)=0$ for all $f$, which by the faithfulness of $\tnu$, ensures that $gf=0$ for all $f$. But we know from Lemma \ref{L2 estimate} that $\mathrm{span}\{\lambda_ta: t\in\mathbb{Q}_D, a\in \mathfrak{m}_\nu(\M), a \mbox{ analytic}\}$ is $\sigma$-weakly dense in 
$\R$. It therefore follows that $g=0$ as required. 

We proceed to show that for all $y$ and $f$ as above, we will for any $x\in\R$ have that 
\begin{equation}\label{10:eqn reduction}
\lim_{n\to\infty}\tnu(y^*(\mathscr{W}_n(x)-x)f)=0.
\end{equation}
Since the subspace of $\R_*$ spanned by functionals of the form $\tnu(y^*\cdot f)$ is norm dense in $\R_*$ and the sequence $(\mathscr{W}_n(x)-x)$ norm-bounded, it will then follow from this that $\lim_{n\to\infty}\rho(\mathscr{W}_n(x)-x)=0$ for all $\rho\in\R_*$ as required. This will establish the $\sigma$-weak density of $\cup_{n\in\bN}\R_{n}$ in $\R$. The fact that $\cup_{n\in\bN}\R_{n}$ is convex, will then ensure that it is also $\sigma$-strong* dense in $\R$. 

We proceed to prove the validity of equation (\ref{10:eqn reduction}). With $y$ and $f$ as before, we may invoke the Cauchy-Schwarz inequality to see that $$|\tnu(y^*(\mathscr{W}_n(x)-x)f)|\leq \tnu(y^*y)^{1/2}\tnu(|(\mathscr{W}_n(x)-x)f|^2)^{1/2} = \tnu(y^*y)^{1/2}\|(\mathscr{W}_n(x)-x)f\|_{\tnu}.$$It now follows from the definition of the expectations $\mathscr{W}_n$, that $\|(\mathscr{W}_n(x)-x)f\|_{\tnu} \leq \sup_{t\in\bR} \|(\sigma_{t}^{\nu_{n}}(x)-x)f\|_{\tnu}$ for each $n$. 
If we combine the above facts, we have that 
$$|\tnu(y^*(\mathscr{W}_n(x)-x)f)|\leq\tnu(y^*y)^{1/2}.\big[\sup_{t\in\bR}\|(\sigma_{t}^{\nu_{n}}(x)-x)f\|_{\tnu}\big].$$
The claimed convergence therefore follows from part (c) of Lemma \ref{L2 estimate}.
\end{proof}

We close this section with a statement of the result for $L^p$-spaces. This should be compared to \cite[Theorem 3.1]{HJX}. The verification of the density claim in part (2) of the theorem requires significantly more work and a very different proof to the one used in \cite[Theorem 3.1]{HJX}. This part of the proof is ultimately achieved through a subtle reworking of the proof of part (a) of \cite[Proposition 2.11]{GL2}.

\begin{theorem}\label{10:T reduction Lp} 
Let $\M$ be a von Neumann algebra equipped with a faithful normal semifinite weight $\nu$ and let $0< p < \infty$ be given. 
Then for $\R=\M\rtimes_\nu\mathbb{Q_D}$ we have that $L^p(\R)$ is a Banach superspace of $L^p(\M)$ isometrically containing $L^p(\M)$. Moreover the sequence $(\R_n)_{n\geq 1}$ of semifinite von Neumann algebras, each equipped
with a faithful normal semifinite trace $\tau_n$, admit an accompanying sequence of isometric embeddings $J_n\, :\, L^p(\R_n,\tau_n) \to L^p(\R)$ such that
\begin{enumerate}
\item the sequence $\big(J_n\big(L^p(\R_n, \tau_n) \big)\big)_{n\geq 1}$ is increasing;
\item $\bigcup_{n\geq 1}J_n\big(L^p(\R_n,\tau_n)\big)$ is norm dense in $L^p(\R)$;
\item in the case where $1\leq p<\infty$, the extension $\mathscr{W}_n^{(p)}$ of $\mathscr{W}_n$ to $L^p(\R)$ contractively maps $L^p(\R)$ onto $J_n\big(L^p(\R_n, \tau_n) \big)$, with $\mathscr{W}_n^{(p)}(f)$ converging weakly in $L^p$ to $f$ for each $f\in L^p(\R)$;
\item in the case where $1\leq p<\infty$, $L^p(\M)$ and each $J_n\big(L^p(\R_n, \tau_n)\big)$ are the images of contractive projections on $L^p(\R)$.
\end{enumerate}
Here $L^p(\R_n, \tau_n)$ is the tracial noncommutative $L^p$-space associated with $(\R_n,\tau_n)$.
\end{theorem}

\begin{proof} The proof makes use of Theorem \ref{red s-finite} and hence we keep all the notation there. The $L^p$ spaces for $\R$ will be constructed using $\tnu$ 
and for $\M$ using $\nu$. For the subalgebras we shall sometimes use their traces and sometimes the weight $\tnu{\upharpoonright}\R_n$ to construct the associated 
$L^p$-spaces. These two variants will respectively be denoted by $L_p(\R_n,\tau_n)$ and $L^p(\R_n)$.

Since by construction $\mathscr{W}_{\mathbb{Q}_D}$ is a faithful normal conditional expectation from $\R$ onto $\M$ for which we have 
that $\nu\circ\mathscr{W}_{\mathbb{Q}_D}=\tnu$, it follows from the discussion in Remark \ref{10:R cond-exp} that $L^p(\M)$ is a 
subspace of $L^p(\R)$ which in the case $1\leq p < \infty$ is the image of a contractive projection. On considering Remark \ref{10:R cond-exp} alongside part (3) of 
Theorem \ref{red s-finite}, it similarly follows that each $L^p(\R_n)$ is a subspace of $L^p(\R)$ which in the case $1\leq p < \infty$ is the image of a contractive 
projection. Given $n\leq k$ it is clear from Lemma \ref{construction of Rn} that the restriction of $\mathscr{W}_n$ to $\R_k$ yields a faithful normal conditional 
expectation from $\R_k$ onto $\R_n$, for which we have that $\tnu\circ \mathscr{W}_n=\tnu$. Thus by Remark \ref{10:R cond-exp} we will then 
clearly have that $L^p(\R_n)\subset L^p(\R_k)$.

For each $n\in \mathbb{N}$ the space $L_p(\R_n, \tau_n)$ is by \cite[Theorem II.37 \& Corollary II.38]{terp} linearly isometric to 
$L^p(\R_n)$. It therefore remains to show that $\cup_{n=1}^\infty L^p(\R_n)$ is dense in $L^p(\R)$, and that in the case $1\leq p <\infty$, the sequence 
$(\mathscr{W}^{(p)}_n(f))$ is for any $f\in L^p(\R)$ weakly convergent to $f$. We shall first prove these claims for the case 
where $1\leq p <\infty$, and then extract the density claim for the general case from this fact.

\textbf{Case 1 ($1\leq p <\infty$):} The claim regarding weak convergence of $(\mathscr{W}^{(p)}_n(f))$ to $f$ for any $f\in L^p(\R)$, clearly ensures that 
$\cup_{n=1}^\infty L^p(\R_n)$ is weak $L^p$ dense in $L^p(\R)$. The convexity of $\cup_{n=1}^\infty L^p(\R_n)$ then ensures that this subspace is in fact norm dense. 
It therefore remains to prove the claim regarding weak convergence.

The case $p=1$ is fairly easy. Given $x\in \R$ we may in this case use the facts noted in Remark \ref{10:R cond-exp} to see that  
 \begin{eqnarray*}
 tr(x\mathscr{W}^{(1)}_n(f))
 &=&
tr(\mathscr{W}^{(1)}_n(x\mathscr{W}^{(1)}_n(f))) = tr(\mathscr{W}_n(x)\mathscr{W}^{(1)}_n(f))\\
&=&tr(\mathscr{W}^{(1)}_n(\mathscr{W}_n(x)f)) =tr(\mathscr{W}_n(x)f).
 \end{eqnarray*}
But we know that $(\mathscr{W}_n(x))$ is $\sigma$-weakly convergent to $x$. Hence 
 $$\lim_{n\to\infty}tr(x\mathscr{W}^{(1)}_n(f))=\lim_{n\to\infty}tr(\mathscr{W}_n(x)f)=tr(xf)$$ 
as required.

Now suppose that $1<p<\infty$ and let $q>1$ be given such that $\frac{1}{p}+\frac{1}{q}=1$. Let $a, b\in \mathfrak{m}(\R)_{\tnu}$ be given. We claim that then 
$\mathscr{W}_n(a),\mathscr{W}_n(b)\in \mathfrak{m}(\R_n)_{\tnu}$ for each $n$. To see this, note that for any $u\in \mathfrak{n}(\R)_{\tnu}$ the operator Schwarz inequality 
for completely positive maps ensures that 
 $$\tnu(\mathscr{W}_n(a)\mathscr{W}_n(u^*))\leq \tnu(\mathscr{W}_n(u^*u))= \tnu(u^*u)<\infty$$ 
which in turn ensures that 
$\mathscr{W}_n(u)\in \mathfrak{n}(\R_n)_{\tnu}$ as claimed. It is moreover clear from the discussion preceding Proposition \ref{exp-props} that $\mathscr{W}^p_n(\mathfrak{i}^{(p)}(a))=\mathfrak{i}^{(p)}(\mathscr{W}_n(a))$ and $\mathscr{W}^q_n(\mathfrak{i}^{(q)}(b))=\mathfrak{i}^{(q)}(\mathscr{W}_n(b))$. This 
fact when combined with repeated applications of the expectation properties noted in Remark \ref{10:R cond-exp} and Proposition \ref{exp-props}, then leads to the conclusion that 
\begin{eqnarray*}
tr(\mathfrak{i}^{(q)}(b)\mathscr{W}^p_n(\mathfrak{i}^{(p)}(a))) &=& tr\circ\mathscr{W}^1_n(\mathfrak{i}^{(q)}(b)\mathscr{W}^p_n(\mathfrak{i}^{(p)}(a)))\\
&=& tr(\mathfrak{i}^{(q)}(\mathscr{W}_n(b))\mathscr{W}^p_n(\mathfrak{i}^{(p)}(a)))\\
&=& tr\circ\mathscr{W}^1_n(\mathfrak{i}^{(q)}(\mathscr{W}_n(b))\mathfrak{i}^{(p)}(a))\\
&=& tr(\mathfrak{i}^{(q)}(\mathscr{W}_n(b))\mathfrak{i}^{(p)}(a)).
\end{eqnarray*} 
By Remark \ref{7:R trwt} we have that $tr(\mathfrak{i}^{(q)}(\mathscr{W}^q_n(b))\mathfrak{i}^{(p)}(a)) = tr(\mathscr{W}_n(b)\mathfrak{i}^{(1)}(a))$ which 
by what we proved in the case $p=1$, must converge to $tr(b\mathfrak{i}^{(1)}(a))$ as $n\to \infty$. 
Since another application of Remark \ref{7:R trwt} shows that $tr(b\mathfrak{i}^{(1)}(a)) =tr(\mathfrak{i}^{(q)}(b)\mathfrak{i}^{(p)}(a))$, it follows that $tr(\mathfrak{i}^{(q)}(b)\mathscr{W}^p_n(\mathfrak{i}^{(p)}(a)))\to tr(\mathfrak{i}^{(q)}(b)\mathfrak{i}^{(p)}(a))$ as $n\to \infty$. Since for any $f\in L^p(\R)$ the sequence 
$(\mathscr{W}^p_n(f))$ is norm bounded and the subspaces $\mathfrak{i}^{(q)}(\mathfrak{m}(\R)_{\tnu})$ and $\mathfrak{i}^{(p)}(\mathfrak{m}(\R)_{\tnu})$ respectively dense in 
$L^q(\R)$ and $L^p(\R)$, approximation by elements of these subspaces now shows that we will for any $f\in L^p(\R)$ and $g\in L^q(\R)$ have that 
$tr(g\mathscr{W}_n^p(f))\to tr(gf)$ as $n\to \infty$. The claim regarding weak convergence therefore follows.

\textbf{Case 2 ($0< p <\infty$):} Let $0<p<\infty$ be given. We explain how to use the density of $\cup_{n\geq 1}L^p(\R_n)$ in 
$L^p(\R)$, to show that $\cup_{n\geq 1}L^{p/2}(\R_n)$ is dense in $L^{p/2}(\R)$. Inductively applying this to what we have already 
proven will then yield the general statement. We have already seen that for $n\leq k$ we will have that $L^{p}(\R_n)\subseteq L^{p}(\R_k)$. This then proves that 
 $$L^{p/2}(\R_n)=L^{p}(\R_n).L^{p}(\R_n)\subseteq L^{p}(\R_n).L^{p}(\R_k)\subseteq L^{p}(\R_k).L^{p}(\R_k)=L^{p/2}(\R_k),$$
and hence that 
$\cup_{n\geq 1}L^{p/2}(\R_n)=\cup_{n,k\geq 1}L^{p}(\R_n).L^{p}(\R_k)$. H\"older's inequality together with the density of $\cup_{n\geq 1}L^p(\R_n)$ in $L^p(\R)$, 
now ensures that 
 $L^{p}(\R_n)\cdot(\cup_{k\geq 1}.L^{p}(\R_k))$ is dense in $L^{p}(\R_n)\cdot L^p(\R)$, with $\cup_{n\geq 1}L^{p}(\R_n).L^{p}(\R)$ 
similarly dense in $L^{p/2}(\R)$. It is now clear that, as was required, $\cup_{n\geq 1}L^{p/2}(\R_n)$ is dense in $L^{p/2}(\R)$. 
\end{proof}

\begin{remark}
Despite the remarkable fact demonstrated by the preceding result, it should be noted that semifinite and type III algebras cannot produce the same $L^p$ spaces. This follows 
from the important work of David Sherman who showed that if for some $1\leq p<\infty$ ($p\neq 2$) the $L^p(\M_0)$ and $L^p(\M_1)$ are linearly isometric, then the underlying 
algebras themselves are Jordan *-isomorphic.   
\end{remark}

\section{Subdiagonality of unital subalgebras of von Neumann algebras}\label{S5}

In the late 1950's and early 1960's, it became clear that many famous theorems about the classical $H^\infty$ space of 
bounded analytic functions on the disk, could be carried over to the setting of abstract function algebras. Several leading researchers contributed 
to the development of these ideas; most notably Helson and Lowdenslager \cite{HL}, and Hoffman \cite{Ho}. This emerging  `commutative generalized $H^p$-theory' 
was then organized and summarized in the mid 1960's in the paper of Srinivasan and Wang \cite{SW}. The notion that Srinivasan and Wang used 
to unify these results was that of {\em weak*-Dirichlet algebras}. This summary of Srinivasan and Wang basically furnishes one with an array of 
properties that are all in some way equivalent to the Szeg\"o formula in the setting of weak*-Dirichlet algebras.

In a parallel development inspired by questions from prediction theory, operator theorists and operator algebraists made great efforts to find noncommutative analogues of the classical
\emph{inner-outer factorization} of analytic functions. In this noncommutative context one wishes, for example,  to find conditions on a positive operator $T$
which imply that $T = |S|$ for an operator $S$ which is in a \emph{noncommutative Hardy class}, or alternatively \emph{outer} in some sense.
(See for example \cite[p.\ 1495]{PX}.) This is an active and important research field with links to many exciting parts of
mathematics. Central parts of this topic still bear further clarification. Note for example the by now classical result of Devinatz \cite{Dev} concerning a Riesz-Szeg\"o like 
factorization of a class of $B(H)$-valued functions on the unit interval, which has resisted generalization in some important directions.

Inspired by these two developments, Arveson introduced his notion of subdiagonal subalgebras of von Neumann algebras as a possible context for extending the 
 of results in \cite{SW} to the noncommutative context \cite{PTAG,AIOA}. The elegance of Arveson's framework is seen in the fact that in the case where 
the ambient von Neumann algebra $\M$ is commutative, the finite maximal subdiagonal subalgebras defined by Arveson correspond exactly to weak*-Dirichlet algebras. Thus Arveson's 
setting canonically extends the notion of weak*-Dirichlet algebras.

The theory of these subdiagonal algebras progressed at a carefully measured pace, until in 2005, Labuschagne \cite{L-Szego} managed to use some of Arveson's ideas to show that in 
the context of finite von Neumann algebras, these maximal subdiagonal algebras satisfy the Szeg\"o formula conjectured by Arveson. 

A sequence of papers by Blecher and Labuschagne followed (\cite{BL1, BL-FMRiesz, BL-outer, BL-Beurling,BL-outer2}), complemented by important contributions from Ueda \cite{Ueda}, 
and Bekjan and Xu \cite{BX}, which together demonstrated that in the context of finite von Neumann algebras the \emph{entire} cycle of results (somewhat surprisingly) survives 
the passage to noncommutativity. Specifically it was shown that the same cycle of results as proferred in \cite{SW} hold true for what Blecher and Labuschagne call tracial subalgebras 
of a finite von Neumann algebra (see \cite{BLsurvey}). 

With the theory of subdiagonal subalgebras of finite von Neumann algebras thereby reaching some level of maturity, authors then turned their attention to the setting of $\sigma$-finite 
von Neumann algebras on the one hand and semifinite von Neumann algebras on the other. In the case of $\sigma$-finite algebras important structural results were obtained by Bekjan, Blecher, Ji, 
Labuschagne, Raikhan, Ohwada, Saito and Xu (\cite{JOS,jisa,Xu,jig2,jig3,L-HpIII,blueda,BeRa}), and Bekjan, Oshanova, Sager, Ueda and others in the semifinite setting (\cite{Bek-sem,BO,sager, Ueda2}).

However the transition from finite to $\sigma$-finite von Neumann algebras cannot be made without some sacrifice. One very costly price that needs to be paid for the passage 
to the $\sigma$-finite case, is the loss of the theory of the Fuglede-Kadison determinant (\cite{FuKa}, \cite{AIOA}). (As was shown by Sten Kaijser \cite{Kai}, the presence 
of such a determinant forces the existence of a finite trace, and hence the theory of the Fuglede-Kadison determinant is essentially a theory of finite von Neumann algebras.) 
In the case of subdiagonal subalgebras of finite von Neumann algebras, this determinant served the role of a `noncommutative' geometric mean, and hence featured very prominently 
in the development of that theory. But how does one in these more general settings even begin to give expression to something like a geometric mean when there is no obvious way 
to make sense of Szeg\"o's formula? As can be seen from \cite{SW} there are a large number of properties that in the setting of weak*-Dirichlet algebras are equivalent to 
Szeg\"o's formula. Whilst Szeg\"o's formula itself may have no meaning in the type III setting, many of these equivalent conditions do extend to the type III setting. See 
for example \cite{L-HpIII, jig2} where one finds aspects like a very detailed Beurling-type theory of invariant subspaces, a very general Gleason-Whitney theorem, and an 
extension of the so-called \emph{unique normal state extension property} and left partial factorization surviving the transition. In the setting of semifinite algebras Bekjan 
and Oshanova and Sager \cite{BO, sager} similarly showed that the unique normal state extension property and the Beurling invariant subspace theory carries over to the semifinite case.

The most recent step forward in the development of this theory was the work of Blecher and Labuschagne \cite{BL-CAOT} on maximal semi-$\sigma$-finite subdiagonal algebras. These 
are subdiagonal subalgebras which are in a very regular way maximal with respect to a strictly semifinite weight on the ambient 
von Neumann algebra. The very regular structure of these algebras enable one to (without reference to Haagerup's reduction theorem) view them as a `limiting case' of the theory of maximal $\sigma$-finite subdiagonal algebras in 
much the same way that Bekjan \cite{Bek-sem} showed that the theory of maximal {\em semifinite} subdiagonal algebras is a limiting case of the finite maximal subdiagonal subalgebras. 
To date this seems to be the most general setting in which aspects like the F \& M Riesz theorem and the full force of the Gleason-Whitney theorem hold true ((GW1) and (GW2) as 
defined in the discussion preceding Theorem \ref{Co}). The concept of subdiagonality for general von Neumann algebras that emerged from all of these studies is the following:

\begin{definition}
Let $\M$ be a von Neumann algebra equipped with a faithful normal semifinite weight $\nu$ and $\mathcal{D}$ be a unital von Neumann subalgebra of $\M$ such that $\nu{\upharpoonright}\mathcal{D}$ is semifinite. Further suppose that there exists a faithful normal conditional expectation $\E:M\to \mathcal{D}$ such that $\nu\circ\E=\nu$ (equivalently $\sigma_t^\nu(\D)=\D$ for all $t\in \mathbb{R}$). We say that a $\sigma$-weakly closed unital subalgebra  $\A$ of $\M$ is \emph{subdiagonal with respect to $\D$} if 
\begin{itemize}
\item $\A+\A^*$ is $\sigma$-weakly dense in $\M$,
\item $\mathcal{D}=\A\cap \A^*$,
\item and $\E$ is multiplicative on $\A$.
\end{itemize}
We will further write $\A_0$ for the ideal $\A\cap\mathrm{ker}(\mathbb{E})$ of $\A$.

\smallskip

\textbf{Note:} Some authors use the terminology \emph{subdiagonal with respect to $\A$} instead of the above.  
\end{definition}

\begin{remark} It easily follows that for any subalgebra $\A$ which is subdiagonal in the above sense, $\mathfrak{n}_\nu(\A)\cap\mathfrak{n}_\nu(\A^*)^*$ and $\mathfrak{n}_\nu(\A)\cap\mathfrak{n}_\nu(\A^*)^*$ are respectively $\sigma$-weakly dense in $\A$ and $\A_0$. This can be seen by selecting a net $(f_\lambda)\subset \D$ of positive analytic elements, satisfying the criteria of Proposition \ref{7:P Terp2}. If for example we are given some $a\in \A$, the fact that $\mathfrak{n}(\M)$ is a left-ideal ensures that $\{f_\gamma a f_\lambda\colon \lambda, \gamma\} \subset \mathfrak{n}(\A)\cap \mathfrak{n}(\A^*)^*$. In the case where $a\in \A_0$, the multiplicativity of $\mathbb{E}$ on $\A$ ensures that then $\{f_\gamma a f_\lambda\colon \lambda, \gamma\}\in \A_0$. The $\sigma$-weak closure of $\mathfrak{n}(\A)\cap \mathfrak{n}(\A^*)^*$ must for each fixed $\lambda$ contain the limit with respect to $\gamma$, namely $(af_\lambda)$. It must therefore also contain $a=\lim_\lambda af_\lambda$.  
\end{remark}

\subsection{Group von Neumann algebras and subdiagonality}

Successful as the theory may be, two challenges remain. We firstly need to show that this theory of noncommutative $H^p$ spaces even extends to general von Neumann algebras. As we shall see in the 
subsequent sections of this paper, when armed with the generalised version of the reduction theorem, this can be done quite successfully. However there is a second challenge that needs attention, and 
that is to find a way to further generalise the notion of subdiagonality to the point where it encompasses Hardy spaces of the upper half-plane. Despite the success of the theory, this has to date 
not been achieved. The problem we need to overcome here is how to deal with notions of $H^\infty$ spaces (described by some unital $\sigma$-weakly closed subalgebra of some von Neumann algebra) where the 
reference weight $\nu$ is not semifinite on $\D=H^\infty\cap (H^\infty)^*$. To gain some intuition on how this may be done, we turn to group von Neumann algebras. 

\subsubsection{Group von Neumann algebra essentials}

We briefly summarise the essentials of group von Neumann algebras of locally compact groups (or just LCGs) $G$. A fuller treatment may be found in \cite[\S 8.6]{GLbook}. We will write $C_r^*(G)$ for the reduced group $C^*$-algebra, and $\vngl$ for the left group von Neumann algebra. In keeping with convention we will wherever there is no danger of confusion simply refer to group von Neumann algebras and write $\vng$ instead of $\vngl$. We pause to biefly review some of these concepts:

For any element $s\in G$ we define an associated unitary $\lambda_s$ on $L^2(G)$ by means of the formula $$\lambda_s(h)(t)=h(s^{-1}t)\qquad h\in L^2(G).$$This is the so-called left-regular representation. More generally for any $f\in L^1(G)$ we may define the operator $\lambda_f$ by $$\lambda_f(h)(t)=(f*h)(t)=\int_G f(s)h(s^{-1}t)\, ds\qquad h\in L^2(G).$$We will write $\rho_s$ for the right-regular representation. With $\delta_G$ denoting the modular function on $G$ the right regular representation may be defined by $$\rho_s(h)(t)=\delta_G^{1/2}(s)h(st)\qquad h\in L^2(G).$$

The reduced group $C^*$-algebra is defined to be $C_r^*(G)=\overline{\lambda(L^1(G))}$, with $\vngl$ being the double commutant of this algebra. We will write $\mathscr{C}(G)$ for the $C^*$-algebra $\overline{\mathrm{span}\{\lambda_s: s\in G\}}$. 

Each group von Neumann algebra $\vng$ admits a canonical weight - the so-called \emph{Plancherel weight} (alt. left Haar weight) - which in a sense encodes left Haar integration at the algebra level. Some preparation is needed to see this. An element $f\in L^2(G)$ is called \emph{left-bounded} if  the formal prescription $\xi\mapsto f*\xi$ yields a bounded operator on $L^2(G)$. More precisely $f\in L^2(G)$ is said to be left-bounded if there exists a constant $C>0$ such that $\|f*g\|_2\leq C\|g\|_2$ for all $g\in C_c(G)$, where $C_c(G)$ denotes the space of continuous functions of compact support. The unique bounded extension of this densely defined operator to all of $L^2(G)$ will, as for the case where symbols are in $L^1(G)$, be denoted by $\lambda_f$. We then specifically have the following (see for example \cite[Definition VII.3.2]{Tak2}):

\begin{theorem}\label{T8.Plancherel} 
The prescription $$\psi_G(x^*x)=\left\{\begin{array}{ll} \|f\|_2^2 &  \quad\mbox{if }x=\lambda(f)\mbox{ for some left-bounded }f\in L^2(G) \\ \infty & \quad\mbox{otherwise}\end{array}\right.$$
defines a faithful normal semifinite weight on $\vng$.  
\end{theorem}

\begin{remark} A neighbourhood $U$ of the group unit $e$ is said to be \emph{symmetric} if $U=U^{-1}$ where $U^{-1}=\{a^{-1}\colon a\in U\}$. A very important fact for our purposes is that for any neighbourhood base $\mathcal{U}$ of the group unit $e$ which is made up of symmetric 
compact sets for which we have that $\{U\cdot U\colon U\in \mathcal{U}\}$ is again a neighbourhood base, the Plancherel weight $\psi_G$ is also given by 
$$\psi_G(x^*x)= \left\{ \begin{array}{cl} \int_G|f|(s)^2\,ds & \mbox{ if } x=\lambda_f \mbox{ for some left-bounded } f\in L^2(G) \\ \infty & \mbox{ otherwise }\end{array}\right.$$ and equivalently by 
$$\psi_G(x^*x)=\left\{\begin{array}{cl} \lim_{U}\omega_U(x^*x) &\mbox{ if }x=\lambda_f\mbox{ for some left-bounded }f\in L^2(G)\\ \infty & \mbox{ otherwise }\end{array}\right.$$where the $\omega_U$'s are vector functionals of the form 
$\omega_U=\frac{1}{|U|^2\gamma_G(U)}\langle(\cdot)\chi_U,\chi_U\rangle$ with $U\in \mathcal{U}$ and $\gamma_G(U)=\sup\{\delta_G(s)\colon s\in U\}$ and where the limit is taken as $U$ decreases to $\{e\}$. See equation (8.5) in the upcoming book of Goldstein and Labuschagne \cite{GLbook} for these facts.
\end{remark}

\subsubsection{Topologically ordered groups and subdiagonality}

We refer the reader to \cite{Folland} for details on locally compact groups. (See also \cite[\S 8.6]{GLbook} for a thumbnail introduction.) We shall here 
be interested in a very specific subclass of locally compact groups. It is well known that when given an ordered discrete group $G$, the $\sigma$-weakly closed 
subalgebra $\A$ generated by $\{\lambda_g\colon g\geq e\}$ is a subdiagonal subalgebra of the group von Neumann algebra $\vng$. For discrete groups the 
Plancherel weight $\psi_G$ is a tracial state on $\vng$. For the group $\mathbb{Z}$ the space $H^2(\A)$ is actually 
up to Fourier transform just $H^2(\mathbb{D})$. The same can be said about Hardy space of the upper half-plane: For $H^2(\mathbb{H})$ one will by the 
Paley-Wiener theorem similarly have that $\mathcal{F}(H^2(\mathbb{H}))\equiv L^2[0,\infty)$ (see for example \cite{LLK}). So by analogy with the previous 
example we here too may use an ordered group (in this case $\mathbb{R}$) to compute $H^2(\mathbb{H})$. These facts then strongly suggest that for a general 
 LCG $G$ it is right and proper to regard the $\sigma$-weak closure $\A$ of $\mathrm{span}\{\lambda_t\colon t\geq e\}$ in $\vng$ as $H^\infty$ of $G$. To get a workable 
theory we do however need to be very specific about what we mean by an ordered group. Not just any group admitting a left-ordering will do. We in particular need 
groups where the ordering harmonises with the topology, namely \emph{topologically ordered groups}. Readers are referred to the papers of Nyikos and Reichel 
\cite{Nyikos} and Venkataraman, Rajagopalan and Soundararajan \cite{Venkat} for background regarding these groups. As pointed out in \cite{Nyikos}, for such 
groups sets of the form $V_a=\{s\in G\colon s<a\}$ and $W_b=\{s\in G\colon s>b\}$ form a subbase for the topology of $G$. To simplify notation we shall when such 
groups are in view with mild abuse of notation simply write $(-\infty, a)$ and $(b,\infty)$ for $V_a$ and $W_b$ respectively, and $(-\infty, a]$, $[b,\infty)$ and 
$[b,a]$ for the order intervals $\{s\in G\colon s\leq a\}$, $\{s\in G\colon s\geq b\}$ and $\{s\in G\colon b\leq s\leq a\}$. It is shown in 
\cite[2.5 \& 5.6]{Venkat} that 
\begin{itemize}
\item if such a group is totally disconnected it is either discrete or contains an open subgroup which is homeomorphic to the Cantor set,
\item and that if it is infinite and not totally disconnected it must as a topological space be homeomorphic to a semi-direct product 
$\mathbb{R}\rtimes_\alpha \Gamma$ where $\mathbb{R}$ is the additive reals and $\Gamma$ a discrete group. 
\end{itemize} 

\subsubsection{The emergent structure} 

Let $G$ be a topologically ordered group. Recall that by \cite{Nyikos} sets of the form $(-\infty,a)$ and $(b,\infty)$ form a subbase for the topology of 
$G$ \cite{Nyikos}. Since the groups we are interested are all homeomorphic to either $\mathbb{R}$, $\Gamma$ or $\mathbb{R}\rtimes_\alpha \Gamma$ where 
$\Gamma$ is a discrete group, we can select the neighbourhood base $\mathcal{U}$ required by the theorem mentioned in the previous subsubsection 
to be such that each $U\in \mathcal{U}$ is of the form $U=[a_U^{-1},a_U]$ for some $a_U>e$. For each fixed $U=[a_U^{-1},a_U]$ we may then let $\A_U$ be 
the $\sigma$-weak closure of $\mathrm{span}(\{\lambda_e\}\cup\{\lambda_t\colon t\geq a_U\})$. Then $\{\A_U\}$ is an increasing net of $\sigma$-weakly closed subalgebras of 
$\A$ for which the following holds: 
\begin{itemize}
\item $\cup_U\A_U$ is dense in $\A$. 
\item $\A+\A^*$ is $\sigma$-weakly dense in the group algebra.
\item For any $U$ there exists a neighbourhood $V_U\subset U$ of $e$ such that the state $\frac{1}{\omega_V(\I)}\omega_V$ will for any $V\subseteq V_U$ be 
multiplicative on $\A_{U}$. If we select $V_U$ small enough so that $V_U\cap tV_U=\emptyset$ for any $t\geq a_U$, we will for any 
$x=\sum_{i=1}^n\alpha_i\lambda_{t_i}$ with $t_1=e$ and $t_i\geq a_U$ have that $\frac{1}{|V_U|}\langle x\chi_{V_U}, \chi_{V_U}\rangle = \alpha_1$. From this it follows that $\frac{1}{\omega_{V_U}(\I)}\omega_{V_U}$ is multiplicative on 
$\mathrm{span}(\{\lambda_e\}\cup\{\lambda_t\colon t\geq a_U\})$ and by continuity therefore on $\A_U$. The same also goes for any $V$ 
smaller than $V_U$. Thus the map $x\to \frac{1}{\omega_{V_U}(\I)}\omega_{V_U}(x)\I$ is a normal conditional expectation from ${\mathrm{VN}}_l(G)$ to 
$\mathbb{C}\I$ which is multiplicative on $\A_U$.
\item Now suppose that $a_{\widetilde{U}}<a_U$. As above we need to select $V_{\widetilde{U}}$ small enough so that $\frac{1}{\omega_V(\I)}\omega_V$ will for any $V\subseteq V_{\widetilde{U}}$ be 
multiplicative on $\A_{\widetilde{U}}$. Since $a_{\widetilde{U}}<a_U$, we will have that $\A_U\subset\A_{\widetilde{U}}$. As above the map $x\to \frac{1}{\omega_{V_{\widetilde{U}}}(\I)}\omega_{V_{\widetilde{U}}}(x)\I$ is a normal conditional expectation which is multiplicative on $\A_{\widetilde{U}}$ and which will send terms of the form 
$x=\sum_{i=1}^n\alpha_i\lambda_{t_i}$ with $t_1=e$ and $t_i\geq a_{\widetilde{U}}$ to $\alpha_1\I$. That means that on $\mathrm{span}(\{\lambda_t\colon t\geq a_U\}\cup\{\lambda_e\})$ the action of this 
expectation agrees with that of the corresponding expectation in the previous bullet. Since the space $\mathrm{span}(\{\lambda_t\colon t\geq a_U\}\cup\{\lambda_e\})$ is $\sigma$-weakly dense in $\A_U$, this 
expectation must therefore on $\A_U$ agree with the expectation constructed in the previous bullet. 
\item For any $U\in \mathcal{U}$ we have that $\A_U^*\cap\A_U=\mathbb{C}\I=\A^*\cap\A$.
\item The set $\mathfrak{n}(\A)+\mathfrak{n}(\A^*)$ embeds norm-densely into the GNS Hilbert space $H_{\varphi_G}$. This follows from the fact that the GNS Hilbert space is a copy of $L^2(G)$ with $\psi_G(\lambda(f)^*\lambda(f))=\|f\|^2_2$ for each left bounded element of $L^2(\M)$, combined with the fact that the continuous functions of compact support are dense in $L^2(G)$ and that each such function $f$ is almost everywhere equivalent to the sum $f\chi_{(e,\infty)} + f\chi_{(-\infty,e)}$. 
\end{itemize}

\begin{proposition}\label{A+A_0} For any $a > e$ and $U=[a,\infty)$, the sets $\{ \lambda_t \colon t \geq a\}$ and $\{\lambda(f)\colon f\in L^1(G), \mathrm{supp}(f)\subset [a,\infty)\}$ generate the same $\sigma$-weakly closed subalgebra (which we shall denote by $\A_{U,0}$. Similarly the sets $\{\lambda_t\colon t > e\}$ and $\{\lambda(f)\colon f\in L^1(G), \mathrm{supp}(f)\subset (e,\infty)\}$ generate the same $\sigma$-weakly closed subalgebra which we shall denote by $\A_0$. \end{proposition}

\begin{proof} On noting that for any neighbourhood of $e$ we have that $\lambda_x\chi_U=\chi_{xU}$, the proof of \cite[Theorem 3.12(a)]{Folland} easily adapts.
\end{proof}

\begin{proposition} For any $a > e$ and $U=[a,\infty)$, the modular automorphism group $\sigma_t^{\psi_G}$ preserves $\A_U$. \end{proposition}

\begin{proof} Recall that for any $y\in G$ we have $\sigma_t(\lambda_y)=\delta_G(y)^{it}\lambda_y$. Given $f\in L^1(G)$ we have 
 $$\lambda(f)(g)=\int f(y)g(y^{-1}\cdot)\,dt = \int f(y)\lambda_y(g)\,dt.$$ 
So in effect $\lambda(f)=\int f(y)\lambda_y\,dt$. But then 
 $$\sigma_t(\lambda(f))= \int f(y)\sigma_t(\lambda_y)\,dt = \int \delta_G^{it}(y)f(y)\lambda_y\,dt=\lambda(\delta_G^{it}f).$$ 
We clearly have that $\delta_G^{it}f$ is supported on $[a,\infty)$ iff $f$ is supported on $[a,\infty)$. The claim therefore follows from the previous proposition.
\end{proof}

\begin{proposition}\label{cts-topord} If $G$ is a topologically ordered continuous group then $\A=\A_0$. \end{proposition}

\begin{proof} If $G$ is a continuous group, there exists a net $(t_\gamma)\subset (e,\infty)$ converging to $e$. By \cite[Prop 2.6]{Folland} we will for any $C_c(G)$ have that 
$\lambda_{t_\gamma}(f)\to f$ uniformly, and hence also in $L^2$-norm. It is now an exercise to use these facts to show that $(\lambda_{t_\gamma})$ is strongly convergent to $\lambda_e$. 
The result follows.
\end{proof}

\subsection{Defining approximate subdiagonality}

In the ``standard'' definition of subdiagonality the fact that the canonical expectation $\mathbb{E}$ onto $\D$ satisfies $\nu\circ\mathbb{E}$ will when combined with the fact that $\mathbb{E}(a^*)\mathbb{E}(a)\leq \mathbb{E}(a^*a)$ for all $a\in \M$ ensure that in this case $\mathfrak{n}(\D)\subset\mathfrak{n}(\A)$, or equvalently $\mathfrak{n}(\A) =\mathfrak{n}(\D)\oplus\mathfrak{n}(\A_0)$. As was shown by Arveson, this encompasses group von Neumann algeras of discrete ordered groups. By Proposition \ref{cts-topord} essentially the same claim holds in the case of a continuous topologically ordered group, but now with $\mathfrak{n}(\D)=\{0\}$. If we formalise the preceding structure into an abstract definition, we obtain the following which is then a context for subdiagonality for general vNAs which can accommodate anti-symmetry for non-$\sigma$-finite contexts.

\begin{definition}\label{def-appsubd} Let $\M$ be a von Neumann algebra equipped with an \emph{fns} weight $\nu$. A $\sigma$-weakly closed unital subalgebra $\A$ of $\M$ is said to be \emph{approximately subdiagonal} with respect to $\\A\cap\A^*$ if there exists a net $(\mathcal{A}_\gamma, c_\gamma, \nu_\gamma, P_\gamma)$ where the $\mathcal{A}_\gamma$'s are $\sigma$-weakly closed subalgebras of $\mathcal{A}$, the $c_\gamma$'s are increasing positive constants, the $\nu_\gamma$'s normal states and the $P_\gamma$'s normal conditional expectations from $\mathcal{M}$ onto $\mathcal{D}_\gamma=\mathcal{A}_\gamma^*\cap\mathcal{A}_\gamma$ such that the following holds:
\begin{enumerate}
\item $\A+\A^*$ is $\sigma$-weakly dense in $\M$.
\item $\mathfrak{n}(\mathcal{A})+\mathfrak{n}(\mathcal{A}^*)$ densely embeds into the GNS Hilbert space $H_\nu$.
\item The triples $(\mathcal{A}_\gamma,\nu_\gamma, P_\gamma)$ are ``subdiagonal-like'' subalgebras with $\A$ appearing as the inductive limit of these triples in the sense described below:
\begin{itemize}
\item $\{\mathcal{A}_\gamma\}$ is an increasing net of subalgebras with $\overline{\cup_\gamma\mathcal{A}_\gamma}^{w*}=\mathcal{A}$. Similarly the subalgebras $\mathcal{A}_\gamma^*\cap\mathcal{A}_\gamma=\mathcal{D}_\gamma$ increase to $\mathcal{D}=\mathcal{A}^*\cap\mathcal{A}$.
\item $\nu(x^*x)= \sup_\gamma c_\gamma\nu_\gamma(x^*x)$ for all $x\in \mathcal{M}$ with $\nu(x^*x)= \lim_\gamma c_\gamma\nu_\gamma(x^*x)$ holding if $x\in \mathfrak{n}_\nu(\mathcal{M})$. 
\item Each $\nu_\gamma$ restricts to a faithful normal state on $\mathcal{D}_\gamma$.
\item $\nu_\gamma= \nu_\gamma\circ P_\gamma$ for all $\gamma$, each $P_\gamma$ is multiplicative on $\mathcal{A}_\gamma$ and if $\gamma\geq \alpha$ then also $P_\gamma{\upharpoonright}\mathcal{A}_\alpha = P_\alpha$.
\end{itemize}
\item The collection $\mathcal{A}_{\gamma,0}=\mathcal{A}_\gamma\cap\mathrm{ker}(P_\gamma)$ of ideals in $\mathcal{A}_\gamma$ increases to the $\sigma$-weakly closed ideal $\mathcal{A}_0=\overline{\cup_\gamma \mathcal{A}_{\gamma,0}}^{w*}$ of $\mathcal{A}$.
\item We also have that $\mathfrak{n}_\nu(\A_{\gamma})\cap\mathfrak{n}_\nu(\A_{\gamma}^*)^*$ and $\mathfrak{n}_\nu(\A_{\gamma,0})\cap\mathfrak{n}_\nu(\A_{\gamma,0}^*)^*$ are each respectively $\sigma$-weakly dense in $\A_\gamma$ and $\A_{\gamma,0}$.
\item Each of $\mathcal{A}_\gamma$, $\mathcal{A}_{\gamma,0}$ and $\mathcal{D}_\gamma$ (and therefore also $\mathcal{A}$, $\mathcal{A}_0$ and $\mathcal{D}$) is preserved by the modular automorphism group $\sigma_t^\nu$.
\end{enumerate}
\end{definition}

\textbf{NOTE:} In the above definition one can more generally assume that the $\nu_\gamma$'s are not states but normal semifinite weights for which $s(\nu_\gamma)\A_\gamma s(\nu_\gamma)\cap s(\nu_\gamma)\A_\gamma^* s(\nu_\gamma)$ is an expected subalgebra of $s(\nu_\gamma)\M s(\nu_\gamma)$. However this change does not seem to have any advantages.

\begin{remark}
\begin{enumerate}
\item In the above definition the criteria in requirement (3) ensure that each $P_\gamma$ maps $\mathfrak{n}_\nu(\A_\gamma)$ back into itself. It therefore follows that $\mathfrak{n}_\nu(\A_\gamma) = \mathfrak{n}_\nu(\D_\gamma) \oplus \mathfrak{n}_\nu(\A_{\gamma,0})$. To see this let $a\in\mathfrak{n}_\nu(\A_\gamma)$ be given. For any $\alpha\geq \gamma$ we therefore have that 
$$\nu_\alpha(P_\gamma(a)^*P_\gamma(a)) = \nu_\alpha(P_\alpha(a)^*P_\alpha(a)) \leq \nu_\alpha(P_\alpha(a^*a)) = \nu_\alpha(a^*a).$$ This clearly ensures that 
$$\nu(P_\gamma(a)^*P_\gamma(a)) =\sup_\alpha c_\alpha\nu_\alpha(P_\gamma(a)^*P_\gamma(a))\leq \sup_\alpha c_\alpha\nu_\alpha(a^*a)=\nu(a^*a)<\infty$$as required.
\item It follows from the above criteria that $\A_0$ is an $\A$-ideal. Given $a\in \A$ and $b\in\A_0$, there exist nets $(a_\alpha)\subset \cup_\gamma\A_\gamma$ and $(b_\beta)\subset \cup_\gamma\A_{\gamma,0}$ respectively converging $\sigma$-weakly to $a$ and $b$. The multiplicativity of the $P_\gamma$'s ensure that each of the products $a_\alpha b_\beta$ belongs to $\cup_\gamma\A_{\gamma,0}$. On respectively taking the limit with respect to $\alpha$ and then $\beta$, it follows that $ab\in \A_0$. A similar argument shows that also $ba\in \A_0$. 
\end{enumerate}
\end{remark}

Criterion (2) of the definition allows for the following refinement:

\begin{proposition}\label{refine-2} Let $\A$ be approximately subdiagonal. Given any $2\leq q <\infty$ we then have that $\mathfrak{j}^{(q)}(\cup_\gamma (\mathfrak{n}_\nu(\A_{\gamma})\cap\mathfrak{n}_\nu(\A_{\gamma}^*)^*))$ and $\mathfrak{j}^{(q)}(\cup_\gamma (\mathfrak{n}_\nu(\A_{\gamma,0})\cap\mathfrak{n}_\nu(\A_{\gamma,0}^*)^*))$ are respectively norm dense in $\mathfrak{j}^{(q)}(\mathfrak{n}_\nu(\A))$ and $\mathfrak{j}^{(q)}(\mathfrak{n}_\nu(\A_0))$.
\end{proposition}

\begin{proof} We shall only prove the second claim. Let $b\in\n(\A_0)$ be given. We may clearly select nets $(a_\gamma)\subset \cup_\gamma (\mathfrak{n}_\nu(\A_{\gamma})\cap\mathfrak{n}_\nu(\A_{\gamma}^*)^*)$ and $(b_\beta)\subset \cup_\gamma (\mathfrak{n}_\nu(\A_{\gamma,0})\cap\mathfrak{n}_\nu(\A_{\gamma,0}^*)^*)$ 
respectively converging $\sigma$-weakly to $\I$ and $b$. By the proof of Proposition \ref{analytic-n(A)} we may for each $\alpha$ and $\beta$ select sequences $(a_{\alpha, n}) \subset \cup_\gamma \mathfrak{n}^\infty_\nu(\A_{\gamma})$ and $(b_{\beta,n}) \subset \cup_\gamma \mathfrak{n}^\infty_\nu(\A_{\gamma,0})$ converging $\sigma$-weakly to $a_\alpha$ and $b_\beta$ respectively. We will 
clearly have that each $b_{\beta,m}\sigma^\nu_{-i/q}(a_{\alpha,n})$ belongs to $\cup_\gamma (\mathfrak{n}_\nu(\A_{\gamma,0})\cap\mathfrak{n}_\nu(\A_{\gamma,0}^*)^*)$. Noting that 
$$\mathfrak{j}^{(q)}(b_{\beta,m}\sigma^\nu_{-i/q}(a_{\alpha,n})) = [(b_{\beta,m}\sigma^\nu_{-i/q}(a_{\alpha,n}))h^{1/q}]=b_{\beta,m}[\sigma^\nu_{-i/q}(a_{\alpha,n})h^{1/q}]$$it is clear 
that as $m\to \infty$ the sequence $(\mathfrak{j}^{(q)}(b_{\beta,m}\sigma^\nu_{-i/q}(a_{\alpha,n})))_{m\in\mathbb{N}}$ will for each fixed $\alpha, \beta$ and $n$ converge weakly to 
$b_\beta[\sigma^\nu_{-i/q}(a_{\alpha,n})h^{1/q}]$. The net $(b_\beta[\sigma^\nu_{-i/q}(a_{\alpha,n})h^{1/q}])_\beta$ will in turn converge weakly to 
$b[\sigma^\nu_{-i/q}(a_{\alpha,n})h^{1/q}]= [b\sigma^\nu_{-i/q}(a_{\alpha,n})h^{1/q}]$. For each fixed $\alpha$ and $n$, we may apply Lemma \ref{GL2-2.4+5} to see that 
$$[b\sigma^\nu_{-i/q}(a_{\alpha,n})h^{1/q}] = b[\sigma^\nu_{-i/q}(a_{\alpha,n})h^{1/q}] = b(h^{1/q}a_{\alpha,n})= [bh^{1/q}]a_{\alpha,n}.$$Now firstly note that for each fixed $\alpha$ the sequence $([bh^{1/q}]a_{\alpha,n})_n$ will converge weakly to $[bh^{1/q}]a_{\alpha}$ and that the net $([bh^{1/q}]a_{\alpha})$ will converge weakly to $[bh^{1/q}]$. We therefore have that $[bh^{1/q}]$ belongs to the weak closure of $\mathfrak{j}^{(q)}(\mathfrak{n}_\nu(\A_0))$ in $L^q(\M)$. Since convexity ensures that the norm and weak closures agree, we are done.  
\end{proof}

We close this section with an analysis of the relation between maximal subdiagonality and approximate subdiagonality.

\begin{proposition}\label{subd-vs-appsubd}
Any maximal subdiagonal subalgebra $\A$ of a von Neumann algebra is approximately subdiagonal.
\end{proposition}

\begin{proof} 
Firstly let $(\M,\A,\mathbb{E}, \nu)$ be a ``standard'' maximal subdiagonal quadruple. That is let $\M$ be a von Neumann algebra equipped with a faithful normal semifinite weight $\nu$, and 
$\mathcal{D}$ be a unital von Neumann subalgebra of $\M$ such that $\nu{\upharpoonright}\mathcal{D}$ is semifinite. Further suppose that there exists a faithful normal conditional expectation 
$\mathbb{E}:M\to \mathcal{D}$ such that $\nu\circ\mathbb{E}=\nu$ (equivalently $\sigma_t^\nu(\D)=\D$ for all $t\in \mathbb{R}$), and that $\A$ is a $\sigma$-weakly closed unital subalgebra $\A$ of $\M$ for 
which  
\begin{itemize}
\item $\A+\A^*$ is $\sigma$-weakly dense in $\M$,
\item $\mathcal{D}=\A\cap \A^*$,
\item $\sigma_t^\nu(\A)=\A$ for all $t$,
\item and $\mathbb{E}$ is multiplicative on $\A$.
\end{itemize}
We may now use Haagerup's description of normal weights to select a net $(c_\gamma, \omega_\gamma)$ of positive scalars and normal states on $\D$ such that $c_\gamma\omega_\gamma$ increases to $\nu{\upharpoonright}\mathcal{D}$. If we now define $\nu_\gamma$ to be $\omega_\gamma\circ \mathbb{E}$, it is then an exercise to see that $c_\gamma\nu_\gamma$ increases to $\nu$. If we further set $\A_\gamma=\A$ and $P_\gamma=\mathbb{E}$, and take note of the claims of Proposition \ref{GNSPropn}, it is clear that this structure satisfies the criteria of the previous definition.
\end{proof}

Given an approximately subdiagonal subalgebra, we will for any $x\in \A_\gamma$ have that $\lim_\alpha P_\alpha(x) = P_\gamma(x)$. Also for any net 
$(x_i)\subset \cup_\gamma(\A_\gamma+\A_\gamma^*)$ converging to say $x\in \A_\gamma$ we will have that $\lim_iP_\alpha(x_i)=P_\alpha(x)$. If for a specific 
approximately subdiagonal subalgebra we in fact have that $\lim_i\lim_\alpha P_\alpha(x_i)=P_\gamma(x)$ and $\D\neq\{0\}$, then that algebra is maximal 
subdiagonal.

\begin{proposition}\label{P.subd-vs-appsubd}
Let $\A$ be an approximately subdiagonal subalgebra of $\M$. If for some $\gamma$ the expectation $P_\gamma$ is non-zero on $\A_\gamma$, the algebra $\D$ is an expected subalgebra of $\M$. If additionally for any net $(x_i)\subset \cup_\gamma(\A_\gamma+\A_\gamma^*)$ converging to say $x\in \cup_\gamma(\A_\gamma+\A_\gamma^*)$ we have that $\lim_i\lim_\alpha P_\alpha(x_i)=\lim_\alpha P_\alpha(x)$, $\A$ is a maximal subdiagonal subalgebra of $\M$. 
\end{proposition} 

\begin{proof}
Suppose that we are given a unital $\sigma$-weakly closed subalgebra $\A$ of $\M$ satisfying the criteria of Definition \ref{def-appsubd}. If for some $\gamma$ the expectation $P_\gamma$ 
is non-zero on $\A_\gamma$, then by the $\sigma$-weak continuity of $P_\gamma$, $P_\gamma(\n_\nu(\A_\gamma)) = \n_\nu(\D_\gamma)$ will be $\sigma$-weakly 
dense in $\D_\gamma$. Since for any $\alpha\geq \gamma$ we moreover have that $P_\alpha{\upharpoonright}\A_\gamma=P_\gamma$, each such $P_\alpha$ will similarly be non-zero 
on $\A_\alpha$ and hence $\mathfrak{n}(\D_\alpha)_\nu$ $\sigma$-weakly dense in $\A_\alpha$. Thus $\cup_\gamma \n_\nu(\D_\gamma)$ (and therefore also $\n_\nu(\D)$) is then in turn $\sigma$-weakly dense in $\D$ (the $\sigma$-weak closure of 
$\cup_\gamma D_\gamma$). But then $\D$ is an expected algebra since $\nu{\upharpoonright}\D$ is semifinite and $\sigma_t^\nu(\D)=\D$ for all $t$. Having verified this fact we now let $\mathbb{E}$ be the canonical expectation of $\M$ onto $\D$ with respect to $\nu$.  

The second claim follows by comparing this expectation to the $P_\gamma$'s. Firstly note that criterion (7) of Definition \ref{def-appsubd} ensures for any 
$\alpha\geq \gamma$ we have that $P_\alpha$ restricted to $\A_\gamma +\A_\gamma^*$ equals the action of $P_\gamma$ on the same space. The pairs 
$(P_\gamma, \A_\gamma +\A_\gamma^*)$ form a type of inductive limit which we can use to uniquely define a contractive operator $P$ on the $\sigma$-weakly dense subspace 
$\cup_\gamma(\A_\gamma +\A_\gamma^*)$ of $\M$ which is multiplicative on the $\sigma$-weakly dense subalgebra $\cup_\gamma \A_\gamma$ of $\A$. This map is moreover an 
idempotent mapping $\cup_\gamma(\A_\gamma +\A_\gamma^*)$ onto the $\sigma$-weakly dense subspace $\cup_\gamma\D_\gamma$ of $\D$.  

We leave it as an exercise to show that each $\n(\A_\gamma)^*\n(\A_\gamma^*)+\n(\A_\gamma^*)^*\n(\A_\gamma)$ is $\sigma$-weakly dense in $\A_\gamma^*+\A_\gamma$. 
Criteria (3) of Definition \ref{def-appsubd} ensures that we will for any $\alpha\geq \gamma$ have that $P_\alpha$ will in its action on 
$\n(\A_\gamma)^*\n(\A_\gamma^*)+\n(\A_\gamma^*)^*\n(\A_\gamma)$ satisfy $\nu_\alpha\circ P_\alpha=\nu_\alpha\circ P_\gamma$. It therefore also follows from 
criterion (3) that we will for any $\gamma$ have that $\nu\circ P_\gamma=\nu$ on $\n(\A_\gamma)^*\n(\A_\gamma^*)+\n(\A_\gamma^*)^*\n(\A_\gamma)$. We therefore 
have that $\nu\circ P=\nu$ on the $\sigma$-weakly dense subspace $\cup_\gamma \n(\A_\gamma)^*\n(\A_\gamma^*)+\n(\A_\gamma^*)^*\n(\A_\gamma)$ of $\M$. So if indeed $P$ 
turned out to be $\sigma$-weakly continuous and was extendible to $\M$, its extension could only be $\mathbb{E}$ which would then in turn have to be multiplicative on $\A$. 
The content of the rest of the proof consists of showing that this is indeed the case.

The limiting condition stated in the hypothesis is precisely what is needed to ensure that $P$ is indeed continuous with respect to the the relative $\sigma$-weakly topology 
on $\cup_\gamma(\A_\gamma+\A_\gamma^*)$. By compactness the unit ball of $\M$ is complete under the $\sigma$-weak topology. 

By Theorem 3.3.3 of \cite{Jarchow} $\M$ and $\cup_\gamma(\A_\gamma+\A_\gamma^*)$ will when equipped with the $\sigma$-weak topology have the same completion. Theorem 3.4.2 
of \cite{Jarchow} then ensures that $P$ extends to a continuous map $\widetilde{P}$ from the completion of $\M$ to the completion of $\D$. We proceed to show that 
$\widetilde{P}$ maps $\M$ into $\D$. Given any $x\in \M$ we may select a net $(x_i)\subset \cup_\gamma(\A_\gamma+\A_\gamma^*)$ converging $\sigma$-weakly to $x$. By the 
Banach-Steinhaus theorem any $\sigma$-weakly convergent net $(x_i)$ is norm bounded. We may therefore without loss of generality assume that $(x_i)\subset 
Ball(\cup_\gamma(\A_\gamma+\A_\gamma^*))$. Since each $P_\gamma$ (and therefore also $P$) is contractive, $P$ will map the net $(x_i)$ onto a net 
$(P(x_i))\subset Ball(\cup_\gamma(\D_\gamma))\subset Ball(\D)$. The continuity of $\widetilde{P}$ ensures that in the completion of $\D$ this net converges to 
$\widetilde{P}(x)$. However by the $\sigma$-weak compactness of $Ball(\D)$, the net $(P(x_i))$ must have a subnet converging to some element of $Ball(\D)$. By uniqueness 
of limits, this limit must agree with the $\widetilde{P}(x)$; in other words we must have that $\widetilde{P}(x)\in \D$. Thus $\widetilde{P}$ restricts to a $\sigma$-weakly 
continuous map from $\M$ to $\D$. As noted earlier, this extension can only be $\mathbb{E}$ which must then be multiplicative on $\A$.
This then suffices to prove that $\A$ is a maximal subdiagonal subalgebra of $\M$. 
\end{proof}

On the basis of the preceding proposition we now introduce the following concept.

\begin{definition}\label{def-regappsubd}
We say that an approximately subdiagonal subalgebra is \emph{regular} if either it is actually maximal subdiagonal or otherwise $\A=\A_0$.
\end{definition}

\section{Maximality of subdiagonal subalgebras of general von Neumann algebras}\label{S6}

In this section we shall use the reduction theorem developed previously to characterise maximality of subdiagonal sublagebras of a general von Neumann algebra. Exel showed that in the case where $\nu$ is a tracial state, all $\sigma$-weakly closed subdiagonal algebras are automatically maximal \cite{Exel}. Ji then later showed that this claim still holds in the case where $\nu$ is a faithful normal semifinite trace (\cite{jig1},cf. \cite[Corollary 4.2]{Xu}). For the non tracial $\sigma$-finite case, we have the following result.

\begin{theorem}[\cite{Xu}, \cite{JOS}] Let $\A$ be a $\sigma$-weakly closed unital subalgebra of $\M$ with $\mathcal{D}$ and $\E$ as before, and assume that additionally $\A+\A^*$ is $\sigma$-weakly dense in $\M$. Then $\A$ is maximal as a subdiagonal subalgebra with respect to $\mathcal{D}$ if and only if  $\sigma_t^\nu(\A)=\A$ for all $t\in \mathbb{R}$.
\end{theorem}
 
The non-trivial proof of this fact makes use of both the reduction theorem and a very elegant maximality criterion of Arveson. Our first objective is to show that the above result holds in the general possibly non-$\sigma$-finite case as well. There are two obstacles to overcome in order to achieve this objective. 
\begin{itemize}
\item The version of the reduction theorem used by Xu in his contribution to the above theorem required the weight $\nu$ to be strictly semifinite on 
$\mathcal{D}$. With the alternative version of the reduction theorem now at our disposal, we are able to remove this restriction. 
\item A further obstacle may be found in Arveson's seminal paper \cite{AIOA}. In that paper Arveson posited a very elegant maximality criterion for subdiagonal algebras \cite[Corollary 2.2.4]{AIOA}, which was exploited by Ji, Ohwada and Saito in their contribution. However Arveson proved his maximality criterion under the assumption that the underlying Hilbert space is separable (and the reference von Neumann algebra therefore $\sigma$-finite). Thus we also need to go back to Arveson's original result and see if it is equally valid in the general case. 
\end{itemize} 

\subsection{Arveson's maximality criterion}

As mentioned earlier Arveson proved his maximality criterion for subdiagonal subalgebras under the assumption that the underlying Hilbert space is separable. Those techniques do not carry over directly to the general case. Here we show that an invocation of the Haagerup-Terp standard form ensures that Arveson's maximality criterion also holds in the general case. If $\A$ is subdiagonal, it is clear that the $\sigma$-weak closure will also be subdiagonal. So in attempting to characterise maximality, we may clearly suppose that $\A$ is $\sigma$-weakly closed.  We shall therefore henceforth restrict attention to $\sigma$-weakly closed subdiagonal subalgebras. In this section we establish the validity of Arveson's maximality criterion for general von Neumann algebras. We note the following crucial fact regarding $\mathfrak{n}(\A_0)$.

\begin{proposition}\label{GNSPropn} Let $\A$ be a $\sigma$-weakly closed subalgebra of $\M$ which is subdiagonal with respect to the expected subalgebra $\mathcal{D}$. Then $\mathfrak{n}(\A_0)$ is $\sigma$-weakly dense in $\A_0$, and $\mathfrak{n}(\A)+\mathfrak{n}(\A^*)$ therefore $\sigma$-weakly dense in $\M$. Moreover $\eta(\mathfrak{n}(\A)+\mathfrak{n}(\A^*))$ is a norm-dense subspace of $H_\nu$ where $H_\nu$ is the Hilbert space produced from the GNS-construction for $\nu$. (Recall that we may equip $\mathfrak{n}_\nu$ with an inner product defined by $\langle a, b \rangle =\nu(b^*a)$, and that $H_\nu$ is the completion of $\mathfrak{n}_\nu$ under the associated norm. The map $\eta$  is then the natural embedding of $\mathfrak{n}_\nu$ into $H_\nu$.)  
\end{proposition} 

\begin{proof}
Since by assumption $\nu$ is semifinite on $\mathcal{D}$, we may select the net $(f_\lambda)$ described in Proposition \ref{7:P Terp2} to be in $\mathcal{D}$. Given any $a\in \A_0$ it is now clear that $(af_\lambda)\subset \mathfrak{n}(\A_0)$, with $(af_\lambda)$ converging strongly to $a$. But the strong and $\sigma$-strong topologies agree on norm bounded sets, and hence the convergence is $\sigma$-strong. Thus as claimed $\mathfrak{n}(\A_0)$ is $\sigma$-weakly dense in $\A_0$. A similar claim clearly holds for $\A_0^*$. The fact that $\nu$ is semifinite on $\mathcal{D}$ ensures that $\mathfrak{n}(\mathcal{D})$ is $\sigma$-weakly dense in $\mathcal{D}$, and hence that $\mathfrak{n}(\A)+\mathfrak{n}(\A^*)=\mathfrak{n}(\A_0)+\mathfrak{n}(\mathcal{D})+\mathfrak{n}(\A_0^*)$ is $\sigma$-weakly dense in $\A+\A^*$ and therefore also in $\M$.

We pass to proving the claim about the norm density of $\eta(\mathfrak{n}(\A)+\mathfrak{n}(\A^*))$. The space $\A+\A^*$ is clearly convex. Hence this subspace is even $\sigma$-strong* dense. Given $b\in \mathfrak{n}_\nu$, we may therefore select a net $(a_\gamma)\subset \mathfrak{n}(\A)+\mathfrak{n}(\A^*)$ converging $\sigma$-strong* to $b$. This ensures that for any fixed $\lambda$ and any 
$x\in H_\nu$, we have that $$\langle\eta(a_\gamma f_\lambda),x\rangle=\langle a_\gamma\eta(f_\lambda),x\rangle\to\langle b\eta(f_\lambda),x\rangle= \langle\eta(bf_\lambda),x\rangle.$$Since the net $(a_\gamma f_\lambda)_\gamma$ belongs to $\mathfrak{n}(\A)+\mathfrak{n}(\A^*)$, this ensures that for any $b\in \mathfrak{n}_\nu$, the weak closure of $\eta(\mathfrak{n}(\A)+\mathfrak{n}(\A^*))$ includes the net $(\eta(bf_\lambda))$. If we are able to show that such a net converges weakly $\eta(b)$, it will follow that $\eta(\mathfrak{n}(\A)+\mathfrak{n}(\A^*))$ is weakly dense in $\eta(\mathfrak{n}_\nu)$. But by convexity, the weak and norm closure of $\eta(\mathfrak{n}(\A)+\mathfrak{n}(\A^*))$ must agree. Thus this will then show that $\eta(\mathfrak{n}(\A)+\mathfrak{n}(\A^*))$ is norm dense in $H_\nu$ as claimed.

To prove that for any $b\in \mathfrak{n}_\nu$ the net $(\eta(bf_\lambda))$ converges weakly to $\eta(b)$, we pass to the Haagerup-Terp standard form for $\M$ with respect to $\nu$ (see Theorem \ref{7:T stdform}). Let $h$ be the density of the dual weight $\widetilde{\nu}$ on the crossed product $\M\rtimes_\nu\mathbb{R}$ with respect to the canonical trace on this crossed product. Recall that $bf_\lambda h^{1/2}$ is then pre-measurable with the closure belonging to $L^2(M)$ (see Proposition \ref{GL2-2.2+3}). We use the fact that in this standard form, the closure $[bh^{1/2}]$ where $b\in\n_\nu$, corresponds to $\eta(b)$.  In the context of this standard form we will then for any $x\in L^2(M)$ have that $\langle\eta(b f_\lambda),x\rangle= tr(x^*[bf_\lambda h^{1/2}])$. We may now further apply Proposition \ref{7:P Terp2} and Lemma \ref{GL2-2.4+5} to this equality to see that 
\begin{eqnarray*}
\langle\eta(bf_\lambda),x\rangle &=& tr(x^*[bf_\lambda h^{1/2}])\\
&=& tr(x^*b[f_\lambda h^{1/2}])\\
&=&tr(x^*b(h^{1/2}\sigma_{i/2}(f_\lambda)))\\
&=&tr(x^*[bh^{1/2}]\sigma_{i/2}(f_\lambda))\\
&\to& tr(x^*[bh^{1/2}])\\
&=&\langle \eta(b),x\rangle.
\end{eqnarray*} 
Thus as required, $\eta(bf_\lambda)$ converges weakly to $\eta(b)$ 
\end{proof}

\begin{definition}\label{Arv-max} Let $\A$ be a $\sigma$-weakly closed subalgebra of $\M$ which is subdiagonal with respect to the expected subalgebra $\mathcal{D}$. We then define $\A_m$ to be the subspace $\A_m=\{f\in\M: \E(afb)=\E(bfa)=0 \mbox{ for all }a\in \A, \,b\in\A_0\}$ 
\end{definition}

\begin{lemma}\label{JOS 2.1-part1}
Let $\M$ be in standard form. With $H_1$, $H_2$ and $H_3$ respectively denoting $[\eta(\n(\A_0))]_2$, $[\eta(\n(\D))]_2$ and $[\eta(\n(\A_0^*))]_2$, we have that 
\begin{enumerate}
\item $H=H_1\oplus H_2\oplus H_3$;
\item each of $H_1$, $H_2$ and $H_3$ is an invariant subspace with respect to the action of $\D$.
\end{enumerate}
\end{lemma}

\begin{proof}
Property (1) follows from the facts that for $a_0,b_0 \in \n(\A_0)$ and $d\in \n(\D)$ we have that $\nu(d^*a_0)=\nu(\E(d^*a_0)) =0$ and similarly that $\nu(a_0b_0)=0$ and $\nu(d^*a_0^*)=0$. Property (2) is a consequence of the fact that $\D\n(\A_0)\subset\n(\A_0)$, $\D\n(\D)\subset\n(\D)$ and $\D\n(\A_0^*)\subset\n(\A_0^*)$. 
\end{proof}

\begin{theorem}\label{arvmax} Let $\A$ be a $\sigma$-weakly closed subalgebra of $\M$ which is subdiagonal with respect to the expected subalgebra $\mathcal{D}$. Then $\A_m$ is a $\sigma$-weakly closed superalgebra of $\A$ which is a maximal subdiagonal subalgebra with respect to the expected subalgebra $\mathcal{D}$.
\end{theorem}

\begin{proof}
The proof follows almost exactly the same lines as Arveson's original proof. However the transition from the $\sigma$-finite to the general setting requires some subtle tweaking of that argument at crucial points. We therefore outline the proof, giving details as necessary. We may firstly note, as Arveson does, that the $\sigma$-weak continuity of $\E$ ensures that $\A_m$ is a $\sigma$-weakly closed subspace of $\M$, and that the multiplicativity of $\E$ on any $\D$-subdiagonal subalgebra $\A_1$ containing $\A$, ensures that $\A_1\subset\A_m$. Hence it remains to show that $\A_m$ is an algebra and that $\E$ is multiplicative on $\A_m$. 

We now assume that $\M$ is represented in the Haagerup-Terp standard form (Theorem \ref{7:T stdform}), and once again follow Arveson by 
introducing the subspace 
 $$\A_M=\{x\in\M: x(H_1\oplus H_2)\subset(H_1\oplus H_2),\quad  x^*(H_2\oplus H_3)\subset(H_2\oplus H_3)\}.$$ 
We remind the reader that - as was shown in the proof of Proposition 
\ref{GNSPropn} - for any $b\in \n_\nu$, $\eta(b)$ corresponds to $[ah^{1/2}]$ in this standard form with $\langle[ah^{1/2}],[bh^{1/2}]\rangle =\nu(b^*a)$. It is clear that $\A_M$ is a subalgebra of $\M$. If we are able to show that 
$\A_M$ is $\D$-subdiagonal and that $\A_m\subset\A_M$ we would have that $\A_m=\A_M$, thereby proving the theorem. The fact that by definition $\nu(a_0fa)=0$ for every $a_0^*\in\n(\A_0^*)$, $a\in\n(\A)$, and $f\in \A_m$, ensures that 
$\langle f\eta(a),\eta(a_0^*)\rangle=\nu(a_0fa)=\nu(\E(a_0fa))=0$. Thus $\A_m(H_1\oplus H_2)\perp [\eta(\n(\A_0^*))]_2=H_3$, which ensures that $\A_m(H_1\oplus H_2)\subset (H_1\oplus H_2)$. Similarly $\A^*(H_2\oplus H_3)\subset(H_2\oplus H_3)$. So $\A_m\subset \A_M$ as required. We claim that
\begin{equation}\label{eqn 1}
\E_2(a^*x)=\E(a^*)\E_2(x)\mbox{ for all }a\in \A\mbox{ with }x\in H_2\oplus H_3.
\end{equation}
To see this, it is enough to note that for any $a,b\in\A$ with $b^*\in\n(\A^*)$, we have that
\begin{eqnarray*}
\E_2(a^*[b^*h^{1/2}])&=& \E_2([a^*b^*h^{1/2}])\\
&=&[\E(a^*b^*)h^{1/2}]\\
&=&[(\E(a^*)\E(b^*))h^{1/2}]\\
&=&\E(a^*)[\E(b^*)h^{1/2}]\\
&=&\E(a^*)\E_2[b^*h^{1/2}].
\end{eqnarray*}
Next let $(f_\lambda)$ be as in the proof of Proposition \ref{GNSPropn}. Since $(f_\lambda)\subset\D$, it is easy to check that $(f_\lambda)\subset\A_M\cap \A_M^*$ with in addition $([f_\lambda h^{1/2}])\subset H_2=\E_2(H_\nu)$. So for any $g\in \A_M$, we have that $(gf_\lambda)\subset\n(\A_M)$ and $(g^*f_\lambda)\subset\n(\A_M^*)$. Given $a\in\A$ and $g\in \A_M$ we therefore have by equation (\ref{eqn 1}) that 
$$[\E(a^*g^*f_\lambda)h^{1/2}]=\E_2[(a^*g^*f_\lambda)h^{1/2}]=\E(a^*)\E_2[(b^*f_\lambda)h^{1/2}]$$and similarly that
$$[\E(a^*)\E(g^*f_\lambda)h^{1/2}]=\E_2[\E(a^*)(g^*f_\lambda)h^{1/2}]=\E(a^*)\E_2[(g^*f_\lambda)h^{1/2}].$$ Clearly $[\E(a^*g^*f_\lambda)h^{1/2}]= [\E(a^*)\E(g^*f_\lambda)h^{1/2}]$. The injectivity of the embedding $\n_\nu\to L^2(\M): g\to [gh^{1/2}]$ (see Proposition \ref{7:P j-embed}) now ensures that $\E(a^*g^*f_\lambda)= \E(a^*)\E(g^*f_\lambda)$. By 
Proposition \ref{7:P Terp2}, the nets $(a^*g^*f_\lambda)$ and $(g^*f_\lambda)$ converge strongly to $a^*g^*$ and $g^*$ respectively. However since the strong and $\sigma$-strong* topologies agree on bounded sets, the convergence is even $\sigma$-strong*. The normality of $\E$ therefore ensures that 
$$\E(a^*g^*) = \lim_\lambda\E(a^*g^*f_\lambda) = \E(a^*)\,\lim_\lambda\E(g^*f_\lambda)= \E(a^*)\E(g^*),$$and hence that $\E(ga)=\E(g)\E(a)$. We claim that we also have 
\begin{equation}\label{eqn 2}
\E_2(gx)=\E(g)\E_2(x)\mbox{ for all }g\in \A_M\mbox{ with }x\in H_1\oplus H_2.
\end{equation}  
(To see this note that a minor modification of the proof of equation (\ref{eqn 1}) shows that $\E_2(g[bh^{1/2}])=\E(g)\E_2[bh^{1/2}]$ for each $b\in\n(\A)$. The claim then follows by continuity.) 

Let $g_0,g_1\in \A_M$ be given. Since $([f_\lambda h^{1/2}])\subset H_2$, we have by the definition of $\A_M$ that $(g_1[f_\lambda h^{1/2}])=([(g_1f_\lambda)h^{1/2}])\subset H_1\oplus H_2$. We may therefore use equation (\ref{eqn 2}) to see that 
\begin{eqnarray*}
[\E(g_0g_1f_\lambda)h^{1/2}]&=&\E_2[(g_0g_1f_\lambda)h^{1/2}]\\
&=&\E_2(g_0[(g_1f_\lambda)h^{1/2}])\\
&=&\E(g_0)\E_2[(g_1f_\lambda)h^{1/2}]\\
&=& \E(g_0)[\E(g_1f_\lambda)h^{1/2}]\\
&=&[\E(g_0)\E(g_1f_\lambda)h^{1/2}].
\end{eqnarray*}
Once again the injectivity of the embedding $\n_\nu\to L^2(\M): g\mapsto [gh^{1/2}]$ ensures that $\E(g_0g_1f_\lambda)=\E(g_0)\E(g_1f_\lambda)$ with the normality of $\E$ further ensuring that $\E(g_0g_1)=\lim_\lambda\E(g_0g_1f_\lambda)=\E(g_0)\lim_\lambda\E(g_1f_\lambda)=\E(g_0)\E(g_1)$as required. This then concludes the proof.
\end{proof}

\subsection{Necessity of $\sigma_t^\nu$-invariance}

Our next objective in achieving our ultimate goal of characterising maximality in the general case, is to show that if $\A$ is maximal subdiagonal subalgebra with respect to $\D$, we necessarily have that $\sigma_t^\nu(\A)=\A$ for all $t\in \bR$. This was shown by Ji, Ohwada and Saito in the $\sigma$-finite case \cite{JOS}. With minor modifications at appropriate points, their proof will also work for the general case. We will outline their proof giving details as appropriate where their proof needs to be modified.
 
We first note that the technology developed in the preceding theorem and its proof now enables us to posit the following companion to Lemma \ref{JOS 2.1-part1}. We leave the proof as an exercise.

\begin{lemma}\label{JOS 2.1-part2}
Let $\M$ be in standard form and let $(\A_m)_0$ denote $\A_m\cap\mathrm{ker}(\E)$. With $H_1$, $H_2$ and $H_3$ as before we have that 
\begin{enumerate}
\item $H_1=[\eta(\n((\A_m)_0))]_2$ and $H_3=[\eta(\n((\A_m)_0^*))]_2$;
\item $(\A_m)_0(H_1+H_2)\subset H_1$ and $(\A_m)_0^*(H_2+H_3)\subset H_3$.
\end{enumerate}
\end{lemma}

The Hilbert space decomposition in Lemmata \ref{JOS 2.1-part1} and \ref{JOS 2.1-part2} now yields the following result. This result is direct generalisation of \cite[Lemma 2.2]{JOS}. The proof closely follows that of \cite{JOS}. The middle part of the proof does however require some delicate adjustments for it to go through.

\begin{lemma}\label{JOS 2.2}
In terms of the Hilbert space decomposition in the preceding theorem we have that
$$\D=\left\{d\in\M: d= \begin{bmatrix} d_{11} & 0 & 0\\ 0 & d_{22} & 0\\ 0 & 0 & d_{33}\end{bmatrix}\right\}$$and
$$(\A_m)_0=\left\{f\in\M: f= \begin{bmatrix} f_{11} & f_{13} & f_{13}\\ 0 & 0 & f_{23}\\ 0 & 0 & f_{33}\end{bmatrix}\right\}.$$
\end{lemma}

\begin{proof} We may once again assume that $\M$ is the Haagerup-Terp standard form. As in \cite{JOS} one starts by considering the algebras 
$$\mathcal{B}=\left\{b\in\M: b= \begin{bmatrix} b_{11} & 0 & 0\\ 0 & b_{22} & 0\\ 0 & 0 & b_{33}\end{bmatrix}\right\} \mbox{ and }
\mathcal{C}=\left\{f\in\M: f= \begin{bmatrix} f_{11} & f_{13} & f_{13}\\ 0 & 0 & f_{23}\\ 0 & 0 & f_{33}\end{bmatrix}\right\}.$$
As noted by \cite{JOS}, it is now easy to see that $\D\subset \mathcal{B}$ and that $(\A_m)_0\subset\mathcal{C}$.

To prove the converse inclusion in this generality, we use the fact that
\begin{equation}\label{eqn 3}
\E_2(f\E_2(x))=\E(f)\E_2(x)\mbox{ for all }f\in \M\mbox{ and }x\in H_\nu.
\end{equation}
This equality can be proven by first noticing that any $f\in\M$ and $d\in\n(\D)$, we have that
\begin{eqnarray*}
\E_2(f[dh^{1/2}])&=& \E_2([fdh^{1/2}])\\
&=&[\E(fd)h^{1/2}]\\
&=&[(\E(f)dh^{1/2}]\\
&=&\E(f)[dh^{1/2}].
\end{eqnarray*} 
This shows that $\E_2(fx)=\E(f)x$ for every $x\in H_2=[\eta(\n(\D))]_2$, which suffices to prove the validity of equation (\ref{eqn 3}). In particular we will for any $g\in \mathrm{ker}(\E)$ have that $\E_2\circ M_g \circ \E_2 =0$ where $M_g$ is the left multiplication operator with symbol $g$. 

Now select $b\in \mathcal{B}$. Since $\E(b) \in\D$, we clearly have that $$\E(b)=\begin{bmatrix} v_{11} & 0 & 0\\ 0 & v_{22} & 0\\ 0 & 0 & v_{33}\end{bmatrix}.$$Since $b-\E(b)\in\mathrm{ker}(\E)$, we therefore have that $0=\E_2\circ(b-\E(b))\circ\E_2$. The matrix forms of both $b$ and $\E(b)$ ensure that $H_2$ is an invariant subspace of $M_{(b-\E(b))}$. So we in fact have that $M_{(b-\E(b))}\circ\E_2=0$. Now let $(f_\lambda)$ be as in the proof of Proposition \ref{GNSPropn}. Since $(f_\lambda)\subset\n(\D)$ and hence $([f_\lambda h^{1/2}])\subset H_2$, we therefore have that $0=(b-\E(b))[f_\lambda h^{1/2}]=[((b-\E(b))f_\lambda) h^{1/2}]=\eta((b-\E(b))f_\lambda)$ for each 
$\lambda$. The injectivity of the embedding $\n_\nu\to H_\nu: f\to \eta(f)$ therefore ensures that $0=(b-\E(b))f_\lambda$ for each $\lambda$. Taking the limit now shows that $0=b-\E(b)$ as was required. The rest of the proof proceeds exactly as in \cite{JOS}.
\end{proof}

It is clear from Remark \ref{7:D defstdfrm} and Theorem \ref{7:T stdform} that in the Haagerup-Terp standard form the densely-defined anti-linear operator $S$ which is the starting point of modular theory in this context, is defined as the closure of the operator $S_0$ with domain $\{[ah^{1/2}]:a\in \n_\nu\cap \n_\nu^*\}$ and action $S_0([ah^{1/2}])=[a^*h^{1/2}]$. We use this fact to prove the following lemma.

\begin{lemma}\label{JOS 2.3}
The closed anti-linear operator $S$ described above has the matrix decomposition 
$$S= \begin{bmatrix} 0 & 0 & S_3\\ 0 & S_2 & 0\\ S_1 & 0 & 0\end{bmatrix}$$with respect to the decomposition in Lemma \ref{JOS 2.1-part1}, where for each $i=1,2,3$ the operator $S_i$ is a densely defined operator on $H_i$ with domain $\mathfrak{F}_i$ such that $S_1(\mathfrak{F}_1)=\mathfrak{F}_3$, $S_2(\mathfrak{F}_2)=\mathfrak{F}_2$ and $S_3(\mathfrak{F}_3)=\mathfrak{F}_1$.
\end{lemma}

\begin{proof} We once again assume that $\M$ is in the Haagerup-Terp standard form. Since $\D$ is an expected algebra with $\nu\circ\E=\nu$, the known theory ensures that $\eta(\n(\D)^*\cap\n(\D))$ is dense in $H_2=L^2(\D)$. We claim that $\eta(\n(\A_0)\cap\n(\A_0^*)^*)$, $\eta(\n(\D)\cap\n(\D)^*)$ and $\eta(\n(\A_0^*)\cap\n(\A_0)^*)$ are respectively dense in $H_1$, $H_2$ and $H_3$. To see this let say $a\in \n(\A_0)$ be given and let $(f_\lambda)\subset \n_\nu\cap\n_\nu^*\cap\D$ be as in Proposition \ref{GNSPropn}. Notice that the proofs all run along similar lines, and hence we only prove the first claim. By the ideal property of $\A_0$ we will still have that $(f_\lambda a)\subset \n(\A_0)$, with the left-ideal property of $\n_\nu$ ensuring that $(a^*f_\lambda)\subset \n_\nu$ and hence that we also have that $(f_\lambda a)\subset \n(\A_0^*)^*$. Since the strong and $\sigma$-strong topology agree on bounded sets, $(f_\lambda)$ converges $\sigma$-strongly to $\I$. This in turn ensures that $(\eta(f_\lambda a))=(f_\lambda[ah^{1/2}])$ converges weakly to $[ah^{1/2}]$ in $L^2(\M)$. Thus the weak closure of $\eta(\n(\A_0)\cap\n(\A_0^*)^*)$ must include the weak closure of $\eta(\n(\A_0))$. Since for convex sets weak and norm closures agree, this suffices to show that $[\eta(\n(\A_0)\cap\n(\A_0^*)^*)]_2=[\eta(\n(\A_0))]_2=H_1$.

The prescription 
$$V_0([ah^{1/2}] + [dh^{1/2}] + [b^*h^{1/2}])= [a^*h^{1/2}] + [d^*h^{1/2}] [bh^{1/2}]$$
$$\mbox{where } a,b\in \n(\A_0)\cap\n(\A_0^*)^*,\quad d\in\n(\D)\cap\n(\D)^*$$therefore yields a densely defined anti-linear operator which is a restriction of the operator $S$. The operator $V_0$ is therefore closable. We show that the closure is $S$ itself. The graph $G(S)$ of $S$ is just the norm closure of $\{[gh^{1/2}]\oplus[g^*h^{1/2}]: g\in \n_\nu\cap\n_\nu^*\}$. We therefore need to show that the closure of $G(V_0)$ includes this subspace (and therefore equals it). Let $g\in \n_\nu\cap\n_\nu^*$ be given and select $(a_\alpha)\subset \A_0+\D+\A_0^*$ such that $(a_\alpha)$ converges $\sigma$-weakly to $g$. The net $(a_\alpha^*)$ will then of course converge $\sigma$-weakly to $g^*$. It is easy to check that then $(f_\gamma a_\alpha f_\lambda)\subset \n(\A_0)\cap\n(\A_0^*)^* + \n(\D)\cap\n(\D)^*+\n(\A_0^*)\cap\n(\A_0)^*$. For each fixed $\lambda$ and $\gamma$ it is then easy to see that the net 
$$[f_\gamma a_\alpha f_\lambda h^{1/2}]\oplus [f_\lambda a_\alpha^* f_\gamma h^{1/2}]= (f_\gamma a_\alpha)[f_\lambda h^{1/2}]\oplus (f_\lambda a_\alpha^*)[f_\gamma h^{1/2}]$$will converge weakly to $([f_\gamma g f_\lambda h^{1/2}]\oplus [f_\lambda g^* f_\gamma h^{1/2}]$ as $\alpha$ increases. Next with $\gamma$ fixed we know from the first part of the proof and the proof of Proposition \ref{GNSPropn}, that $([f_\gamma g f_\lambda h^{1/2}]\oplus [f_\lambda g^* f_\gamma h^{1/2}]$ converges weakly to $([f_\gamma gh^{1/2}]\oplus [g^* f_\gamma h^{1/2}]$ as $\lambda$ increases. These same properties then also ensure that $([f_\gamma gh^{1/2}]\oplus [g^* f_\gamma h^{1/2}]$ converges weakly to $([gh^{1/2}]\oplus [g^*h^{1/2}]$ as $\gamma$ increases. It follows that the weak closure of $G(V_0)$ contains $G(S)$. But since $G(V_0)$ is convex, the weak and norm closures must agree. Thus as required the norm closure of $G(V_0)$ includes $G(S)$. 

Given $x\oplus S(x) \in G(S)$, it follows from what we proved above that we may select sequences $(a_n),(b_n)\subset \n(\A_0)\cap\n(\A_0^*)^*$ and $(d_n)\subset \n(\D)\cap\n(\D)^*$ such that $$\lim_{n\to \infty}(\|[(a_n+d_n+b_n^*)h^{1/2}]-x\|^2_2+\|[(a_n^*+d_n^*+b_n)h^{1/2}]-S(x)\|^2_2=0.$$Once this has been noted, the rest of the proof now follows verbatim as in \cite{JOS}. Specifically with $p_i$ denoting the orthogonal projection from $H_\nu$ onto $H_i$ 
($i=1,2,3$), it now follows from the above that
\begin{eqnarray*}
0&=&\lim_{n\to \infty}(\|[a_nh^{1/2}]-p_1(x)\|^2+\|[d_nh^{1/2}]-p_2(x)\|^2+\|[b_n^*h^{1/2}]-p_3(x)\|^2_2\\
&& +\|[a_n^*h^{1/2}]-p_1(S(x))\|^2+\|[d_n^*h^{1/2}]-p_2(S(x))\|^2+\|[b_nh^{1/2}]-p_3(S(x))\|^2_2.
\end{eqnarray*}
These limit formulas clearly show that $p_i(\mathrm{dom}(S))\subset \mathrm{dom}(S)$ with $S(p_1(\mathrm{dom}(S)))\subset p_3(\mathrm{dom}(S))$, $S(p_2(\mathrm{dom}(S)))\subset p_2(\mathrm{dom}(S))$ and $S(p_3(\mathrm{dom}(S)))\subset p_2(\mathrm{dom}(S))$.
The claim therefore follows on setting $\mathfrak{F}_i=p_i(\mathrm{dom}(S))$ for each $i=1,2,3$.
\end{proof}

With Lemmata \ref{JOS 2.2} and \ref{JOS 2.3} replacing \cite[Lemmata 2.2 \& 2.3]{JOS} the following theorem now follows exactly as in \cite{JOS}. The heart of the proof centres around the fact that 
 $$\Delta=\left[\begin{smallmatrix} S_1^*S_1&0&0\\ 0&S_2^*S_2&0\\0&0&S_3^*S_3\end{smallmatrix}\right].$$

\begin{theorem}\label{necessary}
Let $\M$ be a von Neumann algebra equipped with a faithful normal semifinite weight $\nu$. If $\A$ is a maximal $\D$-subdiagonal subalgebra of $\M$, then $\sigma_t^\nu(\A)=\A$ for every $t\in \mathbb{R}$. 
\end{theorem}

\subsection{Sufficiency of $\sigma_t^\nu$-invariance}

Here we show that on the back of the theory developed above, Xu's proof that in the $\sigma$-finite case [$\sigma$-weak closedness] + 
[$\sigma_t^\nu$-invariance] of a subdiagonal subalgebra guarantees maximality, will with very minor modifications carry over to the general setting. We start by recalling Ji's extension of Exel's maximality theorem.

\begin{theorem}[{\cite{jig1}, cf. \cite[Theorem 4.1]{Xu}}] 
Let $\M$ be a von Neumann algebra equipped with a faithful normal semifinite trace $\tau$, and $\mathcal{D}$ be a unital von Neumann subalgebra of $\M$ such that $\tau{\upharpoonright}\mathcal{D}$ is semifinite. Further let $\E$ be the faithful normal conditional expectation onto $\D$ such that $\tau\circ\E=\tau$. If $\A$ is a $\sigma$-weakly closed $\D$-subdiagonal subalgebra of $\M$, it is a maximal $\D$-subdiagonal subalgebra.
\end{theorem}

Careful perusal of section 3 of \cite{Xu} reveals that if in the non-$\sigma$-finite case
\begin{itemize}
\item the variant of the reduction theorem used in \cite{Xu} is replaced with the one described in the present paper, 
\item and the usage of Exel's maximality theorem is replaced by Ji's extension thereof,
\end{itemize}
the entire proof of \cite[Theorem 1.1]{Xu} and all its underlying lemmata will go through verbatim in the general setting. We therefore obtain the following:

\begin{theorem}\label{sufficient}
Let $\M$ be a von Neumann algebra equipped with a faithful normal semifinite weight $\nu$. If $\A$ is a $\D$-subdiagonal subalgebra of $\M$, then $\A$ will be maximal subdiagonal if $\sigma_t^\nu(\A)=\A$ for every $t\in \mathbb{R}$. 
\end{theorem}

\subsection{The main theorem}

The main theorem of this section, namely the following, now follows from a combination of Theorems \ref{necessary} \& \ref{sufficient}:

\begin{theorem}
Let $\M$ be a von Neumann algebra equipped with a faithful normal semifinite weight $\nu$, and $\mathcal{D}$ be a unital von Neumann subalgebra of $\M$ such that $\nu{\upharpoonright}\mathcal{D}$ is semifinite. Further suppose that there exists a faithful normal conditional expectation $\E:M\to \mathcal{D}$ such that $\nu\circ\E=\nu$. If $\A$ is a $\sigma$-weakly closed $\D$-subdiagonal subalgebra of $\M$, then $\A$ is maximal subdiagonal if and only if $\sigma_t^\nu(\A)=\A$ for every $t\in \mathbb{R}$. 
\end{theorem}

\section{$H^p$-spaces for general von Neumann algebras}\label{Hpsection}

Having dealt with the issue of maximality, we pass to the definition of $H^p(\A)$ spaces. Here the arguments need to be a lot more delicate than in the $\sigma$-finite case given that here it is only $\mathfrak{m}_\nu$ and not all of $\M$, that embeds into $L^p(\M)$ when $1\leq p<\infty$.

\begin{definition}\label{defHp} Let $\M$ be a von Neumann algebra equipped with a faithful normal semifinite weight $\nu$, and let $\A$ be an approximately subdiagonal subalgebra of $\M$. For any $1\leq p<\infty$ we define $H^p(\A)$ to be the norm closure of the subspace $\mathrm{span}(\mathfrak{j}^{(2p)}(\n(\A^*)^*).\mathfrak{j}^{(2p)}(\n(\A))$ of $L^p(\M)$ and $H^p_0(\A)$ the norm closure of the subspace $\mathrm{span}(\mathfrak{j}^{(2p)}(\n(\A^*)^*).\mathfrak{j}^{(2p)}(\n(\A_0))$. For $p=\infty$ we simply define $H^\infty(\A)$ to be $\A$ and $H^\infty(\A_0)$ to be $\A_0$.
\end{definition}

\begin{remark} 
It easily follows from the above definition that if $\A$ is actually maximal subdiagonal, the kernel of the extension $\E_{p}$ of the expectation $\E$ to $L^p(\M)$ will for any $1\leq p<\infty$ contain $H^p_0(\A)$. This can be seen by noting that by definition (see Remark \ref{10:R cond-exp}) this extension will map a term of the form $\mathfrak{j}^{(2p)}(a^*)^*.\mathfrak{j}^{(2p)}(b)=\mathfrak{i}^{(p)}(ab)$ where $a\in \n(\A^*)^*$ and 
$b\in\n(\A_0)$, onto $\mathfrak{i}^{(p)}(\E(ab))=0$.
\end{remark}

Having defined $H^p$ spaces we proceed to develop some technology which will pave the way for a more refined theory of these spaces. 

\begin{theorem}\label{tech-thm} 
Let $\M$ be equipped with a faithful normal semifinite weight $\nu$ and let $\A$ be an approximately subdiagonal subalgebra of $\M$. For any $2\leq q<\infty$ we will then have that 
$$[\mathfrak{i}^{(q)}(\mathrm{span}(\n(\A^*)^*\cdot\n(\A)))]_q=[\mathfrak{j}^{(q)}(\mathrm{span}(\n(\A^*)^*\cdot\n(\A)))]_q$$ 
$$= [\mathfrak{j}^{(q)}(\n(\A)\cap \n(\A^*)^*)]_q =[\mathfrak{j}^{(q)}(\n(\A))]_q=[\mathfrak{j}^{(q)}(\n(\A^*))^*]_q =[\mathfrak{j}^{(q)}(\n(\A^*))]_q^*$$and similarly that 
$$[\mathfrak{i}^{(q)}(\mathrm{span}(\n(\A^*)^*\cdot\n(\A_0)))]_q=[\mathfrak{j}^{(q)}(\mathrm{span}(\n(\A^*)^*\cdot\n(\A_0)))]_q$$ 
$$= [\mathfrak{j}^{(q)}(\n(\A_0)\cap \n(\A_0^*)^*)]_q =[\mathfrak{j}^{(q)}(\n(\A_0))]_q=[\mathfrak{j}^{(q)}(\n(\A_0^*))^*]_q =[\mathfrak{j}^{(q)}(\n(\A_0^*))]_q^*.$$
\end{theorem}

\begin{proof} 
In view of the similarity of the proofs of the two cases, we only prove the second set of equalities. By convexity the subspaces $[\mathfrak{j}^{(q)}(\n(\A_0)\cap \n(\A_0^*)^*]_q$, $[\mathfrak{j}^{(q)}(\n(\A_0)]_q$ and $[\mathfrak{j}^{(q)}(\n(\A_0^*))^*]_q$ are all weakly closed in $L^q$. 

Let $a\in \n(\A_0)$ be given. By Proposition \ref{analytic-n(A)}. The $\sigma$-weak density of $\mathfrak{n}(\A^*)\cap\mathfrak{n}(\A)$ in $\A$ ensures that we may select a sequence $(a_n)\subset \mathfrak{n}^\infty(\A)$ which is $\sigma$-weakly convergent to $\I$. By the ideal property of $\A_0$ and the fact that $\mathfrak{n}_\nu$ is a left ideal, we clearly have that $(a_n a)\subset \n(\A_0)\cap\n(\A_0^*)^*$. With $p$ denoting the conjugate index to $q$ it now for any $b\in L^p(\M)$ follows that $$tr(b[a_nah^{1/q}]) = tr(a_n([ah^{1/q}]b)) \to tr(b[ah^{1/q}]).$$In other words $([a_nah^{1/q}])$ converges $L^q$-weakly to $[ah^{1/q}]$. This proves that $$[\mathfrak{j}^{(q)}(\n(\A_0))]_q\subseteq[\mathfrak{j}^{(q)}(\n(\A_0)\cap \n(\A_0^*)^*)]_q.$$The converse inclusion is clear and hence equality follows. This then proves the third equality. 

Starting with $a\in \n(\A_0^*)^*$, a similar proof to the one above shows that $[\mathfrak{j}^{(q)}(\n(\A_0)\cap \n(\A_0^*)^*)]_q =[\mathfrak{j}^{(q)}(\n(\A_0^*))^*]_q$. The fourth equality therefore also holds. 

It is clear that $[\mathfrak{j}^{(q)}(\mathrm{span}(\n(\A^*)^*\cdot\n(\A_0)))]_q\subset [\mathfrak{j}^{(q)}(\n(\A_0))]_q$. If as before we select some $(a_n)\subset \mathfrak{n}^\infty(\A)$ which is $\sigma$-weakly convergent to $\I$, we will for any $a\in \n(\A_0)$ as before have that $([a_nah^{1/q}])$ converges $L^q$-weakly to $[ah^{1/q}]$, which then establishes the converse inclusion. The second equality therefore follows.

It remains to prove the first and final equalities. Firstly note that the proof of Proposition \ref{refine-2} effectively shows that restricting the embedding $\mathfrak{j}^{(2q)}$ of $\n(\A)\cap \n(\A^*)^*$ into $L^{2q}$ to analytic elements, will produce the same closure as $\mathfrak{j}^{(2q)}(\n(\A)\cap \n(\A^*)^*)$ (and hence by what we've already verified, also of $\mathfrak{j}^{(2q)}(\n(\A)$). Careful 
consideration shows that $$[\mathfrak{i}^{(q)}(\mathrm{span}(\n(\A^*)^*\cdot\n(\A)))]_q = [\mathrm{span}(\mathfrak{j}^{(2q)}(\n(\A^*))^*\cdot\mathfrak{j}^{(2q)}(\n(\A)))]_q.$$So when analysing $[\mathfrak{i}^{(q)}(\mathrm{span}(\n(\A^*)^*\cdot\n(\A)))]_q$, it 
is enough to consider terms of the form $(h^{1/2q}b)[ah^{1/2}]$ where $a$ and $b$ are analytic elements of $\n(\A)\cap \n(\A^*)^*$.

With $a\in\n(\A_0)$ and $b\in \n(\A^*)^*$ analytic, the fact that $$[\mathfrak{i}^{(q)}(\mathrm{span}(\n(\A^*)^*\cdot\n(\A_0)))]_q \subseteq [\mathfrak{j}^{(q)}(\mathrm{span}(\n(\A^*)^*\cdot\n(\A_0)))]_q$$is then a simple consequence of the fact that $$(h^{1/2q}b)[ah^{1/2q}]= (\sigma^\nu_{-i/2q}(b)h^{1/2q})[h^{1/2q}\sigma^\nu_{i/2q}(a)]=$$ $$\sigma^\nu_{-i/2q}(b)[h^{1/q}\sigma^\nu_{i/2q}(a)]= \sigma^\nu_{-i/2q}(b)[\sigma^\nu_{-i/2q}(a)h^{1/q}]=[(\sigma^\nu_{-i/2q}(b)\sigma^\nu_{-i/2q}(a))h^{1/q}].$$(Here we repeatedly used Lemma \ref{GL2-2.4+5}.) The converse inclusion follows by a similar argument. 

Once the proof of the final equality has as above been reduced to a manipulation of analytic elements, it will follow by a similar argument.
\end{proof}

We next remind the reader of the concept of an \emph{analytically conditioned algebra} introduced in \cite{L-HpIII}. This concept is closely related to Ji's concept of an \emph{expectation algebra} (see \cite{jig2}), the difference being that expectation algebras are not required to be $\sigma_t^\nu$-invariant. Both these concepts are type III versions of \emph{tracial algebras}, which were introduced by Blecher and Labuschagne in the setting of finite von Neumann algebras \cite{BL-logmod} who then went on to show that for these algebras a large number of conditions (including the validity of a noncommutative Szeg\"o formula) are all equivalent to maximal subdiagonality (see \cite{BL1, BLsurvey}). In the general case the only difference between a subdiagonal subalgebra and an analytically conditioned subalgebra is that for the latter we do not require $\A+\A^*$ to be $\sigma$-weakly dense in $\M$.

\begin{definition}
A $\sigma$-weakly closed unital subalgebra $\A$ of $\M$ is said to be an \emph{analytically conditioned} subalgebra   
\begin{itemize}
\item if $\sigma_t^\nu(\A)=\A$ for all $t\in \mathbb{R}$,
\item and if the faithful normal conditional expectation $\mathbb{E}:\M\to \mathcal{D}=\A\cap \A^*$ satisfying $\nu\circ\mathbb{E}=\nu$ (ensured by the above condition \cite[Theorem IX.4.2]{Tak2}), is multiplicative on $\A$.
\end{itemize}
\end{definition}

\begin{remark} It is not difficult to modify the proof of Proposition \ref{refine-2} to show that we will for \emph{analytically conditioned} subalgebras also have that $\n(\A)\cap \n(\A^*)^*$ is $\sigma$-weakly dense in $\A$ and that in addition the injection $\mathfrak{j}^{(q)}$ will for any $q\geq 2$ map the analytic elements of $\n(\A)\cap \n(\A^*)^*$ onto a norm dense subspace of 
$[\mathfrak{j}^{(q)}(\n(\A)\cap \n(\A^*)^*)]_q$.
\end{remark} 

Given $1\leq p<\infty$ and an analytically conditioned subalgebra of $\M$, we will in view of the above now adopt the convention of respectively writing $\mathcal{H}^p(\A)$ and $\mathcal{H}^p(\A_0)$ for the closures in $L^p(\M)$ of 
$\mathrm{span}(\mathfrak{j}^{(2p)}(\n(\A^*)^*)\cdot\mathfrak{j}^{(2p)}(\n(\A))$ and \newline $\mathrm{span}(\mathfrak{j}^{(2p)}(\n(\A^*)^*)\cdot\mathfrak{j}^{(2p)}(\n(\A_0))$.

\begin{corollary}\label{gen-perp}
Let $\A$ be an analytically conditioned subalgebra of $\M$. Then $\mathcal{H}^2(\A)$ and $\mathcal{H}^2(\A_0)^*$ are orthogonal subspaces of $L^2(\M)$ 
\end{corollary}

\begin{proof} Given $a\in \mathfrak{n}(\A)$ and $b\in \mathfrak{n}(\A_0^*)^*$, we have that 
$$\langle \mathfrak{j}^{(2)}(a), \mathfrak{j}^{(2)}(b^*) \rangle =tr(\mathfrak{j}^{(2)}(b^*)^*\mathfrak{j}^{(2)}(a))= tr(\mathfrak{i}^{(1)}(ba))=\nu(ba)=\nu(\mathbb{E}(ba))=0.$$
\end{proof}

\begin{proposition}\label{prop2.1} Let $\A$ be an analytically conditioned subalgebra. Given $r\geq 1$ with $\frac{1}{r}=\frac{1}{p}+\frac{1}{q}$, we have that $\mathrm{span}(\mathcal{H}^p(\A)\cdot\mathcal{H}^q(\A))$ is a dense subset of $\mathcal{H}^r(\A)$ with $\E_r(ab)=\E_p(a)\E_q(b)$ for each $a\in \mathcal{H}^p(\A)$ and $b\in\mathcal{H}^q(\A)$.
\end{proposition}

\begin{proof} By definition $\mathcal{H}^p(\A)=[\mathrm{span}(\mathfrak{j}^{(2p)}(\n(\A^*)^*)\cdot\mathfrak{j}^{(2p)}(\n(\A))]_p$ and \newline $\mathcal{H}^q(\A)=[\mathrm{span}(\mathfrak{j}^{(2q)}(\n(\A^*)^*)\cdot\mathfrak{j}^{(2q)}(\n(\A))]_q$. It therefore suffices to prove the 
claims for elements of the form $(h^{1/2p}a)[bh^{1/2p}]$ and $(h^{1/2q}f)[gh^{1/2p}]$, where $a,f\in\n(\A^*)^*$ and $b,g\in\n(\A)$. The same type of argument as was used in the final part of the proof of Theorem \ref{tech-thm}, then suffices to show that we may assume 
that each of $a,\,b,\,f$ and $g$ is an analytic element of $\n(\A^*)^*\cap\n(\A)$ for which the image under each $\sigma_z^\nu$ is still an analytic element of $\n(\A^*)^*\cap\n(\A)$. Having made this assumption, a similar argument to that used in the final part of the proof of Proposition \ref{exp-props}, shows that 
\begin{eqnarray*} 
&&\quad(h^{1/2p}a)[bh^{1/2p}].(h^{1/2q}f)[gh^{1/2q}]\\
&&= (h^{1/2p}a)[bh^{1/2q}].(h^{1/2p}f)[gh^{1/2q}]\\
&&=[\sigma^\nu_{-i/2p}(a)h^{1/2p}](h^{1/2q}\sigma^\nu_{i/2q}(b)).[\sigma^\nu_{-i/2p}(f)h^{1/2p}](h^{1/2q}\sigma^\nu_{i/2q}(g))\\
&&=[\sigma^\nu_{-i/2p}(a)h^{1/2r}]\sigma^\nu_{i/2q}(b)\sigma^\nu_{-i/2p}(f)(h^{1/2r}\sigma^\nu_{i/2q}(g))\\
&&= (h^{1/2r}\sigma^\nu_{i/2r}(\sigma^\nu_{-i/2p}(a)))\sigma^\nu_{i/2q}(b)\sigma^\nu_{-i/2p}(f)[\sigma^\nu_{-i/2r}(\sigma^\nu_{i/2q}(g))h^{1/2r}]\\
&&= (h^{1/2r}\sigma^\nu_{i/2q}(a))\sigma^\nu_{i/2q}(b)\sigma^\nu_{-i/2p}(f)[\sigma^\nu_{-i/2p}(g)h^{1/2r}]\\
\end{eqnarray*} 
This clearly shows that $(h^{1/2p}a)[bh^{1/2p}].(h^{1/2q}f)[gh^{1/2q}]\in \mathcal{H}^r(\A)$ and hence that $\mathcal{H}^p(\A)\cdot\mathcal{H}^q(\A)\subset\mathcal{H}^r(\A)$ as claimed. 

We pass to verifying the reverse inclusion. A similar argument to the one used above shows that it suffices to show that $\mathfrak{j}^{(2r)}(a^*)^*\cdot\mathfrak{j}^{(2r)}(b)\in \mathcal{H}^p(\A)\cdot\mathcal{H}^q(\A)$ for all analytic elements $a$ and $b$ of $\n(\A^*)^*\cap\n(\A)$. Now recall that we saw in the proof of Theorem \ref{tech-thm} that with $(f_\lambda)$ as in the proof of Proposition \ref{GNSPropn}, we will have that $\mathfrak{j}^{(2r)}(f_\lambda a^*)$ and $\mathfrak{j}^{(2r)}(f_\lambda b)$ converge weakly to respectively $\mathfrak{j}^{(2r)}(a^*)$ and $\mathfrak{j}^{(2r)}(b)$in $\mathcal{H}^{2r}(\A)$. Hence the fact that $\mathfrak{j}^{(2r)}(a^*)^*\cdot\mathfrak{j}^{(2r)}(b)\in \mathcal{H}^p\A)\cdot\mathcal{H}^q(\A)$ will follow if we can show that for any $\lambda$ and $\gamma$ we have that $\mathfrak{j}^{(2r)}(f_\lambda a^*)^*\cdot\mathfrak{j}^{(2r)}(f_\gamma b)\in \mathcal{H}^p(\A)\cdot\mathcal{H}^q(\A)$. Using a similar argument as before we deduce that 
\begin{eqnarray*} 
&&\quad (h^{1/2r}af_\lambda)[f_\gamma bh^{1/2r}]\\
&&= [\sigma^\nu_{-i/2r}(a)h^{1/2r}]f_\lambda f_\gamma (h^{1/2r}\sigma^\nu_{i/2r}(b))\\
&&= [\sigma^\nu_{-i/2r}(a)h^{1/2p}](h^{1/2q}f_\lambda)[f_\gamma h^{1/2p}](h^{1/2q}\sigma^\nu_{i/2r}(b))\\
&&= (h^{1/2p}\sigma^\nu_{i/2p}[\sigma^\nu_{-i/2r}(a))](\sigma^\nu_{-i/2q}(f_\lambda)h^{1/2q})(h^{1/2p}\sigma^\nu_{i/2p}(f_\gamma))[\sigma^\nu_{-i/2p}(\sigma^\nu_{i/2r}(b))h^{1/2q}]\\
&&= (h^{1/2p}\sigma^\nu_{-i/2q}(a))[\sigma^\nu_{-i/2q}(f_\lambda)h^{1/2p}](h^{1/2q}\sigma^\nu_{i/2q}(f_\gamma))[\sigma^\nu_{i/2q}(b)h^{1/2q}]\\
&&= (h^{1/2p}\sigma^\nu_{-i/2q}(a))(\sigma^\nu_{-i/2q}(f_\lambda)h^{1/2p})(h^{1/2q}\sigma^\nu_{i/2q}(f_\gamma))[\sigma^\nu_{i/2q}(b)h^{1/2q}]\\
&&=\mathfrak{i}^{(p)}(\sigma^\nu_{-i/2q}(af_\lambda))\cdot\mathfrak{i}^{(q)}(\sigma^\nu_{i/2q}(f_\gamma b))
\end{eqnarray*} 
which clearly shows that $\mathfrak{j}^{(2r)}(f_\lambda a^*)^*\cdot\mathfrak{j}^{(2r)}(f_\gamma b)\in \mathcal{H}^p(\A)\cdot\mathcal{H}^q(\A)$ as was required. 

We may also use the first displayed equaltion above to see that
\begin{eqnarray*}
&&\quad\E_r((h^{1/2p}a)[bh^{1/2p}].(h^{1/2q}f)[gh^{1/2q}])\\
&&=\E_r(h^{1/2r}\sigma^\nu_{i/2q}(a))\sigma^\nu_{i/2q}(b)\sigma^\nu_{-i/2p}(f)[\sigma^\nu_{-i/2p}(g)h^{1/2r}]\\
&&= \E_r(\mathfrak{i}^{(r)}(\sigma^\nu_{i/2q}(a)\sigma^\nu_{i/2q}(b)\sigma^\nu_{-i/2p}(f)\sigma^\nu_{-i/2p}(g))\\
&&= \mathfrak{i}^{(r)}(\E(\sigma^\nu_{i/2q}(a)\sigma^\nu_{i/2q}(b)\sigma^\nu_{-i/2p}(f)\sigma^\nu_{-i/2p}(g)))\\
&&= \mathfrak{i}^{(r)}(\sigma^\nu_{i/2q}(\E(a))\sigma^\nu_{i/2q}(\E(b))\sigma^\nu_{-i/2p}(\E(f))\sigma^\nu_{-i/2p}(\E(g)))\\
&&= (h^{1/2r}\sigma^\nu_{i/2q}(\E(a)))\sigma^\nu_{i/2q}(\E(b))\sigma^\nu_{-i/2p}(\E(f))[\sigma^\nu_{-i/2p}(\E(g))h^{1/2r}].
\end{eqnarray*} 
(Here we silently used the fact noted in the proof of Proposition \ref{exp-props}, that $\E$ preserves analyticity.) We may now use the above formula to see that
\begin{eqnarray*} 
&&\quad\E_p((h^{1/2p}a)[bh^{1/2p}])\cdot\E_q((h^{1/2q}f)[gh^{1/2q}])\\
&&= \mathfrak{i}^{(p)}(\E(ab))\cdot\mathfrak{i}^{(q)}(\E(fg))\\
&&= \mathfrak{i}^{(p)}(\E(a)\E(b))\cdot\mathfrak{i}^{(q)}(\E(f)\E(g))\\
&&= (h^{1/2p}\E(a))[\E(b)h^{1/2q}].(h^{1/2p}\E(f))[\E(g)h^{1/2q}]\\
&&= (h^{1/2r}\sigma^\nu_{i/2q}(\E(a)))\sigma^\nu_{i/2q}(\E(b))\sigma^\nu_{-i/2p}(\E(f))[\sigma^\nu_{-i/2p}(\E(g))h^{1/2r}]\
\end{eqnarray*} 
Hence $$\E_p((h^{1/2p}a)[bh^{1/2p}])\cdot\E_q((h^{1/2q}f)[gh^{1/2q}])=\E_r((h^{1/2p}a)[bh^{1/2p}].(h^{1/2q}f)[gh^{1/2q}])$$as required.\end{proof}

On the basis of the above corollary we now define $H^p(\A)$ for $0<p<\infty$ as below

\begin{definition} For $\frac{1}{2}<p<1$ we define $H^p(\A)$ to be the closure of the linear span of $H^{2p}(\A).H^{2p}(\A)$ in $L^p(\M)$. Carrying on inductively we similarly for any $\frac{1}{2^{n+1}}\leq p<\frac{1}{2^n}$ (where $n=0, 1, 2,\dots$) define $H^p(\A)$ to be the closure of $\mathrm{span}(H^{2p}(\A).H^{2p}(\A))$ in $L^p(\M)$.
\end{definition}

We close this section by showing that $H^p$ spaces for $0<p<1$ exhibit similar behaviour to that noted earlier.

\begin{proposition}\label{prop2.2} Let $\A$ be an analytically conditioned algebra. Given $0<r<1$ with $\frac{1}{r}=\frac{1}{p}+\frac{1}{q}$ for some $0<p,q<\infty$, we have that $\mathrm{span}(\mathcal{H}^p(\A)\cdot\mathcal{H}^q(\A))$ is a dense subset of $\mathcal{H}^r(\A)$.
\end{proposition}

\begin{proof} The result is proved by inductively considering the cases $\frac{1}{2^{n+1}}\leq r<\frac{1}{2^n}$. We shall prove the case $\frac{1}{2}\leq r<1$, leaving the remaining cases as an exercise. Given $0<p,q<\infty$ with $\frac{1}{r}=\frac{1}{p}+\frac{1}{q}$, we may clearly assume that $p\neq q$. That is we may assume that $q>2r>p$. Then of course $q\geq 1$. Suppose that $p\geq \frac{1}{2^{k-1}}$. For any $s>0$ the space $\mathcal{H}^{2^mp}(\A)$ is by definition the closure of $\mathrm{span}(\mathcal{H}^{2^{m+1}s}(\A)\cdot\mathcal{H}^{2^{m+1}s}(\A))$. If we inductively apply this fact to the cases $m=0, 1, \dots k-1$, we obtain that $\mathcal{H}^{s}(\A)$ is the closure of $\mathrm{span}(\mathcal{H}^{2^ks}(\A)\cdot\mathcal{H}^{2^ks}(\A)\dots\mathcal{H}^{2^ks}(\A)$ where on the right we have the span of an $2^k$-fold product of $\mathcal{H}^{2^ks}(\A)$. For the case $p=s$ this means that $\mathrm{span}(\mathcal{H}^p(\A)\cdot\mathcal{H}^q(\A))$ and $\mathrm{span}((\mathcal{H}^{2^kp}(\A)\dots\mathcal{H}^{2^kp}(\A))(\mathcal{H}^{2^kq}(\A)\dots\mathcal{H}^{2^kq}(\A)))$ have the same closures where in each case we have a $2^k$-fold product of spaces. Since now $2^kq, 2^kp\geq 2$, it therefore follows from Proposition \ref{prop2.1} that $\mathrm{span}(\mathcal{H}^{2^kp}(\A)\mathcal{H}^{2^kq}(\A))$ and $\mathrm{span}(\mathcal{H}^{2^kq}(\A)\mathcal{H}^{2^kp}(\A))$ have the same closures, namely $\mathcal{H}^{2^kr}(\A)$. On inductively applying that to what we noted above it is clear that each of \newline $\mathrm{span}(\mathcal{H}^{2^kr}(\A)\dots\mathcal{H}^{2^kr}(\A))$, $\mathrm{span}((\mathcal{H}^{2^kp}(\A)\dots\mathcal{H}^{2^kp}(\A))(\mathcal{H}^{2^kq}(\A)\dots\mathcal{H}^{2^kq}(\A)))$ and \newline $\mathrm{span}((\mathcal{H}^{2^kp}(\A)\cdot\mathcal{H}^{2^kq}(\A))\dots(\mathcal{H}^{2^kp}(\A)\cdot\mathcal{H}^{2^kq}(\A)))$ have the same closures where in each case we have a $2^k$-fold product. But from our earlier analysis the closure of $\mathrm{span}(\mathcal{H}^{2^kr}(\A)\dots\mathcal{H}^{2^kr}(\A))$ is just $\mathcal{H}^r(\A)$, which then proves the result. 
\end{proof}

\section{The Hilbert transform}\label{S9}

Having introduced the concept of $H^p$-spaces for general von Neumann algebras, it now behoves us to develop a theory of Hilbert transforms suited to this context. In the context of finite von Neumann algebras, this was independently done by Narcisse 
Randrianantoanina \cite{Ran}, and Marsalli \& West \cite{MW}. Ji \cite{jig3} then used the reduction theorem presented in \cite{HJX}, to lift this theory to the context of $\sigma$-finite algebras, with Bekjan \cite[\S 4]{Bek-sem} using the techniques described in 
Proposition \ref{Bek-red} to lift the theory to the semifinite setting. We will show how Marsalli and West's approach can be modified so as work for even approximately subdiagonal subalgebras.

The following proposition generalises a crucial part of \cite[Theorem 9]{mar}. 

\begin{proposition} 
\begin{enumerate}
\item[(1)] Let $\A$ be an approximately subdiagonal subalgebra. For any $u\in \mathrm{Re}(\A_\gamma)$ there then exists a unique $v\in\A_\gamma$ with $P_\gamma(\mathrm{Im}(v))=0$ such that 
$u=\mathrm{Re}(v)$. If $u\in \mathrm{Re}(\A_{\gamma,0})$, then $v\in \A_{\gamma, 0}$. Moreover if in fact $u\in \mathrm{Re}[\mathfrak{n}_\nu(\A_{\gamma})\cap\mathfrak{n}_\nu(\A_{\gamma}^*)^*]$, 
then also $v\in[\mathfrak{n}_\nu(\A_{\gamma})\cap\mathfrak{n}_\nu(\A_{\gamma}^*)^*]$. 
\item[(2)]  If $\A$ is even maximal subdiagonal, then similarly there will for any $u\in \mathrm{Re}(\A)$ exist a unique $v\in\A$ with $\mathbb{E}(\mathrm{Im}(v))=0$ such that 
$u=\mathrm{Re}(v)$. If $u\in \mathrm{Re}(\A_0)$, then $v\in \A_0$. Moreover if in fact $u\in \mathrm{Re}[\mathfrak{n}_\nu(\A)\cap\mathfrak{n}_\nu(\A^*)^*]$, then also 
$v\in[\mathfrak{n}_\nu(\A)\cap\mathfrak{n}_\nu(\A^*)^*]$.
\end{enumerate}
\end{proposition}

\begin{proof} We first consider assertion (1). The proofs being similar, we consider only the case where $u\in \mathrm{Re}[\mathfrak{n}_\nu(\A_{\gamma})\cap\mathfrak{n}_\nu(\A_{\gamma}^*)^*]$. Assuming this to hold, there then exists 
$g\in [\mathfrak{n}_\nu(\A_{\gamma})\cap\mathfrak{n}_\nu(\A_{\gamma}^*)^*]$ such that $u=\mathrm{Re}(g)$. Let $a=g-\frac{1}{2}P_\gamma(g-g^*)$. It is an exercise to see that then $a\in \A_\gamma$ with
$u=\mathrm{Re}(a)$ and $P_\gamma(\mathrm{Im}(a))=\frac{1}{2i}P_\gamma(a-a^*)=0$. Since $P_\gamma$ moreover maps $\mathfrak{n}(\A_\gamma)$ into $\mathfrak{n}(\D_\gamma)$ (and similarly for 
$\mathfrak{n}(\A_\gamma^*)$) this shows existence. Now suppose we have another $a_0 \in[\mathfrak{n}_\nu(\A_{\gamma,0})\cap\mathfrak{n}_\nu(\A_{\gamma,0}^*)^*]$ with $u=\mathrm{Re}(a_0)$ and 
$P_\gamma(\mathrm{Im}(a_0))=0$. Then $a+a^*=2u=a_0^*+a_0$ which implies that $$a-a_0=a_0^*-a^*\mbox{ and therefore also }(a-a_0)^*= a_0-a.$$The fact that 
$P_\gamma(\mathrm{Im}(a))=0=P_\gamma(\mathrm{Im}(a_0))$ additionally implies that $$P_\gamma(a)=P_\gamma(a)^* \mbox { and } P_\gamma(a)=P_\gamma(a)^*$$from which we may conclude that 
$$P_\gamma(a)=\frac{1}{2}P_\gamma(a+a^*)=\frac{1}{2}P_\gamma(a_0+a^*_0) =P_\gamma(a_0).$$But for each $\alpha\geq \gamma$, $P_\alpha$ agrees with $P_\gamma$ on $\A_\gamma$.  We will for any 
$\alpha\geq \gamma$ therefore have that $$P_\alpha((a-a_0)^*(a-a_0))= P_\alpha((a_0-a)(a-a_0)) = P_\alpha(a_0-a)P_\alpha(a-a_0) =0.$$But then 
\begin{eqnarray*} 
\nu((a-a_0)^*(a-a_0)) &=& \sup_\alpha c_\alpha\nu_\alpha((a-a_0)^*(a-a_0))\\ 
&=& \sup_\alpha c_\alpha\nu_\alpha(P_\alpha((a-a_0)^*(a-a_0)))\\
&=&0.
\end{eqnarray*}
The faithfulness of $\nu$ therefore ensures that $a=a_0$.

\medskip

On replacing $\A_\gamma$ and $P_\gamma$ with $\A$ and $\mathbb{E}$ respectively, the proof of assertion (2) runs along similar lines as before.
\end{proof} 

We know from the previous proposition that for any $u\in\mathrm{Re}(\A_\gamma)$ there exists a unique $\widetilde{u}\in\mathrm{Re}(\A_\gamma)$ with $P_\gamma(\widetilde{u})=0$ (namely 
$\mathrm{Im}(v)$ in the previous proposition) such that $u+i\widetilde{u}\in \A_\gamma$. If additionally $u\in \mathrm{Re}(\mathfrak{n}_\nu(\A_{\gamma})\cap\mathfrak{n}_\nu(\A_{\gamma}^*)^*)$, then also $u+i\widetilde{u}\in \mathfrak{n}_\nu(\A_{\gamma})\cap\mathfrak{n}_\nu(\A_{\gamma}^*)^*$. In the maximal subdiagonal case the same definitions hold with $\A_\gamma$ and $P_\gamma$ replaced by $\A$ and $\mathbb{E}$. The map $u\to \widetilde{u}$ is nothing more than a non-commutative analogue of harmonic conjugation, the complexification of which (denoted by $u\to \overline{\widetilde{u}}$) is in the present context often referred to as the Hilbert transform. On the other hand the map $\mathfrak{h}:u\to u+i\widetilde{u}$ is often referred to as the Herglotz-Riesz or just Riesz transform. We proceed to describe some technical properties of this transform after which we use those properties to analyse its behaviour with respect to $L^p$ spaces. 

\begin{lemma}[Compare Lemma 5.1; \cite{MW}] Let $\A$ be a maximal subdiagonal subalgebra of a von Neumann algebra $\M$, and let $u\in \mathrm{Re}(\A)$ be given. With $\mathbb{E}$ denoting the canonical normal conditional expectation from $\M$ onto $\D$ with respect to $\nu$, we have that $u-\mathbb{E}(u)\in \mathrm{Re}(\A)$ and $\widetilde{u}=\widetilde{u-\mathbb{E}(u)}$. Moreover $\overset{\approx}{u} =\mathbb{E}(u)-u$. 
\end{lemma}

\begin{proof} The proof of \cite[Lemma 5.1(ii)(iii)]{MW} readily adapts.
\end{proof}

Armed with the above we can now extend the Generalised Riesz Theorem of Marsalli and West \cite[Theorem 5.2]{MW}.

\begin{theorem}\label{mw-5.2} Let $\A$ be an approximately subdiagonal subalgebra of a von Neumann algebra $\M$ with respect to an \emph{fns} weight $\nu$. For any positive even integer $p$ and any $u\in \mathrm{Re}\cup_\gamma(\mathfrak{n}_\nu(\A_{\gamma})\cap\mathfrak{n}_\nu(\A_{\gamma}^*)^*)$  we have that 
 $$\|\mathfrak{i}^{(p)}(\widetilde{u})\|_p\leq\frac{2p}{\log 2}\|\mathfrak{i}^{(p)}(u)\|_p.$$ 
In the maximal subdiagonal case the same statement holds with $\A_\gamma$ replaced by $\A$.
\end{theorem}

\begin{proof} With a few minor modifications the proof of \cite[Theorem 5.2]{MW} goes through. We shall merely indicate those modifications. In the case where $\A$ is not subdiagonal but only approximately subdiagonal we have that $\A=\A_0$. So in this case only the first part of the proof of \cite[Theorem 5.2]{MW} is necessary to prove the claim. In the subdiagonal case we will also need the last four lines of this proof. It is clear from Theorem \ref{tech-thm} that $H^p_0$ may be defined as the closure in the $L^p$-norm of $\mathfrak{j}^{(2p)}(\mathfrak{n}(\A_0^*)^*).\mathfrak{j}^{(2p)}(\mathfrak{n}(\A))$. A repeated application of the equalities in Theorem \ref{tech-thm} now shows that 
 $$(H_0^p)^p=[\mathfrak{j}^{(2p)}(\mathfrak{n}(\A_0^*)^*).\mathfrak{j}^{(2p)}(\mathfrak{n}(\A))]_p^p\subseteq 
 [\mathfrak{j}^{(2)}(\mathfrak{n}(\A_0^*)^*).\mathfrak{j}^{(2)}(\mathfrak{n}(\A))]_1.$$ 
Now recall that by Proposition \ref{refine-2}, 
 $$(\mathfrak{j}^{(2)}(\cup_\gamma (\mathfrak{n}(\A_{\gamma,0})\cap\mathfrak{n}(\A_{\gamma,0}^*)^*))).(\mathfrak{j}^{(2)}(\cup_\gamma (\mathfrak{n}(\A_{\gamma})\cap\mathfrak{n}(\A_{\gamma}^*)^*)))$$ 
 is dense in $[\mathfrak{j}^{(2)}(\mathfrak{n}(\A_0^*)^*).\mathfrak{j}^{(2)}(\mathfrak{n}(\A))]_1$. With $h$ denoting the density of the dual weight, we will for any 
$a \in \mathfrak{n}(\A_{\gamma})\cap\mathfrak{n}(\A_{\gamma}^*)^*$ and $b\in \mathfrak{n}(\A_{\gamma,0})\cap\mathfrak{n}(\A_{\gamma,0}^*)^*)$ (by \cite[Remark 7.41]{GLnotes}) have that 
$$tr(\mathfrak{j}^{(2)}(b^*)^*\mathfrak{j}^{(2)}(a)) =tr([b^*h^{1/2}]^*.[ah^{1/2})=\nu(ba) =$$ $$\lim_\alpha c_\alpha\nu_\alpha(ba)= \lim_\alpha c_\alpha\nu_\alpha(P_\alpha(ba))=0.$$Extending by continuity, we will therefore even in the approximately subdiagonal case have that $tr(H^1_0(\A))=0$ and hence that $tr((H_0^p(\A))^p)=0$. If now the roles of $\tau$, $x$ and $y$ in Marsalli and West's proof are here played by $tr$, $x=\mathfrak{i}^{(p)}(u)$ and $y=\mathfrak{i}^{(p)}(\widetilde{u})$, we then similarly have that $tr((x+iy)^p)=0$. In the expression near the middle of page 350 of \cite{MW} where they look at the behaviour of $|\tau(a)|$ for typical members $a$ of $Q(k,p)$, we look at $|tr(a)|$ where for a typical member of $Q(k,p)$ we have that the factors $x^{\alpha_i}$and  $y^{\beta_i}$ respectively belong to $L^{p/\alpha_i}$ and $L^{p/\beta_i}$. So one can apply the generalised version of Holder's inequality to see that 
\begin{eqnarray*}|tr(u)|\leq tr(|u|)&=& tr(|x^{\alpha_1}y^{\beta_1}\dots x^{\alpha_n}y^{\beta_n}|)\\
&\leq& \|x^{\alpha_1}\|_{p/\alpha_1}\|y^{\beta_1}\|_{p/\beta_1}\dots \|x^{\alpha_n}\|_{p/\alpha_n}\|y^{\beta_n}\|_{p/\beta_n}\\
&\leq& \|x\|_p^{\alpha_1}\|y\|_{p}^{\beta_1}\dots \|x\|_{p}^{\alpha_n}\|y\|_{p}^{\beta_n}= \|x\|_p^k\|_p^{p-k}.
\end{eqnarray*}
The rest of the proof now works as in their paper. The final four lines of their proof are only relevant for the case where $\A\neq \A_0$ which in the present context is the maximal subdiagonal case. Hence for that part we do have access to the preceding lemma, as is required by the proof.
\end{proof}

\begin{proposition} For any approximately subdiagonal subalgebra $\A$ of a von Neumann algebra $\M$ we have that $\mathrm{Re}(H^2(\A))= L^2(\M)_{sa}$. If $\A$ is even maximal subdiagonal this equality holds for all $1\leq p<\infty$
\end{proposition}

\begin{proof} The inclusion $\mathrm{Re}(H^2(\A))\subseteq L^2(\M)_{sa}$ is clear. For the converse we remind the reader that $\mathfrak{i}^{(2)}(\mathfrak{m}(\M))$ is dense in $L^2(\M)$ and $\mathfrak{i}^{(2)}(\mathfrak{n}(\A^*)^*\mathfrak{n}(\A))$ by definition dense in $H^2(\A)$. For the case where $\A$ is approximately subdiagonal, recall that by definition $\mathfrak{n}(\A) +\mathfrak{n}(\A^*)$ embeds densely into the GNS Hilbert space. By the Haagerup-Terp standard form this is equivalent to the claim that $[\mathfrak{j}^{(2)}(\mathfrak{n}(\A) +\mathfrak{n}(\A^*))]_2=L^2(\M)$. It now follows from Theorem \ref{tech-thm} that $[\mathfrak{i}^{(2)}(\mathfrak{n}(\A) +\mathfrak{n}(\A^*))]_2=L^2(\M)$, and hence that $\mathrm{Re}(H^2(\A)) = \mathrm{Re}[\mathfrak{j}^{(2)}(\mathfrak{n}(\A) +\mathfrak{n}(\A^*))]_2=  L^2(\M)_{sa}$ for any will for any $b\in \mathfrak{m}(\M)$.

Now suppose that $\A$ is maximal subdiagonal. We may then select a net $(f_\lambda)\subset \mathcal{D}$ satisfying the criteria of Proposition \ref{7:P Terp2}. The $\sigma$-weak density of $\mathfrak{n}(\A^*)^*\mathfrak{n}(\A) + \mathfrak{n}(\A)^*\mathfrak{n}(\A^*)$ in $\M$ ensures that for any $c\in \mathfrak{m}(\M)$ we may select nets $(a_\alpha), (b_\alpha)\subset \mathfrak{n}(\A^*)^*\mathfrak{n}(\A)$ such that $(a_\alpha+b_\alpha^*)$ converges $\sigma$-weakly to $c$. We clearly have that  $(f_\gamma(a_\alpha+b_\alpha^*)f_\lambda)\subset \mathfrak{n}(\A^*)^*\mathfrak{n}(\A) + \mathfrak{n}(\A)^*\mathfrak{n}(\A^*)$. For fixed $\gamma$ and $\lambda$ we have that $$\mathfrak{i}^{(p)}(f_\gamma(a_\alpha+b_\alpha^*)f_\lambda)=(h^{1/2p}f_\gamma)(a_\alpha+b_\alpha^*)[f_\lambda h^{1/2p}]$$converges $L^p$-weakly to $$(h^{1/2p}f_\gamma)c[f_\lambda h^{1/2p}]= \mathfrak{i}^{(p)}(f_\gamma cf_\lambda)$$as $\alpha$ increases. By Lemma \ref{GL2-2.4+5} $$(h^{1/2p}f_\gamma)c[f_\lambda h^{1/2p}]= \sigma_{-i/2p}(f_\gamma)\mathfrak{i}^{(p)}(c)\sigma_{i/2p}(f_\lambda).$$For each fixed $\gamma$ the net $(\sigma_{-i/2p}(f_\gamma)\mathfrak{i}^{(p)}(c)\sigma_{i/2p}(f_\lambda))_{\lambda\in \Lambda}$ converges $L^p$-weakly to $\sigma_{-i/2p}(f_\gamma)\mathfrak{i}^{(p)}(c)$ as $\lambda$ increases, with $(\sigma_{-i/2p}(f_\gamma)\mathfrak{i}^{(p)}(c))$ converging $L^p$-weakly to $\mathfrak{i}^{(p)}(c)$. Thus $\mathfrak{i}^{(p)}(c)$ belongs to the weak closure of $\mathfrak{i}^{(p)}(\mathfrak{n}(\A^*)^*\mathfrak{n}(\A) + \mathfrak{n}(\A)^*\mathfrak{n}(\A^*))$ which by convexity agrees with the norm closure. Since 
 $$\mathrm{Re}(\mathfrak{i}^{(p)}(\mathfrak{n}(\A^*)^*\mathfrak{n}(\A) + \mathfrak{n}(\A)^*\mathfrak{n}(\A^*)))= \mathrm{Re}(\mathfrak{i}^{(p)}(\mathfrak{n}(\A^*)^*\mathfrak{n}(\A))),$$ 
 we have that $$\mathrm{Re}(\mathfrak{i}^{(p)}(c))\in \mathrm{Re}[\mathfrak{i}^{(p)}(\mathfrak{n}(\A^*)^*\mathfrak{n}(\A)]_p=\mathrm{Re}(H^p(\A)).$$Since $\mathfrak{i}^{(p)}(\mathfrak{m}(\M))$ is dense in $L^p(\M)$, we are done.
\end{proof}

Armed with the above result we may now use Theorem \ref{mw-5.2} to prove the following result. 

\begin{theorem}[Compare \cite{MW}, Theorem 5.4] If $\A$ is an approximately subdiagonal subalgebra of a von Neumann algebra $\M$, the maps 
$\sim:\mathrm{Re}(\A)\to \mathrm{Re}(\A)$ and $\mathfrak{h}:\mathrm{Re}(\A)\to \A$ induce bounded real linear maps $\sim: L^2(\M)_{sa}\to L^2(\M)_{sa}$ and 
$\mathfrak{h}:L^2(\M)_{sa}\to L^2(\M)$. If $\A$ is even maximal subdiagonal we will for any $1<p<\infty$ obtain bounded real linear maps 
$\sim: L^p(\M)_{sa}\to L^p(\M)_{sa}$ and $\mathfrak{h}:L^p(\M)_{sa}\to L^p(\M)$. The complexification of the map $\sim$ given by 
$\overline{\sim}:L^p(\M)\to L^p(\M): a\to \widetilde{\mathrm{Re}(a)} +i \widetilde{\mathrm{Im}(a)}$ is a bounded linear map with norm of order $p$ if $p\geq 2$ 
and norm of order $\frac{1}{p-1}$ if $1<p< 2$.
\end{theorem}

\begin{proof} We leave claim regarding approximately subdiagonal subalgebras as an exercise. For the maximal subdiagonal case we know from Theorem \ref{mw-5.2} 
and the preceding proposition that $u\to\widetilde{u}$ will for any positive even integer $p$ induce a bounded real linear map on $L^p(\M)_{sa}$ with norm majorised by $\frac{2p}{\log 2}$. Given any $a\in L^p(\M)$ we then have that $$\|\overline{\widetilde{a}}\|_p\leq \|\widetilde{\mathrm{Re}(a)}\|_p +\|\widetilde{\mathrm{Im}(a)}\|_p\leq \frac{2p}{\log 2}(\|\mathrm{Re}(a)\|_p +\|\mathrm{Im}(a)\|_p)\leq \frac{4p}{\log 2}\|a\|_p.$$Thus 
$a\to \overline{\widetilde{a}}$ is a bounded operator with norm of order $p$. Haagerup $L^p$ spaces are known to be a complex interpolation scale and hence the 
fact that the same claim also holds for all $p\geq 2$ is a consequence of Riesz-Thorin interpolation. The fact that for $1<p<2$ the action 
$a\to \overline{\widetilde{a}}$ is a bounded operator with norm of order $\frac{1}{p-1}$ follows once we show that up to a sign chage this operator is just the 
adjoint of the operator $\overline{\sim}:L^q(\M)\to L^q(\M)$ where $1=\frac{1}{p}+\frac{1}{q}$. To see this let conjugate indices $p$ and $q$ be given with 
$1<p<2$. For any $x, y\in \mathrm{Re}(\mathfrak{n}(\A^*)^*\mathfrak{n}(\A))$ we then have that 
\begin{eqnarray*}
tr(\mathbb{E}(\mathfrak{i}^{(p)}(u))\mathbb{E}(\mathfrak{i}^{(q)}(v))) &=& tr(\mathbb{E}(\mathfrak{i}^{(p)}(u+i\widetilde{u}))\mathbb{E}(\mathfrak{i}^{(q)}(v+i\widetilde{v})))\\
&=& tr(\mathbb{E}(\mathfrak{i}^{(p)}(u+i\widetilde{u})\mathfrak{i}^{(q)}(v+i\widetilde{v})))\\ 
&=& tr(\mathfrak{i}^{(p)}(u+i\widetilde{u})\mathfrak{i}^{(q)}(v+i\widetilde{v}))\\ 
&=& tr(\mathfrak{i}^{(p)}(u)\mathfrak{i}^{(q)}(v)) - tr(\mathfrak{i}^{(p)}(\widetilde{u})\mathfrak{i}^{(q)}(\widetilde{v}))\\
&&\qquad + i(tr(\mathfrak{i}^{(p)}(\widetilde{u})\mathfrak{i}^{(q)}(v)) + tr(\mathfrak{i}^{(p)}(u)\mathfrak{i}^{(q)}(\widetilde{v})))
\end{eqnarray*}
Since each of $\mathfrak{i}^{(p)}(u)$, $\mathfrak{i}^{(p)}(\widetilde{u})$, $\mathfrak{i}^{(q)}(v)$ and $\mathfrak{i}^{(q)}(\widetilde{v})$ are self-adjoint, the expressions $tr(\mathbb{E}(\mathfrak{i}^{(p)}(u))\mathbb{E}(\mathfrak{i}^{(q)}(v)))$ and $tr(\mathfrak{i}^{(p)}(u)\mathfrak{i}^{(q)}(v)) - tr(\mathfrak{i}^{(p)}(\widetilde{u})\mathfrak{i}^{(q)}(\widetilde{v}))$ are real-valued. Hence we must have that $tr(\mathfrak{i}^{(p)}(\widetilde{u})\mathfrak{i}^{(q)}(v)) = -tr(\mathfrak{i}^{(p)}(u)\mathfrak{i}^{(q)}(\widetilde{v}))$. The density of \newline $\mathrm{Re}(\mathfrak{i}^{(p)}(\mathfrak{n}(\A^*)^*\mathfrak{n}(\A)))$ and $\mathrm{Re}(\mathfrak{i}^{(q)}(\mathfrak{n}(\A^*)^*\mathfrak{n}(\A)))$ in $L^p(\M)_{sa}$ and $L^q(\M)_{sa}$ respectively, now ensures that $tr(\widetilde{a}b) = -tr(a\widetilde{b})$ for each $a \in L^p(\M)_{sa}$ and $b\in L^p(\M)_{sa}$. Thus the adjoint of the real linear $\sim$ on $L^q(\M)$ is the real linear map ${-}{\sim}$ on $L^p(\M)$ which must have norm of order $\frac{1}{p-1}$ (since the pre-adjoint has norm of order $q$). The complexification $\overline{\sim}:L^p(\M)\to L^p(\M)$ therefore also has norm of order $\frac{1}{p-1}$.
\end{proof}

\begin{remark} A slight modification of the argument at the start of the proof of Theorem \ref{mw-5.2} shows that even for approximately subdiagonal subalgebras we have that $H_0^2(\A^*)\perp H^2(\A)$ with respect to the inner product $\langle a, b\rangle =tr (b^*a)$ defined on $L^2(\M)$. When this is combined with the $L^2$-density noted in the previous theorem, it is clear that $L^2(\M)= H^2(\A)\oplus H_0^2(\A^*)$. As we shall shortly see we can do much better in the case of maximal subdiagonal subalgebras.
\end{remark}

We proceed to generalise \cite[Theorem 6.2 \& Corollary 6.3]{MW}.

\begin{lemma} Let $\A$ be a maximal subdiagonal subalgebra of $\M$. For any $x\in H^p(\A)$ we have that 
$$\mathrm{Re}(x)= \mathbb{E}(\mathrm{Re}(x))- \widetilde{\mathrm{Re}(x)} \mbox{ and } \mathrm{Re}(x)= \mathbb{E}(\mathrm{Im}(x))+ \widetilde{\mathrm{Im}(x)}.$$
Moreover $x= \mathbb{E}(x)+i\overline{\widetilde{x}}$.
\end{lemma}
 
The proof follows similar lines as \cite[Lemma 5.1(ii)(iii)]{MW} and is omitted. 

\begin{theorem}\label{complementHp} Let $\A$ be a maximal subdiagonal subalgebra of $\M$ and suppose that $1<p<\infty$. Then $L^p(\M)=H^p_0(\A)\oplus L^p(\D)\oplus H^p_0(\A)^*$ with the corresponding projections 
respectively given by $x\mapsto \frac{1}{2}(x+i\overline{\widetilde{x}}-\mathbb{E}x)$, $x\mapsto \mathbb{E}x$ and $x\mapsto \frac{1}{2}(x-i\overline{\widetilde{x}}-\mathbb{E}x)$. The first and last projection both 
have norm of the same order as $\sim$.
\end{theorem}

\begin{proof} The proof of \cite[Theorem 6.2]{MW} goes through virtually unaltered.
\end{proof}

\begin{corollary} Let $\A$ be a maximal subdiagonal subalgebra of $\M$. For any $1<p,q<\infty$ such that $1=\frac{1}{p}+\frac{1}{q}$ the dual space of $H^p(\A)$ is conjugate isomorphic to $H^q(\A)$.
\end{corollary}

\begin{proof} The proof of \cite[Corollary 6.2]{MW} goes through unaltered.
\end{proof}

With the above technology at our disposal, we are now able to refine Arveson's maximality criterion. The reader should compare this result to Theorem 2.2 of \cite{jisa} where the $\sigma$-finite case was established.

\begin{theorem}\label{newArvmax} Let $\A$ be a maximal subdiagonal subalgebra of $\M$. Then $\A=\{x\in \M\colon \mathbb{E}(xa)=0, a\in \A_0\}=\{x\in \M\colon \mathbb{E}(ax)=0, a\in \A_0\}$.
\end{theorem}

\begin{proof} The proofs of the two claims are similar, and hence we will only show that $\A =\{x\in \M\colon \mathbb{E}(xa)=0, a\in \A_0\}$. We will show that $\A_m=\{x\in \M\colon \mathbb{E}(xa)=0, a\in \A_0\}$ where $\A_m$ is as in Theorem \ref{arvmax}. Write $\mathcal{B}$ for $\{x\in \M\colon \mathbb{E}(xa)=0, a\in \A_0\}$. From the
definition of $\A_m$, we have $\A_m\subseteq \mathcal{B}$. 

Conversely, take any $x\in \mathcal{B}$, and let $(f_\lambda)$ be the net in $\mathfrak{n}(\D)_\nu$ converging strongly to 
$\I$ as guaranteed by Proposition \ref{7:P Terp2}. For now we fix $\lambda$. For any $a\in \mathfrak{n}(\A)$ and $b\in 
\mathfrak{n}(\A_0)$ we then have that $$tr(b(h^{1/2}f_\lambda)x[ah^{1/2}]) = tr((h^{1/2}f_\lambda)xa(h^{1/2}b)).$$Since 
$(h^{1/2}b)\in H^2_0$, we may use Theorem \ref{tech-thm} to select a sequence $(c_n)\subset \mathfrak{n}(\A_0)_\nu$ such that 
$([c_nh^{1/2}])$ converges to $h^{1/2}b$ in $L^2$-norm. So we will then have that 
\begin{eqnarray*}
tr((h^{1/2}f_\lambda)xa(h^{1/2}b))&=&\lim_n tr((h^{1/2}f_\lambda)xa[c_n h^{1/2}])\\
&=&\lim_n tr(\mathfrak{i}^{(1)}f_\lambda xa c_n))\\
&=&\lim_n\nu(f_\lambda xac_n)\\
&=& \lim_n\nu(\mathbb{E}(f_\lambda xac_n))\\
&=& \lim_n\nu(f_\lambda\mathbb{E}(xac_n))\\
&=& 0.
\end{eqnarray*}
If we combine the above two centred equations, we obtain that 
 $$tr([bh^{1/2}]f_\lambda x(ah^{1/2})) = tr(b(h^{1/2}f_\lambda)x(ah^{1/2})) = 0.$$ 
The $\sigma$-strong* convergence of $(f_\lambda)$ to $\I$ ensures that $[bh^{1/2}]f_\lambda$ converges $L^2$-weakly to $[bh^{1/2}]$. Hence $tr([bh^{1/2}]x[ah^{1/2}]) = 0$. Another application of Theorem \ref{tech-thm} shows that $\{[bh^{1/2}]\colon b\in \mathfrak{n}(\A_0)_\nu\}$ is dense in $H^2_0$, and $\{[ah^{1/2}]\colon a\in \mathfrak{n}(\A)_\nu\}$ dense in $H^2$. It therefore follows that $xH^2\perp (H_0^2)^*$ and hence that $xH^2\subset H^2$. 

On swopping the roles of $a$ and $b$ above, a similar argument will show that $tr([ah^{1/2}]x(bh^{1/2}))=0$ where as before $a\in \mathfrak{n}(\A)$ and $b\in \mathfrak{n}(\A_0)$. On applying Theorem \ref{tech-thm}, it now follows that $tr(v^*xw^*)=\overline{tr(wxv)}=0$ for $v\in H^2_0$ and $w\in H^2$. So $x^*(H^2)^*\perp H^2_0$, whence $x^*(H^2)^*\subset (H^2)^*$. The astute reader will now notice that The spaces $H^2$ and $(H^2)^*$ respectively correspond to the spaces $H_1\oplus H_2$ and $H_2\oplus H_3$ described in Lemma \ref{JOS 2.1-part1}. Hence the fact that $xH^2\subset H^2$ and $x^*(H^2)^*\subset (H^2)^*$, ensures that $x$ belongs to the algebra $\A_M$ described in the proof of Theorem \ref{arvmax}. But in that same proof we also show that $\A_M=\A_m$. Hence we are done.
\end{proof}

\section{An analytic reduction technique}\label{S10}

Having already posited a basic theory of $H^p$-spaces for general von Neumann algebras we now briefly describe two reduction techniques that have proven their worth in developing the theory of noncommutative $H^p$-spaces. The first is a technique developed by Bekjan (see \cite{Bek-sem}) for lifting results valid for finite von Neumann algebras equipped with a faithful normal tracial state to the context of semifinite algebras. The second was pioneered by Xu \cite{Xu} and demonstrates the remarkable efficacy of the Haagerup reduction theorem for studying maximal subdiagonal subalgebras.

Let $\M$ be a semifinite von Neumann algebra equipped with a faithful normal semifinite trace $\tau$. Given a maximal subdiagonal subalgebra $\A$ of $\M$ for which the restriction of $\tau$ to $\D=\A^*\cap\A$ is still semifinite, one may then select a net $(e_\alpha)$ of projections in $\D$ all with finite trace increasing to $\I$. Bekjan's technique of lifting the theory of subdiagonal subalgebras of finite von Neumann algebras to semifinite algebras, consists of the following collection of facts all gleaned from \cite{Bek-sem}.

\begin{proposition}\label{Bek-red} Let $\M$, $\A$ and $(e_\alpha)$ be as above.
\begin{itemize}
\item Let $e\in \D$ be a projection with finite trace. Then $\A_e=e\A e$ is a subdiagonal subalgebra of $e\M e$ with respect to $\E_e=\E{\upharpoonright}e\D e$ and with diagonal $\D_e=e\D e$. In addition $(\A_e)_0=e\A_0e$. Moreover for any $0<p<\infty$ we have that $L^p(\M_e,\tau_e)=eL^p(\M,\tau)e$, $H^p(\A_e)=eH^p(\A)e$ and $H^p_0(\A_e)=eH^p_0(\A)e$. (Here $\tau_e$ denotes the restriction of $\tau$ to $e\M e$.)
\item For any $0<p<\infty$, $H^p(\A)$ and $H^p_0(\A)$ are respectively the closures of $\cup_\alpha \A_{e_\alpha}$ and $\cup_\alpha (\A_{e_\alpha})_0$ in $L^p(\M,\tau)$.
\end{itemize}
\end{proposition}

We pass to the second reduction principle, which is a description of the efficacy of Haagerup's reduction theorem for subdiagonal subalgebras. This in some sense constitutes an ``analytic reduction theorem''. When Xu originally developed this technique, his focus was its use in lifting results from the context of finite von Neumann algebras equipped with a faithful normal tracial state, to $\sigma$-finite von Neumann algebras. However with a sharpened version of the reduction theorem now in place, this technique serves as a means for lifting results from the context of semifinite von Neumann algebras to general von Neumann algebras. This technique is contained in \cite[Lemmata 3.1-3.3]{Xu}. The proofs of these lemmata carry over almost verbatim to the present context. The only change that needs to be made, is that references to Exel's maximality result for finite von Neumann algebras \cite{Exel}, should be replaced by references to Ji's maximality result for semifinite algebras \cite{jig1}, and the version of the reduction theorem in the present paper should be used instead of the one in \cite{HJX}. In terms of the notation in the reduction theorem, we specifically obtain the following:

\begin{proposition}\label{anred}
 Let $\A$ be a maximal subdiagonal subalgebra with respect to $\D$ and let $\widehat{\A}$ and $\widehat{\D}$ respectively be the $\sigma$-weak closures of $\{\pi_\nu(a)\lambda_t:a\in \A,\, t\in \mathbb{Q}_D\}$ and $\{\pi_\nu(d)\lambda_t:d\in \D,\, t\in \mathbb{Q}_D\}$ in $\R$. Then the following holds:
\begin{enumerate}
\item[(1)]
\begin{itemize}
\item The expectation $\mathbb{E}:\M\to\D$ extends to a faithful normal conditional expectation $\widehat{\E}:\R\to\widehat{D}$ which 
may be realised by the prescription $\widehat{\E}(\lambda_t a)=\lambda_t\E(a)$ where $a\in \M$ and $t\in \mathbb{Q}_D$. This expectation 
similarly satisfies $\tnu\circ\widehat{\E}=\tnu$ where $\tnu$ is dual weight on $\R$. Moreover 
$\E\circ\mathscr{W}_{\mathbb{Q}_D}=\mathscr{W}_{\mathbb{Q}_D}\circ\widehat{\E}$.

\item $\widehat{\A}$ is a maximal subdiagonal subalgebra with respect to $\widehat{\D}$. Moreover $\mathscr{W}_{\mathbb{Q}_D}$ maps $\widehat{\A}$ onto $\A$. 
\end{itemize}

\item[(2)]
\begin{itemize}
\item For each $n$ we have that $\widehat{\E}\circ\mathscr{W}_n=\mathscr{W}_n\circ\widehat{\E}$. The map $\E_n=\mathscr{W}_n\circ\widehat{\E}$ is therefore a faithful normal conditional expectation from $\R$ onto $\D_n$ where $\D_n=\widehat{\D}\cap \R_n$. For the canonical trace $\tau_n=(\nu_n){\upharpoonright}\R_n$ on $\R_n$ (see Lemma \ref{construction of Rn}) we moreover have that $\tau_n\circ\E_n=\tau_n$. 

\item For each $n\in \mathbb{N}$, $\A_n=\widehat{\A}\cap \R_n$ is a maximal subdiagonal subalgebra with respect to both of the pairings $(\D_n, \tnu{\upharpoonright}\R_n)$ and $(\D_n, \tau_n)$. Moreover $\mathscr{W}_n$ maps $\widehat{\A}$ onto $\A_n$, and $\cup_{n\in \mathbb{N}} \A_n$ is $\sigma$-weakly dense in $\widehat{\A}$. 
\end{itemize}

\item[(3)] Let $1\leq p < \infty$ be given. We shall write $H^p(\A_n)$ for the $H^p$ space computed using $\tnu{\upharpoonright}\R_n$, and $H^p(\A_n,\tau_n)$ for the $H^p$-space computed using the trace $\tau_n$. 
Then the extension $\mathscr{W}^{(p)}_{\mathbb{Q}_D}$ of $\mathscr{W}_{\mathbb{Q}_D}$ to $L^p(\R)$ maps $H^p(\widehat{\A})$ onto $H^p(\A)$. Similarly the extension $\mathscr{W}_n^{(p)} $of 
$\mathscr{W}_n$ to $L^p(\R)$ maps $H^p(\widehat{\A})$ onto $H^p(\A_n)$. Moreover $(H^p(\A_n))$ is an increasing sequence of $H^p$-spaces for which $\cup_{n\geq 1}H^p(\A_n)$ is dense in 
$H^p(\widehat{\A})$. For each $1<p<\infty$ and each $f\in H^p(\A)$, the sequence $(\mathscr{W}_n^{(p)}(f))$ is weakly convergent to $f$. The sequence of isometric embeddings $J_n$ described in the proof 
of Theorem \ref{10:T reduction Lp} will moreover map each $H^p(\A_n,\tau_n)$ onto the corresponding $H^p(\A_n)$. Similar claims hold for $\A_0$. 
\end{enumerate}
\end{proposition}

It should be noted that in the case of subdiagonal subalgebras, the above result can be used to give an alternative proof of the boundedness of the Hilbert transform.

\begin{proof} \textbf{Part (1):} The first part of bullet 1 is an easy consequence of \cite[Theorem 4.1]{HJX}. For the second part observe that both $\E\circ\mathscr{W}_{\mathbb{Q}_D}$ and $\mathscr{W}_{\mathbb{Q}_D}\circ\widehat{\E}$ are normal. Direct checking using formula (\ref{Waction}) now shows that both these operators map terms of the form $\lambda_t a$ where $a\in \M$ and $t\in \mathbb{Q}_D$, onto 0 if $t\neq 0$, and onto $\E(a)$ otherwise. Since $\mathrm{span}\{\lambda_ta: a\in \M,\, t\in \mathbb{Q}_D\}$ is $\sigma$-weakly dense in $\R$, we are done. As noted earlier the second bullet may be verified using the same proof that Xu used in \cite[Lemmata 3.1 \& 3.3]{Xu}, but with references to Exel's maximality result for finite von Neumann algebras \cite{Exel} replaced by references to Ji's maximality result for semifinite algebras \cite{jig1}. The final claim of the second bullet is an easy consequence of formula (\ref{Waction}).

\medskip
 
\textbf{Part (2):} Notice that all the $\lambda_t$'s where $t\in \mathbb{Q}_D$, belong to $\widehat{\D}$. Therefore each $e^{ita_{n}}$ 
where $a_n$ is as in formula (\ref{formula od sfn}), also belongs to $\widehat{\D}$. Since by part (a) $\widehat{\E}$ commutes with 
$\sigma_t^{\tnu}$, it follows from formula (\ref{formula od sfn}) that 
 \begin{eqnarray*}
 \widehat{\E}(\sigma_t^{\nu_n}(x))
 &=& \widehat{\E}(e^{-ita_{n}}\sigma_{t}^{\tnu}(x)e^{ita_{n}})= e^{-ita_{n}}\widehat{\E}(\sigma_{t}^{\tnu}(x))e^{ita_{n}}\\
 &=& e^{-ita_{n}}\sigma_{t}^{\tnu}(\widehat{\E}(x))e^{ita_{n}}=\sigma_t^{\nu_n}(\widehat{\E}(x))
  \end{eqnarray*}
for all $x\in \R$. The fact that 
$\widehat{\E}\circ\mathscr{W}_n=\mathscr{W}_n\circ\widehat{\E}$, now follows from the definition of $\mathscr{W}_n$. Since as noted above 
$e^{-a_{n}/2}$ belongs to $\widehat{\D}$, we may similarly use formula (\ref{def of fn}) to see that 
 \begin{eqnarray*}
 \nu_n(\widehat{\E}(x))
 &=&\tnu(e^{-a_{n}/2}\widehat{\E}(x)e^{-a_{n}/2}) =\tnu(\widehat{\E}(e^{-a_{n}/2}xe^{-a_{n}/2}))\\
 &=&\tnu(e^{-a_{n}/2}xe^{-a_{n}/2})=\nu_n(x)
  \end{eqnarray*}
for all $x\in \R$. This clearly suffices to 
prove the last claim of the first bullet.

The validity of the first part of the second bullet may be verified using essentially the same proof as was used by Xu in \cite[Lemma 3.2]{Xu} (again with references to Exel's maximality result replaced by references to Ji's maximality result). Although not explicitly stated, the claim regarding the range of $\mathscr{W}_n$, was noted by Xu at the end of the proof of \cite[Lemma 3.2]{Xu}. The final claim in this bullet is now an easy consequence of the fact that for any $a\in \widehat{\A}$, $\mathscr{W}_n(a)$ will by the reduction theorem, converge $\sigma$-weakly to $a$.  

\medskip
 
\textbf{Part (3):} All the claims, except for the claims regarding density, are fairly easy consequences of Theorem \ref{10:T reduction Lp} and the techniques developed in its proof. We shall therefore not insult the reader's sensibilities by repeating obvious facts. We pass to commenting on the density claim regarding $\cup_{n\geq 1}H^p_0(\A_n)$. We know that $\cup_{n\geq 1}H^p_0(\A_n)\subset H^p_0(\widehat{\A})$ and that $\cup_{n\geq 1}L^p(\D_n)$ and $\cup_{n\geq 1}L^p(\R_n)$ are respectively dense in $L^p(\D)$ and $L^p(\R)$. For the case of $1<p<\infty$ it now clearly follows from Theorem \ref{complementHp}, that 
 $$\cup_{n\geq 1}L^p(\R_n)=(\cup_{n\geq 1} H^p_0(\A_n))\oplus (\cup_{n\geq 1}L^p(\D))\oplus (\cup_{n\geq 1} H^p_0(\A_n)^*)$$ 
will not be dense in $L^p(\R)=H^p_0(\widehat{\A})\oplus L^p(\widehat{\D}) \oplus H^p_0(\widehat{\A})^*$ if $\cup_{n\geq 1}H^p((\A_n)_0)$ is not dense in $H^p(\widehat{\A}_0)$. Thus by Theorem 
\ref{10:T reduction Lp} the claim holds for the case $1<p<\infty$. 

For the case $p=1$ it follows from Definition \ref{defHp} and Theorem \ref{tech-thm} that $H^1_0(\widehat{\A})=[H^2(\widehat{\A}).H^2_0(\widehat{\A})]_1$. Any $x\in H^1_0(\widehat{\A})$ may therefore be written in the form $x=ab$ where $a\in H^2(\widehat{\A})$ and $b\in H^2_0(\widehat{\A})$. From what we have already shown there exist sequences $(a_n)$ and $(b_n)$ $a_n \in H^2(\A_n)$ and $b_n \in H^2_0(\A_n)$ for each $n$, which respectively converge in $L^2$-norm to $a$ and $b$. But then each $a_nb_n\in H^1_0(\A_n)$ for each $n$, with $(a_nb_n)$ converging in $L^1$-norm to $ab=x$.

The density claim regarding $\cup_{n\geq 1}H^p(\A_n)$ may be verified in a similar fashion.
\end{proof}

\begin{remark}\label{doublered} When these techniques are placed alongside each other, a double reduction method emerges for lifting results from the context finite von Neumann algebras equipped with a faithful normal tracial state to general von Neumann algebras. As a first step results regarding classical $H^p(\mathbb{T})$ spaces are lifted to the setting of finite von Neumann algebras equipped with finite faithful normal tracial states. The next step is to use Bekjan's method as described in Proposition \ref{Bek-red} to lift the results for finite von Neumann algebras equipped with a faithful normal tracial state, to semifinite von Neumann algebras, with the final step being an application of the analytic reduction theorem (Proposition \ref{anred}) to pass from semifinite von Neumann algebras to general von Neumann algebras. We shall have one occasion to demonstrate this double reduction technique - the comparison of left and right Toeplitz operators.   
\end{remark}

\section{A Beurling theory of invariant subspaces}\label{S11}

For a regular approximately subdiagonal subalgebra $\A$ of $\M$ we define a {\em (right) $\A$-invariant subspace} 
of $L^p(\M)$, to be a closed subspace $K$ of $L^p(M)$ such that $K \A \subset K$.  Invariant subspaces may be 
classified in accordance with their structure. In this regard we say that an invariant subspace $K$ is 
{\em simply invariant} if in addition the closure of $K \A_0$ is properly contained in $K$. 
Given a right $\A$-invariant subspace $K$ of $L^2(\M)$, we define the {\em right wandering subspace} 
of $K$ to be the space $W = K \ominus [K \A_0]_2$, and then say that $K$ is {\em type 1} if
$W$ generates $K$ as an $\A$-module (that is, $K = [W \A]_2$), and {\em type 2} if $W = \{0\}$. If $\A$ is a regular 
approximately subdiagonal subalgebra for which $\A=\A_0$, then all right $\A$-invariant subspaces are by default 
of type 2 which then forces a great simplification of the theory. \emph{We will therefore restrict attention to maximal 
subdiagonal subalgebras in the rest of this section (the case where $\A\neq \A_0$).} For consistency, we will not consider 
left invariant subspaces at all, leaving the reader to verify that entirely symmetric 
results pertain in the left invariant case.

The theory of invariant subspaces for $H^2(\mathbb{D})$ was of course pioneered by Beurling, and 
forms a very important part of the classical theory. In \cite{BL-Beurling}, Blecher and Labuschagne 
extended the classical Beurling theory to the setting of finite maximal subdiagonal subalgebras, in 
the process also showing that the noncommutative theory allows for a much more intricate structure 
than the classical theory. This theory was then first extended to maximal subdiagonal subalgebras of 
semi-finite von Neumann algebras by Sager in \cite{sager}, and then almost simultaneously to $\sigma$-finite von Neumann 
algebras by Labuschagne in \cite{L-HpIII} for the case of $p=2$. Subsequently Bekjan and Raikhan showed that the $\sigma$-finite theory also goes through for the case where $p\neq 2$. With the appropriate technology now at our disposal, 
we show that the theory carries over to the context of general von Neumann algebras. From a conceptual point of view, we shall closely follow the outlines of \cite{BL-Beurling} and \cite{L-HpIII}. It should however be noted that the bulk of the development of the theory for the case $p\neq 2$ required completely new proof strategies.

The development of especially the $L^p$-version of the theory of closed right $\A$-invariant subspaces, makes deep use of the concept of a `column $L^p$-sum' as introduced in \cite{js}. Given $1 \leq p < \infty$ and a collection $\{ X_i : i \in I \}$ of closed subspaces of $L^p(M)$, the {\em external} column $L^p$-sum $\oplus^{col}_i \, X_i$ is defined to be the
closure of the restricted algebraic sum in the norm
 $$\Vert (x_i) \Vert_p \overset{def}{=} tr((\sum_i \, x_i^* x_i)^{\frac{p}{2}})^{\frac{1}{p}}.$$
That this is a norm for $1 \leq p < \infty$ is verified in \cite{js}. If $X$ is a subspace of $L^p(M)$, and if 
$\{ X_i : i \in I \}$ is a collection of  subspaces of $X$, which together densely span $X$, with the added property that 
$X_i^* X_j = \{ 0 \}$ if $i \neq j$, then we say that $X$ is the {\em internal} column $L^p$-sum $\oplus^{col}_i \, X_i$. 
We shall not need the concept of an external column sum. So wherever column sum is mentioned below, it shall refer to an 
internal column sum. Note that  if $J$ is a finite subset of $I$, and if $x_i \in X_i\subset L^p$ for all $i \in J$, then we have that $$tr(|\sum_{i \in J} \, x_i|^p)^{1/p} = tr((|\sum_{i \in J} \, x_i|^2)^{\frac{p}{2}})^{1/p}
= tr((\sum_{i \in J} \, x_i^* x_i)^{\frac{p}{2}})^{1/p}.$$This shows that $X$ is then isometrically isomorphic to the external column $L^p$-sum $\oplus^{col}_i \, X_i$. Since the projections onto the summands are clearly
contractive, it follows by routine arguments (or by \cite[Lemma 2.4]{js}) that if $(x_i)\in \oplus^{col}_i \, X_i$, then the net $(\sum_{j \in J} \, x_j)$, indexed by the finite subsets $J$ of $I$, converges in norm to $(x_i)$. 

\subsection{Invariant subspaces of $L^2(\M)$}

The first cycle of results we present are extensions of corresponding results in \S 2 of \cite{L-HpIII}. 
The first result in this regard is basically a restatement of \cite[Theorem 2.4]{L-HpIII}. The exact same 
proof offered in \cite{L-HpIII} goes through in the general setting and hence we forgo the proof. 

\begin{theorem} \label{inv1}
Let $\A$ be an analytically conditioned algebra.    \begin{itemize}
\item [(1)]   Suppose that $X$ is a subspace of $L^2(\M)$
of the form $X = Z \oplus^{col} [Y\A]_2$ where $Z, Y$ are closed
subspaces of $X$, with $Z$
a type 2 invariant subspace,
and $\{y^*x : y, x \in Y \} = Y^*Y \subset L^1({\mathcal D})$.  Then
$X$ is simply right $\A$-invariant if and only if $Y \neq
\{0\}$.
\item [(2)]  If $X$ is as in {\rm (1)},
then $[Y {\mathcal D}]_2 = X \ominus [X\A_0]_2$ (and
$X = [X\A_0]_2 \oplus [Y {\mathcal D}]_2$).
\item [(3)]   If $X$ is as described in {\rm (1)},
then that description also holds if $Y$ is replaced by $[Y {\mathcal D}]_2$.  Thus
(after making this replacement)
we may assume that $Y$ is a ${\mathcal D}$-submodule of $X$.

\item [(4)]   The subspaces $[Y {\mathcal D}]_2$ and $Z$ in the decomposition
in  {\rm (1)} are uniquely determined by $X$.  So is $Y$ if we
take it to be a ${\mathcal D}$-submodule (see {\rm (3)}).
\item [(5)]  If $\A$ is maximal subdiagonal, then any right $\A$-invariant subspace
$X$ of
$L^2(\M)$ is of the form described in {\rm (1)},
with $Y$ the right wandering subspace of $X$.
\end{itemize}
\end{theorem}

Building on Theorem \ref{inv1}, we are now able to present the following rather elegant decomposition of the 
right wandering subspace. This extends \cite[Proposition 2.5]{L-HpIII}. The proof of the general case is quite 
a bit more tricky than that of the $\sigma$-finite case, and hence full details need to be provided. 

\begin{theorem} \label{newpr}  Suppose that $X$ is
as in Theorem {\rm
\ref{inv1}}, and that $W$ is the right wandering subspace of $X$.
Then $W$ may be decomposed as
an orthogonal direct sum $\oplus^{col}_i \, u_i L^2({\mathcal D})$,
where $u_i$ are partial isometries in $\M$ for which 
$u_i(\frac{d\widetilde{\varphi}}{d\tau_L})^{1/2}a\in W$ for each $a\in \mathfrak{n}(\A)_\nu$, with
$u_i^* u_i \in {\mathcal D}$, and $u_j^* u_i = 0$ if
$i \neq j$.   If $W$ has a cyclic vector for the ${\mathcal D}$-action,
then we need only one partial isometry in the above.
\end{theorem}

\begin{proof}   By the theory of representations of a
von Neumann algebra (see e.g.\ the discussion at the
start of Section 3 in
\cite{js}), any normal Hilbert
${\mathcal D}$-module is an $L^2$ direct sum of cyclic
Hilbert ${\mathcal D}$-modules,
and if $K$ is a normal
cyclic Hilbert ${\mathcal D}$-module, then
$K$ is spatially isomorphic to $eL^2({\mathcal D})$, for
an orthogonal projection $e \in {\mathcal D}$.

Suppose that the latter isomorphism is implemented by a unitary ${\mathcal D}$-module map $\psi$. Let $(f_\lambda)$ be the 
net in $\mathfrak{n}(\D)_\nu^*\cap \mathfrak{n}(\D)_\nu$ converging strongly to $\I$ guaranteed by Proposition \ref{7:P Terp2}. If in addition 
$K \subset W$, we will then have that $g_\lambda = \psi(e[f_\lambda h^{1/2}]) \in W$ for each $\lambda$, where 
$h=\frac{d\widetilde{\varphi}}{d\tau_L}$. Then 
$$tr(d^* g_\lambda^* g_\lambda d) = \Vert \psi(e[f_\lambda h^{1/2}] d) \Vert_2^2 = tr(d^*(h^{1/2}f_\lambda^*) e[f_\lambda h^{1/2}] d),$$for each $d \in {\mathcal D}$, and so 
$g_\lambda^*g_\lambda=(h^{1/2}f_\lambda^*) e[f_\lambda h^{1/2}]=|e[f_\lambda h^{1/2}]|^2$. Hence there exists a partial 
isometry $u_\lambda$ majorised by $e$ such that $g_\lambda=u_\lambda e[f_\lambda h^{1/2}] =u_\lambda [f_\lambda h^{1/2}]$. 
By the modular action of $\psi$ we will then have that $\psi(e[f_\lambda h^{1/2}] d)= g_\lambda d=u_\lambda [f_\lambda h^{1/2}] d$ for any $d\in \mathcal{D}$. Since $L^2(\mathcal{D})$ is the closure of 
$\{(h^{1/2}d): d\in \mathfrak{n}(\mathcal{D})_\nu\}$ (see \cite[Proposition 7.40 \& Theorem 7.45]{GLnotes}), and since 
$\psi(e[f_\lambda h^{1/2}] d)= u_\lambda [f_\lambda h^{1/2}] d =(u_\lambda f_\lambda)(h^{1/2} d)$ for each 
$d\in \mathfrak{n}(\mathcal{D})_\nu^*$, it follows $\psi(ef_\lambda b)=u_\lambda f_\lambda b$ for all $b\in L^2(\mathcal{D})$.

When working with $\mathcal{D}$, we may of course assume that $\mathcal{D}$ is in standard form, in which case the 
Haagerup-Terp standard form (Theorem \ref{7:T stdform}) enables us to further identify $L^2(\M)$ with the underlying Hilbert space of $\M$. 
But then the $\sigma$-strong* convergence of $(f_\lambda)$ to $\I$ ensures that $ef_\lambda b$ will for any 
$b\in L^2(\mathcal{D})$ converge in $L^2$-norm to $eb$. Since 
 $$\|eb-ef_\lambda b\|_2 =\|\psi(eb-ef_\lambda b)\|_2=\|\psi(eb)-u_\lambda f_\lambda b\|_2,$$ 
this in turn ensures that $(u_\lambda f_\lambda b)$ converges to 
$\psi(eb)$ in $L^2$-norm. Given that the net $(u_\lambda f_\lambda)$ is in the unit ball of $\M$, it must admit 
a subnet $(u_\gamma f_\gamma)$ which converges to some element $u_e$ of the unit ball of $\M$. For any 
$b\in L^2(\D)$ the net $(u_\gamma f_\gamma b)$ will then converge to $u_eb$ in the $L^2$-weak topology. But 
$(u_\gamma f_\gamma b)$ is also a subnet of $(u_\lambda f_\lambda b)$ which converges to $\psi(eb)$, and will 
therefore itself still be $L^2$-norm convergent to $\psi(eb)$. It is therefore clear that 
$\psi(eb)=u_eb$ for each $b\in L^2(\D)$ and hence that $(u_\lambda f_\lambda b)$ is for each $b\in L^2(\D)$, 
$L^2$-norm convergent to $u_eb$. For any $b\in L^2(\D)$, we now also have that
\begin{eqnarray*}
tr(d^*b^*ebd) &=& \|ebd\|_2^2\\
&=& \lim_\lambda\|ef_\lambda bd\|_2^2\\
&=& \lim_\lambda\|\psi(ef_\lambda bd)\|_2^2\\
&=& \lim_\lambda\|\psi(ef_\lambda b)d\|_2^2\\
&=& \lim_\lambda\|u_\lambda f_\lambda bd\|_2^2\\
&=& \|u_ebd\|_2^2\\
&=& tr(d^*b^*u_e^*u_ebd).
\end{eqnarray*}
This equality firstly ensures that $b^*eb=b^*u_e^*u_eb$ for all $b\in L^2(\D)$, which then in turn ensures that 
$u_e^*u_e=e$. It follows that $u_e$ is a partial isometry with initial projection $e$, and that 
$\psi(eL^2(\mathcal{D}))=u_eL^2(\mathcal{D})$.

Given $u_i$ and $u_j$ with $i\neq j$, we have that $u_iL^2(\mathcal{D}), u_jL^2(\mathcal{D})\subset W$. 
Hence $L^2(\mathcal{D})u_j^* u_iL^2(\mathcal{D}) \subset L^1({\mathcal D})$. Since for 
any $d_0, d_1\in \mathfrak{n}(\mathcal{D})_\nu$ we have that  
\begin{eqnarray*}
tr([d_1^*h^{1/2}]u_j^* u_i(h^{1/2}d_0)) &=&  tr(\psi(e_j(h^{1/2}d_1))^*\psi(e_i(h^{1/2}d_0)))\\
&=& tr([d_1^*h^{1/2}]e_je_i(h^{1/2}d_0))\\
&=&0,
\end{eqnarray*}
the density of $\{(h^{1/2}d): d\in \mathfrak{n}(\mathcal{D})_\nu\}$ in $L^2(\D)$ now ensures that $u_j^* u_i=0$. In the case 
where $i=j$ we of course have that $u_i^*u_i=e_i\in \mathcal{D}$. Putting these facts together,
we see that $W$ is of the desired form.
\end{proof}

The first corollary of the above theorem corresponds to \cite[Corollary 2.5]{L-HpIII}. Here too the proof of 
the general case requires more delicacy than that of the $\sigma$-finite case, and hence we state the proof in full.

\begin{corollary}  Suppose that $X$ is
as in Theorem {\rm \ref{inv1}}, and that $W$ is the right wandering subspace of $X$.
If indeed $X\subset \mathcal{H}^2(\A)$, then $Z\perp L^2(\mathcal{D})$. If additionally $\A$ is maximal subdiagonal, 
then the partial isometries $u_i$ described in the preceding Proposition, all belong to $\A$.
\end{corollary}

\begin{proof} If indeed $X\subset\mathcal{H}^2(\A)$, it is a fairly trivial observation to make that 
 $$Z=[Z\A_0]_2 \subset [X\A_0]_2\subset [\mathcal{H}^2(\A)\A_0]_2=\mathcal{H}^2_0(\A).$$ 
Since 
$L^2(\mathcal{D})\subset \mathcal{H}^2_0(\A)^*\cap\mathcal{H}^2_0(\A)$, it clearly follows from Corollary 
\ref{gen-perp} that $\mathcal{H}^2(\A)=\mathcal{H}^2_0(\A)\oplus L^2(\mathcal{D})$, and hence the first claim follows. 

Now suppose that $\A$ is maximal subdiagonal. To see the second claim recall that in the proof of Theorem \ref{newpr}, 
we showed that $u_iL^2(\mathcal{D})\subset W$ for each $i$.

Hence given any $a_0\in \A_0$, we will for any $b\in L^2(\D)$ therefore have that 
$a_0u_ib\in a_0W \subset \A_0X \subset \A_0H^2(\A) \subset H^2_0(\A)$. But $\mathbb{E}_2$ annihilates 
$H^2_0(\A)$, and hence we must have that $0=\mathbb{E}_2(a_0u_ib)=\mathbb{E}(a_0u_i)b$ for all $b\in L^2(\D)$. This 
can of course only be if $\mathbb{E}(a_0u_i)=0$. Since $a_0\in \A_0$ was arbitrary, we may now apply the sharpened 
Arveson maximality criterion (Theorem \ref{newArvmax}) to see that $u_i\in \A$ as claimed. 
\end{proof}

The next three results correspond to \cite[Corollary 2.7, Proposition 2.8 \& Theorem 2.9]{L-HpIII}. The proofs in 
\cite{L-HpIII} carry over to the general case, and hence we content ourselves with merely stating these results

\begin{corollary} \label{adcor}  If $X$ is an
invariant subspace of the form described in
Theorem {\rm \ref{inv1}}, then $X$ is type 1 if and only if
 $X = \oplus^{col}_i \, u_i \mathcal{H}^2(\A)$, for $u_i$ as
in Theorem {\rm \ref{newpr}}.
\end{corollary}

\begin{proof}  If $X$ is type 1, then $X = [W \A]_2$
where $W$ is the right wandering space, and so the one assertion
follows from Theorem {\rm \ref{newpr}}.
If $X = \oplus^{col}_i \, u_i \mathcal{H}^2(\A)$, for $u_i$ as above, then
$[X \A_0]_2 = \oplus^{col}_i \, u_i \mathcal{H}^2(\A_0)$, and from this
it is easy to argue that  $W =
\oplus^{col}_i \, u_i L^2({\mathcal D})$.
Thus $X = [W \A]_2 = \oplus^{col}_i \, u_i \mathcal{H}^2(\A)$.
  \end{proof}
  
\begin{proposition} \label{typestuff}   Let $X$ be a
closed $\A$-invariant subspace of $L^2(\M)$, where $\A$ is an
analytically conditioned subalgebra of $\M$.
\begin{itemize}  \item [(1)]
If $X = Z \oplus [Y \A]_2$ as in Theorem
 {\rm \ref{inv1}}, then $Z$ is type 2, and $[Y \A]_2$ is type 1.
 \item [(2)]  If $\A$ is a maximal subdiagonal algebra,
and if $X = K_2 \oplus^{col} K_{1}$ where
$K_1$ and $K_{2}$ are types 1 and 2 respectively,
then $K_1$ and $K_2$ are respectively the unique
spaces $Z$ and $[Y \A]_2$ in  Theorem  {\rm \ref{inv1}}.
 \item [(3)]  If $\A$ and $X$ are as in {\rm (2)},
and if $X$ is type 1 (resp.\ type 2),
then the space  $Z$ of Theorem  {\rm \ref{inv1}}
for $X$ is $(0)$ (resp.\ $Z = X$).
 \item [(4)]   If  $X = K_2 \oplus^{col} K_{1}$ where
$K_1$ and $K_{2}$ are types 1 and 2 respectively,
then the right wandering subspace for $X$
equals the right wandering subspace for $K_1$.
  \end{itemize} \end{proposition}

On collating the information contained in the preceding four results, we obtain the following structure theorem for invariant subspaces of $L^2$.

\begin{theorem} \label{main}  If $\A$ is a maximal subdiagonal
subalgebra of $\M$, and if $K$ is a closed right  $\A$-invariant subspace 
of $L^2(\M)$,
then: \begin{itemize}
\item [(1)]   $K$ may be written uniquely as
an (internal) $L^2$-column sum $K_2 \oplus^{col} K_{1}$
of a type 1 and a type 2 invariant subspace of $L^2(\M)$, respectively.
  \item [(2)]  If
$K  \neq (0)$ then $K$ is type 1 if and only if
$K = \oplus_i^{col} \, u_i \, H^2$, for $u_i$  partial isometries
with mutually orthogonal ranges and $|u_i| \in {\mathcal D}$.
\item [(3)]
The right wandering subspace $W$ of $K$
is an $L^2({\mathcal D})$-module in the sense of Junge and Sherman, and in
particular $W^* W \subset L^{1}({\mathcal D})$.
\end{itemize}
\end{theorem}

\subsection{Invariant subspaces of $L^p(\M)$}

In this subsection we shall consistently assume that $\A$ is in fact maximal subdiagonal. In extending the theory of $\A$-invariant subspaces to the case $p\neq 2$, we shall follow a novel path based on an extension of Saito's insightful density theorem \cite{Sai} to the general case. This result was subsequently extended to the semifinite setting by Sager [Proposition 4.1(ii), \cite{sager}]. We take time to briefly reinterpret Sager's result before proving the result we need.

\begin{lemma}[Proposition 4.1(ii), \cite{sager}] Let $1\leq p<\infty$ be given, and let $\M$ be a semifinite von Neumann algebra equipped with a faithful normal semifinite trace $\tau$ and $\A$ be a maximal subdiagonal subalgebra of $\M$ for which the restriction of $\tau$ to $\D=\A^*\cap\A$ is still semifinite. For any right $\A$-invariant subspace $K$ of $L^p(\M,\tau)$, the subspace $\mathfrak{m}_\tau\cap K$ is a norm dense right $\A$-invariant subspace of $K$.
\end{lemma}

\begin{proof} Sager proves a version of this result for $\M\cap K$ rather than $\mathfrak{m}_\tau\cap K$. However her proof suffices to also prove the stated hypothesis. To see this one need only note that the operator $h_1ex$ constructed in the final stages of her proof, does not just belong to $\M$, but even to $\mathfrak{m}_\tau$. To see this notice that in her proof $e\in\mathfrak{m}_\tau$. Since in the tracial case $\mathfrak{m}_\tau$ is a two-sided ideal, the claim now follows from the fact that $h_1ex = eh_1ex$. (Here we used the fact that in her proof $h_1\in eH^\infty e$. This proves the first claim. The right $\A$-invariance of $\mathfrak{m}_\tau\cap K$ follows from the right $\A$-invariance of $K$ and the fact noted earlier that here $\mathfrak{m}_\tau$ is a two-sided ideal. 
\end{proof}

In the ensuing analysis we shall repeatedly use the more general embeddings $\mathfrak{i}_c^{(p)}$ of $\mathfrak{m}_\nu$ into 
$L^p$ introduced at the end of section \ref{S3}. It will however be very important to keep track of the ``powers of $h$'' 
involved in a particular embedding. So for that reason we shall in this section not use the notation $\mathfrak{i}_c^{(p)}$ 
but in for example the case $2\leq p<\infty$, rather respectively write $\clo{h^{c/p}xh^{(1-c)/p}}$ and 
$\clo{h^{c/p}Sh^{(1-c)/p}}$ for $\mathfrak{i}_c^{(p)}(x)$ and $\mathfrak{i}_c^{(p)}(S)$ where $x\in \mathfrak{m}_\nu$ and 
$S\subset \mathfrak{m}_\nu$.

\begin{theorem}\label{techinv-Lp} Let $K$ be a right $\A$-invariant closed subspace of $L^p(\M)$ where $1\leq p\leq \infty$ ($\sigma$-weakly closed in the case of $p=\infty$). 
\begin{enumerate}
\item For any $2\leq p \leq \infty$ and any $0\leq c\leq 1$, there exists a right $\mathfrak{n}(\A)_\nu$-invariant subspace $\mathcal{S}(K)_p^{(c)}$ of $\mathfrak{m}_\nu$ for which $\langle h^{c/p}\mathcal{S}(K)_p^{(c)}h^{(1-c)/p}\rangle$ is norm-dense in $K$ if $p<\infty$ and $\sigma$-weakly dense if $p=\infty$.
\item Given $1\leq p <2$ select $q,r>0$ so that $\frac{1}{r}+\frac{1}{2}=\frac{1}{p}$ and $\frac{1}{p}+\frac{1}{q}=1$.
There exists a right $\mathfrak{n}(\A)_\nu$-invariant subspace $\mathcal{S}(K)_p^{(c)}$ of $\mathfrak{m}_\nu$ for which $\langle h^{c/q}h^{1/r}{S}(K)_p^{(c)}h^{1/r}h^{(1-c)/q}\rangle$ is norm dense in $K$.
\end{enumerate}
\end{theorem} 

The symmetry of the theory of left and right $\A$-invariant subspaces ensures that a version of the above also holds for left $\A$-invariant subspaces, with left $\mathfrak{n}(\A^*)^*_\nu$-invariance taking the place of right $\mathfrak{n}(\A)_\nu$-invariance.

\begin{proof} We shall use the reduction theorem to prove the theorem. All notation will therefore be as in Proposition 
\ref{anred}. For the moment assume that $1\leq p<\infty$ and let $K$ be a right $\A$-invariant subspace of $L^p(\M)$. We start our proof by noting several facts:
\begin{itemize}
\item One may use Proposition \ref{anred} to check that $\widehat{K}=[K\widehat{\A}]_p$ is a right $\widehat{\A}$-invariant subspace of $L^p(\R)$. Since ${\mathscr{W}}^{(p)}_{\mathbb{Q}_D}(K\widehat{\A})=K\A=K$, continuity ensures that ${\mathscr{W}}^{(p)}_{\mathbb{Q}_D}$ maps $\widehat{K}$ onto $K$. 
\item The technology in Proposition \ref{anred} now enables us to conclude from the above that each  $K_n=\mathscr{W}^{(p)}_n(\widehat{K})$ is a right $\A_n$-invariant subspace of $L^p(\R_n)$. 
\item It is clear from for example \cite[Theorem VIII.2.11]{Tak2} and part (2) of Lemma \ref{construction of Rn}, that 
$\frac{d\nu_n}{d\tnu{\upharpoonright}\R_n}=e^{-a_n}$ where $0\leq a_n=2^{n}b_n \leq 2^{n+1}\pi$ by Lemma \ref{bn} and the discussion 
following it.
\item Since by definition $a_n=2^nb_n=-2^ni\log(\lambda_{2^{-n}})$, it is clear that in fact $a_n\in \D_n$ 
\end{itemize} 

By the discussion in section \ref{S3}, the crossed product $\mathcal{C}_n=\R_n\rtimes_{\tnu}\mathbb{R}$ may be regarded as a 
subspace of $\mathcal{C}=\R\rtimes_{\tnu}\mathbb{R}$, in which case we may then identify $h$ with 
$\frac{d\tnu{\upharpoonright}\R_n}{d\tau_{\mathcal{C}_n}}$. (Here $\tau_{\mathcal{C}_n}$ is the canonical \emph{fns} trace on the crossed product 
$\R_n\rtimes_{\tnu}\mathbb{R}$.) We know from \cite[Theorem 6.65]{GLnotes} that $\mathcal{C}_n=\R_n\rtimes_{\tnu}\mathbb{R}$ 
may be identified with $\mathcal{B}_n=\R_n\rtimes_{\nu_n}\mathbb{R}$ by means of an implemented $*$-isomorphism 
$\mathscr{I}$. It moreover follows from \cite[Proposition 7.14]{GLnotes} that an extension of this isomorphism 
homeomorphically identifies $L^p(\R_n)$ constructed with respect to $\tnu{\upharpoonright}\R_n$, with $L^p(\R_n, \nu_n)$ constructed using 
$\nu_n$. It further follows from \cite[Theorems 6.74 \& 7.5]{GLnotes}, that up to Fourier transform, $L^p(\R_n,\nu_n)$ constructed 
using $\nu_n$, is nothing but the space of simple tensors $\{a\otimes \frac{1}{\exp}^{1/p}\colon a\in L^p(\R_n,\tau)\}$ which 
all lie in the space of $\tau$-measurable operators affiliated with the von Neumann algebra tensor product 
$\R_n\overline{\otimes}L^\infty(\mathbb{R})$. In this identification, the operator 
$k_n=\frac{d\widetilde{\nu_n}}{d\tau_{\mathcal{B}_n}}$ corresponds to $\I\otimes \frac{1}{\exp}$. Amongst other facts, this 
ensures that inside $\mathcal{B}_n$, $k_n$ commutes with $\R_n$. The isometry from $L^p(\R_n,\nu_n)$ onto $L^p(\R_n)$ 
guaranteed by Theorem \ref{10:T reduction Lp}, is then up to Fourier transform nothing but the composition of the extension 
of $\mathscr{I}$ to the $\tau$-measurable operators, composed with the map $a\mapsto a\otimes \frac{1}{\exp}^{1/p}$. 

Careful checking of the proof of \cite[Theorem 6.65]{GLnotes} shows that $\mathscr{I}$ maps $\widetilde{\lambda}_t$ 
associated with $\mathcal{C}_n$, onto $(D\tnu{\upharpoonright}\R_n{:}\,D\nu_n)_t\widetilde{\lambda}_t$ where here $\widetilde{\lambda}_t$ is 
associated with $\mathcal{B}_n$. (Here we write $\widetilde{\lambda}_t$ for the shift operators used to construct the crossed 
products $\mathcal{C}_n$ and $\mathcal{B}_n$, to distinguish them from the operators used to construct $\R$.) By for example 
\cite[Theorem VIII.2.11]{Tak2}, we here have that $(D\tnu{\upharpoonright}\R_n{:}\,D\nu_n)_t = (\frac{d\tnu{\upharpoonright}\R_n}{d\nu_n})^{it}$. On now 
invoking \cite[Proposition 6.61]{GLnotes}, it follows that the $*$-isomorphism between the two crossed products maps each 
$h^{it}$ onto $\frac{d\tnu{\upharpoonright}\R_n}{d\nu_n}^{it}k_n^{it}=(e^{a_n})^{it}k_n^{it}$. Since the operators $a_n$ and $k_n$ are 
known to commute, we may now apply the Borel functional calculus to conclude that this isomorphism associates the operator $h$
with $e^{a_n}k_n$.

With the preparation done, we proceed with the proof of parts (1) and (2) under the assumption that $p<\infty$. Fix 
$0\leq c\leq 1$. It is clear that $\mathscr{I}$ and the map $a\mapsto a\otimes\frac{1}{\exp}^{1/p}$ may be used to pull $K_n$ 
back to an $\A_n$ invariant closed subspace $L_n$ of $L^p(\R_n,\nu_n)$. The space $e^{-ca_n/p}L_n$ is then clearly again a 
right invariant subspace of $L^p(\R_n,\nu_n)$. By the lemma, $\mathfrak{m}_{\nu_n}\cap e^{-ca_n/p}L_n$ is a norm dense right 
$\A$-invariant subspace of $e^{-ca_n/p}L_n$. Using the fact that $e^{(1-c)a_n/p}$ is an invertible element of $\D_n$ and 
$\mathfrak{m}_{\nu_n}$ a two sided ideal, it follows that 
 $$e^{ca_n/p}(\mathfrak{m}_{\nu_n}\cap e^{-ca_n/p}L_n)e^{(1-c)a_n/p} = \mathfrak{m}_{\nu_n}\cap L_n$$ 
is a norm dense subspace of $L_ne^{(1-c)a_n/p}=L_n$. We may now first apply the map 
$x\to x\otimes\exp^{-1/p}$ to $\mathfrak{m}_{\nu_n}\cap L_n$, followed by the Fourier transform to transform 
$\mathfrak{m}_{\nu_n}\cap L_n$ to a dense subspace of $\mathscr{I}^{-1}(K_n)$. This subspace is of the form 
$(e^{a_n}k_n)^{c/p}(\mathfrak{m}_{\nu_n}\cap e^{-ca_n/p}L_n)(e^{a_n}k_n)^{(1-c)/p}$. The fact that inside 
$\mathcal{B}_n=\R_n\rtimes_{\nu_n}\mathbb{R}$, $k_n$ commutes with $\R_n$, ensures that no questions regarding the 
$\tau$-measurability of the elements of this subspace arise at this point. Recalling that $\mathscr{I}$ associates 
$e^{a_n}k_n$ with $h$, inside $\mathcal{C}_n=\R_n\rtimes_{\tnu}\mathbb{R}$ this dense subspace of $K_n$ is then of the form 
$h^{c/p}(\mathfrak{m}_{\nu_n}\cap e^{-ca_n/p}L_n)h^{(1-c)/p}$. To allay any concerns about the relation of members of 
$\mathfrak{m}_{\nu_n}$ to $\mathfrak{m}_{\tnu}$, we point out that the very definition of $\nu_n$ ensures that we will for 
any $x\in \mathfrak{m}_{\nu_n}^+$ have that $$\tnu(x)=\nu_n(e^{a_n/2}xe^{a_n/2})\leq \|e^{a_n}\|_\infty\nu_n(x)<\infty.$$

Recall that by the proof of part (5) of the proof of Lemma \ref{construction of Rn}, $h_n=e^{a_{n+1}-a_n}$ belongs to the 
centre of $\R_n$. This and the right $\A$-invariance of $L_n$, ensures that $$e^{-ca_n/p}L_n = e^{-ca_{n+1}/p}L_n(e^{(a_{n+1}-a_n)})^{c/p}= e^{-ca_{n+1}/p}L_n.$$Given 
$x\in \mathfrak{m}_{\nu_n}^+$ we again have by part 
(5) of the proof of Lemma \ref{construction of Rn}, that $$\nu_{n+1}(x)=\nu_n(h_nx)\leq \|h_n\|_\infty\nu_n(x)<\infty.$$We 
therefore also have that $\mathfrak{m}_{\nu_n}\subset \mathfrak{m}_{\nu_{n+1}}$. But then 
$\mathfrak{m}_{\nu_n}\cap e^{-ca_n/p}L_n \subset \mathfrak{m}_{\tau_{n+1}}\cap e^{-ca_{n+1}/p}L_{n+1}$, which ensures that 
$\mathfrak{S}(K)_{p}^{(c)}=\cup_{n\geq 1}(\mathfrak{m}_{\nu_n}\cap e^{-ca_n/p}L_n)$ is a linear space. Hence 
$\clo{h^{c/p}\mathfrak{S}(K)_{p}^{(c)}h^{(1-c)/p}}$ is a subspace of $\cup_{n\geq 1}K_n \subset L^p(\R)$ whose closure 
includes $\cup_{n\geq 1}K_n$. But by part (3) of Theorem \ref{10:T reduction Lp}, $\cup_{n\geq 1}K_n$ is $L^p$-weakly dense 
(and hence norm dense by convexity) in $K$. It is now clear that the closure of 
$\clo{h^{c/p}\mathfrak{S}(K)_{p}^{(c)}h^{(1-c)/p}}$ is $\overline{\cup_{n\geq 1}K_n}=\widehat{K}$. 

On passing from $\R$ to $\M$, we need to differentiate between the cases $2\leq p<\infty$ and $1\leq p<2$. The proofs are 
similar and hence we only prove the first case. The $L^p$-continuity of $\mathscr{W}^{(p)}_{\mathbb{Q}_D}$ and the very specific 
action of this expectation, ensures that it maps $\clo{h^{c/p}\mathfrak{S}(K)_{p}^{(c)}h^{(1-c)/p}}$ onto the dense subspace 
$\clo{h^{c/p}\mathscr{W}_{\mathbb{Q}_D}(\mathfrak{S}(K)_{p}^{(c)})h^{(1-c)/p}}$ of $\mathscr{W}_{\mathbb{Q}_D}(\widehat{K})= K$. Notice that since for any $x\in \R^+$ we have that $\tnu(x)= \nu(\mathscr{W}_{\mathbb{Q}_D}(x))$, we clearly have that 
$\mathscr{W}_{\mathbb{Q}_D}(\mathfrak{S}(K)_{p}{(c)})\subset \mathfrak{m}_\nu$. The right $\A$-invariance of $K$, now ensures 
that $$\clo{h^{c/p}\mathscr{W}_{\mathbb{Q}_D}(\mathfrak{S}(K)_{p}^{(c)})h^{(1-c)/p}}\mathfrak{n}(\A^*)^*_\nu\subset K,$$with 
the fact that 
 $$[\clo{h^{(1-c)/p}\mathfrak{n}(\A^*)^*_\nu}]_{(1-c)/p}= [\clo{\mathfrak{n}(\A)_\nu h^{(1-c)/p}}]_{(1-c)/p}$$ 
ensuring that $$[\clo{h^{c/p}\mathscr{W}_{\mathbb{Q}_D}(\mathfrak{S}(K)_{p}^{(c)})\mathfrak{n}(\A)_\nu h^{(1-c)/p}}]_p= [\clo{h^{c/p}\mathscr{W}_{\mathbb{Q}_D}(\mathfrak{S}(K)_{p}^{(c)})h^{(1-c)/p}}\mathfrak{n}(\A^*)^*_\nu]_{p}\subset K.$$It is 
an easy exercise to check that $\mathscr{W}_{\mathbb{Q}_D}(\mathfrak{S}(K)_{p}^{(c)})\mathfrak{n}(\A)_\nu\subset \mathfrak{m}_\nu$. It is the subspace $\mathscr{W}_{\mathbb{Q}_D}(\mathfrak{S}(K)_{p}^{(c)})\mathfrak{n}(\A)_\nu$ that we 
define to be $\mathcal{S}(K)_p^{(c)}$. This subspace is clearly right $\mathfrak{n}(\A)_\nu$-invariant. To conclude the first 
part of the proof we therefore need to show that $[\clo{h^{c/p}\mathcal{S}(K)_p^{(c)}h^{(1-c)/p}}]_p$ is all of $K$. This will 
follow if we can show that it includes $\clo{h^{c/p}\mathscr{W}_{\mathbb{Q}_D}(\mathfrak{S}(K)_p^{(c)})h^{(1-c)/p}}$. For this we shall use the net $(f_\lambda)\subset \D$ described in Proposition \ref{7:P Terp2}. For any $\lambda$ and any 
$x\in \mathscr{W}_{\mathbb{Q}_D}(\mathfrak{S}(K)_{p}^{(c)})$, $\clo{h^{c/p}x f_\lambda h^{(1-c)/p}}$ belongs to 
$\clo{h^{c/p}\mathcal{S}(K)_p^{(c)}h^{(1-c)/p}}$. Now notice that 
\begin{eqnarray*}
\clo{h^{c/p}x f_\lambda h^{(1-c)/p}}&=&(h^{c/p}x)[f_\lambda h^{(1-c)/p}]\\
&=&(h^{c/p}x) (h^{(1-c)/p}\sigma_{i(1-c)/p}^\nu(f_\lambda))\\
&=& \clo{h^{c/p}xh^{(1-c)/p}} \sigma_{i(1-c)/p}^\nu(f_\lambda).
\end{eqnarray*}
Since $(\sigma_{i(1-c)/p}^\nu(f_\lambda))$ is $\sigma$-weakly convergent to $\I$, 
$(\clo{h^{c/p}xh^{(1-c)/p}}\sigma_{i(1-c)/p}^\nu(f_\lambda))$ is $L^p$-weakly convergent to $\clo{h^{c/p}xh^{(1-c)/p}}$, 
which then ensures that $$\clo{h^{c/p}xh^{(1-c)/p}}\in [\clo{h^{c/p}\mathcal{S}(K)_p^{(c)}h^{(1-c)/p}}]_p$$as required.

In conclusion we pass to proving the validity of claim (1) in the case $p=\infty$. Let $K$ be a $\sigma$-weakly closed right 
$\A$-invariant subspace of $\M$. Then the polar set $K^\circ$ is a left $\A$-invariant subspace of $L^1(\M)$. From the left invariant 
version of what we have already proven, there exists a left $\mathfrak{n}(\A^*)_\nu^*$-invariant subspace 
$\mathcal{S}(K^\circ)_1$ of $\mathfrak{m}_\nu$ such that $\clo{h^{1/2}\mathcal{S}(K^\circ)_1h^{1/2}}$ is norm dense in 
$K^\circ$. We now use $\mathcal{S}(K^\circ)_1$ to construct the closed subspace $K_2=[\clo{\mathcal{S}(K^\circ)_1h^{1/2}}]_2$ 
of $L^2(\M)$. The left $\mathfrak{n}(\A^*)_\nu^*$-invariance of $\mathcal{S}(K^\circ)_1$ ensures that $K_2$ itself is left 
$\mathfrak{n}(\A^*)_\nu^*$-invariant, and hence left $\A$-invariant (using the fact that $\mathfrak{n}(\A^*)_\nu^*$ is 
$\sigma$-weakly dense in $\A$). Then the polar $K_2^\circ$ is right $\A$-invariant. By what we have already shown, 
$K_2^\circ$ is of the form $K_2^\circ=[\clo{\mathcal{S}h^{1/2}}]_2$ for some right $\mathfrak{n}(\A)_\nu$-invariant subspace 
$\mathcal{S}$ of $\mathfrak{m}_\nu$. 

We will show that $K=\overline{\mathcal{S}}^{w^*}$ which will prove the theorem. Firstly notice that if $x\in \mathcal{S}$, 
we will for any $y\in \mathcal{S}(K^\circ)_1$ have that $0=tr([xh^{1/2}].[yh^{1/2}])=tr(x\clo{h^{1/2}yh^{1/2}})$. This follows because $[yh^{1/2}]\in K_2$ 
and $[xh^{1/2}]\in K_2^\circ$. Hence $x\in \clo{h^{1/2}\mathcal{S}(K^\circ)_1h^{1/2}}^\circ= (K^\circ)^\circ=K$, which shows that 
$\overline{\mathcal{S}}^{w^*}\subseteq K$. 

To prove the reverse containment, we firstly note that the right $\mathfrak{n}(\A)_\nu$-invariance of $\mathcal{S}$, ensures 
that $\overline{\mathcal{S}}^{w^*}$ is right invariant with respect to multiplication by elements of 
$\overline{\mathfrak{n}(\A)_\nu}^{w^*}=\A$. Thus $(\overline{\mathcal{S}}^{w^*})^\circ$ is a left $\A$-invariant subspace of 
$L^1(\M)$. Hence there exists a left $\mathfrak{n}(\A^*)_\nu^*$-invariant subspace $\mathcal{S}_0$ of $\mathfrak{m}_\nu$ such 
that $(\overline{\mathcal{S}}^{w^*})^\circ=[\clo{h^{1/2}\mathcal{S}_0h^{1/2}}]_1$. Given $w\in\mathcal{S}_0$, we will for any 
$z\in \mathcal{S}$ have that $$0=tr(z.\clo{h^{1/2}wh^{1/2}})=tr([zh^{1/2}].[wh^{1/2}]),$$which ensures that then 
$[wh^{1/2}]\in [\clo{\mathcal{S}h^{1/2}}]^\circ=K_2$; in other words that $[\clo{\mathcal{S}_0h^{1/2}}]_2\subset K_2^\circ$. 
For any given $w\in\mathcal{S}_0$, we may then use the fact that $K_2^\circ=[\clo{\mathcal{S}(K^\circ)_1h^{1/2}}]_2$, to 
select a sequence $(s_n)\subset \mathcal{S}(K^\circ)_1$ such that $[s_nh^{1/2}]\to [wh^{1/2}]$ in $L^2$-norm. Now let 
$(f_\lambda)\subset \D$ be as before. For any fixed $\lambda$, H\"older's inequality then ensures that 
$$f_\lambda\clo{h^{1/2}s_nh^{1/2}}=[f_\lambda h^{1/2}][s_nh^{1/2}]\to[f_\lambda h^{1/2}][wh^{1/2}]= f_\lambda\clo{h^{1/2}wh^{1/2}}$$in 
$L^1$-norm. The left $\A$-invariance of $K^\circ$ will, when combined with the fact that 
$K^\circ = [\clo{h^{1/2}\mathcal{S}(K^\circ)_1h^{1/2}}]$, ensure that each $f_\lambda\clo{h^{1/2}s_nh^{1/2}}$ belongs to 
$K^\circ$ and hence that the net $(f_\lambda\clo{h^{1/2}wh^{1/2}})$ is contained in $K^\circ$. The fact $(f_\lambda)$ is 
$\sigma$-weakly convergent to $\I$, now ensures that $(f_\lambda\clo{h^{1/2}wh^{1/2}})$ is $L^1$-weakly convergent to 
$\clo{h^{1/2}wh^{1/2}}$. But by convexity $K^\circ$ is $L^1$-weakly closed, which then ensures that 
$\clo{h^{1/2}wh^{1/2}}\in K^\circ$; i.e. that $\clo{h^{1/2}\mathcal{S}_0h^{1/2}}\subseteq K^\circ$. Taking the  
closure then shows that $(\overline{\mathcal{S}}^{w^*})^\circ\subseteq K^\circ$, or equivalently that 
$K\subseteq \overline{\mathcal{S}}^{w^*}$ as was required.
\end{proof}

The following result extends \cite[Corollary 4.3]{BL-Beurling}. Whereas this was an easy corollary in the setting of finite von Neumann algebras, the passage to the general case demands some deep analysis. 

\begin{theorem}\label{latt-isom} For any $1\leq p, q\leq \infty$ there is a lattice isomorphism from the $\sigma$-weakly closed right 
$\A$-invariant subspaces of $L^p$ to those of $L^q$. We in particular have the following:
\begin{enumerate}
\item Given $2\leq p < \infty$ and a right $\A$-invariant closed subspace $K$ of $L^p(\M)$, the prescription 
$K\to \overline{\mathcal{S}(K)_p}^{w^*}$ where $\mathcal{S}(K)_p$ is a right $\mathfrak{n}(\A)_\nu$-invariant subspace of 
$\mathfrak{n}_\nu$ for which $K=[\clo{\mathcal{S}(K)_ph^{1/p}}]_p$, realises a lattice isomorphism from the right 
$\A$-invariant subspaces of $L^p(\M)$ to those of $L^\infty(\M)$.
\item Given $1\leq p < 2$ and a right $\A$-invariant closed subspace $K$ of $L^p(\M)$, the prescription 
$K\to [\clo{h^{1/2}\mathcal{S}(K)_p}]_2$ (where $r>0$ is chosen so that $\frac{1}{p}=\frac{1}{2}+\frac{1}{r}$ and where 
$\mathcal{S}(K)_p$ is a right $\mathfrak{n}(\A)_\nu$-invariant subspace of $\mathfrak{m}_\nu$ for which 
$K= [\clo{h^{1/2}\mathcal{S}(K)_ph^{1/r}}]_p$) realises a lattice isomorphism from the right $\A$-invariant subspaces of 
$L^p(\M)$ to those of $L^2(\M)$.
\end{enumerate}
\end{theorem}

\begin{proof} \textbf{Case 1 $(2\leq p<\infty)$:} Fix $p$ where $2\leq p<\infty$, and write $\mathfrak{I}_r(L^p)$ for the 
closed right $\A$-invariant subspaces of $L^p$ (with $\mathfrak{I}_r(L^\infty)$ denoting $\sigma$-weakly closed right 
$\A$-invariant subspaces). Given $K\in\mathfrak{I}_r(L^p)$, we note that there is a unique largest right 
$\mathfrak{n}(\A)_\nu$-invariant subspace $\mathbf{S}(K)_p$ of $\mathfrak{m}_\nu$ for which 
$K=[\clo{\mathbf{S}(K)_ph^{1/p}}]_p$. 
To see this simply define $\mathbf{S}(K)_p$ to be the span of the union of all the $\mathfrak{n}(\A)_\nu$-invariant subspaces 
$\mathcal{S}$ of $\mathfrak{m}_\nu$ for which $K=[\clo{\mathcal{S}h^{1/p}}]_p$. It is an exercise to see that for any such 
$\mathcal{S}$ we have that $\clo{\mathcal{S}h^{1/p}}\subseteq \clo{\mathbf{S}(K)_ph^{1/p}}\subseteq K$, which then 
ensures that $K=[\clo{\mathbf{S}(K)_ph^{1/p}}]_p$. When combined with the fact that $\mathfrak{n}(\A)_\nu$ is $\sigma$-weakly 
dense in $\A$, the right $\mathfrak{n}(\A)_\nu$-invariance of $\mathbf{S}(K)_p$ ensures that 
$\mathbf{S}(K)_pa\subset \overline{\mathbf{S}(K)_p}^{w^*}$ for any $a\in \A$, and hence that 
$\overline{\mathbf{S}(K)_p}^{w^*}a\subset \overline{\mathbf{S}(K)_p}^{w^*}$ for any $a\in \A$. Thus the prescription 
$$\mathcal{L}_{p,\infty}:[\clo{\mathbf{S}(K)_ph^{1/p}}]_p \to \overline{\mathbf{S}(K)_p}^{w^*}$$yields a well-defined map from 
$\mathfrak{I}_r(L^p)$ to $\mathfrak{I}_r(L^\infty)$.

We proceed with describing the properties of this map.
(i) \emph{Injectivity:} Suppose we are given $K,L \in\mathfrak{I}_r(L^p)$ with 
$\overline{\mathbf{S}(K)_p}^{w^*}=\overline{\mathbf{S}(L)_p}^{w^*}$. Given any $s\in \mathbf{S}(K)_p$, we may 
therefore find a net $(t_\alpha)\subset \overline{\mathbf{S}(L)_p}^{w^*}$ which is $\sigma$-weakly convergent to $s$. Now let 
$(f_\lambda)\subset \D$ be the net described in Proposition \ref{7:P Terp2}. Since $(f_\lambda)\subset \mathfrak{n}(\A)_\nu$, 
we will for any fixed $\lambda$ have that $([t_\alpha f_\lambda h^{1/p}])\subset [\clo{\mathbf{S}(L)_ph^{1/p}}]_p$. But 
$[t_\alpha f_\lambda h^{1/p}]= t_\alpha[f_\lambda h^{1/p}]$ is then $L^p$ weakly convergent to 
$s[f_\lambda h^{1/p}]=[sf_\lambda h^{1/p}]$, ensuring that $([sf_\lambda h^{1/p}])\subset L$ for any $\lambda$. However the 
fact that for each $\lambda$ we have $[s f_\lambda h^{1/p}]=[s h^{1/p}]\sigma^\nu_{i/p}(f_\lambda)$ with 
$(\sigma^\nu_{i/p}(f_\lambda))$ $\sigma$-weakly convergent to $\I$ by Proposition \ref{7:P Terp2}, ensures that 
$([sf_\lambda h^{1/p}])$ is $L^p$-weakly convergent to $[sh^{1/p}]$. We therefore have that 
$\clo{\mathbf{S}(K)_ph^{1/p}}\subset L$ and hence that $K\subseteq L$. Reversing the roles of $K$ and $L$ in the above 
argument then shows that we also have that $L\subseteq K$, and hence that $K=L$.

(ii) \emph{Uniqueness of the realisation of $\mathcal{L}_{p,\infty}$:} Let $K \in\mathfrak{I}_r(L^p)$ and let $\mathcal{S}$ be a right 
$\mathfrak{n}(\A)_\nu$ invariant subspace of $\mathfrak{n}_\nu$ for which $K=[\clo{\mathcal{S}h^{1/p}}]_p$. We claim that then 
$\mathcal{L}_{p,\infty}(K)= \overline{\mathcal{S}}^{w^*}$. We proceed with proving this claim.

It is an exercise to on the one hand see that for the subspace $\mathcal{T}=\mathcal{S}+\mathbf{S}(K)_p$ we then still have 
that $K=[\clo{\mathcal{T}h^{1/p}}]_p$, and on the other that each of $\overline{\mathcal{S}}^{w^*}$ and 
$\overline{\mathcal{T}}^{w^*}$ are then right $\A$-invariant subspaces of $L^\infty$. It is clear that 
$\overline{\mathcal{S}}^{w^*}\subseteq \overline{\mathcal{T}}^{w^*}$. 

We show that in fact $\overline{\mathcal{S}}^{w^*}=\overline{\mathcal{T}}^{w^*}$. To achieve this we firstly note that 
$(\overline{\mathcal{S}}^{w^*})^\circ$ is a left $\A$-invariant subspace of $L^1$, and then use Theorem \ref{techinv-Lp} to 
select a left $\mathfrak{n}(\A^*)_\nu^*$-invariant subspace $\mathcal{S}_0$ of $\mathfrak{m}_\nu$ such that 
$(\overline{\mathcal{S}}^{w^*})^\circ = [\clo{h^{1/2}\mathcal{S}_0h^{1/2}}]$. Given any $t\in \mathcal{T}$, we use the fact 
that $[\clo{\mathcal{T}h^{1/p}}]_p=K=[\clo{\mathcal{S}h^{1/p}}]_p$ to select a sequence $(s_n)\subset \mathcal{S}$ for which 
$([s_mh^{1/p}])$ converges to $[th^{1/p}]$ in $L^p$-norm. Let $r>0$ be given such that $\frac{1}{2}=\frac{1}{p}+\frac{1}{r}$. 
For each $n$ we then have that $$0=tr(s_n\clo{h^{1/2}\mathcal{S}_0h^{1/2}})= 
tr([s_nh^{1/p}]\clo{h^{1/r}\mathcal{S}_0h^{1/2}}),$$ which in turn ensures that $$0=\lim_{n\to\infty}tr([s_nh^{1/p}]\clo{h^{1/r}\mathcal{S}_0h^{1/2}})= tr([th^{1/p}]\clo{h^{1/r}\mathcal{S}_0h^{1/2}})= tr(t\clo{h^{1/2}\mathcal{S}_0h^{1/2}}).$$That is $$\mathcal{T}\subset \clo{h^{1/2}\mathcal{S}_0h^{1/2}}^\circ= (\overline{\mathcal{S}}^{w^*})^{\circ\circ}=\overline{\mathcal{S}}^{w^*}.$$We then clearly have that 
$\overline{\mathcal{T}}^{w^*}\subseteq\overline{\mathcal{S}}^{w^*}$, which suffices to ensure that 
$\overline{\mathcal{S}}^{w^*}=\overline{\mathcal{T}}^{w^*}$. 

A similar argument shows that we also have that $\overline{\mathbf{S}(K)_p}^{w^*}=\overline{\mathcal{T}}^{w^*}$, and hence 
that $\overline{\mathbf{S}(K)_p}^{w^*}=\overline{\mathcal{S}}^{w^*}$ as claimed.

(iii) \emph{Surjectivity of $\mathcal{L}_{p,\infty}$:} Let $K\in\mathfrak{I}_r(L^\infty)$ be given. We know from Theorem 
\ref{techinv-Lp} that any $K$ admits a $\sigma$-weakly dense right $\mathfrak{n}(\A)_\nu$-invariant subspace $\mathcal{S}(K)$ of 
$\mathfrak{m}_\nu\cap K$. Consider $K_p=[\clo{\mathcal{S}(K)h^{1/p}}]_p\subset L^p$. The fact that 
 $$[\clo{\mathfrak{n}(\A)_\nu h^{1/p}}]_p=H^p=[\clo{h^{1/p}\mathfrak{n}(\A^*)^*_\nu}]_p$$ 
now ensures that 
  \begin{eqnarray*}
  \clo{\mathcal{S}(K)h^{1/p}}\mathfrak{n}(\A^*)^*_\nu 
  &=&\mathcal{S}(K)\clo{h^{1/p}\mathfrak{n}(\A^*)^*_\nu} \subset\mathcal{S}(K)[\clo{\mathfrak{n}(\A)_\nu h^{1/p}}]_p\\
  &\subset& [\clo{\mathcal{S}(K)\mathfrak{n}(\A)_\nu h^{1/p}}]_p \subset [\clo{\mathcal{S}(K)h^{1/p}}]_p = K_p.
  \end{eqnarray*} 
 This ensures that $K_p$ is right $\mathfrak{n}(\A^*)^*_\nu$-invariant. But 
$\mathfrak{n}(\A^*)^*_\nu$ is $\sigma$-weakly dense in $\A$. So given $a\in \A$, we may select a net 
$(a_\gamma)\subset \mathfrak{n}(\A^*)^*_\nu$ converging $\sigma$-weakly to $a$. For any $k\in K_p$, what we showed above 
ensures that $(ka_\gamma)\subset K_p$. Since by convexity $K_p$ is $L^p$-weakly closed with $(ka_\gamma)$ $L^p$-weakly 
convergent to $ka$, we have that $ka\in K_p$ for any $k\in K_p$ and any $a\in \A$. Thus $K_p\in \mathfrak{I}_r(L^p)$. The 
fact verified in the previous step now ensures that $\mathcal{L}_{p,\infty}(K_p)=K$.

(iv) \emph{$\mathcal{L}_{p,\infty}$ preserves closures of sums:} Let $K_\alpha\in \mathfrak{I}_r(L^p)$ ($\alpha\in A$) be a 
family of spaces in $\mathfrak{I}_r(L^p)$. Each $K_\alpha$ is of the form $K_\alpha=[\clo{\mathbf{S}(K_\alpha)_ph^{1/p}}]_p$, where 
$\mathbf{S}(K_\alpha)_p$ is as at the start of the proof. We then clearly have that 
$$K_\beta = [\clo{\mathbf{S}(K_\beta)_ph^{1/p}}]_p\subset [\clo{\mathrm{span}(\cup_\alpha\mathbf{S}(K_\alpha)_p)h^{1/p}}]_p\subset 
\overline{\mathrm{span}(\cup_\alpha K_\alpha)},$$which ensures that 
$$[\clo{\mathrm{span}(\cup_\alpha\mathbf{S}(K_\alpha)_p)h^{1/p}}]_p = \overline{\mathrm{span}(\cup_\alpha K_\alpha)}.$$(Here 
$\mathrm{span}(\cup_\alpha\mathbf{S}(K_\alpha)_p)$ and $\mathrm{span}(\cup_\alpha K_\alpha)$ are of course the sets of finite 
sums of terms respectively taken from the $\mathbf{S}(K_\alpha)_p$s and $K_\alpha$s. It is not difficult to see that then 
$\overline{\mathrm{span}(\cup_\alpha K_\alpha)}$ is again a right $\A$-invariant closed subspace of $L^p$, and that 
$\mathrm{span}(\cup_\alpha\mathbf{S}(K_\alpha)_p)$ is a right $\mathfrak{n}(\A)$-invariant subspace of $\mathfrak{m}_\nu$. By 
what we showed earlier $\mathcal{L}_{p,\infty}$ will map $\overline{\mathrm{span}(\cup_\alpha K_\alpha)}$ onto 
$\overline{\mathrm{span}(\cup_\alpha \mathbf{S}(K_\alpha)_p)}^{w^*}$. Since for any $\beta$ we have that 
$$\mathbf{S}(K_\beta)_p\subset \mathrm{span}(\cup_\alpha \mathbf{S}(K_\alpha)_p)\subset \mathrm{span}(\cup_\alpha \overline{\mathbf{S}(K_\alpha)_p}^{w^*})= 
\mathrm{span}(\cup_\alpha \mathcal{L}_{p,\infty}(K_\alpha)),$$it follows that 
$$\mathcal{L}_{p,\infty}(K_\beta) = \overline{\mathbf{S}(K_\beta)_p)}^{w^*} \subset \overline{\mathrm{span}(\cup_\alpha \mathbf{S}(K_\alpha)_p)}^{w^*} =\mathcal{L}_{p,\infty}(\overline{\mathrm{span}(\cup_\alpha K_\alpha)})$$for each $\beta$, 
and that $$\overline{\mathrm{span}(\cup_\alpha \mathbf{S}(K_\alpha)_p)}^{w^*} \subset \overline{\mathrm{span}(\cup_\alpha \mathcal{L}_{p,\infty}(K_\alpha))}^{w^*}.$$Together 
these inclusions suffice to ensure that 
in fact $$\mathcal{L}_{p,\infty}(\overline{\mathrm{span}(\cup_\alpha K_\alpha)}) = \overline{\mathrm{span}(\cup_\alpha \mathcal{L}_{p,\infty}(K_\alpha))}^{w^*}.$$

(v) \emph{$\mathcal{L}_{p,\infty}$ preserves intersections:} Let $q\in (1,2]$ be given such that $1=\frac{1}{p}+\frac{1}{q}$. 
The observation made at the start of this section regarding the relation between closed right $\A$-invariant subspaces of 
$L^p$ and left $\A$-invariant subspaces of $L^q$, ensures that the map $\mathcal{L}_{p,\infty}$ induces a bijection 
$\mathcal{L}_{p,\infty}^\circ$ from the closed left $\A$-invariant subspaces of $L^q$ (denoted by $\mathfrak{I}_\ell(L^q)$) 
to those of $L^1$ (namely $\mathfrak{I}_\ell(L^1)$) given by the formula $\mathcal{L}_{p,\infty}^\circ(K)= (\mathcal{L}_{p,\infty}(K^\circ))^\circ$. When combined with known properties of the polar operation, the fact that 
$\mathcal{L}_{p,\infty}$ preserves closures of sums, ensures that $\mathcal{L}_{p,\infty}^\circ$ preserves intersections. If 
we can show that $\mathcal{L}_{p,\infty}^\circ$ also preserves closures of sums, then by symmetry, $\mathcal{L}_{p,\infty}$ 
will preserve intersections.

Firstly select $r\geq 1$ such that $\frac{1}{q}=\frac{1}{r} + \frac{1}{2}$. For any $\widetilde{K}\in \mathfrak{I}_\ell(L^q)$ 
there exists a unique largest left $\mathfrak{n}(\A^*)_\nu^*$-invariant subspace $\widetilde{\mathcal{S}}(\widetilde{K})_q$ 
of $\mathfrak{m}_\nu$ such that $\widetilde{K} =[\clo{h^{1/r}\widetilde{\mathcal{S}}(K)_qh^{1/2}}]_q$. (This can be shown by 
the same sort of argument we used at the start of the proof.) 

We claim that the space $[\clo{h^{1/2}\widetilde{\mathcal{S}}(\widetilde{K})_qh^{1/2}}]_1$ is left 
$\mathfrak{n}(\A)_\nu$-invariant. If that is the case, then since $\overline{(\mathfrak{n}(\A^*)_\nu^*)}^{w^*}=\A$, it will 
even be $\A$-invariant. To see the left $\mathfrak{n}(\A)_\nu$-invariance, we may the use the fact that 
$[\clo{\mathfrak{n}(\A)_\nu h^{1/2}}]_2= [\clo{h^{1/2}\mathfrak{n}(\A^*)_\nu^*}]_2$, to see that  
\begin{eqnarray*}
\mathfrak{n}(\A)_\nu\clo{h^{1/2}\widetilde{\mathcal{S}}(K)_qh^{1/2}}&\subset& [\clo{\mathfrak{n}(\A)_\nu h^{1/2}}]_2 \clo{\widetilde{\mathcal{S}}(K)_qh^{1/2}}\\
&\subset& [\clo{h^{1/2}\mathfrak{n}(\A^*)_\nu^*}]_2 \clo{\widetilde{\mathcal{S}}(K)_qh^{1/2}}\\
&\subset& [\clo{h^{1/2}\mathfrak{n}(\A^*)_\nu^*\widetilde{\mathcal{S}}(K)_qh^{1/2}}]_1\\ 
&\subset& [\clo{h^{1/2}\widetilde{\mathcal{S}}(K)_qh^{1/2}}]_1.
\end{eqnarray*}
The prescription 
\begin{equation}\label{eq:tech-invss} 
[\clo{h^{1/r}\widetilde{\mathcal{S}}(\widetilde{K})_qh^{1/2}}]_q \to [\clo{h^{1/2}\widetilde{\mathcal{S}}(\widetilde{K})_qh^{1/2}}]_1
\end{equation}
therefore yields a well defined map from $\mathfrak{I}_\ell(L^q)$ into $\mathfrak{I}_\ell(L^1)$. The same sort 
of arguments as those used in parts (ii) and (iv) above, may now be used to show that 
\begin{itemize}
\item this map may firstly alternatively be realised by simply sending $[\clo{h^{1/r}\widetilde{\mathcal{S}}h^{1/2}}]_q$ to 
$[\clo{h^{1/2}\widetilde{\mathcal{S}}h^{1/2}}]_1$, where $\widetilde{\mathcal{S}}$ is an arbitrary left $\mathfrak{n}(\A^*)^*$-invariant subspace of $\mathfrak{m}_\nu$ for which $K=[\clo{h^{1/r}\widetilde{\mathcal{S}}(K)_qh^{1/2}}]_q$
\item and secondly that it preserves closures of sums.
\end{itemize}

If therefore we can show that the map described by equation (\ref{eq:tech-invss}) above is precisely 
$\mathcal{L}_{p,\infty}^\circ$, the proof will be complete. With this task in mind let $K\in \mathfrak{I}_r(L^p)$ be given 
and let $\widetilde{K}=K^\circ$ as above. Let $\mathcal{S}(K^\circ)_q$ and 
$\mathcal{S}(\mathcal{L}_{p,\infty}(K^\circ)^\circ)_1$ respectively be the unique largest left 
$\mathfrak{n}(\A^*)^*$-invariant subspaces for which $K^\circ=[\clo{h^{1/r}\mathcal{S}(K^\circ)_qh^{1/2}}]_q$ and 
$\mathcal{L}_{p,\infty}(K)^\circ=[\clo{h^{1/2}\mathcal{S}(\mathcal{L}_{p,\infty}(K)^\circ)_1h^{1/2}}]_1$.
With $\mathbf{S}(K)_p$ as at the start of the proof, the action of $\mathcal{L}_{p,\infty}$ is of course to send 
$K^\circ=[\clo{\mathbf{S}(K^\circ)_ph^{1/p}}]_p$ to $\overline{\mathbf{S}(K^\circ)_p}^{w^*}$. Let $\mathbf{S}(K)_p$ be as at 
the start of the proof. The fact that $[\clo{h^{1/r}\mathcal{S}(K^\circ)_qh^{1/2}}]_q = [\clo{\mathbf{S}(K)_ph^{1/p}}]_p^\circ$ then ensures that we will for any $s\in \mathbf{S}(K)_p$ have that 
$$0=tr([sh^{1/p}]\clo{h^{1/r}\mathcal{S}(K^\circ)_qh^{1/2}})= tr(s\clo{h^{1/2}\mathcal{S}(K^\circ)_qh^{1/2}}),$$and hence that 
$$\clo{h^{1/2}\mathcal{S}(K^\circ)_qh^{1/2}}\subset \mathbf{S}(K)_p^\circ =(\overline{\mathbf{S}(K)_p}^{w^*})^\circ =$$ $$\mathcal{L}_{p,\infty}(K)^\circ = [\clo{h^{1/2}\mathcal{S}(\mathcal{L}_{p,\infty}(K)^\circ)_1h^{1/2}}]_1.$$This fact then in 
turn ensures that $\mathcal{S}_0=\mathcal{S}(K^\circ)_q+\mathcal{S}(\mathcal{L}_{p,\infty}(K)^\circ)_1$ is another left 
$\mathfrak{n}(\A^*)^*_\nu$-invariant subspace of $\mathfrak{m}_\nu$ for which 
$$[\clo{h^{1/2}\mathcal{S}_0h^{1/2}}]_1=[\clo{h^{1/2}\mathcal{S}(\mathcal{L}_{p,\infty}(K)^\circ)_1h^{1/2}}]_1=\mathcal{L}_{p,\infty}(K)^\circ.$$The fact that $\mathcal{S}(\mathcal{L}_{p,\infty}(K)^\circ)_1$ must be the largest such space ensures that 
$\mathcal{S}_0=\mathcal{S}(\mathcal{L}_{p,\infty}(K)^\circ)_1$.

Similarly the fact that $$[\clo{h^{1/2}\mathcal{S}(\mathcal{L}_{p,\infty}(K)^\circ)_1h^{1/2}}]_1 = \mathcal{L}_{p,\infty}(K)^\circ=(\overline{\mathbf{S}(K)_p})^\circ = \mathbf{S}(K)_p^\circ,$$then ensures that we 
will for any $s\in \mathbf{S}(K)_p$ have that $$0=tr(s\clo{h^{1/2}\mathcal{S}(\mathcal{L}_{p,\infty}(K)^\circ)_1h^{1/2}})= tr([sh^{1/p}]\clo{h^{1/r}\mathcal{S}(\mathcal{L}_{p,\infty}(K)^\circ)_1h^{1/2}}),$$and hence that 
$$\clo{h^{1/r}\mathcal{S}(\mathcal{L}_{p,\infty}(K)^\circ)_1h^{1/2}}\subset [\clo{\mathbf{S}(K)_ph^{1/p}}]_p^\circ =K^\circ = [\clo{h^{1/r}\mathcal{S}(K^\circ)_1h^{1/2}}]_q.$$Thus here 
$\mathcal{S}_0=\mathcal{S}(K^\circ)_q+\mathcal{S}(\mathcal{L}_{p,\infty}(K)^\circ)_1$ is another left 
$\mathfrak{n}(\A^*)^*_\nu$-invariant subspace of $\mathfrak{m}_\nu$ for which 
$$[\clo{h^{1/r}\mathcal{S}_0h^{1/2}}]_q=[\clo{h^{1/r}\mathcal{S}(K^\circ)_qh^{1/2}}]_q=K^\circ.$$As before the fact that 
$\mathcal{S}(K^\circ)_q$ must be the largest such space ensures that $\mathcal{S}_0=\mathcal{S}(K^\circ)_q$. 

It follows that $\mathcal{S}(K^\circ)_q=\mathcal{S}(\mathcal{L}_{p,\infty}(K)^\circ)_1$. Now observe that by equation 
(\ref{eq:tech-invss}), $\mathcal{L}_{p,\infty}^\circ$ will map $K^\circ = [\clo{h^{1/r}\mathcal{S}(K^\circ)_qh^{1/2}}]_q$ onto 
$[\clo{h^{1/2}\mathcal{S}(K^\circ)_qh^{1/2}}]_1$. But by what we have just shown, this latter subspace is precisely 
$[\clo{h^{1/2}\mathcal{S}(\mathcal{L}_{p,\infty}(K)^\circ)_1h^{1/2}}]_1=\mathcal{L}_{p,\infty}(K)^\circ$. This is what was 
required and hence concludes this part of the proof.

\medskip

\emph{Observation:} The pair of maps $\mathcal{L}_{p,\infty}$ and $\mathcal{L}_{p,\infty}^\circ$ are both lattice 
isomorphisms and hence enough to establish the first assertion of the theorem. We therefore do not need to prove the validity 
of claim (2) if all we are interested in is the verification of the first assertion. Our interest in claim (2) is rather in its usefulness as a technical tool for our subsequent analysis.

\textbf{Case 2 $(1\leq p< 2)$:} Let $p\in [1,2)$ be given and select $q\in (2,\infty]$ so that $1=\frac{1}{p}+\frac{1}{q}$. 
The existence of the lattice isomorphism $\mathcal{L}_{p,2}$ from $\mathfrak{I}_\ell(L^p)$ to $\mathfrak{I}_\ell(L^2)$ described in the 
hypothesis of (2), can be verified by suitably modifying the preceding argument. There is however a shortcut to demonstrating 
its existence. To see this, we firstly need to note that for any $1\leq s_0, s_1 \leq \infty$ with $1=\frac{1}{s_0}+\frac{1}{s_1}$, $K$ is a closed right $\A$-invariant 
subspace of $L^{s_0}$ if and only if $K^\circ$ is a left $\A$-invariant subspace of $L^{s_1}$. To construct 
$\mathcal{L}_{p,2}$, we may therefore firstly apply the construction in the first part of the proof to the closed left 
$\A$-invariant subspaces of $L^q$ to obtain the lattice homomorphism $\widetilde{\mathcal{L}}_{q,\infty}$ from  
the closed left $\A$-invariant subspaces of $L^q$ to those of $L^\infty$ with the action as described in equation 
(\ref{eq:tech-invss}), where for such a subspace $\widetilde{K}$ the set $\widetilde{\mathcal{S}}(\widetilde{K})_q$ is here 
the unique largest left $\mathfrak{n}(\A^*)_\nu^*$-invariant subspace of $\mathfrak{m}_\nu$ for which 
$\widetilde{K}=[\clo{h^{1/r}\widetilde{\mathcal{S}}(\widetilde{K})_qh^{1/2}}]_q$. The prescription
$\mathcal{L}_{p,1}=\widetilde{\mathcal{L}}_{q,\infty}^\circ$ where $\widetilde{\mathcal{L}}_{q,\infty}^\circ(K)=\widetilde{\mathcal{L}}_{q,\infty}(K^\circ)^\circ$ then yields a lattice homomorphism from $\mathfrak{I}_r(L^p)$ to 
$\mathfrak{I}_r(L^1)$ with the action of sending any $K$ of the form $[\clo{h^{1/2}\mathcal{S}h^{1/r}}]_p$ (where 
$\mathcal{S}$ is a right $\mathfrak{n}(\A)_\nu$-invariant subspace of $\mathfrak{m}_\nu$), to 
$[\clo{h^{1/2}\mathcal{S}h^{1/r}}]_1$. The map $\mathcal{L}_{p,1}$ we seek is then nothing but 
$\mathcal{L}_{2,1}^{-1}\mathcal{L}_{p,1}$.
\end{proof}

As in the case of $L^2$, we say that a closed right $\A$-invariant subspace $K$ of $L^p$ is a type 2 invariant subspace if $K=[K\A_0]_p$. With this concept in place, we are now able to present the following analogue of Theorem \ref{main}. This result extends each of \cite[Theorem 4.6]{sager} and \cite[Theorems 3.6 \& 3.8]{BeRa}. The proof of the second part closely follows that of \cite[Theorem 4.5]{BL-Beurling}.

\begin{theorem} \label{lp-inv}  Let $\A$ be a maximal subdiagonal subalgebra of $\M$, and suppose that $K$ is a closed 
$\A$-invariant subspace of $L^p(\M)$, for $1 \leq p \leq \infty$.  (For $p = \infty$ we assume that $K$ is $\sigma$-weakly closed.) 
\begin{enumerate}
\item The space $K$ may then be written as an $L^p$-column sum of the form $Z \oplus^{col} (\oplus_i^{col} \, u_i H^p)$ where $Z$ is a type 2 closed right $\A$-invariant subspace of $L^p$, and where the $u_i$'s are partial isometries in $\M \cap K$ with $u^*_ju_i = 0$ if $i \neq j$ and $u_i^* u_i \in {\mathcal D}$. Moreover, for each $i$, $u_i^* Z = (0)$, left multiplication by the $u_i u_i^*$ are contractive projections from $K$ onto the summands $u_i H^p(\A)$, and left multiplication
by $1 - \sum_i \, u_i u_i^*$ is a contractive projection from $K$ onto $Z$.
\item Let $K$ be in the form $Z \oplus^{col} (\oplus_i^{col} \, u_i H^p)$ described above. Then there exists a contractive projection from $K$ onto $\oplus_i^{col} \, u_i L^p({\mathcal D})$ and along $[K\A_0]_p$. The quotient $K/[K A_0]_p$ is therefore isometrically ${\mathcal D}$-isomorphic to the subspace $\oplus_i^{col} \, u_i L^p({\mathcal D})$.
(Here $[\cdot]_\infty$ is as usual the $\sigma$-weak closure.)
\end{enumerate}
\end{theorem}

\begin{proof} \textbf{Claim 1:} We shall prove the first claim in two stages.

\textbf{Case 1 $(2< p\leq \infty)$:} By part (1) of Theorem \ref{latt-isom} the prescription 
$K\to \overline{\mathcal{S}(K)_p}^{w^*}$ where $\mathcal{S}(K)_p$ is a right $\mathfrak{n}(\A)_\nu$-invariant subspace of 
$\mathfrak{n}_\nu$ for which $K=[\clo{\mathcal{S}(K)_ph^{1/p}}]_p$, yields a lattice isomorphism $\mathcal{L}_{p,\infty}$ from the right $\A$-invariant subspaces of $L^p(\M)$ to those of $L^\infty(\M)$. 

\emph{Observation 1.1: $K$ is type 2 if and only if $\mathcal{L}_{p,\infty}(K)$ is.}$\quad$ Let $K$ be a closed right $\A$-invariant 
subspace of $L^p$ of the form $K=[\clo{\mathcal{S}h^{1/p}}]_p$ for some right $\mathfrak{n}(\A)$-invariant subspace of 
$\mathfrak{m}_\nu$. First suppose that $K$ is type 2. Since both $\clo{\mathcal{S}h^{1/p}}$ and $K\A_0$ are norm dense in 
$K$, $\clo{\mathcal{S}h^{1/p}}\A_0$ is norm-dense in $K$. The fact that $\mathfrak{n}(\A_0^*)^*_\nu$ is $\sigma$-weakly dense 
in $\A_0$, then ensures that $\clo{\mathcal{S}h^{1/p}}\mathfrak{n}(\A_0^*)^*_\nu$ is $L^p$-weakly dense (and hence norm 
dense) in $K$. We may now use the fact that $$[\clo{\mathcal{S}h^{1/p}}]_p = K = [\clo{h^{1/p}\mathfrak{n}(\A_0^*)^*_\nu}]_p=[\clo{\mathfrak{n}(\A_0)_\nu h^{1/p}}]_p$$to conclude that 
$[\clo{\mathcal{S}h^{1/p}}\mathfrak{n}(\A_0^*)^*_\nu]$. This clearly shows that in fact 
$[\clo{\mathcal{S}\mathfrak{n}(\A_0)_\nu h^{1/p}}]_p = [\clo{\mathcal{S}h^{1/p}}]_p$. It now follows from Theorem 
\ref{latt-isom} that $\overline{\mathcal{S}\mathfrak{n}(\A_0)_\nu}^{w^*}= \mathcal{L}_{p,\infty}(K)$. In view of the fact that 
$\mathcal{S}\mathfrak{n}(\A_0)_\nu\subset \overline{\mathcal{S}}^{w^*}\,\A_0\subset \overline{\mathcal{S}}^{w^*},$ this 
suffices to show that $\mathcal{L}_{p,\infty}(K)$ is type 2. 

Conversely suppose that $\mathcal{L}_{p,\infty}(K)$ is type 2. Select a right $\mathfrak{n}(\A)_\nu$-invariant subspace 
$\mathcal{S}$ of $\mathfrak{m}_\nu$ such that $K=[\clo{\mathcal{S}h^{1/p}}]_p$. Then  
$\overline{\mathcal{S}}^{w^*}=\mathcal{L}_{p,\infty}(K)$. We know that $\overline{\mathcal{S}}^{w^*}\,\A_0$ is $\sigma$-weakly 
dense in $\mathcal{L}_{p,\infty}(K)$, $\overline{\mathcal{S}}^{w^*}\,\mathfrak{n}(\A_0)_\nu$ $\sigma$-weakly dense in 
$\overline{\mathcal{S}}^{w^*}\,\A_0$, and $\mathcal{S}\mathfrak{n}(\A_0)_\nu$ $\sigma$-weakly dense in 
$\overline{\mathcal{S}}^{w^*}\,\mathfrak{n}(\A_0)_\nu$. Thus $\mathcal{S}\mathfrak{n}(\A_0)_\nu$ is $\sigma$-weakly dense in 
$\overline{\mathcal{S}}^{w^*}=\mathcal{L}_{p,\infty}(K)$. By Theorem \ref{latt-isom}, this ensures 
$\mathcal{L}_{p,\infty}(K)$ will map both of the spaces $K=[\clo{\mathcal{S}h^{1/p}}]_p$ and 
$[\clo{\mathcal{S}\mathfrak{n}(\A_0)_\nu h^{1/p}}]_p$ onto $\overline{\mathcal{S}}^{w^*}$. The injectivity of this map then 
ensures that $[\clo{\mathcal{S}h^{1/p}}]_p = [\clo{\mathcal{S}\mathfrak{n}(\A_0)_\nu h^{1/p}}]_p$. A similar argument to that 
used in the preceding paragraph now shows that $\clo{\mathcal{S}h^{1/p}}\mathfrak{n}(\A_0^*)^*_\nu$ is norm-dense in 
$[\clo{\mathcal{S}\mathfrak{n}(\A_0)_\nu h^{1/p}}]_p=K$. This clearly ensures that the larger space 
$\clo{\mathcal{S}h^{1/p}}\A_0$ (and therefore also $K\A_0$) is norm dense in $K$.

\emph{Observation 1.2: $\mathcal{L}_{p,\infty}$ and its inverse preserves column sums.}$\quad$ We demonstrate the validity of the 
claim in the simplest case of the column sum of two subspaces. The proof easily extends to the case of arbitrary columns 
sums. Let $K$, $K_1$ and $K_2$ be closed right invariant subspaces of $L^p$ for which $K=\overline{K_1+K_2}$. The fact that 
$\mathcal{L}_{p,\infty}$ preserves closures of sums ensures that 
$\mathcal{L}_{p,\infty}(K)=\overline{\mathcal{L}_{p,\infty}(K_1)+\mathcal{L}_{p,\infty}(K_2)}$. For each of $K_1$ and $K_2$ 
we may select a corresponding right $\mathfrak{n}(\A)_\nu$-invariant subspace $\mathcal{S}_i$ of $\mathfrak{m}_\nu$ such that 
$K_1=[\clo{\mathcal{S}_1h^{1/p}}]_p$ and $K_2=[\clo{\mathcal{S}_2h^{1/p}}]_p$. It is then clear that $(K_1)^*K_2=0$ if and 
only if $\clo{\mathcal{S}_1h^{1/p}}^*\clo{\mathcal{S}_2h^{1/p}}=\clo{h^{1/p}\mathcal{S}_1^*\mathcal{S}_2h^{1/p}}=0$ if and 
only if $\mathcal{S}_1^*\mathcal{S}_2 =0$ if and only if $\mathcal{L}_{p,\infty}(K_1)^*\mathcal{L}_{p,\infty}(K_2)= (\overline{\mathcal{S}_1}^{w^*})^*\overline{\mathcal{S}_2}^{w^*} =0$. 

\emph{Observation 1.3: $K=uH^p(\A)$ for some partial isometry $u\in \M$ with $u^*u\in \D$ if and only $\mathcal{L}_{p,\infty}(K)= 
u\A$.} $\quad$ To see this notice that since $uH^p(\A)=[\clo{u\mathfrak{n}(\A)_\nu h^{1/p}}]_p$, $\mathcal{L}_{p,\infty}$ 
will by Theorem \ref{latt-isom} map this subspace onto $\overline{u\mathfrak{n}(\A)_\nu}^{w^*}= u\A$. By the injectivity of 
$\mathcal{L}_{p,\infty}$, it is only the space $uH^p(\A)=[\clo{u\mathfrak{n}(\A)_\nu h^{1/p}}]_p$ that can map onto 
$\overline{u\mathfrak{n}(\A)_\nu}^{w^*}= u\A$.

\emph{Conclusion:} $\quad$ Any closed right $\A$-invariant subspace of $L^p$ is the image of a closed right $\A$-invariant 
subspace of $L^2$ under the map $\mathcal{L}_{p,\infty}^{-1}\mathcal{L}_{2,\infty}$. In view of the above three observations, 
the validity of the theorem for $L^p$ where $2<p\leq \infty$ willtherefore follow from applying this map to the invariant subspaces of 
$L^2$.

\textbf{Case 2 $(1\leq p< 2)$:} In this case we directly use the map $\mathcal{L}_{p,2}$ described in part (2) of Theorem 
\ref{latt-isom}, to transfer the known structure of closed right invariant subspaces of $L^2$ to those of $L^p$. In the rest of the proof we shall consistently assume that $r\geq 1$ has been chosen so that $\frac{1}{p}=\frac{1}{2}+\frac{1}{r}$. The action of this map is then to send subspaces of the form $K=[\clo{h^{1/2}\mathcal{S}(K)_ph^{1/r}}]_p$ (where $\mathcal{S}(K)_p$ is a right $\mathfrak{n}(\A)_\nu$-invariant subspace of 
$\mathfrak{n}_\nu$) to $[\clo{h^{1/2}\mathcal{S}(K)_p}]$. 

\emph{Observation 2.1: If $\mathcal{L}_{p,2}(K)$ is type 2, then so is $K$.}$\quad$ Let $K$ be a closed right $\A$-invariant subspace of $L^p$ of the form $K=[\clo{h^{1/2}\mathcal{S}h^{1/r}}]_p$ for some right $\mathfrak{n}(\A)$-invariant subspace of $\mathfrak{m}_\nu$.  Suppose that $\mathcal{L}_{p,2}(K) = [\clo{h^{1/2}\mathcal{S}}]_2$. We know that $\clo{h^{1/2}\mathcal{S}}\A_0$ is norm dense in $\mathcal{L}_{p,2}(K) = [\clo{h^{1/2}\mathcal{S}}]_2$, and $\clo{h^{1/2}\mathcal{S}}\mathfrak{n}(\A_0)_\nu$ $L^2$-weakly dense (and therefore norm dense) in 
$\clo{h^{1/2}\mathcal{S}}\A_0$, So $\clo{h^{1/2}\mathcal{S}}\mathfrak{n}(\A_0)_\nu$ is norm dense in $\mathcal{L}_{p,2}(K)$. 
Since by Theorem \ref{latt-isom}, $\mathcal{L}_{p,2}$ will map $[\clo{h^{1/2}\mathcal{S}\mathfrak{n}(\A_0)_\nu h^{1/r}}]_p$
onto $[\clo{h^{1/2}\mathcal{S}\mathfrak{n}(\A_0)_\nu}]_2=\mathcal{L}_{p,2}(K)$, the injectivity of this map ensures that $K=[\clo{h^{1/2}\mathcal{S}h^{1/r}}]_p$ and $[\clo{h^{1/2}\mathcal{S}\mathfrak{n}(\A_0)_\nu h^{1/r}}]_p$ agree. A similar argument to those used before now shows that $\clo{h^{1/2}\mathcal{S}h^{1/r}}\mathfrak{n}(\A_0^*)^*_\nu$ is norm-dense in 
$[\clo{\mathcal{S}\mathfrak{n}(\A_0)_\nu h^{1/p}}]_p=K$. This clearly ensures that the larger space 
$\clo{h^{1/2}\mathcal{S}h^{1/r}}\A_0$ (and therefore also $K\A_0$) is norm dense in $K$.

\emph{Observation 2.2: $\mathcal{L}_{p,2}^{-1}$ preserves column sums.}$\quad$ As before we demonstrate the validity 
of the claim in the simplest case of the column sum of two subspaces. Let $K$, $K_1$ and $K_2$ be closed right invariant 
subspaces of $L^p$ for which $K=\overline{K_1+K_2}$. We know that $\mathcal{L}_{p,2}^{-1}$ is a lattice isomorphism and hence
preserves closures of sums. We therefore only need to check the preservation of the ``orthogonality'' property inherent in 
column sums. For each of $K_1$ and $K_2$ we may select a corresponding right $\mathfrak{n}(\A)_\nu$-invariant subspace 
$\mathcal{S}_i$ of $\mathfrak{m}_\nu$ such that $K_1=[\clo{h^{1/2}\mathcal{S}_1h^{1/r}}]_p$ and 
$K_2=[\clo{h^{1/2}\mathcal{S}_2h^{1/r}}]_p$. By Theorem \ref{latt-isom} we then have that $\mathcal{L}_{p,2}(K_i)=[\clo{h^{1/2}\mathcal{S}_i}]_2$ for $i=1,2$. It is then clear that \begin{eqnarray*}
&&\mathcal{L}_{p,2}(K_1)^*\mathcal{L}_{p,2}(K_2)=0\\
&\Rightarrow& \clo{h^{1/2}\mathcal{S}_1}^*\clo{h^{1/2}\mathcal{S}_2} = \clo{\mathcal{S}_1^*h^{1/2}}\clo{h^{1/2}\mathcal{S}_2}=0\\
&\Rightarrow& \clo{h^{1/2}\mathcal{S}_1h^{1/r}}^*\clo{h^{1/2}\mathcal{S}_2h^{1/r}}= \clo{h^{1/r}\mathcal{S}_1^*h^{1/2}}\clo{h^{1/2}\mathcal{S}_2h^{1/r}}=0\\
&\Rightarrow& K_1^*K_2=0.
\end{eqnarray*}
\emph{Observation 2.3: If $K$ is a closed right $\A$-invariant subspace for which $\mathcal{L}_{p,2}(K)= 
uH^2(\A)$ for some partial isometry $u\in \M$ with $u^*u\in \D$, then $K=uH^p(\A)$.} $\quad$ Let $K$ be given such that 
$\mathcal{L}_{p,2}(K)= uH^2(\A)$. Select a right $\mathfrak{n}(\A)_\nu$ invariant subspace $\mathcal{S}$ of 
$\mathfrak{m}_\nu$ such that $K=[\clo{h^{1/2}\mathcal{S}h^{1/r}}]_p$. By Theorem \ref{latt-isom} we will then have that 
$uH^2(\A) = \mathcal{L}_{p,2}(K) =[\clo{h^{1/2}\mathcal{S}}]_2$. Following what should by now be fairly familiar paths, one may then 
use the fact that $[\clo{h^{1/r}\mathfrak{n}(\A^*)^*_\nu}]_r= [\clo{\mathfrak{n}(\A)_\nu h^{1/r}}]_r$ to see that 
$K\mathfrak{n}(\A^*)^*_\nu =[\clo{h^{1/2}\mathcal{S}h^{1/r}}]_p\mathfrak{n}(\A^*)^*_\nu$ is norm-dense in 
$[\clo{h^{1/2}\mathcal{S}\mathfrak{n}(\A)_\nu h^{1/r}}]_p$. We leave it as an exercise to verify that 
\begin{eqnarray*}
&&[\clo{h^{1/2}\mathcal{S}\mathfrak{n}(\A)_\nu h^{1/r}}]_p=[[\clo{h^{1/2}\mathcal{S}}]_2\clo{\mathfrak{n}(\A)_\nu h^{1/r}}]_p\\
&=& [uH^2(\A)\clo{\mathfrak{n}(\A)_\nu h^{1/r}}]_p = u[H^2(\A)\clo{\mathfrak{n}(\A)_\nu h^{1/r}}]_p\\
&=& u[\clo{h^{1/2}\mathfrak{n}(\A^*)^*_\nu}\clo{\mathfrak{n}(\A)_\nu h^{1/r}}]_p =u[\clo{h^{1/2}\mathfrak{n}(\A^*)^*_\nu\mathfrak{n}(\A)_\nu h^{1/r}}]_p\\
&=& uH^p(\A).
\end{eqnarray*}
But since $K\mathfrak{n}(\A^*)^*_\nu$ is weakly $L^p$-dense (and therefore norm dense) in $K\A=K$, we must have that $K=uH^p(\A)$ as claimed.

Having verified the above three observations, the conclusion now follows from an application of $\mathcal{L}_{p,2}$ to Theorem \ref{main}. 

To see the final claim, notice that since left multiplication by $u_i u_i^*$ annihilates $Z$ and $u_jH^p(\A)$ if $j \neq i$, 
left multiplication by the $u_i u_i^*$'s are contractive projections from $K$ onto the summands $u_iH^p(\A)$. Easy checking also shows that  $1-\sum_i \, u_i u_i^*$ is a projection which will by left multiplication map $K$ onto $Z$.

\textbf{Claim 2:} Let $\E$ be the canonical faithful normal conditional expectation onto $\D$ for which $\nu\circ \E=\nu$. Our task here, is to show that the map $\theta:L^p\to L^p$ defined by $\theta(f)=\sum_iu_i\E_p(u_i^*f)$ is a contractive idempotent map. To see this we first let $F$ be a finite subfamily of the index set $\{i\colon i\in I\}$. It then easily follows from the properties of the $u_i$'s, that for $\theta_F$ defined by $\theta_F(f)=\sum_{i\in F}u_i\E_p(u_i^*f)$, we have that $$\theta(\theta(f))=\theta(\sum_{i\in F}u_i\E_p(u_i^*f)) =$$ $$\sum_{i,j\in F}u_j\E_p(u_j^*u_i\E_p(u_i^*f))=\sum_{i\in F}u_i\E_p(u_i^*f).$$We pass to verifying the contractivity. 

\emph{Case 1:} First let $p\geq 2$. If say $p<\infty$, the complete positivity of $\E$ combined with the very particular action of the 
extension of $\E$ to $L^p$ and $L^{p/2}$, ensures that for any $x\in \mathfrak{n}_\nu$ we have that 
$$|\E_p([xh^{1/p}])|^2= \E_p([xh^{1/p}])^*\E_p([xh^{1/p}])= \clo{h^{1/p}\E(x^*)\E(x)h^{1/p}}\leq$$ $$\clo{h^{1/p}\E(x^*x)h^{1/p}}=\E_{p/2}(\clo{h^{1/p}x^*xh^{1/p}}).$$The norm density of $\clo{\mathfrak{n}_\nu h^{1/p}}$ in $L^p$ 
combined with the $L^p$ and $L^{p/2}$-continuity of the relevant extensions of $\E$, ensures that we in fact have that 
$|\E_p(f)|^2\leq \E_{p/2}(|f|^2)$ for any $f\in L^p(\M)$. As far as $\theta$ is concerned, the properties of the $u_i$'s ensure that 
we will for any $f\in L^p$ have that 
$$\theta_F(f)^*\theta_F(f) = (\sum_{i\in F}\E_p(u_i^*f)^*u_i^*)(\sum_{j\in F}u_j\E_p(u_j^*f))=$$ $$\sum_{i\in F}\E_p(u_i^*f)^*u_i^*u_i\E_p(u_i^*f) = \sum_{i\in F}\E_p(u_i^*f)^*\E_p(u_i^*f) \leq \sum_{i\in F}\E_{p/2}(f^*u_iu_i^*f).$$(In the second to last (in)equality we used the fact that $u^*_iu_i$ is a projection in $\D$. Easy 
checking reveals that $\sum_{i\in F}u_iu_i^*$ is itself in fact a projection, and hence the above inequality may be further 
refined to yield the fact that $$\theta_F(f)^*\theta_F(f)\leq \sum_{i\in F}\E_{p/2}(f^*u_iu_i^*f)\leq \E_{p/2}(f^*f).$$The contractivity of 
the extension of $\E$ to $L^{p/2}$, then ensures that $$\|\theta_F(f)\|_p^2=\|\,|\theta_F(f)|^2\|_{p/2}\leq \|\E_{p/2}(f^*f)\|_{p/2}\leq \|f^*f\|_{p/2}=\|f\|_p^2$$for all $f\in L^p(\M)$. We leave the verification of a similar claim for the case $p=\infty$ to the reader. It follows fairly directly from the fact that $|\E(f)|^2\leq \E(|f|^2)$ for all 
$f\in L^\infty$. 

\emph{Case 2:} Now let $1\leq p<2$, and let $q>2$ be the conjugate index of $p$. For any $f\in L^p$ and $g\in L^q$, we may then use the fact that the extension of $\E$ to $L^p$ is the dual of its extension to $L^q$, to see that 
$$tr(g^*\theta_F(f))=\sum_{j\in F}tr(g^*u_i\E_p(u_i^*f)) = \sum_{j\in F}tr((u_i\E_q(u_i^*g))^*f)= tr(\theta_F(g)^*f).$$Since $\theta_F$ acts contractively on $L^q$, it therefore follows that $$|tr(g^*\theta_F(f))|\leq \|\theta_F(g)\|_q\|f\|_p\leq \|g\|_q\|f\|_p,$$which by $L^p$ duality ensures that $\|\theta_F(f)\|_p\leq \|f\|_p$.

The above facts now clearly imply that the map $\theta:f \to \lim_F\sum_{i\in F}u_i\E(u_i^*f)=\sum_iu_i\E(u_i^*f)$ is a 
contractive idempotent map on $L^p$. Our last order of business is to investigate the action of $\theta$ on $K$. Since 
$u_i^*Z=\{0\}$ for any $i$, it is clear that $\theta$ annihilates $Z$. For a typical element $f=\sum_iu_i x_i$ (where $x_i\in H^p$ 
for each $i$) of $\oplus^{col}_iu_iH^p(\A)$, the fact that $u_j^*u_i=0$ if $j\neq i$, ensures that then 
$$\theta(f)= \sum_iu_i\E_p(u_i^*f) = \sum_iu_i\E_p(u_i^*u_ix_i).$$But since $u_i^*u_i\in \D$ for each $i$, it then further follows 
that $$\theta(f)=  \sum_iu_i\E_p(u_i^*u_ix_i) = \sum_iu_iu_i^*u_i\E_p(x_i) = \sum_iu_i\E_p(x_i).$$It now clearly follows that 
$\theta$ maps $K=Z\oplus^{col}(\oplus^{col}_iu_iH^p(\A))$ onto $\oplus^{col}_iu_i\E_p(H^p(\A))= \oplus^{col}_iu_iL^p(\D)$. As 
far as the kernel of $\theta$s action on $K$ is concerned, it therefore clearly corresponds to 
$Z\oplus^{col}(\oplus^{col}_iu_iH^p_0(\A))$. Since $K\A_0\subseteq [Z\A_0]_p\oplus^{col}(\oplus^{col}_iu_i[H^p(\A)\A_0]_p) = Z\oplus^{col}(\oplus^{col}_iu_iH^p_0(\A))$, the kernel must include $[K\A_0]_p$. However conversely the fact that 
$Z\subset K$ and $u_iH^p(\A)$ for each $i$, similarly ensures that $Z=[Z\A_0]_p\subset [K\A_0]_p$ and that $H^p_0(\A)= [H^p(\A)\A_0]_p\subset [K\A_0]_p$ for each $i$. Therefore $[K\A_0]_p= Z\oplus^{col}(\oplus^{col}_iu_iH^p_0(\A))= \mathrm{ker}(\theta)$, as is required. Thus the restriction of $\theta$ to $K$ is a contractive projection onto $\oplus^{col}_iu_iL^p(\D)$ and along $[K\A_0]$, and hence induces an isometric ${\mathcal D}$-module map from
$K/{\rm ker}(\theta)$ onto $\oplus^{col}_i \, u_i L^p(\D)$.
\end{proof}

Armed with the wisdom of part (2) of the above theorem the following definition of the right wandering subspace for the case $p\neq 2$ now makes perfect sense. It is clear from part (2) of Theorem \ref{inv1} that in the case $p=2$ this definition gives us exactly the right wandering subspace defined earlier.

\begin{definition} Let $K$ be a right $\A$-invariant subspace of $L^p$ where $1\leq p \leq\infty$. With $Z$ as in the preceding theorem, we may now on the basis of the above theorem define the right wandering subspace $W(K)$ of $K$ to be the space $\oplus_i^{col} \, u_i L^p({\mathcal D})$. We say that $K$ is a type 1 invariant subspace if $K=[W(K)\mathcal{A}]$ and a type 2 invariant subspace if $K=[K\mathcal{A}_0]$.
\end{definition}

Against the backdrop of the above definition, the first assertion of the preceding theorem may be interpreted as the statement that any closed right $\A$-invariant subspace of $L^p$ may be written as a column sum of a type 1 and type 2 invariant subspace. The following analogue of Beurling's characterization of $\sigma$-weakly closed ideals of $H^\infty(\mathbb{D})$ now also readily follows from Theorem \ref{lp-inv}. This extends \cite[Corollary 4.8]{BL-Beurling} where this fact was noted for the case of finite von Neumann algebras.

\begin{corollary} \label{rid}  If  $\A$ is maximal subdiagonal, then the type 1 $\sigma$-weakly closed right ideals of $\A$
are precisely those right ideals of the form $\oplus_i^{col} \, u_i \A$, for partial isometries $u_i \in \A$ with mutually 
orthogonal ranges and $|u_i| \in {\mathcal D}$.
\end{corollary}

\section{Characterizations of maximal subdiagonal subalgebras}\label{S12}

With the technology of both section \ref{Hpsection} and a stronger analytic reduction theorem at our disposal (Proposition \ref{anred}), we are now able to extend the equivalences in 
\cite[Theorem 3.4]{L-HpIII}, to the general case. The primary focus is to describe conditions which characterise subdiagonality of analytically conditioned subalgebras of a von Neumann 
algebra. In the case of finite von Neumann algebras equipped with a faithful normal tracial state, all such conditions turn out to be equivalent to the validity of Szeg\"o's formula 
\cite{BLsurvey}. Hence even in settings where there are no Fuglede-Kadison determinants and hence no non-commutative Szeg\"o formula, one may nevertheless regard such conditions as echoes 
of Szeg\"o's formula. 

The same line of attack as was used to prove \cite[Theorem 3.4]{L-HpIII} will go through in the present context. However several of the steps require 
a significantly more refined and delicate argument, and so 
for the sake of the reader we give full details in the main theorem. We firstly note that \cite[Lemma 3.1]{L-HpIII} carries over almost verbatim to the present context. We formulate this result 
for the convenience of the reader.

\begin{lemma}
 Let $\A$ be an analytically conditioned algebra and let $\widehat{\A}$ be the $\sigma$-weak closure of $\{\pi_\nu(a)\lambda_t:a\in \A,\, t\in \mathbb{Q}_D\}$ in $\R=\M\rtimes\mathbb{Q}_D$. Then with notation as in Proposition \ref{anred}, $\widehat{\A}$ is an analytically conditioned subalgebra of $\R$, and each $\widehat{\A}_n=\widehat{\A}\cap \R_n$ an analytically conditioned subalgebra of $\R_n$.
\end{lemma}

The result \cite[Lemma 3.1]{L-HpIII} similarly carries over to the general setting, although we have a bit more work to do here.

\begin{lemma}\label{L2}
 Let $\A$ be an analytically conditioned algebra.  If $L^2(\M)=\mathcal{H}^2(\A)\oplus(\mathcal{H}_0^2(\A))^*$, then also 
$L^2(\R)=\mathcal{H}^2(\widehat{\A})\oplus(\mathcal{H}_0^2(\widehat{\A}))^*$, and $L^2(\R_n)=\mathcal{H}^2(\widehat{\A}_n)\oplus(\mathcal{H}_0^2(\widehat{\A}_n))^*$ for each $n$.
\end{lemma}

\begin{proof} Firstly observe that since $\M$ is an expected subalgebra of $\R$, we may assume that $h_{\M}=h_{\R}$. By assumption $L^2(\M)=\mathcal{H}^2(\A)\oplus(\mathcal{H}_0^2(\A))^*$. By Theorem \ref{tech-thm}, this means that $\mathfrak{j}^{(2)}(\mathfrak{n}(\A)+\mathfrak{n}(\A_0^*))$ is dense in $L^2(\M)$.  The claim will therefore follow if we are able to conclude from this that $\mathrm{span}\{\lambda_t[(a+b^*)h^{1/2}]\colon \in \mathbb{Q}_D,\, a\in \mathfrak{n}(\A),\, b\in \mathfrak{n}(\A_0^*)^*\}$ is dense in $L^2(\R)$. By assumption the $L^2$-closure of $\mathrm{span}\{\lambda_t[(a+b^*)h^{1/2}]\colon \in \mathbb{Q}_D,\, a\in \mathfrak{n}(\A),\, b\in \mathfrak{n}(\A_0^*)^*\}$ certainly includes $\mathrm{span}\{\lambda_tf\colon \in \mathbb{Q}_D,\, f\in L^2(\M)\}$, which in turn includes $\mathrm{span}\{\lambda_t[ch^{1/2}]\colon t\in \mathbb{Q}_D,\, c\in\mathfrak{n}_\nu\}$. We show that this subspace is dense in $L^2(\R)$. We achieve this through a modification of the second part of the proof of Theorem \ref{10:T reduction Lp} 

Suppose that $z\in L^2(\R)$. The claim will follow if we can show that we must have that $z=0$ whenever $tr(z\lambda_t[ch^{1/2}])=0$ for each $t\in\mathbb{Q}_D$ and each $c\in \mathfrak{n}_\nu$. So let this be the case. Given $x\in \M$, the fact that each $\mathfrak{n}_{\nu}$ is a left-ideal ensures that for each fixed $t\in\mathbb{Q}_D$ and $c\in \mathfrak{n}_\nu$, we have that 
$$x\lambda_t[ch^{1/2}]=\lambda_t(\lambda_t^*x\lambda_t)[ch^{1/2}]=\lambda_t\sigma_t^\nu(x)[ch^{1/2}]= \lambda_t[(\sigma_t^\nu(x)c)h^{1/2}]$$where $\sigma_t^\nu(x)c\in \mathfrak{n}_\nu$. So for any $t\in\mathbb{Q}_D$, $x\in\M$ and $f\in\mathrm{span}\{\lambda_t[ch^{1/2}]\colon t\in \mathbb{Q}_D,\, c\in\mathfrak{n}_\nu\}$, we have that $x\lambda_t f\in \mathrm{span}\{\lambda_t[ch^{1/2}]\colon \in \mathbb{Q}_D,\, c\in\mathfrak{n}_\nu\}$. This then ensures that $0=tr(z(x\lambda_t)f)$ for each $t\in\mathbb{Q}_D$, $x\in \M$ and $f\in \mathrm{span}\{\lambda_t[ch^{1/2}]\colon t\in \mathbb{Q}_D,\, c\in\mathfrak{n}_\nu\}$. Given that $\{x\lambda_t:x\in\M, \, t\in\mathbb{Q}_D\}$ is $\sigma$-weakly dense in $\R$, this can only be the case if in fact $fz=0$ for all $f\in \mathrm{span}\{\lambda_t[ch^{1/2}]\colon t\in \mathbb{Q}_D,\, c\in\mathfrak{n}_\nu\}$.

If we restrict to $\mathrm{span}\{[ch^{1/2}]\colon c\in \n_\nu\}$ and apply Proposition \ref{P density1} to this fact, it 
follows that $L^2(\M)z=$ (equivalently $z^*L^{2}(\M)=0$), and hence that $z^*[ah^{1/q}]=0$ for each $a\in \mathfrak{n}_{\nu}$. Fixing 
$a\in \mathfrak{n}_{\nu}$, it is easy to see that $z^*[ah^{1/2}]=0$ if and only if $|z^*|[ah^{1/2}]=0$. Since trivially $z^*=0$ if 
and only if $|z^*|=0$, it follows that we may assume that $z^*\geq 0$. Having made this assumption one may then further note that 
$z^*=0$ if and only if $(z^*\chi_{[0,\gamma]}(z^*))=0$ for every $\gamma>0$. Since in this setting the equality $z^*[ah^{1/2}]=0$ 
ensures that $0= \chi_{[0,\gamma]}(z^*)z^*[ah^{1/2}]$ and hence that $0=(z^*\chi_{[0,\gamma]}(z^*))[ah^{1/2}]$, it follows that we may 
further assume $z^*$ to be bounded. Hence given $\xi \in \mathrm{dom}(h^{1/2})\subset \mathrm{dom}([ah^{1/2}])$, we then have that 
$z^*a=0$ on the range of $h^{1/2}$, which must be dense by the fact that $h$ is non-singular and positive. Therefore $z^*a=0$. But 
$a\in \mathfrak{n}_{\nu}$ was arbitrary. So for $(f_\lambda)$ as in Proposition \ref{7:P Terp2}, we have that $z^*=\lim_\lambda z^*f_\lambda=0$ as required
\end{proof}

A version of \cite[Lemma 3.3]{L-HpIII} valid for general algebras now follows by exactly the same proof as the one used in \cite{L-HpIII}.

\begin{lemma}\label{unsep}
 Let $\A$ be an analytically conditioned algebra.  If any $f\in L^1(\M)^+$ which is in the annihilator of $\A_0$ belongs to $L^1(\mathcal{D})$, then also 
\begin{itemize}
\item any $f\in L^1(\R)^+$ which is in the annihilator of $\widehat{\A}_0$ belongs to $L^1(\widehat{\mathcal{D}})$,
\item and for any $n$, any $f\in L^1(\R_n)^+$ which is in the annihilator of $(\widehat{\A}_n)_0$, belongs to $L^1(\mathcal{D}_n)$.
\end{itemize}
\end{lemma}

All the lemmata used to prove \cite[Theorem 3.4]{L-HpIII}, therefore carry over to the general setting. Armed with these lemmata, we now have the technology to prove the result below. Some terminology is required to comprehend the formulation of the theorem. We say that an extension in the Banach dual $\M^\star$ of $\M$ of a functional  in $\A^\star$ (the Banach dual of $\A$) is a {\em Hahn-Banach extension} if it has the same norm as the original functional.
If $\A$ is a $\sigma$-weakly closed subalgebra of $\M$ then we say that $\A$ has property (GW1) if every
Hahn-Banach extension to $\M$ of any normal functional on $\A$, is normal on $\M$.
We say that $\A$ has property (GW2) if there is at most one normal Hahn-Banach extension to $\M$ of
any normal functional on $\A$. We say that $\A$ has the {\em Gleason-Whitney property} (GW)
if it possesses both (GW1) and (GW2).

\begin{theorem}  \label{Co}
 Let $\A$ be an analytically conditioned algebra.  Then the following are equivalent:
\begin{itemize} \item [(1)]  $\A$ is maximal subdiagonal,
\item [(2)]  For every right $\A$-invariant subspace $X$ of $L^2(M)$, the right wandering subspace $W$ of $X$ satisfies
$W^* W \subset L^1({\mathcal D})$, and $W^* (X \ominus [W A]_2) = \{0\}$
\item[(3)] $L^2(\M)=\mathcal{H}^2(\A)\oplus\mathcal{H}_0^2(\A)^*$, and any $f\in L^1(\M)^+$ which is in the annihilator of $\A_0$ belongs to $L^1(\mathcal{D})$.
\item[(4)] $\A$ satisfies (GW2).
\end{itemize}
\end{theorem}

\begin{proof} The fact that (1) implies (2i) is proven in Theorem \ref{inv1}.  We proceed to prove that (2) implies (3). 
To this end, let $g \in L^1(\M)^+$ be given with $tr(g \A_0) = 0$. Let $f = |g|^{\frac{1}{2}}$. Clearly $f\in L^2(\M)$, 
and $f^2=g$. Now set $X = [f \A]_2$. Note that $f \perp [f \A_0]_2$ since if $a_n \in \A_0$ with $f a_n \to k$ in $L^2$-norm, 
then $\langle k,f\rangle =tr(f k) =\lim_n tr(f^2 a_n) = \lim_n tr(g a_n)=0$. In particular, the fact that $f \perp
[f \A_0]_2=[X\A_0]_2$, ensures that $f \in X \ominus [X\A_0]_2 = W$. So by hypothesis, $f^2 = g \in L^1({\mathcal D})$.

Next set $X = L^2(\M)\ominus \mathcal{H}^2_0(\A)^*$. We will deduce that $\A$ satisfies $L^2$-density. That is that 
$X = \mathcal{H}^2(\A)$. Our first task is to show that $X$ is right $\A$-invariant. To see this recall that 
since $\A$ is analytically conditioned, $\{h^{1/2}a_0^*: a_0\in \mathfrak{n}(\A_0)\}$ is norm dense in 
$\mathcal{H}^2_0(\A)^*$. So $f \in L^2(\M)$ is orthogonal to $(\mathcal{H}^2_0(\A))^*$ if and only if $tr([a_0h^{1/2}]f)=tr((h^{1/2}a_0^*)^*f)=
\langle f,(h^{1/2}a^*)\rangle =0$ for every $a_0\in \mathfrak{n}(\A_0)$. 
Given such an $f\in X$ and $a\in \A$ and $a_0\in \mathfrak{n}(\A_0)$, the fact that then $aa_0\in \mathfrak{n}(\A_0)$, ensures that we will then also 
have that $tr([a_0h^{1/2}](fa))=tr([aa_0h^{1/2}]f)=0$ for every $a_0\in \mathfrak{n}(\A_0)$. Hence 
$fa\in L^2(\M)\ominus \mathcal{H}^2_0(\A)^*=X$ as required. 

Now let $(f_\lambda)$ be the net in $\mathfrak{n}(\D)_\nu^*\cap \mathfrak{n}(\D)_\nu$ converging strongly to $\I$ as 
guaranteed by Proposition \ref{7:P Terp2}. We first show that $(h^{1/2}f_\lambda) \in X$ for each $\lambda$, where 
$h=\frac{d\widetilde{\nu}}{d\tau}$. This is a consequence of the fact that $\{[a_0h^{1/2}]: a_0\in \mathfrak{n}(\A_0)\}$ 
is dense in $\mathcal{H}^2_0(\A)$, and that $$tr((h^{1/2}a^*)(h^{1/2}f))=tr((h^{1/2}f_\lambda)[ah^{1/2}])=\nu(f_\lambda a)=\nu(\mathbb{E}(f_\lambda a)) =0$$for all $a\in \mathfrak{n}(\A_0)$. In fact we even have $h^{1/2}f_\lambda \in W=X\ominus[X\A_0]_2$ since for any 
$a_0\in \mathfrak{n}(\A_0)$ and $x\in X$ the fact that $(h^{1/2}F_\lambda a_0^*)\in \mathcal{H}^2_0(\A)^*$ ensures that 
 $$0=\langle (h^{1/2}f_\lambda a_0^*), x\rangle = tr(x^*(h^{1/2}f_\lambda a_0^*))=tr((xa_0)^*(h^{1/2}f_\lambda))= \langle (h^{1/2}f_\lambda), xa_0\rangle.$$ 
Here we again used the density of $\{[a_0h^{1/2}]: a_0\in \mathfrak{n}(\A_0)\}$ 
in $\mathcal{H}^2_0(\A)$.

This then ensures that 
$$(X \ominus [W \A]_2)^*(h^{1/2}f_\lambda)\subset (X \ominus [W \A]_2)^*W=\{0\}.$$For any $b\in \mathfrak{n}_\nu$ we then have that $$\{0\}=(X \ominus [W \A]_2)^*(h^{1/2}f_\lambda)b=(X \ominus [W \A]_2)^*\sigma^\nu_{-i/2}(f_\lambda)(h^{1/2}b)$$(see Lemma \ref{GL2-2.4+5}). The 
$\sigma$-weak convergence of $\sigma^\nu_{-i/2}(f_\lambda)$ to $\I$ guarantees that $\sigma^\nu_{-i/2}(f_\lambda)(h^{1/2}b)$ is weak-$L^2$ convergent to $(h^{1/2}b)$. Hence we in fact have that $\{0\}=(X \ominus [W \A]_2)^*(h^{1/2}b)$ for each $b\in \mathfrak{n}_\nu^*$. But because of the density of $\{h^{1/2}b \colon b\in \mathfrak{n}_\nu^*\}$ in $L^2(\M)$, this can only be if $\{0\}=(X \ominus [W \A]_2)$, that is $X=[W\A]_2$.

However the fact that $h^{1/2}f_\lambda \in W$ for each $\lambda$, also ensures that $W^*(h^{1/2}f_\lambda) \subset W^* W  \subset L^1({\mathcal D})$. For any $w\in W$ we will then have that 
$$[w^*\sigma^\nu_{-i/2}(f_\lambda)h^{1/2}]=w^*(h^{1/2}f_\lambda)=\mathbb{E}_1(w^*(h^{1/2}f_\lambda))=$$ $$\mathbb{E}_2(w^*)(h^{1/2}f_\lambda)=[\mathbb{E}_2(w^*)\sigma^\nu_{-i/2}(f_\lambda)h^{1/2}].$$For any $b\in \mathfrak{n}_\nu^*$, we then have that 
 \begin{eqnarray*}
 (w^*\sigma^\nu_{-i/2}(f_\lambda))(h^{1/2}b)
 &=&[w^*\sigma^\nu_{-i/2}(f_\lambda)h^{1/2}]b\\
 &=&[\mathbb{E}_2(w^*)\sigma^\nu_{-i/2}(f_\lambda)h^{1/2}]= (\mathbb{E}_2(w^*)\sigma^\nu_{-i/2}(f_\lambda))(h^{1/2}b).
  \end{eqnarray*}
As noted earlier $\sigma^\nu_{-i/2}(f_\lambda))(h^{1/2}b)$ converges $L^2$-weakly to $(h^{1/2}b)$, which then leads to the conclusion that $w^*(h^{1/2}b)= \mathbb{E}_2(w^*)(h^{1/2}b)$ for each $b\in \mathfrak{n}_\nu^*$. The density of $\{h^{1/2}b \colon b\in \mathfrak{n}_\nu^*\}$ in $L^2(\M)$, then ensures that $w^*=\mathbb{E}_2(w^*)$, and hence that $w=\mathbb{E}_2(w)\in L^2(\mathcal{D})$. So $X = [W\A]_2 \subset [L^2(\mathcal{D})\A]_2\subset \mathcal{H}^2(\A)$. The converse inclusion 
$\mathcal{H}^2(\A) \subset X$ follows from the fact that $\mathcal{H}^2(\A)$ is orthogonal to $\mathcal{H}^2_0(\A)^*$.

We next note that Bekjan and Oshanova \cite{BO} proved the equivalence of (3) and (1) for the case of semifinite algebras equipped with a faithful normal semifinite trace. If therefore we take the proof of (3)$\Rightarrow$(1) in \cite[Theorem 3.4]{L-HpIII}, and replace the version of the reduction theorem used there with the one proved in this paper, then that will enable us to lift the implication (3)$\Rightarrow$(1) in the Bekjan-Oshanova result to the setting of general von Neumann algebras. The equivalence (1)$\Leftrightarrow$(4) may be proven using exactly the same argument as was used to prove the equivalence (1)$\Leftrightarrow$(4) in \cite[Theorem 3.4]{L-HpIII}. 
\end{proof}

\section{Toeplitz operators for general von Neumann algebras}\label{S13}

Let $\A$ be an approximately subdiagonal subalgebra of a von Neumann algebra $M$. For any $a\in \M$ we define the \emph{left} $T_a$ and \emph{right} $\prescript{}{a}{T}$ Toeplitz operators on $H^2(\A)$ by $T_a:f\mapsto P_+(af)$ and $\prescript{}{a}{T}:f\mapsto P_+(fa)$ respectively, where $P_+$ is the projection from $L^2(\M)$ onto $H^2(\A)$. We shall for the most part focus on left Toeplitz operators. Unless otherwise specified, all references to ``Toeplitz operators'' will therefore have left Toeplitz opertors in mind. The following basic operational properties hold for these operators. These properties are proved for left Toeplitz operators at the start of \S 2 of \cite{MW-Toeplitz} in the context of  finite maximal subdiagonal subalgebras, but on replacing the trace in their proof with the tracial functional $tr$, essentially the same proofs will with minor alterations go through in the general case.

\begin{proposition}\label{P:Toeplitz-props} Let $\A$ be an approximately subdiagonal subalgebra of a von Neumann algebra $M$. 
\begin{enumerate} 
\item For any $a\in \M$ we have that $T_a^* = T_{a^*}$ and $\prescript{}{a}{T^*} = \prescript{}{a^*}{T}$.
\item Let $a,b \in \M$ be given. If either $a\in \A$ or $b\in \A^*$, then $T_{ab}=T_aT_b$ and $\prescript{}{a}{T}\prescript{}{b}{T} = \prescript{}{ab}{T}$.
\end{enumerate}
\end{proposition}  

\subsection{Left vs Right Toeplitz operators}

Let $\mathscr{L}=\{T_a\colon a\in \A\}$ and $\mathscr{R}=\{\prescript{}{a}{T}\colon a\in \A\}$. In the case of finite von Neumann algebras equipped with a faithful 
normal tracial state Marsalli proved that $\mathscr{L}'=\mathscr{R}$ and $\mathscr{L}=\mathscr{R}'$ \cite[Theorem 1]{mar}. This was then lifted to the setting of $\sigma$-finite algebras by 
Ji \cite[Theorem 2.3]{jig3} using the version of the reduction theorem presented in \cite{HJX}. To the best of the authors' knowledge, this result is not yet known in the semifinite setting.
This is therefore the perfect opportunity to demonstrate the efficacy of the double reduction technique described in Remark \ref{doublered}, by first lifting this result from 
the setting of finite von Neumann algebras, to maximal subdiagonal subalgebras of semifinite von Neumann algebras and from there to maximal subdiagonal subalgebras of general von Neumann algebras. At its very root (the case considered by Marsalli) the proof depends on factorization, and hence at this stage it is not clear to the authors how this result could be extended to approximately subdiagonal subalgebras. The first step in our double reduction, is to use Bekjan's technique to lift the result from the setting of finite von Neumann algebras equipped with a faithful normal tracial state, to the semifinite setting.

\begin{lemma}\label{semifinLR}
Let $\M$ be a semifinite von Neumann algebra equipped with a faithful normal semifinite trace, and $\A$ a maximal subdiagonal subalgebra. Then $\mathscr{L}'=\mathscr{R}$ and $\mathscr{L}=\mathscr{R}'$
\end{lemma}

\begin{proof} It clearly suffices to only prove the first equality. It is easy to see that $\mathscr{R}\subset\mathscr{L}'$. We pass to 
proving the converse inclusion. So let $X\in \mathscr{L}'$ be given and let $(e_\gamma)$ be a net of projections in $\D$ increasing to 
$\I$. Write $V_\gamma$ for the operator $V_\gamma(x)= e_\gamma xe_\gamma$ on $H^2(\A)$. We claim that the restriction of  
$V_\gamma XV_\gamma$ to $H^2(e_\gamma\A e_\gamma)$ belongs to $\mathscr{L}_\gamma'$, where $\mathscr{L}_\gamma$ is the set of operators 
$\{T_a\colon a\in e_\gamma\A e_\gamma\}$ considered as operators on $H^2(e_\gamma\A e_\gamma)$. Given $a\in e_\gamma\A e_\gamma$ we will 
for the sake of clarity hereafter write $T_a^{(\gamma)}$ for the operator induced on $H^2(e_\gamma\A e_\gamma)$, and $T_a$ for the 
operator induced on $H^2(\A)$. Recall that $H^2(e_\gamma\A e_\gamma)=e_\gamma H^2(\A)e_\gamma\subset H^2(\A)$. Then notice that in this 
case we will clearly have that 
$V_\gamma T_a^{(\gamma)}V_\gamma(g)=T_a^{(\gamma)}(g)=T_a(g)$ for $f\in H^2(\A)$ and $g\in H^2(e_\gamma\A e_\gamma)$. Given any 
$x\in H^2(e_\gamma\A e_\gamma)$, it is therefore clear that 
$$T_a^{(\gamma)}(V_\gamma XV_\gamma(x))=T_a^{(\gamma)}(V_\gamma (X(x)))= ae_\gamma(X(x))e_\gamma 
= e_\gamma a(X(x))e_\gamma$$ $$= e_\gamma T_a(X(x))e_\gamma=e_\gamma X(T_a(x))e_\gamma = e_\gamma (X(e_\gamma axe_\gamma))e_\gamma = V_\gamma XV_\gamma(T_a^{(\gamma)}(x)).$$Thus the claim follows.

By Marsalli's result there will for any $\gamma$ exist a unique $a_\gamma\in e_\gamma\A e_\gamma$ such that $V_\gamma XV_\gamma= \prescript{}{a_\gamma}{T^{(\gamma)}}$. 
Now let $\gamma\leq\delta$ be given. Then of course $e_{\gamma}\leq e_{\delta}$, whence 
$H^2(e_{\gamma}\A e_{\gamma})\subset H^2(e_{\delta}\A e_{\delta})$. So for any $x\in H^2(e_{\gamma}\A e_{\gamma})$ we will have that 
$$\prescript{}{e_\gamma a_\delta e_\gamma}{T^{(\gamma)}}(x)= x e_\gamma a_\delta e_\gamma = e_\gamma x a_\delta e_\gamma = 
V_\gamma[(\prescript{}{a_\delta}{T^{(\delta)}}(x))]=$$ $$V_\gamma[V_\delta X V_\delta(x))] = V_\gamma[V_\delta X V_\delta(V_\gamma(x)))]= 
V_\gamma X V_\gamma(x)=\prescript{}{a_\gamma}{T^{(\gamma)}}(x).$$In other words 
$x(e_\gamma a_\delta e_\gamma)= xa_\gamma$ for any $x\in H^2(e_\gamma\A e_\gamma)$. This can only be if in fact 
\begin{equation}\label{eq-toeplitz} e_\gamma a_\delta e_\gamma= a_\gamma. \end{equation} For each $\gamma$ we moreover have that 
$\|a_\gamma\|=\|\prescript{}{a_\gamma}{T^{(\gamma)}}\|=\|V_\gamma X V_\gamma\|\leq\|X\|$. By passing to a subnet if necessary, we may therefore 
assume that $(a_\gamma)$ is $\sigma$-weakly convergent to some $a\in \A$. For any fixed $\delta$, $(e_\delta a_\gamma e_\delta)_\gamma$ 
is then $\sigma$-weakly convergent to $e_\delta a e_\delta$. But by equation (\ref{eq-toeplitz}), that limit must be $a_\delta$. Now let 
$\gamma$, $\lambda$ be given and select $\delta$ such that $\delta\geq \gamma$ and $\delta\ge \lambda$. It is an easy exercise to see 
that the Hilbert adjoint of the operator $V_\gamma$ on $H^2(\A)$ is just $V_\gamma$. For any $x\in H^2(e_\gamma\A e_\gamma)$ and 
$y\in H^2(e_\lambda\A e_\lambda)$ we may then use this observation and the fact that $e_\delta\geq e_\gamma$ and 
$e_\delta\geq e_\lambda$, to see that 
$$\langle X(x),y\rangle = \langle XV_\delta(x),V_\delta y\rangle = \langle V_\delta XV_\delta (x), y\rangle 
= \langle \prescript{}{a_\delta}{T^{(\delta)}}(x), y\rangle=$$ $$\langle \prescript{}{e_\delta a e_\delta}{T^{(\delta)}}(x), y\rangle=\langle x a e_\delta, y\rangle= \tau(y^*x a e_\delta)=\tau(e_\delta y^*x a)=\tau(y^*xa)= \langle \prescript{}{a}{T}(x), y\rangle.$$Since $\cup_\gamma H^2(e_\gamma\A e_\gamma)$ is norm dense in $H^2(\A)$, 
we have that $X=\prescript{}{a}{T}$ on a norm-dense subspace of $H^2(\A)$, and hence on all of $H^2(\A)$. This clearly proves the lemma.
\end{proof}

We now pass to the second part of the double reduction, which is the application of the version of the Haagerup reduction theorem proved 
in this paper, to lift the result from semifinite von Neumann algebras, to the general case. We shall achieve this by modifying the 
technique Ji used when he proved this result for $\sigma$-finite algebras using the version of the reduction theorem in \cite{HJX}. In 
his proof Ji often made intimate use of the fact that $\sigma$-finite algebras may be assumed to admit a cyclic and separating vector. 
Not all aspects of his proof strategy therefore go through in the general case. Some ingenuity is required to fill the gaps.

\begin{theorem}
Let $\A$ be a maximal subdiagonal subalgebra of some von Neumann algebra. Then $\mathscr{L}'=\mathscr{R}$ and $\mathscr{L}=\mathscr{R}'$
\end{theorem}

\begin{proof} The first half of Ji's proof of the $\sigma$-finite case will in this setting go through largely unaltered. However in the second half Ji regularly used tricks only valid in the $\sigma$-finite setting. The idea of that part of the proof still works, but much more care is needed to make it work in this setting. For the sake of the reader we briefly sketch the first part of the proof, before providing details for the second part. 

Given the similarity of the proofs, it clearly suffices to only prove the first equality. As in the semifinite case, it is easy to see that $\mathscr{R}\subseteq\mathscr{L}'$. Given $X\in \mathscr{L}'$, our task is therefore to show that $X\in \mathscr{R}$. The first step in doing this is to lift $X$ to an operator on $H^2(\widehat{\A})$. To do this one firstly shows that $\mathrm{span}\{\lambda_t x\colon x\in L^2(\M), \, t\in \mathbb{Q}_D\}$ and $\mathrm{span}\{\lambda_t a\colon a\in H^2(\A), \, t\in \mathbb{Q}_D\}$ are respectively dense in $L^2(\R)$ and $H^2(\widehat{A})$. The prescription $\widehat{X}(\lambda_t a)=\lambda_tX(a)$ where $a\in H^2(\A)$ and $t\in \mathbb{Q}_D$ therefore defines an operator on a dense subspace of $H^2(\widehat{A})$. However one is able to show that this operator is bounded and therefore continuously extends to all of $H^2(\widehat{A})$, and that in fact $\|\widehat{X}\|=\|X\|$. 

By careful checking one is then able to show that for any $x\in \mathrm{span}\{\lambda_t a\colon a\in H^2(\A), \, t\in \mathbb{Q}_D\}$, any $a\in \A$ and any 
$s\in \mathbb{Q}_D$, we have that $\widehat{X}(\lambda_sax)=\lambda_sa\widehat{X}(x)$. So by continuity this equality then holds for all $x\in H^2(\widehat{A})$. Now recall that $\mathrm{span}\{\lambda_s a: a\in \A, s\in \mathbb{Q}_D\}$ is $\sigma$-weakly dense in $\widehat{A}$. Since $\mathrm{span}\{\lambda_s a: a\in \A, s\in \mathbb{Q}_D\}$ is convex, this subspace is even $\sigma$-strongly dense. This fact then enables us to conclude that in fact $\widehat{X}(fx)=f\widehat{X}(x)$ for all $f\in\widehat{\A}$ and all $x\in H^2(\widehat{\A})$. This is however just another way of saying that $\widehat{X}$ is in the commutant of $\{T_f\colon f\in \widehat{A}\}$ - the class of left-Toeplitz operators on $H^2(\widehat{\A})$ with symbols in $\widehat{\A}$. 

At this point we pass to using reduction to show that then $\widehat{X}\in \{\prescript{}{f}{T}\colon f\in \widehat{A}\}$ and to then conclude from that fact that $X\in \{\prescript{}{a}{T}\colon a\in \A\}$. Here we need to give details. We will denote the extension of the conditional expectation $\mathscr{W}_n:\R\to \R_n$ to $L^2(\R)\to L^2(\R_n)$ by $\mathscr{W}_n^{(2)}$. Recall that $\mathscr{W}_n^{(2)}$ maps $H^2(\widehat{\A})$ onto $H^2(\A_n)$. For each $n\in \mathbb{N}$, we may therefore define a map $X_n$ on $H^2(\A_n)$ by means of the prescription $$X_n(\mathscr{W}_n^{(2)}(x))=\mathscr{W}_n^{(2)}(\widehat{X}(x))\mbox{ for all }x\in H^2(\widehat{\A}).$$For each $n$, we then clearly have that $X_n\in B(H^2(\A_n))$ with $\|X_n\|\leq \|X\|$. Moreover for any $n$, $a\in \A_n$ and $f\in H^2(\A_n)$, we then further have that 
 $$X_naf= X_n\mathscr{W}_n^{(2)}(af)=\mathscr{W}_n^{(2)}(\widehat{X}af) =  \mathscr{W}_n^{(2)}(a\widehat{X}f) = a\mathscr{W}_n^{(2)}(\widehat{X}f)=aX_n(f).$$ 

We know from \cite[Theorem II.37]{terp} and its proof that $\R_n\rtimes_{\tnu{\upharpoonright}\R_n}\mathbb{R}$ is unitarily equivalent to 
$\R_n\rtimes_{\tau_n}\mathbb{R}$. For the sake of clarity we shall momentarily suspend our practice of identifying $\R_n$ with the 
copies $\pi_{\tnu{\upharpoonright}\R_n}(\R_n)$ and $\pi_{\tau_n}(\R_n)$ inside $\R_n\rtimes_{\tnu{\upharpoonright}\R_n}\bR$ and $\R_n\rtimes_{\tau_n}\bR$ respectively. Let $u$ be the unitary realising the unitary equivalence. Then the map $V_u(x)=u^*xu$ maps 
$\R_n\rtimes_{\tnu{\upharpoonright}\R_n}\mathbb{R}$ onto $\R_n\rtimes_{\tau_n}\mathbb{R}$ with terms of the form $\pi_{\tnu/\R_n}(f)$ where $f\in \R_n$, mapping onto $\pi_{\tau_n}(f)$. We also know from the remark on page 62 of \cite{terp} that 
the space $L^2(\R_n)$ realised inside $\R_n\rtimes_{\tau_n}\mathbb{R}$, is a copy of $L^2(\R_n,\tau_n)\otimes \exp^{1/2}$. For the sake 
of simplicity of exposition we shall identify $L^2(\R_n,\tau_n)\otimes \exp^{1/2}$ with $L^2(\R_n,\tau_n)$ when convenient. We know from 
the above computation that in the context of $\R_n\rtimes_{\tnu/\R_n}\mathbb{R}$ we have that $X_naf = aX_n(f)$ for any 
$a\in \pi_{\tnu/\R_n}(\A_n)$ and $f\in H^2(\A_n)$. This then transfers to the claim that in the context of 
$\R_n\rtimes_{\tau_n}\mathbb{R}$, we have that $$V_uX_nV_{u^*}(af) = V_uX_n(V_{u^*}(a)V_{u^*}(f))=  V_uV_{u^*}(a)X_nV_{u^*}(f)
= aV_uX_nV_{u^*}(f)$$for $a\in \pi_{\tau_n}(\A_n)$ and 
$f\in H^2(\A_n,\tau_n)$. So by Lemma \ref{semifinLR} there must exist some $b_n\in \pi_{\tau_n}(\A_n)$ such that 
$V_uX_nV_{u^*}=\prescript{}{b_n}{T}$ on $H^2(\A_n,\tau_n)$; that is $V_uX_nV_{u^*}(f)=fb_n$ for each $f\in H^2(\A_n,\tau_n)$. Then also 
 $$\|b_n\|=\|\prescript{}{b_n}{T}\|=\|V_uX_nV_{u^*}\|=\|X_n\|.$$ On transferring back to the $\R_n\rtimes_{\tnu/\R_n}\mathbb{R}$ context, it can now 
easily be checked that there exists $c_n\in\pi_{\tnu/\R_n}(\A_n)\subset \widehat{\A}$ such that $X(f)=fc_n$ for each $f\in H^2(\A_n)$, 
with in addition $\|c_n\|=\|X_n\|\leq\|X\|$. (In fact $c_n=V_{u^*}(b_n)$.)

Now suppose we have $k,n\in \mathbb{N}$ with $k\geq n$. Since then $H^2(\A_n)\subset H^2(\A_k)$, we have that $X_kf=fc_k$ for any 
$f\in H^2(\A_n)$. By passing to a subsequence if necessary we may assume that $(c_k)$ is $\sigma$-weakly convergent to some $c\in\widehat{\A}$. For any 
$f\in H^2(\A_n)$, we then have that 
 $$\widehat{X}(f)= \lim_{k\to\infty}\mathscr{W}_k(\widehat{X}(f))= \lim_{k\to\infty}X_k(f)= \lim_{k\to\infty}fc_k=fc= 
\prescript{}{c}{T}(f).$$ 
Thus $\widehat{X}$ and $\prescript{}{c}{T}$ (as an operator on $H^2(\widehat{A})$) agree on the dense subspace 
$\cup_{n\geq 1}H^2(\A_n)$ of $H^2(\widehat{A})$, and hence on all of $H^2(\widehat{A})$. For any $f\in H^2(\A)\subset H^2(\widehat{A})$, 
we see from the definition of $\widehat{X}$ that $fc=\widehat{X}(f)=X(f)\in H^2(\A)$. But this means that we will for any $f\in H^2(\A)$ 
have that $X(f)=fc=\mathscr{W}^{(2)}_{\mathbb{Q}_D}(fc) = f\mathscr{W}_{\mathbb{Q}_D}(c)$ for all $f\in H^2(\A)$. Thus $X$ agrees with 
the operator $\prescript{}{g}{T}$ on $H^2(\A)$ where $g=\mathscr{W}_{\mathbb{Q}_D}(c)\in\A$. The result therefore follows
\end{proof}

\section{Fredholm Toeplitz operators}\label{S14} 

We remind the reader that closed densely defined operator $T$ between Banach space $X$ and $Y$ with closed range, is said to upper semi-Fredholm (denote by $T\in \Phi_+$) if the kernel of $T$ is finite dimensional. If on the other hand the quotient space $Y/(T(X))$ is finite dimensional, we say that $T$ is lower semi-Fredholm (denoted by $T\in \Phi_-$). If $T$ is both upper and lower semi-Fredholm we will just call it Fredholm, and will denote this class by $\Phi$. It is worth noting that if $T$ is a densely defined closed operator between Banach spaces with the quotient space $Y/(T(X))$ finite dimensional, then its range is actually automatically closed.

\subsection{Indices of $T_f$ for $f\in (L^\infty(\M))^{-1}$}

We shall have need of the following fact, which is an easy consequence of the Haagerup-Terp standard form (Theorem \ref{7:T stdform}). In the following we will at times use this fact without comment.

\begin{lemma}\label{Mfvsf} Let $f\in \M$ be given and suppose that $\M$ is in standard form in the sense of Definition \ref{7:D defstdfrm}. Then the multiplication map $M_f:L^2(\M)\to L^2(\M):a\mapsto fa$ is bounded below if and only if the same is true of $f$ (that is if and only if $|f|$ is invertible).
\end{lemma} 

In this subsection $\A$ is taken to be an approximately subdiagonal subalgebra of the von Neumann algebra $\M$. 
We will define indices for $T_f$ for invertible elements of $\M$. The basic idea behind this definition, is to try and quantify the extent to which left multiplication by an element of $\M^{-1}$ disturbs the a priori structure of $H^2(\A)$. In other words, given say $f\in \M^{-1}$, how much of $fH^2(\A)$ lies in $(H^2_0)^*$, and how different is $fH^2+(H^2_0)^*$ from $L^2=H^2\oplus (H^2_0)^*$? 

\begin{definition} Given any $f\in \M$ we define the quantity $\alpha(f)$ to be $\alpha(f)=\mathrm{dim}(L^2/(fH^2+(H^2_0)^*))$. If $f$ is even bounded below, we define $\beta(f)$ to be $\beta(f)=\mathrm{dim}(H^2\cap (M_f)^{-1}(H^2_0)^*)$, where $(M_f)^{-1}$ is the partially defined inverse of $M_f$. Whenever either $\alpha(f)$ or $\beta(f)$ is finite, the index of an $f\in \M$ which is bounded below, is then defined to be $$\mathrm{ind}(f)=\beta(f)-\alpha(f).$$
\end{definition}

\begin{remark} In the case of a unitary $u$, we have that $$\beta(u)=\mathrm{dim}(H^2\cap u^{-1}(H^2_0)^*)=\mathrm{dim}(H^2\cap u^*(H^2_0)^*)=\mathrm{dim}(u(H^2)\cap (H^2_0)^*).$$We pause to justify these definitions. Essentially $\alpha(u)$ attempts to quantify how much of $H^2$ is ``lost'' when $H^2$ is replaced by $uH^2$, whereas $\beta(u)$ quantifies how much of $(H^2_0)^*$ is ``lost'' through being absorbed into $uH^2$. As such the philosophy of this definition seems to fit that of Definition 3 of \cite{mir}. We will denote the class of all elements of $\M$ with finite index by $\Phi(\M)$. 
\end{remark}

\begin{proposition}\label{ind-u} Let $f\in \M$ be given and let $T_f$ be the associated Toeplitz operator on $H^2(\A)$. 
\begin{enumerate}
\item[(1)] If $f\in \M$ is bounded below, then 
 $$\mathrm{dim(ker}(T_f))=\beta(f)\;\text{ and }\;\mathrm{dim}(H^2/(T_fH^2))=\alpha(f).$$ 
It follows that such an $f$ has an index if and only if $T_f$ does, in which case $\mathrm{ind}(T_f)=-\mathrm{ind}(f)$. Moreover $T_f\in \Phi_-$ if and only if $\alpha(f)<\infty$.
\item[(2)] In the case where $f$ is invertible we further have that $\alpha(f)\geq \beta(f^*)$, with equality holding whenever $\alpha(f)<\infty$. In such a case $fH^2+(H^2_0)^*$ is closed. 
\end{enumerate} 
\end{proposition}

\begin{proof} We first prove part (1) of the proposition. Let $f\in \M$ be bounded below. To see the first equality, observe that by the definition of Toeplitz operators, $b\in \mathrm{ker}(T_f)$ if and only if $fb\in (H^2_0)^*$. From this fact it is now easy to conclude that $\mathrm{ker}(T_f)=H^2\cap (M_f)^{-1}(H^2_0)^*$ , where $f^{-1}$ is the partially defined inverse of $f$. 

To see the second claim, observe that it is an easy exercise to see that as vector spaces, the quotient spaces $$\frac{H^2}{T_f(H^2)}\mbox{ and }\frac{H^2\oplus(H_0^2)^*}{T_f(H^2)\oplus(H^2_0)^*}=\frac{L^2}{T_f(H^2)\oplus(H^2_0)^*}$$may be canonically identified with each other. The claim now follows from the observation that $T_f(H^2)\oplus(H^2_0)^*=fH^2+(H^2_0)^*$. 

It follows from classical Fredholm theory that $T_f\in \Phi_-$ if and only if \newline $\mathrm{dim}(H^2/(T_fH^2))<\infty$. In view of what we've already proved, this establishes the final claim of part (a). 

In view of the fact that $T_f^*=T_{f^*}$, the claim in part (2) is now simply a consequence of part (1) combined with the known fact that $\mathrm{dim}(H^2/(T_fH^2))\geq \mathrm{dim(ker}(T_f^*))$, with equality holding whenever $\mathrm{dim}(H^2/(T_fH^2))<\infty$. It remains to prove that $fH^2+(H^2_0)^*$ is closed whenever $\alpha(f)<\infty$.

If indeed $\alpha(f)<\infty$, the continuous map $Q:H^2\to L^2/((H^2_0)^*):a\to fa+(H^2_0)^*$ will have a range with finite codimension. By known theory \cite[IV.1.13]{gb}, the range of $Q$ must then be closed. But saying that $\{fa+(H^2_0)^*:a\in H^2\}$ is a closed subspace of $L^2/((H^2_0)^*)$, is the same as saying that $fH^2+(H^2_0)^*$ is a closed subspace of $L^2$.
\end{proof}

The following corollary easily follows from the above. In the next section we will prove a result which shows that requiring $f$ to be bounded below is a natural restriction to make.

\begin{corollary} Let $f\in \M$ be bounded below. Then $T_f$ is Fredholm if and only if $f$ has a finite index.
\end{corollary}

For the case of a unitary, the following example gives some intuition about what each of the quantities $\alpha(u)$ and $\beta(u)$ represents in elementary complex analysis. 

\begin{example} Let $g$ be analytic on $\overline{\mathbb{D}}$ with $m$ zeros in the interior of $\overline{\mathbb{D}}$. Then for the classical Hardy spaces of the disc, we have that $m= \mathrm{dim}(L^2/([gH^2]+(H^2_0)^*))$. To see this one may argue as follows:

If $g$ is as above, then on some neighbourhood of $\overline{\mathbb{D}}$, it can be written in the form $$g(z)=z^{m_0}\Pi_{i=1}^k(z-a_i)^{m_i}g_0(z),$$where $\sum_{i=0}^k m_i=m$ with each $m_i\geq 0$, where $a_i\in \mathbb{D}\backslash\{0\}$ for each $i\geq 1$, and where $g_0$ is analytic on $\overline{\mathbb{D}}$ with no zeros in the interior of $\overline{\mathbb{D}}$. It 
is clear that for each $i\geq 1$ the function $h_i(z)=(1-z\overline{a_i})^{m_i}$ is invertible in both $L^\infty(\mathbb{T})$ and the disc 
algebra $\A(\mathbb{D})$. Now consider the function $f=\Pi_{i=1}^k h_i$. Using the facts just noted regarding the $h_i$'s, it 
is an exercise to firstly see that $\overline{f}^{(-1)}(H^2_0)^*= (H^2_0)^*$, and that $\overline{f}^{(-1)}L^2=L^2$. Hence \begin{eqnarray*}\mathrm{dim}(L^2/([gH^2]+(H^2_0)^*)) &=& \mathrm{dim}((\overline{f}^{(-1)}L^2)/([(\overline{f}^{(-1)}g)H^2]+\overline{f}^{(-1)}(H^2_0)^*))\\
&=&\mathrm{dim}(L^2/([(\overline{f}^{(-1)}g)H^2]+(H^2_0)^*)).
\end{eqnarray*}
But in its action on the circle group $\mathbb{T}$, we have that $$\overline{f}^{(-1)}(z)= \Pi_{i=1}^k(1-\overline{z}a_i)^{-m_i}=\Pi_{i=1}^k(\frac{z}{z-a_i})^{m_i}.$$So as a subspace of $L^2(\mathbb{T})$, $[(\overline{f}^{(-1)}g)H^2]+(H^2_0)^*)$ may be identified with $[z^mg_0H^2]+(H^2_0)^*$. Now recall that by construction both $g_0$ 
and $\frac{1}{g_0}$ are analytic on $\overline{\mathbb{D}}$. So $[z^mg_0H^2]=[z^mH^2]$. Since $L^2(\mathbb{T})= H^2\oplus (H^2_0)^*$ with $H^2= \mathrm{span}\{z^k:0\leq k\leq m-1\}\oplus [z^mH^2]$, the claim follows. 
\end{example}

\bigskip

If we combine the above observation with the classical argument principle of complex analysis (which describes winding numbers in terms of the difference between the zeros and poles), it provides some intuition for the famous classical result of Gohberg and Krein, which states that if for some $f\in C(\mathbb{T})$ the operator $T_f:H^2(\mathbb{T})\to H^2(\mathbb{T})$ is Fredholm, its index will be minus the winding number of the curve traced out by $f$ with respect to the origin. In situations where classical results of the above type on Fredholm properties of Toeplitz operators are extended to group von Neumann algebras, $C_r^*(G)$ may be used as a noncommutative substitute for $C(\mathbb{T})$. As an alternative to $C_r^*(G)$, we may on occasion also use $\mathscr{C}(G)$. To see this note that $\vng$ appears as the double commutant of both $C_r^*(G)$ and $\mathscr{C}(G)$.

In closing this section we present a proposition which offers some insight into the index of specific operators. Let $G$ be a countable and discrete ordered group. In this case $\vng$ will be a finite von Neumann algebra equipped with a faithful normal tracial state, the $\sigma$-weakly closed subalgebra $\A$ generated by $\{\lambda_g\colon g\geq e\}$ will be maximal subdiagonal, and the algebras $C^*_r(G)$ and $\mathscr{C}(G)$ will agree. 

\begin{proposition}\label{lambdag} Let $g\in G$ be given. 
\begin{itemize}
\item If $g\geq e$, $T_{\lambda_g}$ will be an isometry with range $\overline{\mathrm{span}\{\lambda_t: t\geq g\}}$. The map $T_{\lambda_g}$ is then a $\Phi_+$ map with index $\mathrm{card}[e,g)$. (Here we take the cardinality of the empty set to be 0.) 
\item If $g\leq e$, $T_{\lambda_g}$ is a surjection with kernel $\overline{\mathrm{span}\{\lambda_t: e\leq t< g^{-1}\}}$. The map $T_{\lambda_g}$ is then a $\Phi_-$ map with index $-\mathrm{card}(e, g^{-1}]$.
\end{itemize}   
\end{proposition}

\begin{proof} First suppose that $g\geq e$. Then $\lambda_g\in \A$, which means that for any $b\in H^2(\A)$, $T_{\lambda_g}b=\lambda_g b$. So for any $b\in H^2$ we will trivially have that $\|T_{\lambda_g}b\|_2=\|\lambda_g b\|_2=\|b\|_2$. Now recall that as elements of 
$H^2(\A)$, $\{\lambda_t:t\geq e\}$ is here an orthonormal basis for $H^2$, and hence $\mathrm{span}\{\lambda_t:t\geq e\}$ is norm dense in $H^2(\A)$. Since $\lambda_g$
 is a unitary element of $\vng$, we must by continuity of multiplication have that $$\lambda[\overline{\mathrm{span}\{\lambda_t: t\geq e\}}]=\overline{\mathrm{span}\{\lambda_{gt}: t\geq e\}}=\overline{\mathrm{span}\{\lambda_s: s\geq g\}}.$$In other words the range of $T_{\lambda_g}$ is 
closed and is precisely $\overline{\mathrm{span}\{\lambda_s: s\geq g\}}$. So as claimed $T_{\lambda_g}$ is a $\Phi_+$ map. It is an easily 
verifiable fact that $$\mathrm{ran}(T_{\lambda_g})^\perp=H^2\ominus (\overline{\mathrm{span}\{\lambda_s: s\geq g\}})=\overline{\mathrm{span}\{\lambda_s: e\leq s < g\}}.$$In the case $g=e$ this subspace is of course empty. If $g>e$ then $\{\lambda_s: e\leq s <g\}$ is an orthonormal basis. The 
cardinality of this orthonormal basis is precisely $\mathrm{card}[e,g)$, and hence as claimed, we have that
 $$\mathrm{ind}(T_{\lambda_g})= \mathrm{dim}(\mathrm{ran}(T_{\lambda_g})^\perp)=\mathrm{card}[e,g).$$
 
Now suppose that $g\leq e$. Then of course $g^{-1}\geq e$. Given that $\lambda_g^*=\lambda_{g^{-1}}$, we may then apply what we've already proved to the operator $T_{\lambda_g^*}=T_{\lambda_g}^*$, and use classical duality theory to conclude that $T_{\lambda_g}$ is a surjective $\Phi_-$ operator with $\mathrm{ind}(T_{\lambda_g})= -\mathrm{ind}(T_{\lambda_g^*})=-\mathrm{card}[e,g^{-1})$.
\end{proof}

\subsection{Semi-Fredholm Toeplitz operators}

Throughout this subsection $G$ will be a topologically ordered locally compact group. Whenever $\vng$ is in view, $\A$ will denote the approximately subdiagonal subalgebra generated by 
$\{\lambda_g\colon g\geq e\}$.

\begin{definition} When we write $s\nearrow \infty$ where $s\in G$, we mean that for any compact neighbourhood $K$ of the group unit 
$e$, there exists some $s_0\geq e$ such that $s\geq s_0\Rightarrow s\not\in K$. The concept $s\searrow -\infty$ is defined similarly.
\end{definition}

\begin{lemma}\label{wk8to0} The operators $\lambda_s$ converge $\sigma$-weakly to 0 in $\vng$ as either $s\nearrow \infty$ or $s\searrow -\infty$.
\end{lemma}

\begin{proof}
We will prove the claim for the case $s\nearrow \infty$. Since the net $\{\lambda_s: s\in G\}$ is norm-bounded, it is relatively $\sigma$-weakly compact. To prove our claim we need only show that 0 is the only $\sigma$-weak cluster point of this net. Let $v$ be any $\sigma$-weak  cluster point. In that case there must be a subnet $\{\lambda_{s_t}:s_t\in G\}$ which is $\sigma$-weakly convergent to $v$. But then this subnet is also weak operator convergent to $v$. So for any two elements $f,g \in L^2(G)$, we should have $\langle \lambda_{s_t} f, g\rangle\to \langle vf, g\rangle$. 

Now suppose that $f$ and $g$ are supported on compact sets $K_f$ and $K_g$. Then $t\mapsto f(s^{-1}t)$ will of course be supported on 
$sK_f$. It is clear that as $s\nearrow \infty$ there must exist some $s_0$ such that $sK_f\cap K_g=\emptyset$ for all $s\geq s_0$. In other words we will then have that $\langle \lambda_s f, g\rangle = \int_G \overline{g(t)}f(s^{-1}t)\,dt=0$ for all $s\geq s_0$. This in turn ensures that $\langle vf,g\rangle=0$ for all $f,g \in C_{00}(G)$. Since $C_{00}(G)$ is dense in $L^2(G)$, we must therefore have that $\langle vf,g\rangle=0$ for all $f,g \in L^2(G)$, and hence that $v=0$ as required.
\end{proof}

The following theorem provides strong support for a practice that we have already been implicitly applying, which is that when trying to describe the symbols $f$ for which $T_f$ is $\Phi_+$, it is not unreasonable to restrict to the case where $f$ is bounded below. 

\begin{theorem}\label{semi-fred} Let $f\in \vng$ be given. If $T_f\in \Phi_+$, then $f$ is bounded below. 
\end{theorem}

\begin{proof}
Assume that $T_f\in \Phi_+$, and let $P_+$ be the projection of $L^2(\vng)\equiv L^2(G)$ onto $H^2(\vng)$. The kernel of $T_f$ is finite dimensional, so there exists a compact projection $K$ from 
$L^2(\vng)$ onto $\mathrm{ker}(T_f)$. The fact that $T_f\in \Phi_+$, ensures that $T_f$ is bounded below on $(\mathrm{ker}(T_f))^\perp$. Using this fact we may conclude that there exists a constant 
$\delta>0$ so that $$\|T_f(a)\|^2+\|K(a)\|^2\geq\delta(\|(\I-K)a\|^2+\|Ka\|^2)=\delta\|a\|^2$$for all $a\in H^2(\vng)$. Now let $g\in L^2(G)$ and $h\in L^1(G)$ be given, and assume that $h$ has compact support. Using the left-translation invariance of Haar measure, we then have that
\begin{eqnarray*}
\lambda_h\lambda_{s_0}(g)(t) &=& \lambda_h(g_{s_0})(t)\\
&=& \int_G h(s)g_{s_0}(s^{-1}t)\,ds\\
&=& \int_G h(s)g(s_0^{-1}s^{-1}t)\,ds\\
&=& \int_G h(rs_0^{-1})g(r^{-1}t)\,dr\quad\mbox{set }r=ss_0\\
&=& \lambda_{\rho_{s_0^{-1}}(h)}(g)(t)
\end{eqnarray*}
As far as the action of $r\mapsto h(rs_0^{-1})=\rho_{s_0^{-1}}(h)(r)$ is concerned, notice that $r\in \mathrm{supp}(\rho_{s_0^{-1}}(h))$ iff $rs_0^{-1}\in \mathrm{supp}(h)$. 
So any $t\in \rho_{s_0^{-1}}(h)$ is of the form $t=rs_0^{-1}$ for some $r\in \mathrm{supp}(h)$, with $t=rs_0^-1 \geq e$  iff $s_0^{-1}\geq r^{-1}$. If therefore we select 
$s_0$ so that $s_0^{-1}\geq r^{-1}$ for every $r\in \mathrm{supp}(h)$, the support of $\rho_{s_0^{-1}}(h)$ will then be contained in $G^+$. The same will of course be true 
for any $s\geq s_0$. This means that $\lambda_{\rho_{s}(h)}\in H^2(\vng)$ for all $s^{-1}\geq s_0{-1}$. 
 
Given any $b\in \mathrm{span}\{\lambda_h: h\in C_c(G)\}$, the above discussion ensures that we can find some $s_b\in G$ so that $b\lambda_s\in H^2(\A)$ for all $s^{-1}\geq s_b^{-1}$. 
So for $s$ large enough, we will have that $P_+(b\lambda_s)=b\lambda_s$. By Lemma \ref{wk8to0}, we also have that $\lambda_{s^{-1}}=\lambda_s^*$ converges $\sigma$-weakly to 0  as $s^{-1}\nearrow\infty$, 
and hence that $b\lambda_s$ then converges weakly to 0 in $L^2(\vng)$. On replacing $a$ with such a $b\lambda_s$ in the computation in the first part of the proof, we will then have that 
\begin{eqnarray*}
\|fb\|^2+\|K(b\lambda_s)\|^2 &=& \|fb\lambda_s\|^2+\|K(b\lambda_s)\|^2\\
&\geq& \|P_+(fb\lambda_s)\|^2+\|K(b\lambda_s)\|^2\\
&\geq& \delta\|b\lambda_s\|^2\\
&=& \delta \|b\|^2
\end{eqnarray*}
The compactness of the projection $K$ ensures that it transforms weak convergence to strong convergence. So the terms $K(b\lambda_s)$ will converge strongly to 0 as $s\nearrow \infty$. This fact then yields the inequality $$\|\varphi b\|\geq \delta \|b\|\mbox{ for all }b\in \mathrm{span}\{\lambda_h: h\in C_{c}(G)\}.$$The next thing to note is that $\mathrm{span}\{\lambda_h: h\in C_{c}(G)\}$ is norm-dense in $L^2(\vng)$. Hence $$\|fa\|\geq \delta \|a\|\mbox{ for all }a\in L^2(\vng).$$The Lemma \ref{Mfvsf} now ensures that $f$ is bounded below. 
\end{proof}

\begin{corollary} Let $\varphi\in \vng$ be given. If $T_\varphi\in \Phi_+\cap\Phi_-$, then $\varphi\in \vng^{-1}$. 
\end{corollary}

\begin{proof} Recall that $T_\varphi^*=T_{\varphi^*}$. Since $T_\varphi\in\Phi_-$ iff $T_\varphi^*\in\Phi_+$, the preceding result ensures that both $\varphi$ and $\varphi^*$ are bounded below. So $\varphi$ must be invertible.
\end{proof}

In closing we present the following very elegant result. This is a faithful reproduction of a well-known classical result. A version of this result appears in \cite{pru} (see Corollary 4.4(iii) of that paper). However we hasten to point out that Prunaru's Toeplitz operators map from $H^\infty$ to $H^2$, not $H^2$ to $H^2$. Hence we cannot directly apply his result. The final part of the current proof relies on an application of the noncommutative Riesz factorization theorem which at this stage is only known to hold without restriction for finite maximal subdiagonal subalgebras.

\begin{proposition}\label{prop:linv} Let $\M$ be a finite von Neumann algebra equipped with a faithful normal tracial state and let $\A$ be a finite maximal subdiagonal subalgebra of $\M$. For any unitary $u\in \M$, $T_u$ is bounded below if and only if $d(u,\A)=\inf\{\|u-a\|_\infty : a\in \A\}<1$.
\end{proposition}

\begin{proof}
Let $u\in \M$ be given with $d(u,\A)<1$. Then there exists $a\in \A$ with $\|\I-a^*u\|=\|u-a\|<1$. For any $f\in H^2(\A)$ we have that $\|f-T_{a^*u}f\|_2=\|P(f-a^*uf)\|_2\leq \|\I-a^*u\|_\infty\|f\|_2$. Hence $\|I-T_{a^*u}\|\leq d(u,\A)<1$. This ensures that $T_{a^*u}$ is invertible. But by Proposition \ref{P:Toeplitz-props} $T_{a^*u} =T_{a^*}T_u$. Hence $T_u$ is bounded below.

Conversely suppose $T_{a^*u}$ to be left invertible. That means we can find some $1>\epsilon>0$ for which 
$$\|P(uf)\|_2\geq \epsilon\|f\|_2=\epsilon\|uf\|_2\mbox{ for every }f\in H^2.$$Therefore for any $f\in H^2$, we will have that
$$\|uf\|_2^2=\|P(uf)\|_2^2+\|(I-P)(uf)\|_2^2\geq \epsilon^2\|uf\|^2_2+\|(I-P)(uf)\|_2^2,$$and hence that 
 $$|(I-P)(uf)\|_2^2\leq (1-\epsilon^2)^{1/2}\|uf\|_2.$$ 
Given $f\in H^2$ and $g\in H^2_0$, we may then use the fact that $P(uf)\perp g^*$ to see that then 
$$|\tau(ufg)|=|\tau([(I-P)(uf)]g)|\leq \|(I-P)(uf)\|_2\|g\|_2\leq (1-\epsilon^2)^{1/2}\|f\|_2\|g\|_2.$$We may now argue as in the last part of the proof of \cite[Theorem 3.9]{LX} to conclude from this inequality that 
$$d(u,\A)=\sup\{|\tau(uh)|:h\in H^1_0, \|h\|_1\leq 1\}\leq (1-\epsilon^2)^{1/2}.$$
\end{proof}

\subsection{Hankel maps and the existence of Fredholm Toeplitz operators}

Let $\A$ be an approximately subdiagonal subalgebra of $\M$. Given $f\in \M$ we define the Hankel map $\mathcal{H}_f$ with symbol $f$ to be the map $\mathcal{H}_f:H^2(\A)\to(H^2_0(\A))^*: a\mapsto (\I-P_+)(f a)$.

\medskip

\textbf{Warning:} \emph{This definition of a Hankel map is different from the one given in \cite{LX}, but is the same as the one given in \cite{CMNX}.}

\medskip

It turns out that an easy test for the Fredholmness of $T_f$, is compactness of the Hankel map. This fact should however be tempered with the observation that some group algebras admit NO compact Hankel maps (see the introduction to section 5 of \cite{Exel_1}).

\begin{proposition} Let $f\in \M$ be bounded below. Then $T_f$ will be $\Phi_+$ whenever $\mathcal{H}_f$ is compact.
\end{proposition} 

\begin{proof}
Let $f\in \M$ be bounded below. But then so is $M_f$ by Lemma \ref{Mfvsf}. For the sake of convenience we will in the remainder of the proof regard both $T_f$ and $\mathcal{H}_f$ as maps into $L^2(\M)$. (This can be done without loss of generality, since all we need to check is the closedness of the range and the dimension of the kernel of $T_f$.) Since $\mathcal{H}_f$ is compact and $M_f\in \Phi_+$, it then trivially follows from the compact perturbation theorem for semi-Fredholm operators that $T_f$ is $\Phi_+$.
\end{proof}

The above result provides criteria under which we are assured of the existence of Fredholm Toeplitz operators. However given that some group algebras admit no compact Hankel maps, it is incumbent on us to find criteria which do guarantee the existence of such maps.

A very satisfactory set of criteria may be found in the setting of antisymmetric finite maximal subdiagonal subalgebras. \emph{So for the remainder of this subsection we will assume that $\M$ is a finite von Neumann algebra equipped with a faithful normal tracial state $\tau$, and that $\A$ is an anti-symmetric maximal subdiagonal subalgebra.}

\begin{definition}\label{Gleason-part} Let $\rho$, $\omega$ be two states which are multiplicative on $\A$. We say that $\rho$ is in the Gleason part of $\omega$ if $\|(\omega-\rho)|_{\A}\|<2$. We denote such membership by $\rho\in G(\omega)$.
\end{definition}

In the context of anti-symmetric weak*-Dirichlet algebras, Curto, Muhly, Nakazi and Xia \cite{CMNX} showed that the Gleason part of $\tau$ is trivial, that is  $G(\tau)=\{\tau\}$, if and only if the only compact Hankel map is the zero operator. For the commutative case this result characterises those algebras which allow non-trivial compact Hankel maps. Using some of their ideas, we will show that if indeed the Gleason part is non-trivial, then in the noncommutative setting that will in certain cases also guarantee the existence of non-trivial compact Hankel maps.

In the case of weak*-Dirichlet algebras Curto, Muhly, Nakazi and Xia \cite[Theorem 2]{CMNX} show that if $G(\tau)\neq \{\tau\}$, there exists an outer function $z$ (the so-called Wermer embedding function) for which we have that $zH^2=H^2_0$. They then go on to show that given the existence of such a function, a Hankel map $\mathcal{H}_f$ (where $f\in L^\infty$) will be compact if and only if $f$ belongs to the algebra generated by $\A$ and $z^*$. 

The issue of Gleason parts for finite maximal subdiagonal subalgebras was considered in some detail in \cite{BL-Gleason}. The crucial fact for us is that the theory developed there may be used provide criteria under which Wermer embedding functions exist in even the noncommutative context. Given the existence of such a map, we may then follow \cite{CMNX} by showing that compactness of a Hankel map $\mathcal{H}_f$ (where $f\in \M$) is controlled by membership of $f$ to the closed subalgebra of $\M$ generated by $\A$ and $z^{-1}$. 

We start our analysis by stating a result result which forms the foundation of the classical proof of the existence of the Wermer Embedding function. We point out that in their analysis of Gleason parts for subdiagonal algebras, Blecher and Labuschagne did not restrict themselves to the anti-symmetric case. However the equivalence we need is only known to be true in the anti-symmetric case, and hence we will restrict to this case. 

\begin{theorem}[\cite{BL-Gleason}]\label{Gleason} Let $\A$ be an antisymmetric finite maximal subdiagonal subalgebra and let $\rho$, $\omega$ be two states which are multiplicative on $\A$. Then the following are equivalent:
\begin{itemize} 
\item $\rho\in G(\omega)$ (that is $\|(\omega-\rho)|_{\A}\|<2$);
\item $\|\omega{\upharpoonright}\A_\rho\|<1$ where $\A_\rho=\{a\in \A: \rho(a)=0\}$;
\item there are constants $c, d>0$ such that $c\rho\leq\omega$ and $d\omega\leq\rho$.
\end{itemize}
\end{theorem}

The above equivalence is the foundation on which the proof of the following crucial theorem (proved in \cite{BL-Gleason}) is built. Given the importance of this result for the present endeavour, we state the proof in full. 
 
\begin{theorem}[Wermer Embedding function \cite{BL-Gleason}] Let $\A$ be as before with $\omega$ a normal state in the Gleason part of $\tau$, distinct from $\tau$. Then there exists an element $z_r\in \A_0$ which is invertible in $\M$ such that $H^2(\A)z_r=H^2_0(\A)$. This element is of the form $h^{1/2}v_rh^{-1/2}$ for some unitary $v_r\in \M$ where $h\in \M_+^{-1}$ is the density for which $\omega=\tau(h\cdot)$. If in fact $\omega$ is also tracial, we have that $v$ commutes with $h$ and hence that $v_r=z_r$. There similarly also exists an element $z_l\in \A_0$ which is invertible in $\M$ such that $z_lH^2(\A)=H^2_0(\A)$. This element is of the form $h^{-1/2}v_lh^{1/2}$ for some unitary $v_l\in \M$. As before if $\omega$ is tracial, we have that $v_l=z_l$.
\end{theorem}

\begin{proof}  We prove the existence of the element $z_r$. This proof will clearly also suffice to establish the existence of an element $w_r\in \A_0^*$ of the form $h^{1/2}u_rh^{-1/2}$ for some unitary $u_r$, for which $H^2(\A^*)w_r=H^2_0(\A^*)$. The second claim then follows by simply setting $z_l=w_r^*$.  For ease of notation, we will therefore drop the subscripts in the proof, and simply write $z$ and $v$ for $z_r$ and $v_r$.

Suppose that $\omega\in G(\tau)$ with $\omega\neq\tau$. The completion of $\A$, $\A_0$ and $\M$ under the $L^2$-norm generated by $\omega$, will respectively be denoted by $H^2(\omega)$, $H^2_0(\omega)$ and $L^2(\omega)$. We know from the preceding theorem that there exist $\alpha, \beta >0$ such that for all $g\in \M^+$, $\alpha\tau(g)\leq \omega(g) \leq \beta\tau(g)$. This ensures that the spaces $H^2(\omega)$, $H^2_0(\omega)$ and $L^2(\omega)$ are effectively just equivalent renormings of $H^2(\A)$, $H^2_0(\A)$ and $L^2(\M)$. The space $L^1(\omega)$ is similarly an equivalent renorming of $L^1(\M)$. The action of the state $\omega$ admits a natural extension to the space $L^1(\omega)$, which we will still denote by $\omega$. It is an exercise to see that for this extension we have that $\omega(b^*a)=\langle a,b\rangle_\omega$ for all $a,b\in L^2(\M)=L^2(\omega)$.

Now let $e\in H^2$ be the projection of $\I\in \A$ onto $H^2_0(\omega)$ with respect to the inner product $\langle\cdot,\cdot\rangle_\omega$ coming from $\omega$. So $e\in H^2_0(\omega)$, with $\I-e$ orthogonal to $H^2_0$ in $L^2(\omega)$.
Let $c^2=\langle e,e\rangle_\omega=\|e\|^2_\omega\neq 0$, for otherwise $\I$ will be orthogonal to $\A_0$ with respect to $\langle\cdot,\cdot\rangle_\omega$, which would in turn ensure that $\omega$ annihilates $\A_0$. But that would force $\omega= \tau$, which would contradict our assumption. Hence we may let $z=\frac{1}{c}e$. 

Let $f\in \A$ be given. Since then $fe\in H^2_0$, we have that $fe\perp_\omega(\I-e)$, and hence that $\omega(fe)=\omega(e^*fe)$. In particular for $f=\I$, we get $\omega(e)=c^2$. It is an exercise to see that the multiplicativity of $\omega$ on $\A$ ensures that $\omega(ab)=\omega(a)\omega(b)$ for all $a\in \A$, $b\in H^2$. From this it now follows that $$c^2\omega(f)=\omega(fe)=\omega(e^*fe) \quad\mbox{for all }f\in \A.$$

We proceed to show that $c^2\omega(a)=\omega(|e|^2a)$ for all $a\in \M$. To see this, firstly note that by construction, the functional $\gamma: \M\to \mathbb{C}: a\mapsto \frac{1}{c^2}\omega(e^*ae)$ is well-defined and positive on $\M$, and assumes the value 1 at $\I$. Hence it is a state. It is however a state which agrees with $\omega$ on $\A$. Therefore the claim follows by the noncommutative Gleason-Whitney theorem, namely part (4) of Theorem \ref{Co}.

Let $h\in L^1(\M)_+$ be the density for which $\omega=\tau(h\cdot)$. The fact that there exist $\alpha, \beta >0$ such that for all $g\in \M^+$, $\alpha\tau(g)\leq \omega(g) \leq \beta\tau(g)$, may alternatively be formulated as the claim that $\alpha\I \leq h \leq \beta\I$, or equivalently that $h\in \M_+^{-1}$ as claimed.

The fact that for every $f\in \M$ we have that $$\tau(hf)= \omega(f) =\frac{1}{c^2} \omega(e^*fe)=\frac{1}{c^2}\tau(he^*fe)=\frac{1}{c^2}\tau(ehe^*f),$$ensures that as affiliated operators of $\M$, $h=\frac{1}{c^2}ehe^*$. This may be reformulated as the claim that $\I=\frac{1}{c^2}|h^{1/2}e^*h^{-1/2}|^2$. Since $\M$ is finite, this in turn ensures that 
$v = \frac{1}{c}h^{-1/2}eh^{1/2}$ is a unitary element of $\M$. It follows that $z=\frac{1}{c}e$ is of the form $z=h^{1/2}vh^{-1/2}$. In view of the fact that $H^2_0(\A)=H^2_0(\omega)$, this description of $z$ moreover proves that $z\in H^2_0(\A)\cap \M=\A_0$.   

Now observe that if $\omega$ is actually tracial, that would ensure that for any $a,b\in \M$ we will have that $$\tau((ha)b)=\omega(ab)=\omega(ba)=\tau(hba)=\tau((ah)b).$$It follows that then $ha=ah$ for any $a\in \M$, in other words $h\eta\mathcal{Z}(\M)$. In this case we will therefore have that $z = \frac{1}{c}h^{-1/2}eh^{1/2}=\frac{1}{c}e$.

It remains to prove that $H^2(\A)z=H^2_0(\A)$, or equivalently that $H^2(\A)e=H^2_0(\A)$. Since $e\in \A_0$, it is clear that $H^2e\subset H^2_0$. Given that $e$ is an invertible element of $\M$, $H^2e$ must be a closed subspace of $L^2(\M)$. Let $g\in H^2_0\ominus_\omega H^2e$ be given. If we are able to show that we then necessarily have that $g=0$, it will follow that $H^2_0=H^2e$ as required. 

Since for any $f\in \A$ we have that $fg\in H^2_0$, we will then also have that $(\I-e)\perp fg$ with respect to the inner product $\langle\cdot,\cdot\rangle_\omega$. In other words for any $f\in \A$ we have that
$$0=\langle fg, (\I-e)\rangle_\omega=\omega(fg)-\omega(e^*fg).$$Next observe that $\omega(fg) =\omega(f)\omega(g)$ for any $f\in \A$. To see this select any sequence $\{a_n\}\subset \A$ converging to $g$ in the $L^2$-norm, and notice that we then have that 
 $$\omega(fg) =\lim_{n\to\infty}\omega(fa_n)=\lim_{n\to\infty}\omega(f).\omega(a_n)=\omega(f)\omega(g).$$ 
Therefore 
 $\omega(fg)=0$ for all $f\in \A_\omega=\{a\in \A: \omega(a)=0\}.$
When combined with the previously centered equation, this ensures that
 $$0=\omega(e^*fg)=\omega((f^*e)^*g)= \langle g,f^*e\rangle_\omega\quad\mbox{for all }f\in \A_\omega.$$

We have therefore shown that $g\perp (\A+\A_\omega^*)e$ with respect to the inner product $\langle\cdot,\cdot\rangle_\omega$. But 
 $$\A+\A_\omega^*=(\A_0+\mathbb{C}\I)+\A_\omega^*=\A_0+(\mathbb{C}\I+\A_\omega)^*=\A_0+\A^*,$$ 
and $\A_0+\A^*$
is known to be norm dense in 
$L^2(\M)$. The equivalence of the norms generated by $\tau$ and $\omega$, therefore ensures that $\A_0+\A^*$ is $\|\cdot\|_\omega$-dense in $L^2$. Since $e\in \M^{-1}$, $(\A_0+\A^*)e$ is similarly $\|\cdot\|_\omega$-dense in $L^2$. Hence $g$ is orthogonal to $L^2$ with respect to $\langle\cdot,\cdot\rangle_\omega$, ensuring that $\|g\|_\omega=0$ as required.
\end{proof}

\begin{lemma} Suppose there exists an element $z\in \A_0$ which is unitary in $\M$, such that $zH^2(\A)=H^2_0(\A)$. Then the left multiplication operators $\{M_{z^n}\}\subset B(L^2(\M))$ (respectively $\{M_{(z^*)^n}\}$) converge to 0 in the weak operator topology.
\end{lemma}

\begin{proof} We will prove the claim for the sequence $\{M_{z^n}\}$. The other claim is entirely analogous. 

First note that the hypothesis ensures that $(z^*)^mz^n=z^{n-m}$ belongs to either $H^1_0(\A)$, or $H^1_0(\A)^*$ whenever $n\neq m$. That means that $\langle z^n,z^m\rangle=\tau(z^{n-m})=0$ whenever $n\neq m$. Thus as a subset of $L^2(\M)$, $\{z^n\}$ is an ONS. Hence $z^n\to 0$ weakly in $L^2(\M)$. That in turn ensures that for any $f\in L^2(\M)$, $z^nf\to 0$ weakly in $L^1(\M)$. 

We claim that in fact for any $f\in L^2(\M)$, $\{z^nf\}$ will converge weakly to 0 in $L^2(\M)$. To see this note that since $\{z^n\}$ is a bounded subset of $\M$, $\{z^nf\}$ is a bounded subset of $L^2(\M)$, and hence relatively weakly compact. If we can show that 0 is the only cluster point of this set in $L^2(\M)$, that will suffice to prove the claim regarding weak convergence. Let $g_0$ be a cluster point of $\{z^nf\}$ in $L^2(\M)$. Then there exists a subnet $\{z^{n_\lambda}f\}$ converging weakly to $g_0$ in $L^2(\M)$. Since $\M$ admits a tracial state, $L^2(\M)$ contractively embeds into $L^1(\M)$. So as elements of $L^1(\M)$, $\{z^{n_\lambda}f\}$ converges weakly to $g_0$ in $L^1(\M)$. But since in $L^1(\M)$ the sequence  $\{z^{n_\lambda}f\}$ converges weakly to 0, the subnet $\{z^{n_\lambda}f\}$ must also converge weakly to 0 in $L^1(\M)$. In other words $g_0=0$. Thus the only cluster point of $\{z^nf\}$ in $L^2(\M)$, is 0. As noted earlier, this suffices to show that $\{z^nf\}$ converges weakly to 0 in $L^2(\M)$. In other words for any $f,g\in L^2(\M)$, $\langle M_{z^n}f,g\rangle=\langle z^nf,g\rangle\to 0$ as $n\to\infty$, which is what we needed to prove.
\end{proof}

We are now finally ready to establish existence criteria for compact Hankel maps. The result we present is a faithful non-commutative version of \cite[Theorem 2]{CMNX}. We closely follow the proof offered in \cite{CMNX}. 

\begin{proposition}
Let $\A$ be an antisymmetric finite maximal subdiagonal subalgebra and suppose that there exists an element $z\in \A_0$ (invertible in $\M$) such that $zH^2(\A)=H^2_0(\A)$. Given $f\in \M$, the Hankel map $\mathcal{H}_f$ will be compact if $f$ belongs to the norm closed subalgebra generated by $z^{-1}$ and $\A$. If indeed $z$ is unitary in $\M$, then whenever $\mathcal{H}_f$ is compact, $f$ will conversely necessarily belong to the norm closed subalgebra generated by $z^{-1}=z^*$ and $\A$ .
\end{proposition}

\begin{proof} Suppose there exists an element $z\in \A_0\cap \M^{-1}$ such that $zH^2(\A)=H^2_0(\A)$, and let $f$ belongs to the norm closed 
subalgebra generated by $z^{-1}$ and $\A$. We prove that $\mathcal{H}_f$ is then compact. Now let $E$ be a subspace of $L^2(\M)$ which contains $H^2(\A)$ and for which $\mathrm{dim}\frac{E}{H^2}<\infty$. We show that then $\mathrm{dim}\frac{z^{-1}E}{H^2}<\infty$ and that $\mathrm{dim}\frac{aE}{H^2}<\infty$ for any $a\in \A$. To see this observe that any such subspace $E$ is of the form 
$E=F\oplus H^2=F\oplus\mathbb{C}\I\oplus H^2_0$ where $F$ is finite dimensional. From the assumptions on $z$, it is clear that we then have that 
$z^{-1}(F\oplus H^2)=z^{-1}(F\oplus\mathbb{C}\I)+ H^2$ with $z^{-1}(F\oplus\mathbb{C}\I)$ obviously finite dimensional. In other words $$\mathrm{dim}\frac{F\oplus H^2}{H^2}<\infty\Rightarrow \mathrm{dim}\frac{z^{-1}(F\oplus H^2)}{H^2}<\infty$$as claimed. Using the fact that $aH^2(\A)\subset H^2(\A)$ for any $a\in \A$, it is a trivial exercise to see that for any such $a$, we similarly have 
$$\mathrm{dim}\frac{F\oplus H^2}{H^2}<\infty\Rightarrow \mathrm{dim}\frac{a(F\oplus H^2)+H^2}{H^2}<\infty.$$ 

Now let $b$ be a finite algebraic combination of powers of $z^{-1}$ and elements of $\A$. It follows from what we proved above that then $\mathrm{dim}\frac{b(H^2)+H^2}{H^2}<\infty$, and hence that $\mathcal{H}_b$ is a finite rank operator, and therefore compact. However any $f$ belonging to the norm closed subalgebra generated by $z^{-1}$ and $\A$, is the norm limit of such $b$'s. Since we trivially have that $ \|\mathcal{H}_f-\mathcal{H}_b\|=\|\mathcal{H}_{(f-b)}\|\leq \|f-b\|_\infty$, the map $\mathcal{H}_f$ will then be a norm limit of compact maps, and hence itself compact. 

It remains to prove the converse. To this end let $f\in \M$ be given such that $\mathcal{H}_f$ is compact. Observe that
\begin{eqnarray*}
\|\mathcal{H}_f\| &=& \sup\{\|(\I-P_+)(fa)\|: a\in H^2(\A), \|a\|_2\leq 1\}\\
&=& \sup\{|(\I-P_+)(fa), \,b\rangle|: a\in H^2(\A), b\in H^2_0(\A)^*, \|a\|_2\leq 1, \|b\|_2\leq 1\}\\
&=& \sup\{|\langle fa,\, b\rangle|: a\in H^2(\A), b\in H^2_0(\A)^*, \|a\|_2\leq 1, \|b\|_2\leq 1\}\\
&=& \sup\{|\tau(fab^*)|: a\in H^2(\A), b\in H^2_0(\A)^*, \|a\|_2\leq 1, \|b\|_2\leq 1\}\\
&=& \sup\{|\tau(fF)|: F\in H^1_0(\A), \tau(|F|)\leq 1\}.
\end{eqnarray*}
Here the last equality follows from the Noncommutative Riesz Factorisation theorem from \cite{Sai} and \cite{MW}. Next notice that the argument in the last eleven lines of the proof of \cite[Theorem 3.9]{LX} suffices to show that
$$\sup\{|\tau(fF)|: F\in H^1_0(\A), \tau(|F|)\leq 1\} = \inf\{\|f+a\|_\infty : a \in \A a\}.$$Combining these observations now leads to the fact that $\|\mathcal{H}_f\|=\inf\{\|f+a\|_\infty : a \in \A a\}$. This fact together with the Lemma, now provides us with all the technology we need for the proof of \cite[Lemma 2.3]{CMNX} to go through almost verbatim in the present setting.
\end{proof}

\end{document}